\documentclass{amsart}
\usepackage{amsmath,amssymb}
\usepackage{mathrsfs}
\usepackage{amscd}

\newtheorem{thm}{Theorem}
\newtheorem{lem}{Lemma}

\newtheorem{ex}{Example}
\newcommand{\relmiddle}[1]{\mathrel{}\middle#1\mathrel{}}
\numberwithin{equation}{section}
\numberwithin{figure}{section}
\numberwithin{df}{section}
\numberwithin{prp}{section}
\numberwithin{thm}{section}
\numberwithin{lem}{section}
\numberwithin{rem}{section}
\numberwithin{con}{section}
\numberwithin{ex}{section}
\begin{document}

\title[rigged configuration bijection]
{An Explicit Algorithm of Rigged Configuration Bijection for the Adjoint Crystal of Type $G_{2}^{(1)}$}

\author{Toya Hiroshima}
\address{OCAMI, Osaka City University,
3-3-138 Sugimoto, Sumiyoshi-ku, Osaka 558-8585, Japan}
\email{toya-hiroshima@outlook.jp}


\date{}

\begin{abstract}
We construct an explicit algorithm of the static-preserving bijection between the rigged configurations and the highest weight paths of the form $(B^{2,1})^{\otimes L}$ in the $G_{2}^{(1)}$ adjoint crystals.
\end{abstract}


\maketitle

\tableofcontents

\section{Introduction}

Kerov, Kirillov, and Reshetikhin introduced a combinatorial object, called a rigged configuration, through Bethe Ansatz analysis of the isotropic Heisenberg spin chain \cite{KKR86}.
They constructed a bijection between rigged configurations and the semistandard Young tableaux \cite{KKR86,KR88}, through which a bijection between rigged configurations and highest weight element of a tensor product of Killirov-Reshetikhin (KR) crystals or highest weight paths was formulated.
The rigged configuration possess a natural statistic and it coincides with the charge introduced by Lascoux and Sch$\ddot{\mathrm{u}}$tzenberger~\cite{LS78}.
On the other hand, the highest weight path carries the statistic called energy by Nakayashiki and Yamada~\cite{NY97}.
The bijection, which is called the rigged configuration bijection, is a bijection such that the charge of a rigged configuration is sent to the energy of the corresponding path.

The bijection of Kerov, Kirillov, and Reshetikhin is a bijection from rigged configurations to the paths of the form $(B^{1,1})^{\otimes L}$ in type $A_{n}^{(1)}$ where $B^{r,s}$ denotes the KR crystal.
Until now, this has been extended in more general setting and in various types~\cite{DS06,KSS02,OS12,OSS03a,OSS03b,OSS03c,OSS13,OSS+17,Sak14,Sch05,Scr16,Scr20,SS06a}.
For nonexceptional types, the generalization of the rigged configuration bijection has been established~\cite{OSS18}.
On the other hand, there remain unsolved problems in exception types.
This paper concerns one of these problems.
For the $G_{2}^{(1)}$ adjoint crystals, the explicit algorithm of the rigged configuration bijection $\Phi$ is not known as pointed out in ~\cite{Scr20} though the crystal structure is very simple (see Fig.~\ref{fig:graph}). 
In this paper, we construct a map $\Phi$ from rigged configurations to highest weight elements of $(B^{2,1})^{\otimes L}$ by executing a fundamental procedure $\delta_{\theta}$ repeatedly.
Our result provides an alternate but direct proof of $X=M$ conjecture of~\cite{HKO+99,HKO+02b} in our setting, which has been proved previously by Naoi~\cite{Nao12}.

This paper is organized as follows.
In Section 2, we give the necessary background on KR crystals and paths for $G_{2}^{(1)}$.
Rigged configurations and the bijection $\Phi$ are described in Section 3, where we state our main theorem (Theorem~\ref{th:main}).
In Section 4, we provide an explicit description of the bijection $\delta_{\theta}$ as well as its inverse $\Tilde{\delta}_{\theta}$ for $G_{2}^{(1)}$ adjoint crystals. 
Section 5 is devoted to the proof of Theorem~\ref{th:main}.
In Section 6, we explain some forbidden rules in the algorithm.

\section{Affine Algebra $G_{2}^{(1)}$ and the KR Crystal}

\subsection{Affine algebra $G_{2}^{(1)}$}
We consider in this paper the exceptional untwisted affine algebra $G_{2}^{(1)}$.
The Dynkin diagram is depicted in Figure~\ref{fig:Dynkin}.
We follow \cite{Kac90} for the labeling of the Dynkin nodes.
Let $I$ be the index set of the Dynkin nodes and let $\alpha_{i}$, $\alpha_{i}^{\vee}$, $\Lambda_{i}$ $(i\in I)$ be simple roots, simple coroots, fundamental weights, respectively.
Following the notation in \cite{Kac90} we denote the projection of $\Lambda_{i}$ onto the weight space of $G_{2}$ by $\Bar{\Lambda}_{i}$ ($i\in I_{0}:=I\backslash {0}$) and set $\Bar{P}=\bigoplus_{i\in I_{0}}\mathbb{Z}\Bar{\Lambda}_{i}$, $\Bar{P}^{+}=\bigoplus_{i\in I_{0}}\mathbb{Z}_{\geq 0}\Bar{\Lambda}_{i}$.
Let $(A_{i,j})_{i,j\in I}$ stand for the Cartan matrix for $G_{2}^{(1)}$.
The canonical pairing $\left<\;,\;\right> : P^{\vee}\times P \rightarrow \mathbb{Z}$ is given by $\left< \alpha_{i}^{\vee},\alpha_{j}\right>=A_{i,j}$.

\setlength{\unitlength}{12pt}

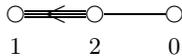
\begin{figure} \label{fig:Dynkin}
\begin{center}
\begin{picture}(5.5,2)

\put(0,1.5){\circle{0.5}}
\put(2.5,1.5){\circle{0.5}}
\put(5,1.5){\circle{0.5}}

\put(0.3,1.4){\line(1,0){1.9}}
\put(0.3,1.5){\line(1,0){1.9}}
\put(0.3,1.6){\line(1,0){1.9}}
\put(2.8,1.5){\line(1,0){1.9}}

\put(0.75,1){\makebox(1,1){\Large{$<$}}}

\put(-0.5,0){\makebox(1,1){\small{$1$}}}
\put(2,0){\makebox(1,1){\small{$2$}}}
\put(4.5,0){\makebox(1,1){\small{$0$}}}

\end{picture} 
\end{center}
\caption{The Dynkin diagram for $G_{2}^{(1)}$ in the Bourbaki labeling.}
\end{figure}

\subsection{KR crystal}
Let $\mathfrak{g}$ be any affine algebra and $U_{q}^{\prime}(\mathfrak{g})$ the corresponding quantized enveloping algebra without the degree operator.
Among finite-dimensional $U_{q}^{\prime}(\mathfrak{g})$-modules there is a distinguished family called Kirillov-Reshetikhin (KR) modules~\cite{Her10,KNT02,Nak03b}.
One of the remarkable properties of KR modules is that they admit crystal bases introduced by Kashiwara~\cite{Kas90,Kas91}.
Such a crystal base is called a KR crystal.
The KR crystal is indexed by $(a,i)\in\mathcal{H}_{0}:=I_{0}\times \mathbb{Z}_{>0}$ and denoted by $B^{a,i}$.
For exceptional types the KR crystal is known to exist when the KR module is irreducible or the index $a$ is adjacent to $0$~\cite{KKM+92b}.

The KR crystal we are interested in in this paper is $G_{2}^{(1)}$ adjoint crystal $B^{2,1}$, which is a level 1 perfect crystal and is constructed in \cite{BFKL06}.
The crystal structure is depicted in Figure~\ref{fig:graph}.
See also \cite{MMO10,MOW12} (note the difference indexing of the Dynkin diagram).
Here vertices in the graph signify elements of $B^{2,1}$ and $b\stackrel{i}{\longrightarrow}b^{\prime}$ stands for $f_{i}b=b^{\prime}$ or equivalently $b=e_{i}b^{\prime}$.
we adopt the anti-Kashiwara convention for the tensor products of crystals as in \cite{BS17}.
Namely, if $B_{1}$ and $B_{2}$ are crystals, then for $b_{1}\otimes b_{2}\in B_{1}\otimes B_{2}$ the action of $e_{i}$ is defined as

\[
e_{i}(b_{1}\otimes b_{2})=
\begin{cases}
e_{i}b_{1}\otimes b_{2} &  if \; \varphi_{i}(b_{2})<\varepsilon_{i}(b_{1}), \\
b_{1}\otimes e_{i}b_{2} &  if \; \varphi_{i}(b_{2})\geq\varepsilon_{i}(b_{1}),
\end{cases}
\]
where $\varepsilon_{i}(b)=\max\left\{  k\geq0 \relmiddle| e_{i}^{k}b\in B\right\}$ and 
$\varphi_{i}(b)=\max\left\{  k\geq0 \relmiddle| f_{i}^{k}b\in B\right\}$.

In what follow in this paper $B=B^{2,1}$.
The set of \emph{classically restricted paths} in $B^{\otimes L}$ of weight $\lambda \in \Bar{P}^{+}$ is by definition
\[
\mathcal{P}(\lambda,L)=\left\{ b\in B^{\otimes L} \relmiddle| \mathrm{wt}(b)=\lambda \text{ and } e_{i}b=0 \text{ for all } i\in I_{0}   \right\}.
\]
One may check that the following are equivalent for $b=b_{1}\otimes b_{2}\otimes \cdots \otimes b_{L}\in B^{\otimes L}$ and $\lambda \in \Bar{P}^{+}$.

\begin{itemize}
\item[(1)] 
$b$ is a classically restricted path of weight $\lambda \in \Bar{P}^{+}$.

\item[(2)] 
 $b_{2}\otimes \cdots \otimes b_{L}$ is a classically restricted path of weight $\lambda -\mathrm{wt}(b_{1})$ and $\varepsilon_{i}\leq \left<\alpha_{i}^{\vee},\lambda -\mathrm{wt}(b_{1})\right>$ for all $i\in I_{0}$.
\end{itemize}

The weight function $\mathrm{wt}:B\rightarrow \Bar{P}$ is given by 
$\mathrm{wt}(b)=\sum_{i\in I_{0}}(\varphi_{i}(b)-\varepsilon_{i}(b))\Bar{\Lambda}_{i}$.
The weight function $\mathrm{wt}:B^{\otimes L}\rightarrow \Bar{P}$ is defined by 
$\mathrm{wt}(b_{1}\otimes\cdots\otimes b_{L})=\sum_{j=1}^{L}\mathrm{wt}(b_{j})$.

In Figure~\ref{fig:graph}, we show the crystal subgraph for $B_{0}$, which is obtained by ignoring the $0$-arrow from the crystal graph for $B^{2,1}$.
All the $0$-arrows are listed below.

\setlength{\unitlength}{12pt}

\begin{center}
\begin{picture}(22,1.5)

\put(0,0){\framebox(1,1){\small{$10$}}}
\put(3,0){\framebox(1,1){\small{$2$}}}

\put(6,0){\framebox(1,1){\small{$11$}}}
\put(9,0){\framebox(1,1){\small{$3$}}}

\put(12,0){\framebox(1,1){\small{$12$}}}
\put(15,0){\framebox(1,1){\small{$4$}}}

\put(18,0){\framebox(1,1){\small{$13$}}}
\put(21,0){\framebox(1,1){\small{$6$}}}

\put(1,0,5){\vector(1,0){2}}
\put(7,0,5){\vector(1,0){2}}
\put(13,0,5){\vector(1,0){2}}
\put(19,0,5){\vector(1,0){2}}

\put(1.5,0.5){\makebox(1,1){\tiny{$0$}}}
\put(7.5,0.5){\makebox(1,1){\tiny{$0$}}}
\put(13.5,0.5){\makebox(1,1){\tiny{$0$}}}
\put(19.5,0.5){\makebox(1,1){\tiny{$0$}}}
\end{picture} 
\end{center}
and

\begin{center}
\begin{picture}(8,1.5)

\put(0,0){\framebox(1,1){\small{$14$}}}
\put(3,0){\framebox(1,1){\small{$\emptyset$}}}
\put(6,0){\framebox(1,1){\small{$1$}}}
\put(7,0){\makebox(1,1){.}}

\put(1,0,5){\vector(1,0){2}}
\put(4,0,5){\vector(1,0){2}}

\put(1.5,0.5){\makebox(1,1){\tiny{$0$}}}
\put(4.5,0.5){\makebox(1,1){\tiny{$0$}}}
\end{picture} 
\end{center}

\setlength{\unitlength}{12pt}

\begin{figure} \label{fig:graph}
\begin{center}
\begin{picture}(19,7)

\put(0,0){\framebox(1,1){\small{$3$}}}
\put(0,3){\framebox(1,1){\small{$2$}}}
\put(0,6){\framebox(1,1){\small{$1$}}}

\put(3,3){\framebox(1,1){\small{$4$}}}

\put(6,0){\framebox(1,1){\small{$6$}}}
\put(6,6){\framebox(1,1){\small{$5$}}}

\put(9,0){\framebox(1,1){\small{$8$}}}
\put(9,6){\framebox(1,1){\small{$7$}}}

\put(12,0){\framebox(1,1){\small{$10$}}}
\put(12,6){\framebox(1,1){\small{$9$}}}

\put(15,3){\framebox(1,1){\small{$11$}}}

\put(18,0){\framebox(1,1){\small{$14$}}}
\put(18,3){\framebox(1,1){\small{$13$}}}
\put(18,6){\framebox(1,1){\small{$12$}}}

\put(0.5,3){\vector(0,-1){2}}
\put(0.5,6){\vector(0,-1){2}}

\put(1,1){\vector(1,1){2}}

\put(4,3){\vector(1,-1){2}}
\put(4,4){\vector(1,1){2}}

\put(7,0.5){\vector(1,0){2}}
\put(7,6.5){\vector(1,0){2}}

\put(10,0.5){\vector(1,0){2}}
\put(10,6.5){\vector(1,0){2}}

\put(13,1){\vector(1,1){2}}
\put(13,6){\vector(1,-1){2}}

\put(16,4){\vector(1,1){2}}

\put(18.5,3){\vector(0,-1){2}}
\put(18.5,6){\vector(0,-1){2}}

\put(-0.5,1.5){\makebox(1,1){\tiny{$1$}}}
\put(-0.5,4.5){\makebox(1,1){\tiny{$2$}}}

\put(2,1.5){\makebox(1,1){\tiny{$1$}}}

\put(4,1.5){\makebox(1,1){\tiny{$1$}}}
\put(4,4.5){\makebox(1,1){\tiny{$2$}}}

\put(7.5,-0.5){\makebox(1,1){\tiny{$2$}}}
\put(7.5,6.5){\makebox(1,1){\tiny{$1$}}}

\put(10.5,-0.5){\makebox(1,1){\tiny{$2$}}}
\put(10.5,6.5){\makebox(1,1){\tiny{$1$}}}

\put(14,1.5){\makebox(1,1){\tiny{$1$}}}
\put(14,4.5){\makebox(1,1){\tiny{$2$}}}

\put(16,4.5){\makebox(1,1){\tiny{$1$}}}

\put(18.5,1.5){\makebox(1,1){\tiny{$2$}}}
\put(18.5,4.5){\makebox(1,1){\tiny{$1$}}}
\end{picture} 
\end{center}
\caption{Crystal graph of $B_{0}$ for $B^{2,1}$.}
\end{figure}
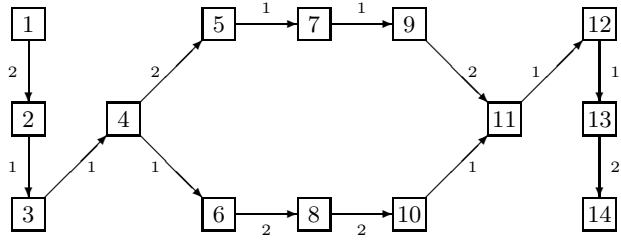

The \emph{energy function} $D: B^{\otimes L}\rightarrow \mathbb{Z}$ gives the grading on $B^{\otimes L}$.
In our case where a path is an element of the tensor product of a single KR crystal it takes a simple form.
Due to the existence of the universal $R$-matrix and the fact that $B\otimes B$ is connected, by \cite{KKM+92a} there is a unique (up to global additive constant) function $H: B\otimes B\rightarrow \mathbb{Z}$ called the \emph{local energy function}, such that
\[
H(e_{i}(b\otimes b^{\prime}))=
\begin{cases}
H(b\otimes b^{\prime})-1 & \text{ if } i=0 \text{ and } e_{0}(b\otimes b^{\prime})=e_{0}b \otimes b^{\prime}, \\
H(b\otimes b^{\prime})+1 & \text{ if } i=0 \text{ and } e_{0}(b\otimes b^{\prime})=b\otimes e_{0}b^{\prime}, \\
H(b\otimes b^{\prime}) & \text{otherwise}.
\end{cases}
\]
We normalize $H$ by the condition $H(\framebox{$1$}\otimes\framebox{$1$})=0$.

The crystal graph of $B_{0}\otimes B_{0}$ decomposes into five connected components as
\begin{equation} \label{eq:decomposition}
B_{0}\otimes B_{0}=B(2\Bar{\Lambda}_{2})\oplus
B(3\Bar{\Lambda}_{1})\oplus B(2\Bar{\Lambda}_{1})\oplus B(\Bar{\Lambda}_{2}) \oplus B(0),
\end{equation}
where $B(\lambda)$ stands for the highest weight $G_{2}$ crystal of weight $\lambda$ and the highest vector are given by $\framebox{$1$}\otimes\framebox{$1$}$ ($\lambda=2\Bar{\Lambda}_{2}$), $\framebox{$2$}\otimes\framebox{$1$}$ ($\lambda=3\Bar{\Lambda}_{1}$), $\framebox{$5$}\otimes\framebox{$1$}$ ($\lambda=2\Bar{\Lambda}_{1}$), $\framebox{$8$}\otimes\framebox{$1$} $ ($\lambda=\Bar{\Lambda}_{2}$), and $\framebox{$14$}\otimes\framebox{$1$}$ ($\lambda=0$).
$H$ is constant on each component, and takes the value $H(\framebox{$1$}\otimes\framebox{$1$})=0$, $H(\framebox{$2$}\otimes\framebox{$1$})=-1$, and $H=-2$ for the rest.
It is easily verified that $H(\framebox{$\emptyset$}\otimes\framebox{$\emptyset$})=-2$ and $H(\framebox{$\emptyset$}\otimes\framebox{$i$})=H(\framebox{$i$}\otimes\framebox{$\emptyset$})=-1$ $(1\leq i\leq 14)$.
With this $H$, the energy function $D$ is defined by
\begin{equation} \label{eq:D}
D(b_{1}\otimes \cdots \otimes b_{L})=\sum_{j=1}^{L}j H(b_{j}\otimes b_{j+1}),
\end{equation}
where $b_{L+1}=\framebox{$1$}$.

\section{Rigged Configuration and the Bijection}

\subsection{Rigged configuration}

Fix $\lambda\in \Bar{P}^{+}$ and a matrix $\boldsymbol{L}=(L_{i}^{(a)})_{(a,i)\in \mathcal{H}_{0}}$ of nonnegative integers.
Let $\nu=(m_{i}^{(a)})_{(a,i)\in\mathcal{H}_{0}}$.
Say that $\nu$ is a $\lambda$-configuration if it satisfies

\begin{equation} \label{eq:admissible}
\sum_{(a,i)\in\mathcal{H}_{0}}i m_{i}^{(a)}\alpha_{a}=\sum_{(a,i)\in\mathcal{H}_{0}}i L_{i}^{(a)}\Bar{\Lambda}_{a}-\lambda
\end{equation}
and $p_{i}^{(a)}\geq 0$ for all $(a,i)\in\mathcal{H}_{0}$, where
\begin{equation} \label{eq:vacancy}
p_{i}^{(a)}=\sum_{j\in \mathbb{Z}_{>0}} L_{j}^{(a)}\min (i,j)-\sum_{b\in I_{0}}\frac{A_{ab}}{\gamma_{b}}\sum_{j\in \mathbb{Z}_{>0}}\min (\gamma_{a}i,\gamma_{b}j)m_{j}^{(b)}
\end{equation}
is the \emph{vacancy numbers} and $\gamma_{a}$ ($a\in I_{0}$) are scaling factors which are given by $\gamma_{1}=1$ and $\gamma_{2}=3$ for $G_{2}$.

Let $\nu$ be an admissible configuration.
We identify $\nu$ with a sequence of partitions $\{ \nu^{(a)}\}_{a\in I_{0}}$ such that $\nu^{(a)}=(1^{m_{1}^{(a)}}2^{m_{2}^{(a)}}\cdots)$.
Let $J=\{ J^{(a,i)} \}_{(a,i)\in \mathcal{H}_{0}}$ be a double sequence of partitions.
Then a rigged configuration (RC) is a pair $(\nu,J)$ subject to the restriction Eq.~\eqref{eq:admissible} and the requirement that $J^{(a,i)}$ be a partition contained in a $m_{i}^{(a)}\times p_{i}^{(a)}$ rectangle.
The set of rigged configurations for fixed $\lambda$ and $\boldsymbol{L}$ is denoted by $\mathrm{RC}(\lambda,\boldsymbol{L})$.

Define the \emph{charge} of a configuration $\nu$ by
\begin{align*}
c(\nu)=&\frac{1}{2}\sum_{a,b\in I_{0}}\frac{A_{ab}}{\gamma_{b}}\sum_{i,j\in \mathbb{Z}_{>0}}\min (\gamma_{a}i,\gamma_{b}j)m_{i}^{(a)}m_{j}^{(b)} \\
&-\sum_{a\in I_{0},j,k\in \mathbb{Z}_{>0}}L_{j}^{(a)}\min (j,k)m_{k}^{(a)}.
\end{align*} 
To obtain the charge of a rigged configuration $(\nu, J)$, we add all of the riggings to $c(\nu)$;
\[
c(\nu,J)=c(\nu)+|J|,
\]
where $|J|=\sum_{(a,i)\in \mathcal{H}_{0}}|J^{(a,i)}|$.

We now set 
\begin{equation} \label{eq:restriction}
L_{i}^{(a)}=L\delta_{a,2}\delta_{i,1} \quad ((a,i)\in \mathcal{H}_{0}),
\end{equation}
which corresponds to considering paths in $(B^{2,1})^{\otimes L}$.
The set $\mathrm{RC}(\lambda,\boldsymbol{L})$ with the restriction Eq.~\eqref{eq:restriction} is denoted by $\mathrm{RC}(\lambda,L)$.
The vacancy numbers $p_{i}^{(a)}$ and the charge $c(\nu)$ take the following forms.

\begin{align}
p_{i}^{(1)}&=-2\sum_{j\in \mathbb{Z}_{>0}}\min (i,j)m_{j}^{(1)}+\sum_{j\in \mathbb{Z}_{>0}}\min (i,3j)m_{j}^{(2)}, \label{eq:vacancy1} \\
p_{i}^{(2)}&=L+\sum_{j\in \mathbb{Z}_{>0}}\min (3i,j)m_{j}^{(1)}-2\sum_{j\in \mathbb{Z}_{>0}}\min (i,j)m_{j}^{(2)} \label{eq:vacancy2},
\end{align} 
and
\begin{align} \label{eq:charge}
c(\nu)=&\sum_{i,j\in \mathbb{Z}_{>0}}\min (i,j)m_{i}^{(1)}m_{j}^{(1)}
-\sum_{i,j\in \mathbb{Z}_{>0}}\min (i,3j)m_{i}^{(1)}m_{j}^{(2)} \\
&+\sum_{i,j\in \mathbb{Z}_{>0}}\min (i,j)m_{i}^{(2)}m_{j}^{(2)} \nonumber
-L\sum_{j\in \mathbb{Z}_{>0}}m_{j}^{(2)}.
\end{align}

\subsection{The bijection from RCs to paths.}
We now describe the bijection $\Phi : \mathrm{RC}(\lambda,L)\rightarrow \mathcal{P}(\lambda,L)$.
Let $(\nu,J)\in \mathrm{RC}(\lambda,L)$.
We shall define a map $\gamma : \mathrm{RC}(\lambda,L)\rightarrow B$ which associates to $(\nu,J)$ an element of $B$.
Define by $\mathrm{RC}_{b}(\lambda,L)$ the elements of $\mathrm{RC}(\lambda,L)$ such that $\gamma (\nu,L)=b$.
We shall define a bijection $\delta_{\theta} : \mathrm{RC}_{b}(\lambda,L)\rightarrow \mathrm{RC}(\lambda-\mathrm{wt}(b),L-1)$.
The disjoint union of these bijections then defines a bijection $\delta_{\theta} : \mathrm{RC}(\lambda,L)\rightarrow \bigsqcup _{b\in B}\mathrm{RC}(\lambda-\mathrm{wt}(b),L-1)$.

The bijection $\Phi$ is defined recursively as follows.
For $b\in B$ let $\mathcal{P}_{b}(\lambda,L)$ be the set of paths in $B^{\otimes L}$ that have $b$ as leftmost tensor factor.
For $L=0$ the bijection $\Phi$ sends the empty rigged configuration (the only element of the set $\mathrm{RC}(\lambda,L)$) to the empty path (the only element of $\mathcal{P}(\lambda,L)$).
Otherwise assume that $\Phi$ has been defined for $B^{\otimes (L-1)}$ and define it for $B^{\otimes L}$ by the commutative diagram

\[
\begin{CD}
\mathrm{RC}_{b}(\lambda,L) @>{\Phi}>> \mathcal{P}_{b}(\lambda,L) \\
@V{\delta_{\theta}}VV @VVV \\
\mathrm{RC}(\lambda-\mathrm{wt}(b),L-1) @>{\Phi}>> \mathcal{P}(\lambda-\mathrm{wt}(b),L-1)
\end{CD}
\]
where the right-hand side vertical map removes the leftmost tensor factor $b$.

Here follows the main theorem of our paper.

\begin{thm} \label{th:main}
$\Phi : \mathrm{RC}(\lambda,L)\rightarrow \mathcal{P}(\lambda,L)$ is a bijection such that
\begin{equation} \label{eq:main}
c(\nu,J)=D(\Phi(\nu,J))\text{ for all }(\nu,J)\in \mathrm{RC}(\lambda,L).
\end{equation}
\end{thm}

\section{The Bijection}

In this section, for $(\nu,J)\in \mathrm{RC}(\lambda,L)$, an algorithm is given which defines $b=\gamma (\nu,J)$, the new smaller rigged configuration $(\Tilde{\nu},\Tilde{J})=\delta (\nu,J)$ such that $(\Tilde{\nu},\Tilde{J})\in \mathrm{RC}(\rho,L-1)$ where $\rho=\lambda -\mathrm{wt}(b)$, and the new vacancy numbers $\Tilde{p}_{i}^{(a)}$ in terms of the old $p_{i}^{(a)}$.
Before describing the algorithm, we give an example of $\Phi$, where we fix notation and give several definitions.

\begin{ex} \label{ex:def}
The algorithm $\Phi$ for an RC of $L=4$ and $\lambda=2\Bar{\Lambda}_{1}$ is depicted at each step $\delta_{\theta}$ below, where partitions $\nu^{(1)}$ and $\nu^{(2)}$ are illustrated as the left and the right Young diagrams in the RC.
A row in a Young diagram is called a \emph{string}.

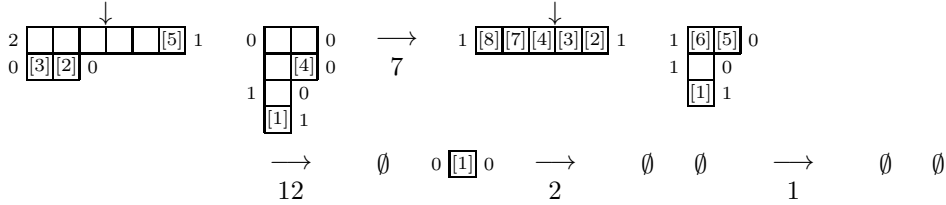
\setlength{\unitlength}{10pt}
\begin{flushleft}
\begin{picture}(29,5)
\put(1,2){\line(1,0){2}}
\put(1,3){\line(1,0){6}}
\put(1,4){\line(1,0){6}}
\put(1,2){\line(0,1){2}}
\put(2,2){\line(0,1){2}}
\put(3,2){\line(0,1){2}}
\put(4,3){\line(0,1){1}}
\put(5,3){\line(0,1){1}}
\put(6,3){\line(0,1){1}}
\put(7,3){\line(0,1){1}}

\put(3,4){\makebox(2,1){$\downarrow$}}
\put(0,2){\makebox(1,1){{\scriptsize $0$}}}
\put(0,3){\makebox(1,1){{\scriptsize $2$}}}
\put(3,2){\makebox(1,1){{\scriptsize $0$}}}
\put(7,3){\makebox(1,1){{\scriptsize $1$}}}
\put(1,2){\makebox(1,1){{\scriptsize $[3]$}}}
\put(2,2){\makebox(1,1){{\scriptsize $[2]$}}}
\put(6,3){\makebox(1,1){{\scriptsize $[5]$}}}

\put(10,0){\line(1,0){1}}
\put(10,1){\line(1,0){1}}
\put(10,2){\line(1,0){2}}
\put(10,3){\line(1,0){2}}
\put(10,4){\line(1,0){2}}

\put(10,0){\line(0,1){4}}
\put(11,0){\line(0,1){4}}
\put(12,2){\line(0,1){2}}

\put(9,1){\makebox(1,1){{\scriptsize $1$}}}
\put(9,3){\makebox(1,1){{\scriptsize $0$}}}
\put(11,0){\makebox(1,1){{\scriptsize $1$}}}
\put(11,1){\makebox(1,1){{\scriptsize $0$}}}
\put(12,2){\makebox(1,1){{\scriptsize $0$}}}
\put(12,3){\makebox(1,1){{\scriptsize $0$}}}

\put(10,0){\makebox(1,1){{\scriptsize $[1]$}}}
\put(11,2){\makebox(1,1){{\scriptsize $[4]$}}}

\put(14,3){\makebox(2,1){$\longrightarrow$}}
\put(14,2){\makebox(2,1){$7$}}

\put(18,3){\line(1,0){5}}
\put(18,4){\line(1,0){5}}
\put(18,3){\line(0,1){1}}
\put(19,3){\line(0,1){1}}
\put(20,3){\line(0,1){1}}
\put(21,3){\line(0,1){1}}
\put(22,3){\line(0,1){1}}
\put(23,3){\line(0,1){1}}

\put(20,4){\makebox(2,1){$\downarrow$}}
\put(17,3){\makebox(1,1){{\scriptsize $1$}}}
\put(23,3){\makebox(1,1){{\scriptsize $1$}}}
\put(18,3){\makebox(1,1){{\scriptsize $[8]$}}}
\put(19,3){\makebox(1,1){{\scriptsize $[7]$}}}
\put(20,3){\makebox(1,1){{\scriptsize $[4]$}}}
\put(21,3){\makebox(1,1){{\scriptsize $[3]$}}}
\put(22,3){\makebox(1,1){{\scriptsize $[2]$}}}

\put(26,1){\line(1,0){1}}
\put(26,2){\line(1,0){1}}
\put(26,3){\line(1,0){2}}
\put(26,4){\line(1,0){2}}
\put(26,1){\line(0,1){3}}
\put(27,1){\line(0,1){3}}
\put(28,3){\line(0,1){1}}
\put(25,2){\makebox(1,1){{\scriptsize $1$}}}
\put(25,3){\makebox(1,1){{\scriptsize $1$}}}
\put(27,1){\makebox(1,1){{\scriptsize $1$}}}
\put(27,2){\makebox(1,1){{\scriptsize $0$}}}
\put(28,3){\makebox(1,1){{\scriptsize $0$}}}
\put(26,1){\makebox(1,1){{\scriptsize $[1]$}}}
\put(26,3){\makebox(1,1){{\scriptsize $[6]$}}}
\put(27,3){\makebox(1,1){{\scriptsize $[5]$}}}
\end{picture}
\end{flushleft}

\setlength{\unitlength}{10pt}
\begin{flushright}
\begin{picture}(26,2)
\put(0,1){\makebox(2,1){$\longrightarrow$}}
\put(0,0){\makebox(2,1){$12$}}

\put(4,1){\makebox(1,1){{$\emptyset$}}}

\put(7,1){\line(1,0){1}}
\put(7,2){\line(1,0){1}}
\put(7,1){\line(0,1){1}}
\put(8,1){\line(0,1){1}}
\put(6,1){\makebox(1,1){{\scriptsize $0$}}}
\put(8,1){\makebox(1,1){{\scriptsize $0$}}}
\put(7,1){\makebox(1,1){{\scriptsize $[1]$}}}

\put(10,1){\makebox(2,1){$\longrightarrow$}}
\put(10,0){\makebox(2,1){$2$}}

\put(14,1){\makebox(1,1){{$\emptyset$}}}
\put(16,1){\makebox(1,1){{$\emptyset$}}}

\put(19,1){\makebox(2,1){$\longrightarrow$}}
\put(19,0){\makebox(2,1){$1$}}

\put(23,1){\makebox(1,1){{$\emptyset$}}}
\put(25,1){\makebox(1,1){{$\emptyset$}}}
\end{picture}
\end{flushright}
so that the first RC corresponds to the path 
$\framebox{$7$}\otimes \framebox{$12$}\otimes \framebox{$2$}\otimes \framebox{$1$}$.
The energy of this path is computed to be $-8$.
The computation uses Eq.~\eqref{eq:D} and Eq.~\eqref{eq:Hb}.
The charge of the first RC is computed to be $-8$ using Eq.~\eqref{eq:charge}.

\end{ex}

The down arrows marked in Example~\ref{ex:def} are called \emph{delimiters}, which are marked at each end of the consecutive sequence of three boxes of a string in $\nu^{(1)}$.
The numbers marked on the left of strings are vacancy numbers, which are common for the strings of the same length.
The numbers marked on the right of strings are riggings.
We follow the convention that riggings in strings of the same length are sorted in an increasing order (from the top).
We call a string of length $i$ in $\nu^{(a)}$ \emph{singular}, \emph{q-singular}, and \emph{qq-singular} if its rigging is equal to $p_{i}^{(a)}$, $p_{i}^{(a)}-1$, and $p_{i}^{(a)}-2$, respectively.
Here, ``q-singular'' is the abbreviation of ``quasi singular'' firstly introduced in \cite{Moh12b} and commonly used in the literature~\cite{OSS18,Scr20}.
The ``qq-singular'' strings are firstly introduced in this paper.
A singular or q-singular string is written as a singular/q-singular string.
The other cases are similar. 
We say that the string which is not singular is \emph{q-singular at best}.
A string of \emph{qq-singular at best} is defined similarly.
A string which is qq-singular at best but not qq-singular is called \emph{qqq-singular at best}.
The boxes marked by $[n]$ are deleted at each step $\delta_{\theta}$.
The number $n$ indicates the order of marking process.
We call this process the \emph{box marking} and say that the box or the string is marked by $[n]$. 
We also say that $[n]$ is marked in a box or a string when the box or the string is marked by $[n]$.
The string marked by $[n]$ is called the \emph{selected $i_{n}$-string} or simply the \emph{$i_{n}$-string}.
The length of the $i_{n}$-string is also denoted by $i_{n}$.
Therefore, the $i_{5}$-string in $\nu^{(1)}$ in the first RC is also a 6-string.
In the first RC in Example~\ref{ex:def}, the $i_{2}$-string coincide with the $i_{3}$-string.
We write $i_{2}=i_{3}$ in this case.
The length of the $i_{5}$-string is larger than that of the $i_{3}$-string.
We write $i_{3}<i_{5}$ in this case.
This notation is generalized in an obvious manner.
Strings in $\nu^{(1)}$ are classified by their types.

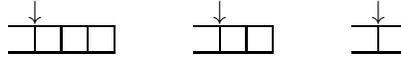
\setlength{\unitlength}{10pt}
\begin{center}
\begin{picture}(16,2)
\put(0,0){\line(1,0){4}}
\put(0,1){\line(1,0){4}}
\put(1,0){\line(0,1){1}}
\put(2,0){\line(0,1){1}}
\put(3,0){\line(0,1){1}}
\put(4,0){\line(0,1){1}}
\put(0,1){\makebox(2,1){$\downarrow$}}

\put(7,0){\line(1,0){3}}
\put(7,1){\line(1,0){3}}
\put(8,0){\line(0,1){1}}
\put(9,0){\line(0,1){1}}
\put(10,0){\line(0,1){1}}
\put(7,1){\makebox(2,1){$\downarrow$}}

\put(13,0){\line(1,0){2}}
\put(13,1){\line(1,0){2}}
\put(14,0){\line(0,1){1}}
\put(15,0){\line(0,1){1}}
\put(13,1){\makebox(2,1){$\downarrow$}}
\end{picture}
\end{center}
Strings of the first, the second, and the third types are called \emph{type-0}, \emph{type-I}, and \emph{type-II}, respectively.
A string of type-0 or type-I is written as a string of type-0/I.
The other cases are similar.
The \emph{effective length} of a string of length $i$ in $\nu^{(1)}$ is defined by $i^{eff}:=\lceil i/3 \rceil$.
For example, $i_{3}^{eff}=1$ and $i_{5}^{eff}=2$ in the first RC in Example~\ref{ex:def}.
The real length $i$ is recovered as $3i^{eff}$, $3i^{eff}-1$, or $3i^{eff}-2$ if the string is type-0, type-I, or type-II, respectively.
We form a new RC by removing marked boxes, adjusting the values of riggings of the box-deleted strings, and keeping the other the same.

\subsection{Algorithm $\delta_{\theta}$}

The algorithm consists of box marking and adjusting the values of riggings of the box-deleted strings.
We begin with the algorithm of box marking.

\begin{description}

\item[(BM-1)] Marking $[1]$.

Find the singular string of minimum length in $\nu^{(2)}$.
If such a string exists, then mark the rightmost box of the string by $[1]$.
If not, then terminate the algorithm and return $\framebox{$1$}$.
When $[1]$ is marked, we prescribe that boxes of length $(\leq 3(i_{1}-1))$ in $\nu^{(1)}$ cannot be marked.

\item[(BM-2)] Marking $[2]$.

Find the strings of effective length $i_{1}$ in $\nu^{(1)}$.
If such strings exist, then mark the rightmost box of one of such strings by [2] according to the following preferential rule.

\begin{center}
\begin{tabular}{c|ccc}
                & type-0 & type-I & type-II \\ \hline
singular     & 1 & 2 & 3 \\
q-singular & 4 & 5 & \\
qq-singular & 6 & & 
\end{tabular}
\end{center}
For example, firstly find a singular string of length $3i_{1}$.
If such a string exists, then mark the string by $[2]$.
If such strings do not exist, then find the singular string of length $3i_{1}-1$, etc.
If $[2]$ cannot be marked in a string of effective length $i_{1}$, then find the string of minimum length $(\geq 3i_{1}+1)$ which is singular/q-singular/qq-singular.
In this search, the type-II (resp. type-I) qq-singular string must be ignored if there exists a type-I (resp. type-0) singular string of the same effective length (see Example~\ref{ex:BM-2a}).
We prescribe that the type-II qq-singular string of length ($\geq 3i_{1}+1$) is ignored if there exists a type-0 singular string of the same effective length.
We also prescribe that the type-I q-singular (resp. qq-singular) string of length ($\geq 3i_{1}+1$) is ignored if there exists a type-0 singular (resp. q-singular) string of the same effective length (see Example~\ref{ex:BM-2b}).
If the search is successful, then mark the rightmost box of the founded string by $[2]$.
If not, then delete the marked box and return $\framebox{$2$}$.
If [2] is marked, then we prescribe that boxes of north and northwest of the box marked by [2] cannot be marked.
This rule is also applied in any box marking in $\nu^{(1)}$ (not in $\nu^{(2)}$).

\begin{ex} \label{ex:BM-2a}
$L=6$ and $\lambda=2\Bar{\Lambda}_{1}+\Bar{\Lambda}_{2}$.

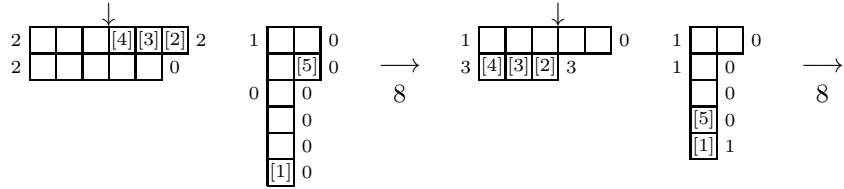
\setlength{\unitlength}{10pt}
\begin{center}
\begin{picture}(35,7)
\put(1,4){\line(1,0){5}}
\put(1,5){\line(1,0){6}}
\put(1,6){\line(1,0){6}}
\put(1,4){\line(0,1){2}}
\put(2,4){\line(0,1){2}}
\put(3,4){\line(0,1){2}}
\put(4,4){\line(0,1){2}}
\put(5,4){\line(0,1){2}}
\put(6,4){\line(0,1){2}}
\put(7,5){\line(0,1){1}}

\put(3,6){\makebox(2,1){$\downarrow$}}
\put(0,4){\makebox(1,1){{\scriptsize $2$}}}
\put(0,5){\makebox(1,1){{\scriptsize $2$}}}
\put(6,4){\makebox(1,1){{\scriptsize $0$}}}
\put(7,5){\makebox(1,1){{\scriptsize $2$}}}
\put(4,5){\makebox(1,1){{\scriptsize $[4]$}}}
\put(5,5){\makebox(1,1){{\scriptsize $[3]$}}}
\put(6,5){\makebox(1,1){{\scriptsize $[2]$}}}

\put(10,0){\line(1,0){1}}
\put(10,1){\line(1,0){1}}
\put(10,2){\line(1,0){1}}
\put(10,3){\line(1,0){1}}
\put(10,4){\line(1,0){2}}
\put(10,5){\line(1,0){2}}
\put(10,6){\line(1,0){2}}
\put(10,0){\line(0,1){6}}
\put(11,0){\line(0,1){6}}
\put(12,4){\line(0,1){2}}
\put(9,3){\makebox(1,1){{\scriptsize $0$}}}
\put(9,5){\makebox(1,1){{\scriptsize $1$}}}
\put(11,0){\makebox(1,1){{\scriptsize $0$}}}
\put(11,1){\makebox(1,1){{\scriptsize $0$}}}
\put(11,2){\makebox(1,1){{\scriptsize $0$}}}
\put(11,3){\makebox(1,1){{\scriptsize $0$}}}
\put(12,4){\makebox(1,1){{\scriptsize $0$}}}
\put(12,5){\makebox(1,1){{\scriptsize $0$}}}
\put(10,0){\makebox(1,1){{\scriptsize $[1]$}}}
\put(11,4){\makebox(1,1){{\scriptsize $[5]$}}}

\put(14,4){\makebox(2,1){$\longrightarrow$}}
\put(14,3){\makebox(2,1){$8$}}

\put(18,4){\line(1,0){3}}
\put(18,5){\line(1,0){5}}
\put(18,6){\line(1,0){5}}
\put(18,4){\line(0,1){2}}
\put(19,4){\line(0,1){2}}
\put(20,4){\line(0,1){2}}
\put(21,4){\line(0,1){2}}
\put(22,5){\line(0,1){1}}
\put(23,5){\line(0,1){1}}

\put(20,6){\makebox(2,1){$\downarrow$}}
\put(17,4){\makebox(1,1){{\scriptsize $3$}}}
\put(17,5){\makebox(1,1){{\scriptsize $1$}}}
\put(21,4){\makebox(1,1){{\scriptsize $3$}}}
\put(23,5){\makebox(1,1){{\scriptsize $0$}}}
\put(18,4){\makebox(1,1){{\scriptsize $[4]$}}}
\put(19,4){\makebox(1,1){{\scriptsize $[3]$}}}
\put(20,4){\makebox(1,1){{\scriptsize $[2]$}}}

\put(26,1){\line(1,0){1}}
\put(26,2){\line(1,0){1}}
\put(26,3){\line(1,0){1}}
\put(26,4){\line(1,0){1}}
\put(26,5){\line(1,0){2}}
\put(26,6){\line(1,0){2}}
\put(26,1){\line(0,1){5}}
\put(27,1){\line(0,1){5}}
\put(28,5){\line(0,1){1}}
\put(25,4){\makebox(1,1){{\scriptsize $1$}}}
\put(25,5){\makebox(1,1){{\scriptsize $1$}}}
\put(27,1){\makebox(1,1){{\scriptsize $1$}}}
\put(27,2){\makebox(1,1){{\scriptsize $0$}}}
\put(27,3){\makebox(1,1){{\scriptsize $0$}}}
\put(27,4){\makebox(1,1){{\scriptsize $0$}}}
\put(28,5){\makebox(1,1){{\scriptsize $0$}}}
\put(26,1){\makebox(1,1){{\scriptsize $[1]$}}}
\put(26,2){\makebox(1,1){{\scriptsize $[5]$}}}

\put(30,4){\makebox(2,1){$\longrightarrow$}}
\put(30,3){\makebox(2,1){$8$}}

\end{picture}
\end{center}
which corresponds to the path 
$\framebox{$8$}\otimes \framebox{$8$}\otimes \framebox{$1$}\otimes \framebox{$5$}\otimes \framebox{$8$}\otimes \framebox{$1$}$.
The charge of the first RC and the energy of this path coincide, which is computed to be $-19$.
If we selected the qq-singular string for the $i_{2}$-string in the first RC, then we would have

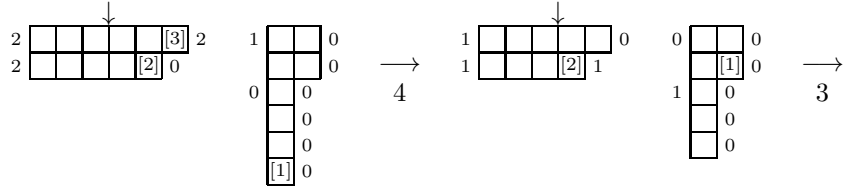
\setlength{\unitlength}{10pt}
\begin{center}
\begin{picture}(35,7)
\put(1,4){\line(1,0){5}}
\put(1,5){\line(1,0){6}}
\put(1,6){\line(1,0){6}}
\put(1,4){\line(0,1){2}}
\put(2,4){\line(0,1){2}}
\put(3,4){\line(0,1){2}}
\put(4,4){\line(0,1){2}}
\put(5,4){\line(0,1){2}}
\put(6,4){\line(0,1){2}}
\put(7,5){\line(0,1){1}}

\put(3,6){\makebox(2,1){$\downarrow$}}
\put(0,4){\makebox(1,1){{\scriptsize $2$}}}
\put(0,5){\makebox(1,1){{\scriptsize $2$}}}
\put(6,4){\makebox(1,1){{\scriptsize $0$}}}
\put(7,5){\makebox(1,1){{\scriptsize $2$}}}
\put(5,4){\makebox(1,1){{\scriptsize $[2]$}}}
\put(6,5){\makebox(1,1){{\scriptsize $[3]$}}}

\put(10,0){\line(1,0){1}}
\put(10,1){\line(1,0){1}}
\put(10,2){\line(1,0){1}}
\put(10,3){\line(1,0){1}}
\put(10,4){\line(1,0){2}}
\put(10,5){\line(1,0){2}}
\put(10,6){\line(1,0){2}}
\put(10,0){\line(0,1){6}}
\put(11,0){\line(0,1){6}}
\put(12,4){\line(0,1){2}}
\put(9,3){\makebox(1,1){{\scriptsize $0$}}}
\put(9,5){\makebox(1,1){{\scriptsize $1$}}}
\put(11,0){\makebox(1,1){{\scriptsize $0$}}}
\put(11,1){\makebox(1,1){{\scriptsize $0$}}}
\put(11,2){\makebox(1,1){{\scriptsize $0$}}}
\put(11,3){\makebox(1,1){{\scriptsize $0$}}}
\put(12,4){\makebox(1,1){{\scriptsize $0$}}}
\put(12,5){\makebox(1,1){{\scriptsize $0$}}}
\put(10,0){\makebox(1,1){{\scriptsize $[1]$}}}

\put(14,4){\makebox(2,1){$\longrightarrow$}}
\put(14,3){\makebox(2,1){$4$}}

\put(18,4){\line(1,0){4}}
\put(18,5){\line(1,0){5}}
\put(18,6){\line(1,0){5}}
\put(18,4){\line(0,1){2}}
\put(19,4){\line(0,1){2}}
\put(20,4){\line(0,1){2}}
\put(21,4){\line(0,1){2}}
\put(22,4){\line(0,1){2}}
\put(23,5){\line(0,1){1}}

\put(20,6){\makebox(2,1){$\downarrow$}}
\put(17,4){\makebox(1,1){{\scriptsize $1$}}}
\put(17,5){\makebox(1,1){{\scriptsize $1$}}}
\put(22,4){\makebox(1,1){{\scriptsize $1$}}}
\put(23,5){\makebox(1,1){{\scriptsize $0$}}}
\put(21,4){\makebox(1,1){{\scriptsize $[2]$}}}

\put(26,1){\line(1,0){1}}
\put(26,2){\line(1,0){1}}
\put(26,3){\line(1,0){1}}
\put(26,4){\line(1,0){2}}
\put(26,5){\line(1,0){2}}
\put(26,6){\line(1,0){2}}
\put(26,1){\line(0,1){5}}
\put(27,1){\line(0,1){5}}
\put(28,4){\line(0,1){2}}
\put(25,3){\makebox(1,1){{\scriptsize $1$}}}
\put(25,5){\makebox(1,1){{\scriptsize $0$}}}
\put(27,1){\makebox(1,1){{\scriptsize $0$}}}
\put(27,2){\makebox(1,1){{\scriptsize $0$}}}
\put(27,3){\makebox(1,1){{\scriptsize $0$}}}
\put(28,4){\makebox(1,1){{\scriptsize $0$}}}
\put(28,5){\makebox(1,1){{\scriptsize $0$}}}
\put(27,4){\makebox(1,1){{\scriptsize $[1]$}}}

\put(30,4){\makebox(2,1){$\longrightarrow$}}
\put(30,3){\makebox(2,1){$3$}}

\end{picture}
\end{center}
which corresponds to the path 
$\framebox{$4$}\otimes \framebox{$3$}\otimes \framebox{$\emptyset$}\otimes \framebox{$5$}\otimes \framebox{$8$}\otimes \framebox{$1$}$ 
whose energy is computed to be $-20$.
\end{ex}

\begin{ex} \label{ex:BM-2b}
$L=6$ and $\lambda=2\Bar{\Lambda}_{2}+\Bar{\Lambda}_{2}$ as in Example~\ref{ex:BM-2a}.

\setlength{\unitlength}{10pt}
\begin{center}
\begin{picture}(35,7)
\put(1,4){\line(1,0){5}}
\put(1,5){\line(1,0){6}}
\put(1,6){\line(1,0){6}}
\put(1,4){\line(0,1){2}}
\put(2,4){\line(0,1){2}}
\put(3,4){\line(0,1){2}}
\put(4,4){\line(0,1){2}}
\put(5,4){\line(0,1){2}}
\put(6,4){\line(0,1){2}}
\put(7,5){\line(0,1){1}}

\put(3,6){\makebox(2,1){$\downarrow$}}
\put(0,4){\makebox(1,1){{\scriptsize $2$}}}
\put(0,5){\makebox(1,1){{\scriptsize $2$}}}
\put(6,4){\makebox(1,1){{\scriptsize $0$}}}
\put(7,5){\makebox(1,1){{\scriptsize $1$}}}
\put(5,4){\makebox(1,1){{\scriptsize $[2]$}}}
\put(6,5){\makebox(1,1){{\scriptsize $[3]$}}}

\put(10,0){\line(1,0){1}}
\put(10,1){\line(1,0){1}}
\put(10,2){\line(1,0){1}}
\put(10,3){\line(1,0){1}}
\put(10,4){\line(1,0){2}}
\put(10,5){\line(1,0){2}}
\put(10,6){\line(1,0){2}}
\put(10,0){\line(0,1){6}}
\put(11,0){\line(0,1){6}}
\put(12,4){\line(0,1){2}}
\put(9,3){\makebox(1,1){{\scriptsize $0$}}}
\put(9,5){\makebox(1,1){{\scriptsize $1$}}}
\put(11,0){\makebox(1,1){{\scriptsize $1$}}}
\put(11,1){\makebox(1,1){{\scriptsize $0$}}}
\put(11,2){\makebox(1,1){{\scriptsize $0$}}}
\put(11,3){\makebox(1,1){{\scriptsize $0$}}}
\put(12,4){\makebox(1,1){{\scriptsize $1$}}}
\put(12,5){\makebox(1,1){{\scriptsize $0$}}}
\put(10,0){\makebox(1,1){{\scriptsize $[1]$}}}
\put(11,4){\makebox(1,1){{\scriptsize $[4]$}}}

\put(14,4){\makebox(2,1){$\longrightarrow$}}
\put(14,3){\makebox(2,1){$5$}}

\put(18,4){\line(1,0){4}}
\put(18,5){\line(1,0){5}}
\put(18,6){\line(1,0){5}}
\put(18,4){\line(0,1){2}}
\put(19,4){\line(0,1){2}}
\put(20,4){\line(0,1){2}}
\put(21,4){\line(0,1){2}}
\put(22,4){\line(0,1){2}}
\put(23,5){\line(0,1){1}}

\put(20,6){\makebox(2,1){$\downarrow$}}
\put(17,4){\makebox(1,1){{\scriptsize $-1$}}}
\put(17,5){\makebox(1,1){{\scriptsize $0$}}}

\put(26,1){\line(1,0){1}}
\put(26,2){\line(1,0){1}}
\put(26,3){\line(1,0){1}}
\put(26,4){\line(1,0){1}}
\put(26,5){\line(1,0){2}}
\put(26,6){\line(1,0){2}}
\put(26,1){\line(0,1){5}}
\put(27,1){\line(0,1){5}}
\put(28,5){\line(0,1){1}}
\put(25,4){\makebox(1,1){{\scriptsize $1$}}}
\put(25,5){\makebox(1,1){{\scriptsize $2$}}}

\end{picture}
\end{center}
The second RC is not admissible.

\end{ex}

\item[(BM-3)] Marking $[3]$.

\begin{itemize}
\item[(1)] $i_{2}=3i_{1}-2$.

Find the q-singular string of length $3i_{1}$ ignoring q-singular strings of length $3i_{1}-1$ even if they exist.
If such s string of length $3i_{1}$ exists, then mark the rightmost box of the founded string by $[3]$.
If such a string does not exist, then find the string (singular or q-singular) of length $(\geq 3i_{1}+1)$.
If such a string exists, then mark the rightmost box of the founded string by $[3]$.
If not, then delete the marked boxes and return $\framebox{$3$}$.

\item[(2)] $i_{2}=3i_{1}-1$.

\begin{itemize}
\item[(a)] The selected $i_{2}$-string is singular.

Mark the box on the left of $[2]$ by $[3]$.

\item[(b)] The selected $i_{2}$-string is q-singular.

Find the singular/q-singular string of minimum length $(\geq 3i_{1}+1)$.
If such a string exists, then mark the rightmost box of the founded string by $[3]$.
If not, then delete the marked boxes and return $\framebox{$3$}$.

\end{itemize}

\item[(3)] $i_{2}=3i_{1}$.

\begin{itemize}
\item[(a)] The selected $i_{2}$-string is q-singular.

Mark the box on the left of $[2]$ by $[3]$.

\item[(b)] The selected $i_{2}$-string is qq-singular.

Find the singular/q-singular string of minimum length $(\geq i_{2}+1)$.
If such a string exists, then mark the rightmost box of the founded string by [3].
If not, then delete the marked boxes and return $\framebox{$3$}$.

\end{itemize}

\item[(4)] $i_{2}\geq 3i_{1}+1$.

\begin{itemize}

\item[(a)] The selected $i_{2}$-string is singular/q-singular.

Mark the box on the left of [2] by [3].

\item[(b)] The selected $i_{2}$-string is qq-singular.

Follow the same as in case (3-b).

\end{itemize}

\end{itemize}
In cases (3) and (4), if the selected $i_{2}$-string is qq-singular of type-0 and there exists a singular string of length $i_{2}+1$, then discard the previous $i_{2}$-string selection and do the following box marking.

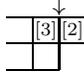
\setlength{\unitlength}{10pt}
\begin{center}
\begin{picture}(3,3)
\put(0,0){\line(1,0){2}}
\put(0,1){\line(1,0){3}}
\put(0,2){\line(1,0){3}}
\put(1,0){\line(0,1){2}}
\put(2,0){\line(0,1){2}}
\put(3,1){\line(0,1){1}}
\put(1,2){\makebox(2,1){$\downarrow$}}

\put(1,1){\makebox(1,1){{\scriptsize $[3]$}}}
\put(2,1){\makebox(1,1){{\scriptsize $[2]$}}}
\end{picture}
\end{center}

\begin{ex} \label{ex:BM-3}
$L=5$ and $\lambda=\Bar{\Lambda}_{1}+2\Bar{\Lambda}_{2}$.

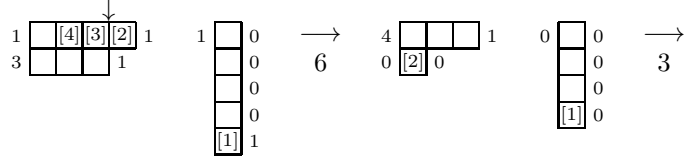
\setlength{\unitlength}{10pt}
\begin{center}
\begin{picture}(26,6)
\put(1,3){\line(1,0){3}}
\put(1,4){\line(1,0){4}}
\put(1,5){\line(1,0){4}}
\put(1,3){\line(0,1){2}}
\put(2,3){\line(0,1){2}}
\put(3,3){\line(0,1){2}}
\put(4,3){\line(0,1){2}}
\put(5,4){\line(0,1){1}}
\put(3,5){\makebox(2,1){$\downarrow$}}
\put(0,3){\makebox(1,1){{\scriptsize $3$}}}
\put(0,4){\makebox(1,1){{\scriptsize $1$}}}
\put(4,3){\makebox(1,1){{\scriptsize $1$}}}
\put(5,4){\makebox(1,1){{\scriptsize $1$}}}
\put(2,4){\makebox(1,1){{\scriptsize $[4]$}}}
\put(3,4){\makebox(1,1){{\scriptsize $[3]$}}}
\put(4,4){\makebox(1,1){{\scriptsize $[2]$}}}

\put(8,0){\line(1,0){1}}
\put(8,1){\line(1,0){1}}
\put(8,2){\line(1,0){1}}
\put(8,3){\line(1,0){1}}
\put(8,4){\line(1,0){1}}
\put(8,5){\line(1,0){1}}
\put(8,0){\line(0,1){5}}
\put(9,0){\line(0,1){5}}
\put(7,4){\makebox(1,1){{\scriptsize $1$}}}
\put(9,0){\makebox(1,1){{\scriptsize $1$}}}
\put(9,1){\makebox(1,1){{\scriptsize $0$}}}
\put(9,2){\makebox(1,1){{\scriptsize $0$}}}
\put(9,3){\makebox(1,1){{\scriptsize $0$}}}
\put(9,4){\makebox(1,1){{\scriptsize $0$}}}
\put(8,0){\makebox(1,1){{\scriptsize $[1]$}}}

\put(11,4){\makebox(2,1){$\longrightarrow$}}
\put(11,3){\makebox(2,1){$6$}}

\put(15,3){\line(1,0){1}}
\put(15,4){\line(1,0){3}}
\put(15,5){\line(1,0){3}}
\put(15,3){\line(0,1){2}}
\put(16,3){\line(0,1){2}}
\put(17,4){\line(0,1){1}}
\put(18,4){\line(0,1){1}}
\put(14,3){\makebox(1,1){{\scriptsize $0$}}}
\put(14,4){\makebox(1,1){{\scriptsize $4$}}}
\put(16,3){\makebox(1,1){{\scriptsize $0$}}}
\put(18,4){\makebox(1,1){{\scriptsize $1$}}}
\put(15,3){\makebox(1,1){{\scriptsize $[2]$}}}

\put(21,1){\line(1,0){1}}
\put(21,2){\line(1,0){1}}
\put(21,3){\line(1,0){1}}
\put(21,4){\line(1,0){1}}
\put(21,5){\line(1,0){1}}
\put(21,1){\line(0,1){4}}
\put(22,1){\line(0,1){4}}
\put(20,4){\makebox(1,1){{\scriptsize $0$}}}
\put(22,1){\makebox(1,1){{\scriptsize $0$}}}
\put(22,2){\makebox(1,1){{\scriptsize $0$}}}
\put(22,3){\makebox(1,1){{\scriptsize $0$}}}
\put(22,4){\makebox(1,1){{\scriptsize $0$}}}
\put(21,1){\makebox(1,1){{\scriptsize $[1]$}}}

\put(24,4){\makebox(2,1){$\longrightarrow$}}
\put(24,3){\makebox(2,1){$3$}}
\end{picture}
\end{center}
which corresponds to the path 
$\framebox{$6$}\otimes \framebox{$3$}\otimes \framebox{$3$}\otimes \framebox{$5$}\otimes \framebox{$1$}$.
The charge of the first RC and the energy of this path coincide, which is computed to be $-14$.
If we overrode the above rule, then we would have

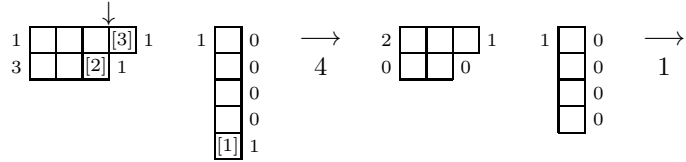
\setlength{\unitlength}{10pt}
\begin{center}
\begin{picture}(26,6)
\put(1,3){\line(1,0){3}}
\put(1,4){\line(1,0){4}}
\put(1,5){\line(1,0){4}}
\put(1,3){\line(0,1){2}}
\put(2,3){\line(0,1){2}}
\put(3,3){\line(0,1){2}}
\put(4,3){\line(0,1){2}}
\put(5,4){\line(0,1){1}}
\put(3,5){\makebox(2,1){$\downarrow$}}
\put(0,3){\makebox(1,1){{\scriptsize $3$}}}
\put(0,4){\makebox(1,1){{\scriptsize $1$}}}
\put(4,3){\makebox(1,1){{\scriptsize $1$}}}
\put(5,4){\makebox(1,1){{\scriptsize $1$}}}
\put(3,3){\makebox(1,1){{\scriptsize $[2]$}}}
\put(4,4){\makebox(1,1){{\scriptsize $[3]$}}}

\put(8,0){\line(1,0){1}}
\put(8,1){\line(1,0){1}}
\put(8,2){\line(1,0){1}}
\put(8,3){\line(1,0){1}}
\put(8,4){\line(1,0){1}}
\put(8,5){\line(1,0){1}}
\put(8,0){\line(0,1){5}}
\put(9,0){\line(0,1){5}}
\put(7,4){\makebox(1,1){{\scriptsize $1$}}}
\put(9,0){\makebox(1,1){{\scriptsize $1$}}}
\put(9,1){\makebox(1,1){{\scriptsize $0$}}}
\put(9,2){\makebox(1,1){{\scriptsize $0$}}}
\put(9,3){\makebox(1,1){{\scriptsize $0$}}}
\put(9,4){\makebox(1,1){{\scriptsize $0$}}}
\put(8,0){\makebox(1,1){{\scriptsize $[1]$}}}

\put(11,4){\makebox(2,1){$\longrightarrow$}}
\put(11,3){\makebox(2,1){$4$}}

\put(15,3){\line(1,0){2}}
\put(15,4){\line(1,0){3}}
\put(15,5){\line(1,0){3}}
\put(15,3){\line(0,1){2}}
\put(16,3){\line(0,1){2}}
\put(17,3){\line(0,1){2}}
\put(18,4){\line(0,1){1}}
\put(14,3){\makebox(1,1){{\scriptsize $0$}}}
\put(14,4){\makebox(1,1){{\scriptsize $2$}}}
\put(17,3){\makebox(1,1){{\scriptsize $0$}}}
\put(18,4){\makebox(1,1){{\scriptsize $1$}}}

\put(21,1){\line(1,0){1}}
\put(21,2){\line(1,0){1}}
\put(21,3){\line(1,0){1}}
\put(21,4){\line(1,0){1}}
\put(21,5){\line(1,0){1}}
\put(21,1){\line(0,1){4}}
\put(22,1){\line(0,1){4}}
\put(20,4){\makebox(1,1){{\scriptsize $1$}}}
\put(22,1){\makebox(1,1){{\scriptsize $0$}}}
\put(22,2){\makebox(1,1){{\scriptsize $0$}}}
\put(22,3){\makebox(1,1){{\scriptsize $0$}}}
\put(22,4){\makebox(1,1){{\scriptsize $0$}}}

\put(24,4){\makebox(2,1){$\longrightarrow$}}
\put(24,3){\makebox(2,1){$1$}}

\end{picture}
\end{center}
which corresponds to the path 
$\framebox{$4$}\otimes \framebox{$1$}\otimes \framebox{$7$}\otimes \framebox{$5$}\otimes \framebox{$1$}$.
However, the energy of this path is computed to be $-15$. 

\end{ex}

\item[(BM-4)] Marking $[4]$.

\begin{itemize}

\item[(1)] The selected $i_{3}$-string is singular.

The $i_{3}$-string is possibly marked by [2].

\begin{itemize}
\item[(a)] The selected $i_{3}$-string is type-0.

Mark the box on the left of $[3]$ by $[4]$.

\item[(b)] The selected $i_{3}$-string is type-I.

If $i_{3}^{eff}>i_{1}$, then mark the box on the left of $[3]$ by $[4]$.

If $i_{3}^{eff}=i_{1}$, i.e. $i_{3}=3i_{1}-1$, then the $i_{3}$-string is marked by [2].
Find the singular string of minimum length $(\geq i_{3}^{eff})$ in $\nu^{(2)}$.
If such a string exists, then let $l_{4}^{(2)}$ be the length of the founded string.
If not, then set $l_{4}^{(2)}=\infty$.
Find the singular string in $\nu^{(1)}$ of minimum length $(\geq i_{3}+1)$.
Note that the singular strings of length $3i_{1}$ do not exists in this case. 
If such a string exists, then let $l_{4}^{(1)}$ be the effective length of the founded string.
If not, then set $l_{4}^{(1)}=\infty$.
If $l_{4}^{(1)}\geq l_{4}^{(2)}$, then mark the rightmost box of the founded singular string of length $l_{4}^{(2)}$ in $\nu^{(2)}$ by [4]. 
If $l_{4}^{(1)}< l_{4}^{(2)}$, then mark the rightmost box of the founded singular string of effective length $l_{4}^{(1)}$ in $\nu^{(1)}$ by $[4]$. 
If $l_{4}^{(1)}=l_{4}^{(2)}=\infty$, then delete the marked boxes and return $\framebox{$4$}$.

\item[(c)] The selected $i_{3}$-string is type-II.

Find the singular string in $\nu^{(2)}$ of length $i_{3}^{eff}-1$ (the effective length reduction).
If such a string exists, then mark the rightmost box of the founded string by $[4]$ (see Examples~\ref{ex:BM-4a} and \ref{ex:BM-4b}).
If not, then find the singular string in $\nu^{(2)}$ of length $i_{3}^{eff}$.
If such a string exists, then mark the rightmost box of the founded string by $[4]$.
If not, then mark the box on the left of $[3]$ by $[4]$. 

\end{itemize}

\item[(2)] The selected $i_{3}$-string is q-singular.

Follow the same as in (1-b).

If the selected $i_{4}$-string is type-0 and the selected $i_{3}$-string is type-I with the same effective length, then discard the $i_{3}$-string selection and do the following box marking.

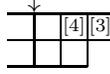
\setlength{\unitlength}{10pt}
\begin{center}
\begin{picture}(5,3)
\put(0,0){\line(1,0){3}}
\put(0,1){\line(1,0){4}}
\put(0,2){\line(1,0){4}}
\put(1,0){\line(0,1){2}}
\put(2,0){\line(0,1){2}}
\put(3,0){\line(0,1){2}}
\put(4,1){\line(0,1){1}}
\put(0,2){\makebox(2,1){$\downarrow$}}
\put(2,1){\makebox(1,1){{\scriptsize $[4]$}}}
\put(3,1){\makebox(1,1){{\scriptsize $[3]$}}}
\end{picture}
\end{center}
This rule guarantees the bijection when the q-singular $i_{5}$-string is in $\nu^{(2)}$ with $i_{4}^{eff}=i_{5}$ and when the singular $i_{5}$-string is in $\nu^{(2)}$ with $i_{4}^{eff}=i_{5}-1$ or $i_{4}^{eff}=i_{5}$ (see Examples~\ref{ex:BM-4c} and \ref{ex:BM-4d}). 

\end{itemize}

\begin{ex} \label{ex:BM-4a}
$L=3$ and $\lambda=\Bar{\Lambda}_{1}+\Bar{\Lambda}_{2}$.
The only admissible configuration $\nu$ is

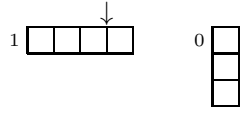
\setlength{\unitlength}{10pt}
\begin{center}
\begin{picture}(9,4)
\put(1,2){\line(1,0){4}}
\put(1,3){\line(1,0){4}}
\put(1,2){\line(0,1){1}}
\put(2,2){\line(0,1){1}}
\put(3,2){\line(0,1){1}}
\put(4,2){\line(0,1){1}}
\put(5,2){\line(0,1){1}}
\put(3,3){\makebox(2,1){$\downarrow$}}
\put(0,2){\makebox(1,1){{\scriptsize $1$}}}

\put(8,0){\line(1,0){1}}
\put(8,1){\line(1,0){1}}
\put(8,2){\line(1,0){1}}
\put(8,3){\line(1,0){1}}
\put(8,0){\line(0,1){3}}
\put(9,0){\line(0,1){3}}
\put(7,2){\makebox(1,1){{\scriptsize $0$}}}
\end{picture}
\end{center}
which has the charge $c(\nu)=-5$.
The first two steps of $\Phi$ are depicted as

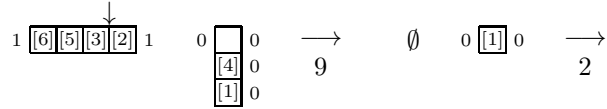
\setlength{\unitlength}{10pt}
\begin{center}
\begin{picture}(25,4)
\put(1,2){\line(1,0){4}}
\put(1,3){\line(1,0){4}}
\put(1,2){\line(0,1){1}}
\put(2,2){\line(0,1){1}}
\put(3,2){\line(0,1){1}}
\put(4,2){\line(0,1){1}}
\put(5,2){\line(0,1){1}}
\put(3,3){\makebox(2,1){$\downarrow$}}
\put(0,2){\makebox(1,1){{\scriptsize $1$}}}
\put(5,2){\makebox(1,1){{\scriptsize $1$}}}
\put(1,2){\makebox(1,1){{\scriptsize $[6]$}}}
\put(2,2){\makebox(1,1){{\scriptsize $[5]$}}}
\put(3,2){\makebox(1,1){{\scriptsize $[3]$}}}
\put(4,2){\makebox(1,1){{\scriptsize $[2]$}}}

\put(8,0){\line(1,0){1}}
\put(8,1){\line(1,0){1}}
\put(8,2){\line(1,0){1}}
\put(8,3){\line(1,0){1}}
\put(8,0){\line(0,1){3}}
\put(9,0){\line(0,1){3}}
\put(7,2){\makebox(1,1){{\scriptsize $0$}}}
\put(9,0){\makebox(1,1){{\scriptsize $0$}}}
\put(9,1){\makebox(1,1){{\scriptsize $0$}}}
\put(9,2){\makebox(1,1){{\scriptsize $0$}}}
\put(8,0){\makebox(1,1){{\scriptsize $[1]$}}}
\put(8,1){\makebox(1,1){{\scriptsize $[4]$}}}

\put(11,2){\makebox(2,1){$\longrightarrow$}}
\put(11,1){\makebox(2,1){$9$}}

\put(15,2){\makebox(1,1){{$\emptyset$}}}

\put(18,2){\line(1,0){1}}
\put(18,3){\line(1,0){1}}
\put(18,2){\line(0,1){1}}
\put(19,2){\line(0,1){1}}
\put(17,2){\makebox(1,1){{\scriptsize $0$}}}
\put(19,2){\makebox(1,1){{\scriptsize $0$}}}
\put(18,2){\makebox(1,1){{\scriptsize $[1]$}}}

\put(21,2){\makebox(2,1){$\longrightarrow$}}
\put(21,1){\makebox(2,1){$2$}}
\end{picture}
\end{center}
which corresponds to the path $\framebox{$9$}\otimes \framebox{$2$}\otimes \framebox{$1$}$ whose energy is $-4$ and

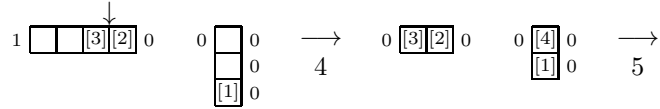
\setlength{\unitlength}{10pt}
\begin{center}
\begin{picture}(25,4)
\put(1,2){\line(1,0){4}}
\put(1,3){\line(1,0){4}}
\put(1,2){\line(0,1){1}}
\put(2,2){\line(0,1){1}}
\put(3,2){\line(0,1){1}}
\put(4,2){\line(0,1){1}}
\put(5,2){\line(0,1){1}}
\put(3,3){\makebox(2,1){$\downarrow$}}
\put(0,2){\makebox(1,1){{\scriptsize $1$}}}
\put(5,2){\makebox(1,1){{\scriptsize $0$}}}
\put(3,2){\makebox(1,1){{\scriptsize $[3]$}}}
\put(4,2){\makebox(1,1){{\scriptsize $[2]$}}}

\put(8,0){\line(1,0){1}}
\put(8,1){\line(1,0){1}}
\put(8,2){\line(1,0){1}}
\put(8,3){\line(1,0){1}}
\put(8,0){\line(0,1){3}}
\put(9,0){\line(0,1){3}}
\put(7,2){\makebox(1,1){{\scriptsize $0$}}}
\put(9,0){\makebox(1,1){{\scriptsize $0$}}}
\put(9,1){\makebox(1,1){{\scriptsize $0$}}}
\put(9,2){\makebox(1,1){{\scriptsize $0$}}}
\put(8,0){\makebox(1,1){{\scriptsize $[1]$}}}

\put(11,2){\makebox(2,1){$\longrightarrow$}}
\put(11,1){\makebox(2,1){$4$}}

\put(15,2){\line(1,0){2}}
\put(15,3){\line(1,0){2}}
\put(15,2){\line(0,1){1}}
\put(16,2){\line(0,1){1}}
\put(17,2){\line(0,1){1}}
\put(14,2){\makebox(1,1){{\scriptsize $0$}}}
\put(17,2){\makebox(1,1){{\scriptsize $0$}}}
\put(15,2){\makebox(1,1){{\scriptsize $[3]$}}}
\put(16,2){\makebox(1,1){{\scriptsize $[2]$}}}

\put(20,1){\line(1,0){1}}
\put(20,2){\line(1,0){1}}
\put(20,3){\line(1,0){1}}
\put(20,1){\line(0,1){2}}
\put(21,1){\line(0,1){2}}
\put(19,2){\makebox(1,1){{\scriptsize $0$}}}
\put(21,1){\makebox(1,1){{\scriptsize $0$}}}
\put(21,2){\makebox(1,1){{\scriptsize $0$}}}
\put(20,1){\makebox(1,1){{\scriptsize $[1]$}}}
\put(20,2){\makebox(1,1){{\scriptsize $[4]$}}}

\put(23,2){\makebox(2,1){$\longrightarrow$}}
\put(23,1){\makebox(2,1){$5$}}
\end{picture}
\end{center}
which corresponds to the path  $\framebox{$4$}\otimes \framebox{$5$}\otimes \framebox{$1$}$ whose energy is $-5$.
\end{ex}

\begin{ex} \label{ex:BM-4b}
$L=4$ and $\lambda=2\Bar{\Lambda}_{1}+\Bar{\Lambda}_{2}$.

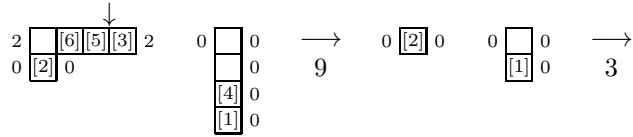
\setlength{\unitlength}{10pt}
\begin{center}
\begin{picture}(25,5)
\put(1,2){\line(1,0){1}}
\put(1,3){\line(1,0){4}}
\put(1,4){\line(1,0){4}}
\put(1,2){\line(0,1){2}}
\put(2,2){\line(0,1){2}}
\put(3,3){\line(0,1){1}}
\put(4,3){\line(0,1){1}}
\put(5,3){\line(0,1){1}}
\put(3,4){\makebox(2,1){$\downarrow$}}
\put(0,2){\makebox(1,1){{\scriptsize $0$}}}
\put(0,3){\makebox(1,1){{\scriptsize $2$}}}
\put(2,2){\makebox(1,1){{\scriptsize $0$}}}
\put(5,3){\makebox(1,1){{\scriptsize $2$}}}
\put(1,2){\makebox(1,1){{\scriptsize $[2]$}}}
\put(2,3){\makebox(1,1){{\scriptsize $[6]$}}}
\put(3,3){\makebox(1,1){{\scriptsize $[5]$}}}
\put(4,3){\makebox(1,1){{\scriptsize $[3]$}}}

\put(8,0){\line(1,0){1}}
\put(8,1){\line(1,0){1}}
\put(8,2){\line(1,0){1}}
\put(8,3){\line(1,0){1}}
\put(8,4){\line(1,0){1}}
\put(8,0){\line(0,1){4}}
\put(9,0){\line(0,1){4}}
\put(7,3){\makebox(1,1){{\scriptsize $0$}}}
\put(9,0){\makebox(1,1){{\scriptsize $0$}}}
\put(9,1){\makebox(1,1){{\scriptsize $0$}}}
\put(9,2){\makebox(1,1){{\scriptsize $0$}}}
\put(9,3){\makebox(1,1){{\scriptsize $0$}}}
\put(8,0){\makebox(1,1){{\scriptsize $[1]$}}}
\put(8,1){\makebox(1,1){{\scriptsize $[4]$}}}

\put(11,3){\makebox(2,1){$\longrightarrow$}}
\put(11,2){\makebox(2,1){$9$}}

\put(15,3){\line(1,0){1}}
\put(15,4){\line(1,0){1}}
\put(15,3){\line(0,1){1}}
\put(16,3){\line(0,1){1}}
\put(14,3){\makebox(1,1){{\scriptsize $0$}}}
\put(16,3){\makebox(1,1){{\scriptsize $0$}}}
\put(15,3){\makebox(1,1){{\scriptsize $[2]$}}}

\put(19,2){\line(1,0){1}}
\put(19,3){\line(1,0){1}}
\put(19,4){\line(1,0){1}}
\put(19,2){\line(0,1){2}}
\put(20,2){\line(0,1){2}}
\put(18,3){\makebox(1,1){{\scriptsize $0$}}}
\put(20,2){\makebox(1,1){{\scriptsize $0$}}}
\put(20,3){\makebox(1,1){{\scriptsize $0$}}}
\put(19,2){\makebox(1,1){{\scriptsize $[1]$}}}

\put(22,3){\makebox(2,1){$\longrightarrow$}}
\put(22,2){\makebox(2,1){$3$}}
\end{picture}
\end{center}
which corresponds to the path $\framebox{$9$}\otimes \framebox{$3$}\otimes \framebox{$2$}\otimes \framebox{$1$}$ whose energy is $-7$.
\end{ex}

\begin{ex} \label{ex:BM-4c}
$L=6$ and $\lambda=\Bar{\Lambda}_{1}+\Bar{\Lambda}_{2}$.

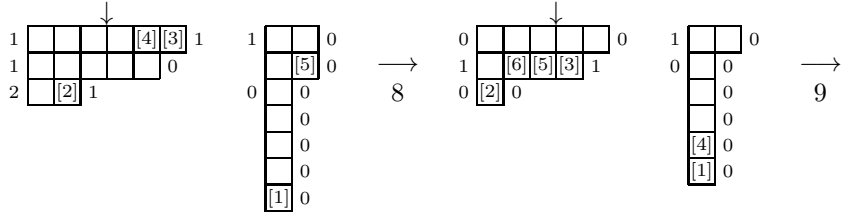
\setlength{\unitlength}{10pt}
\begin{center}
\begin{picture}(35,8)
\put(1,4){\line(1,0){2}}
\put(1,5){\line(1,0){5}}
\put(1,6){\line(1,0){6}}
\put(1,7){\line(1,0){6}}
\put(1,4){\line(0,1){3}}
\put(2,4){\line(0,1){3}}
\put(3,4){\line(0,1){3}}
\put(4,5){\line(0,1){2}}
\put(5,5){\line(0,1){2}}
\put(6,5){\line(0,1){2}}
\put(7,6){\line(0,1){1}}

\put(3,7){\makebox(2,1){$\downarrow$}}
\put(0,4){\makebox(1,1){{\scriptsize $2$}}}
\put(0,5){\makebox(1,1){{\scriptsize $1$}}}
\put(0,6){\makebox(1,1){{\scriptsize $1$}}}
\put(3,4){\makebox(1,1){{\scriptsize $1$}}}
\put(6,5){\makebox(1,1){{\scriptsize $0$}}}
\put(7,6){\makebox(1,1){{\scriptsize $1$}}}
\put(2,4){\makebox(1,1){{\scriptsize $[2]$}}}
\put(5,6){\makebox(1,1){{\scriptsize $[4]$}}}
\put(6,6){\makebox(1,1){{\scriptsize $[3]$}}}

\put(10,0){\line(1,0){1}}
\put(10,1){\line(1,0){1}}
\put(10,2){\line(1,0){1}}
\put(10,3){\line(1,0){1}}
\put(10,4){\line(1,0){1}}
\put(10,5){\line(1,0){2}}
\put(10,6){\line(1,0){2}}
\put(10,7){\line(1,0){2}}
\put(10,0){\line(0,1){7}}
\put(11,0){\line(0,1){7}}
\put(12,5){\line(0,1){2}}
\put(9,4){\makebox(1,1){{\scriptsize $0$}}}
\put(9,6){\makebox(1,1){{\scriptsize $1$}}}
\put(11,0){\makebox(1,1){{\scriptsize $0$}}}
\put(11,1){\makebox(1,1){{\scriptsize $0$}}}
\put(11,2){\makebox(1,1){{\scriptsize $0$}}}
\put(11,3){\makebox(1,1){{\scriptsize $0$}}}
\put(11,4){\makebox(1,1){{\scriptsize $0$}}}
\put(12,5){\makebox(1,1){{\scriptsize $0$}}}
\put(12,6){\makebox(1,1){{\scriptsize $0$}}}
\put(10,0){\makebox(1,1){{\scriptsize $[1]$}}}
\put(11,5){\makebox(1,1){{\scriptsize $[5]$}}}

\put(14,5){\makebox(2,1){$\longrightarrow$}}
\put(14,4){\makebox(2,1){$8$}}

\put(18,4){\line(1,0){1}}
\put(18,5){\line(1,0){4}}
\put(18,6){\line(1,0){5}}
\put(18,7){\line(1,0){5}}
\put(18,4){\line(0,1){3}}
\put(19,4){\line(0,1){3}}
\put(20,5){\line(0,1){2}}
\put(21,5){\line(0,1){2}}
\put(22,5){\line(0,1){2}}
\put(23,6){\line(0,1){1}}

\put(20,7){\makebox(2,1){$\downarrow$}}
\put(17,4){\makebox(1,1){{\scriptsize $0$}}}
\put(17,5){\makebox(1,1){{\scriptsize $1$}}}
\put(17,6){\makebox(1,1){{\scriptsize $0$}}}
\put(19,4){\makebox(1,1){{\scriptsize $0$}}}
\put(22,5){\makebox(1,1){{\scriptsize $1$}}}
\put(23,6){\makebox(1,1){{\scriptsize $0$}}}
\put(18,4){\makebox(1,1){{\scriptsize $[2]$}}}
\put(19,5){\makebox(1,1){{\scriptsize $[6]$}}}
\put(20,5){\makebox(1,1){{\scriptsize $[5]$}}}
\put(21,5){\makebox(1,1){{\scriptsize $[3]$}}}

\put(26,1){\line(1,0){1}}
\put(26,2){\line(1,0){1}}
\put(26,3){\line(1,0){1}}
\put(26,4){\line(1,0){1}}
\put(26,5){\line(1,0){1}}
\put(26,6){\line(1,0){2}}
\put(26,7){\line(1,0){2}}
\put(26,1){\line(0,1){6}}
\put(27,1){\line(0,1){6}}
\put(28,6){\line(0,1){1}}
\put(25,5){\makebox(1,1){{\scriptsize $0$}}}
\put(25,6){\makebox(1,1){{\scriptsize $1$}}}
\put(27,1){\makebox(1,1){{\scriptsize $0$}}}
\put(27,2){\makebox(1,1){{\scriptsize $0$}}}
\put(27,3){\makebox(1,1){{\scriptsize $0$}}}
\put(27,4){\makebox(1,1){{\scriptsize $0$}}}
\put(27,5){\makebox(1,1){{\scriptsize $0$}}}
\put(28,6){\makebox(1,1){{\scriptsize $0$}}}
\put(26,1){\makebox(1,1){{\scriptsize $[1]$}}}
\put(26,2){\makebox(1,1){{\scriptsize $[4]$}}}

\put(30,5){\makebox(2,1){$\longrightarrow$}}
\put(30,4){\makebox(2,1){$9$}}

\end{picture}
\end{center}
which corresponds to the path 
$\framebox{$8$}\otimes \framebox{$9$}\otimes \framebox{$3$}\otimes \framebox{$5$}\otimes \framebox{$8$}\otimes \framebox{$1$}$.
The charge of the first RC and the energy of this path coincide, which is computed to be $-22$. 
If we selected the q-singular string for the $i_{3}$-string in the first RC, then we would have

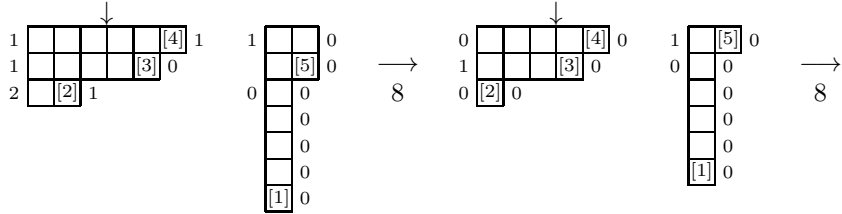
\setlength{\unitlength}{10pt}
\begin{center}
\begin{picture}(35,8)
\put(1,4){\line(1,0){2}}
\put(1,5){\line(1,0){5}}
\put(1,6){\line(1,0){6}}
\put(1,7){\line(1,0){6}}
\put(1,4){\line(0,1){3}}
\put(2,4){\line(0,1){3}}
\put(3,4){\line(0,1){3}}
\put(4,5){\line(0,1){2}}
\put(5,5){\line(0,1){2}}
\put(6,5){\line(0,1){2}}
\put(7,6){\line(0,1){1}}

\put(3,7){\makebox(2,1){$\downarrow$}}
\put(0,4){\makebox(1,1){{\scriptsize $2$}}}
\put(0,5){\makebox(1,1){{\scriptsize $1$}}}
\put(0,6){\makebox(1,1){{\scriptsize $1$}}}
\put(3,4){\makebox(1,1){{\scriptsize $1$}}}
\put(6,5){\makebox(1,1){{\scriptsize $0$}}}
\put(7,6){\makebox(1,1){{\scriptsize $1$}}}
\put(2,4){\makebox(1,1){{\scriptsize $[2]$}}}
\put(5,5){\makebox(1,1){{\scriptsize $[3]$}}}
\put(6,6){\makebox(1,1){{\scriptsize $[4]$}}}

\put(10,0){\line(1,0){1}}
\put(10,1){\line(1,0){1}}
\put(10,2){\line(1,0){1}}
\put(10,3){\line(1,0){1}}
\put(10,4){\line(1,0){1}}
\put(10,5){\line(1,0){2}}
\put(10,6){\line(1,0){2}}
\put(10,7){\line(1,0){2}}
\put(10,0){\line(0,1){7}}
\put(11,0){\line(0,1){7}}
\put(12,5){\line(0,1){2}}
\put(9,4){\makebox(1,1){{\scriptsize $0$}}}
\put(9,6){\makebox(1,1){{\scriptsize $1$}}}
\put(11,0){\makebox(1,1){{\scriptsize $0$}}}
\put(11,1){\makebox(1,1){{\scriptsize $0$}}}
\put(11,2){\makebox(1,1){{\scriptsize $0$}}}
\put(11,3){\makebox(1,1){{\scriptsize $0$}}}
\put(11,4){\makebox(1,1){{\scriptsize $0$}}}
\put(12,5){\makebox(1,1){{\scriptsize $0$}}}
\put(12,6){\makebox(1,1){{\scriptsize $0$}}}
\put(10,0){\makebox(1,1){{\scriptsize $[1]$}}}
\put(11,5){\makebox(1,1){{\scriptsize $[5]$}}}

\put(14,5){\makebox(2,1){$\longrightarrow$}}
\put(14,4){\makebox(2,1){$8$}}

\put(18,4){\line(1,0){1}}
\put(18,5){\line(1,0){4}}
\put(18,6){\line(1,0){5}}
\put(18,7){\line(1,0){5}}
\put(18,4){\line(0,1){3}}
\put(19,4){\line(0,1){3}}
\put(20,5){\line(0,1){2}}
\put(21,5){\line(0,1){2}}
\put(22,5){\line(0,1){2}}
\put(23,6){\line(0,1){1}}

\put(20,7){\makebox(2,1){$\downarrow$}}
\put(17,4){\makebox(1,1){{\scriptsize $0$}}}
\put(17,5){\makebox(1,1){{\scriptsize $1$}}}
\put(17,6){\makebox(1,1){{\scriptsize $0$}}}
\put(19,4){\makebox(1,1){{\scriptsize $0$}}}
\put(22,5){\makebox(1,1){{\scriptsize $0$}}}
\put(23,6){\makebox(1,1){{\scriptsize $0$}}}
\put(18,4){\makebox(1,1){{\scriptsize $[2]$}}}
\put(21,5){\makebox(1,1){{\scriptsize $[3]$}}}
\put(22,6){\makebox(1,1){{\scriptsize $[4]$}}}

\put(26,1){\line(1,0){1}}
\put(26,2){\line(1,0){1}}
\put(26,3){\line(1,0){1}}
\put(26,4){\line(1,0){1}}
\put(26,5){\line(1,0){1}}
\put(26,6){\line(1,0){2}}
\put(26,7){\line(1,0){2}}
\put(26,1){\line(0,1){6}}
\put(27,1){\line(0,1){6}}
\put(28,6){\line(0,1){1}}
\put(25,5){\makebox(1,1){{\scriptsize $0$}}}
\put(25,6){\makebox(1,1){{\scriptsize $1$}}}
\put(27,1){\makebox(1,1){{\scriptsize $0$}}}
\put(27,2){\makebox(1,1){{\scriptsize $0$}}}
\put(27,3){\makebox(1,1){{\scriptsize $0$}}}
\put(27,4){\makebox(1,1){{\scriptsize $0$}}}
\put(27,5){\makebox(1,1){{\scriptsize $0$}}}
\put(28,6){\makebox(1,1){{\scriptsize $0$}}}
\put(26,1){\makebox(1,1){{\scriptsize $[1]$}}}
\put(27,6){\makebox(1,1){{\scriptsize $[5]$}}}

\put(30,5){\makebox(2,1){$\longrightarrow$}}
\put(30,4){\makebox(2,1){$8$}}

\end{picture}
\end{center}
which corresponds to the path 
$\framebox{$8$}\otimes \framebox{$8$}\otimes \framebox{$7$}\otimes \framebox{$4$}\otimes \framebox{$5$}\otimes \framebox{$1$}$.
However, the energy of this path is computed to be $-23$.

\end{ex}

\begin{ex} \label{ex:BM-4d}
$L=7$ and $\lambda=3\Bar{\Lambda}_{1}+\Bar{\Lambda}_{2}$.

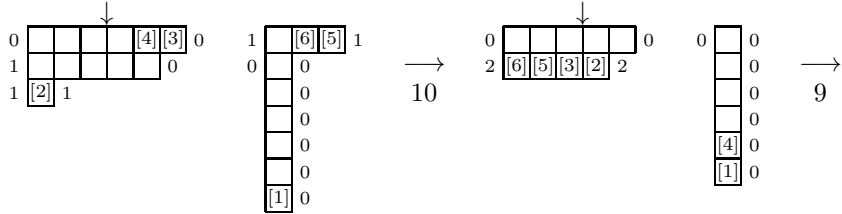
\setlength{\unitlength}{10pt}
\begin{center}
\begin{picture}(35,8)
\put(1,4){\line(1,0){1}}
\put(1,5){\line(1,0){5}}
\put(1,6){\line(1,0){6}}
\put(1,7){\line(1,0){6}}
\put(1,4){\line(0,1){3}}
\put(2,4){\line(0,1){3}}
\put(3,5){\line(0,1){2}}
\put(4,5){\line(0,1){2}}
\put(5,5){\line(0,1){2}}
\put(6,5){\line(0,1){2}}
\put(7,6){\line(0,1){1}}

\put(3,7){\makebox(2,1){$\downarrow$}}
\put(0,4){\makebox(1,1){{\scriptsize $1$}}}
\put(0,5){\makebox(1,1){{\scriptsize $1$}}}
\put(0,6){\makebox(1,1){{\scriptsize $0$}}}
\put(2,4){\makebox(1,1){{\scriptsize $1$}}}
\put(6,5){\makebox(1,1){{\scriptsize $0$}}}
\put(7,6){\makebox(1,1){{\scriptsize $0$}}}
\put(1,4){\makebox(1,1){{\scriptsize $[2]$}}}
\put(5,6){\makebox(1,1){{\scriptsize $[4]$}}}
\put(6,6){\makebox(1,1){{\scriptsize $[3]$}}}

\put(10,0){\line(1,0){1}}
\put(10,1){\line(1,0){1}}
\put(10,2){\line(1,0){1}}
\put(10,3){\line(1,0){1}}
\put(10,4){\line(1,0){1}}
\put(10,5){\line(1,0){1}}
\put(10,6){\line(1,0){3}}
\put(10,7){\line(1,0){3}}
\put(10,0){\line(0,1){7}}
\put(11,0){\line(0,1){7}}
\put(12,6){\line(0,1){1}}
\put(13,6){\line(0,1){1}}
\put(9,5){\makebox(1,1){{\scriptsize $0$}}}
\put(9,6){\makebox(1,1){{\scriptsize $1$}}}
\put(11,0){\makebox(1,1){{\scriptsize $0$}}}
\put(11,1){\makebox(1,1){{\scriptsize $0$}}}
\put(11,2){\makebox(1,1){{\scriptsize $0$}}}
\put(11,3){\makebox(1,1){{\scriptsize $0$}}}
\put(11,4){\makebox(1,1){{\scriptsize $0$}}}
\put(11,5){\makebox(1,1){{\scriptsize $0$}}}
\put(13,6){\makebox(1,1){{\scriptsize $1$}}}
\put(10,0){\makebox(1,1){{\scriptsize $[1]$}}}
\put(11,6){\makebox(1,1){{\scriptsize $[6]$}}}
\put(12,6){\makebox(1,1){{\scriptsize $[5]$}}}

\put(15,5){\makebox(2,1){$\longrightarrow$}}
\put(15,4){\makebox(2,1){$10$}}

\put(19,5){\line(1,0){4}}
\put(19,6){\line(1,0){5}}
\put(19,7){\line(1,0){5}}
\put(19,5){\line(0,1){2}}
\put(20,5){\line(0,1){2}}
\put(21,5){\line(0,1){2}}
\put(22,5){\line(0,1){2}}
\put(23,5){\line(0,1){2}}
\put(24,6){\line(0,1){1}}

\put(21,7){\makebox(2,1){$\downarrow$}}
\put(18,5){\makebox(1,1){{\scriptsize $2$}}}
\put(18,6){\makebox(1,1){{\scriptsize $0$}}}
\put(23,5){\makebox(1,1){{\scriptsize $2$}}}
\put(24,6){\makebox(1,1){{\scriptsize $0$}}}
\put(19,5){\makebox(1,1){{\scriptsize $[6]$}}}
\put(20,5){\makebox(1,1){{\scriptsize $[5]$}}}
\put(21,5){\makebox(1,1){{\scriptsize $[3]$}}}
\put(22,5){\makebox(1,1){{\scriptsize $[2]$}}}

\put(27,1){\line(1,0){1}}
\put(27,2){\line(1,0){1}}
\put(27,3){\line(1,0){1}}
\put(27,4){\line(1,0){1}}
\put(27,5){\line(1,0){1}}
\put(27,6){\line(1,0){1}}
\put(27,7){\line(1,0){1}}
\put(27,1){\line(0,1){6}}
\put(28,1){\line(0,1){6}}
\put(26,6){\makebox(1,1){{\scriptsize $0$}}}
\put(28,1){\makebox(1,1){{\scriptsize $0$}}}
\put(28,2){\makebox(1,1){{\scriptsize $0$}}}
\put(28,3){\makebox(1,1){{\scriptsize $0$}}}
\put(28,4){\makebox(1,1){{\scriptsize $0$}}}
\put(28,5){\makebox(1,1){{\scriptsize $0$}}}
\put(28,6){\makebox(1,1){{\scriptsize $0$}}}
\put(27,1){\makebox(1,1){{\scriptsize $[1]$}}}
\put(27,2){\makebox(1,1){{\scriptsize $[4]$}}}

\put(30,5){\makebox(2,1){$\longrightarrow$}}
\put(30,4){\makebox(2,1){$9$}}

\end{picture}
\end{center}
which corresponds to the path 
$\framebox{$10$}\otimes \framebox{$9$}\otimes \framebox{$3$}\otimes \framebox{$1$}\otimes \framebox{$9$}\otimes \framebox{$2$}\otimes \framebox{$1$}$.
The charge of the first RC and the energy of this path coincide, which is computed to be $-24$. 
If we selected the q-singular string for the $i_{3}$-string in the first RC, then we would have

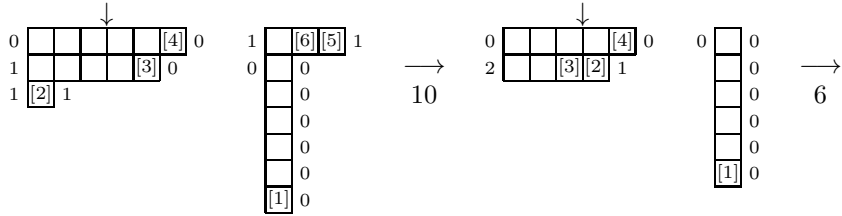
\setlength{\unitlength}{10pt}
\begin{center}
\begin{picture}(35,8)
\put(1,4){\line(1,0){1}}
\put(1,5){\line(1,0){5}}
\put(1,6){\line(1,0){6}}
\put(1,7){\line(1,0){6}}
\put(1,4){\line(0,1){3}}
\put(2,4){\line(0,1){3}}
\put(3,5){\line(0,1){2}}
\put(4,5){\line(0,1){2}}
\put(5,5){\line(0,1){2}}
\put(6,5){\line(0,1){2}}
\put(7,6){\line(0,1){1}}

\put(3,7){\makebox(2,1){$\downarrow$}}
\put(0,4){\makebox(1,1){{\scriptsize $1$}}}
\put(0,5){\makebox(1,1){{\scriptsize $1$}}}
\put(0,6){\makebox(1,1){{\scriptsize $0$}}}
\put(2,4){\makebox(1,1){{\scriptsize $1$}}}
\put(6,5){\makebox(1,1){{\scriptsize $0$}}}
\put(7,6){\makebox(1,1){{\scriptsize $0$}}}
\put(1,4){\makebox(1,1){{\scriptsize $[2]$}}}
\put(5,5){\makebox(1,1){{\scriptsize $[3]$}}}
\put(6,6){\makebox(1,1){{\scriptsize $[4]$}}}

\put(10,0){\line(1,0){1}}
\put(10,1){\line(1,0){1}}
\put(10,2){\line(1,0){1}}
\put(10,3){\line(1,0){1}}
\put(10,4){\line(1,0){1}}
\put(10,5){\line(1,0){1}}
\put(10,6){\line(1,0){3}}
\put(10,7){\line(1,0){3}}
\put(10,0){\line(0,1){7}}
\put(11,0){\line(0,1){7}}
\put(12,6){\line(0,1){1}}
\put(13,6){\line(0,1){1}}
\put(9,5){\makebox(1,1){{\scriptsize $0$}}}
\put(9,6){\makebox(1,1){{\scriptsize $1$}}}
\put(11,0){\makebox(1,1){{\scriptsize $0$}}}
\put(11,1){\makebox(1,1){{\scriptsize $0$}}}
\put(11,2){\makebox(1,1){{\scriptsize $0$}}}
\put(11,3){\makebox(1,1){{\scriptsize $0$}}}
\put(11,4){\makebox(1,1){{\scriptsize $0$}}}
\put(11,5){\makebox(1,1){{\scriptsize $0$}}}
\put(13,6){\makebox(1,1){{\scriptsize $1$}}}
\put(10,0){\makebox(1,1){{\scriptsize $[1]$}}}
\put(11,6){\makebox(1,1){{\scriptsize $[6]$}}}
\put(12,6){\makebox(1,1){{\scriptsize $[5]$}}}

\put(15,5){\makebox(2,1){$\longrightarrow$}}
\put(15,4){\makebox(2,1){$10$}}

\put(19,5){\line(1,0){4}}
\put(19,6){\line(1,0){5}}
\put(19,7){\line(1,0){5}}
\put(19,5){\line(0,1){2}}
\put(20,5){\line(0,1){2}}
\put(21,5){\line(0,1){2}}
\put(22,5){\line(0,1){2}}
\put(23,5){\line(0,1){2}}
\put(24,6){\line(0,1){1}}

\put(21,7){\makebox(2,1){$\downarrow$}}
\put(18,5){\makebox(1,1){{\scriptsize $2$}}}
\put(18,6){\makebox(1,1){{\scriptsize $0$}}}
\put(23,5){\makebox(1,1){{\scriptsize $1$}}}
\put(24,6){\makebox(1,1){{\scriptsize $0$}}}
\put(21,5){\makebox(1,1){{\scriptsize $[3]$}}}
\put(22,5){\makebox(1,1){{\scriptsize $[2]$}}}
\put(23,6){\makebox(1,1){{\scriptsize $[4]$}}}

\put(27,1){\line(1,0){1}}
\put(27,2){\line(1,0){1}}
\put(27,3){\line(1,0){1}}
\put(27,4){\line(1,0){1}}
\put(27,5){\line(1,0){1}}
\put(27,6){\line(1,0){1}}
\put(27,7){\line(1,0){1}}
\put(27,1){\line(0,1){6}}
\put(28,1){\line(0,1){6}}
\put(26,6){\makebox(1,1){{\scriptsize $0$}}}
\put(28,1){\makebox(1,1){{\scriptsize $0$}}}
\put(28,2){\makebox(1,1){{\scriptsize $0$}}}
\put(28,3){\makebox(1,1){{\scriptsize $0$}}}
\put(28,4){\makebox(1,1){{\scriptsize $0$}}}
\put(28,5){\makebox(1,1){{\scriptsize $0$}}}
\put(28,6){\makebox(1,1){{\scriptsize $0$}}}
\put(27,1){\makebox(1,1){{\scriptsize $[1]$}}}

\put(30,5){\makebox(2,1){$\longrightarrow$}}
\put(30,4){\makebox(2,1){$6$}}

\end{picture}
\end{center}
which corresponds to the path 
$\framebox{$10$}\otimes \framebox{$6$}\otimes \framebox{$5$}\otimes \framebox{$1$}\otimes \framebox{$9$}\otimes \framebox{$2$}\otimes \framebox{$1$}$.
However, the energy of this path is computed to be $-25$.

\end{ex}

\item[(BM-5)] Marking $[5]$.

\begin{itemize}

\item[(1)] The selected $i_{4}$-string is in $\nu^{(1)}$.

Find the singular/q-singular string of minimum length $(\geq i_{4}^{eff})$ in $\nu^{(2)}$.
Suppose that such a string exists.
If the founded string is singular and the length is $1$, then delete the marked boxes and return $\framebox{$\emptyset$}$.
In this case, the box marking of $\nu^{(1)}$ is

\setlength{\unitlength}{10pt}
\begin{center}
\begin{picture}(3,3)
\put(0,2){\line(1,0){3}}
\put(0,3){\line(1,0){3}}
\put(0,2){\line(0,1){1}}
\put(1,2){\line(0,1){1}}
\put(2,2){\line(0,1){1}}
\put(3,2){\line(0,1){1}}
\put(2,0.9){\makebox(1,1){$\uparrow$}}
\put(0,2){\makebox(1,1){{\scriptsize $[4]$}}}
\put(1,2){\makebox(1,1){{\scriptsize $[3]$}}}
\put(2,2){\makebox(1,1){{\scriptsize $[2]$}}}
\put(2,0){\makebox(1,1){{\scriptsize $3$}}}
\end{picture}
\end{center}
If the founded string is q-singular or singular with length $(\geq 2)$, then the rightmost box of the founded string by $[5]$.
If not, then delete the marked boxes and return $\framebox{$6$}$.

\item[(2)] The selected $i_{4}$-string is in $\nu^{(2)}$.

Find the singular or q-singular string of effective length $(\geq i_{4})$ in $\nu^{(1)}$.
If such a string exists and it is not a type-I/II q-singular string of effective length $i_{4}$, then mark the rightmost box of the founded string by $[5]$.
If such a string does not exist, then delete the marked boxes and return $\framebox{$5$}$.

\begin{ex}
$L=4$ and $\lambda=3\Lambda_{1}$.

\setlength{\unitlength}{10pt}
\begin{center}
\begin{picture}(18,5)
\put(1,2){\line(1,0){1}}
\put(1,3){\line(1,0){5}}
\put(1,4){\line(1,0){5}}
\put(1,2){\line(0,1){2}}
\put(2,2){\line(0,1){2}}
\put(3,3){\line(0,1){1}}
\put(4,3){\line(0,1){1}}
\put(5,3){\line(0,1){1}}
\put(6,3){\line(0,1){1}}

\put(3,4){\makebox(2,1){$\downarrow$}}

\put(0,2){\makebox(1,1){{\scriptsize $0$}}}
\put(0,3){\makebox(1,1){{\scriptsize $2$}}}
\put(2,2){\makebox(1,1){{\scriptsize $0$}}}
\put(6,3){\makebox(1,1){{\scriptsize $1$}}}
\put(1,2){\makebox(1,1){{\scriptsize $[2]$}}}
\put(5,3){\makebox(1,1){{\scriptsize $[3]$}}}

\put(9,0){\line(1,0){1}}
\put(9,1){\line(1,0){1}}
\put(9,2){\line(1,0){1}}
\put(9,3){\line(1,0){2}}
\put(9,4){\line(1,0){2}}
\put(9,0){\line(0,1){4}}
\put(10,0){\line(0,1){4}}
\put(11,3){\line(0,1){1}}
\put(8,2){\makebox(1,1){{\scriptsize $0$}}}
\put(8,3){\makebox(1,1){{\scriptsize $0$}}}
\put(10,0){\makebox(1,1){{\scriptsize $0$}}}
\put(10,1){\makebox(1,1){{\scriptsize $0$}}}
\put(10,2){\makebox(1,1){{\scriptsize $0$}}}
\put(11,3){\makebox(1,1){{\scriptsize $0$}}}
\put(9,0){\makebox(1,1){{\scriptsize $[1]$}}}
\put(10,3){\makebox(1,1){{\scriptsize $[4]$}}}

\put(13,3){\makebox(2,1){$\longrightarrow$}}
\put(13,2){\makebox(2,1){$5$}}
\end{picture}
\end{center}
which corresponds to the path $\framebox{$5$}\otimes \framebox{$9$}\otimes \framebox{$2$} \otimes \framebox{$1$}$.
The charge of the first RC and the energy of this path coincide, which is computed to be $-8$.
For the rest steps, see Example~\ref{ex:BM-4a}.
If [5] were marked on the left of [3], then we would have

\setlength{\unitlength}{10pt}
\begin{center}
\begin{picture}(28,5)
\put(1,2){\line(1,0){1}}
\put(1,3){\line(1,0){5}}
\put(1,4){\line(1,0){5}}
\put(1,2){\line(0,1){2}}
\put(2,2){\line(0,1){2}}
\put(3,3){\line(0,1){1}}
\put(4,3){\line(0,1){1}}
\put(5,3){\line(0,1){1}}
\put(6,3){\line(0,1){1}}

\put(3,4){\makebox(2,1){$\downarrow$}}

\put(0,2){\makebox(1,1){{\scriptsize $0$}}}
\put(0,3){\makebox(1,1){{\scriptsize $2$}}}
\put(2,2){\makebox(1,1){{\scriptsize $0$}}}
\put(6,3){\makebox(1,1){{\scriptsize $1$}}}
\put(1,2){\makebox(1,1){{\scriptsize $[2]$}}}
\put(4,3){\makebox(1,1){{\scriptsize $[5]$}}}
\put(5,3){\makebox(1,1){{\scriptsize $[3]$}}}

\put(9,0){\line(1,0){1}}
\put(9,1){\line(1,0){1}}
\put(9,2){\line(1,0){1}}
\put(9,3){\line(1,0){2}}
\put(9,4){\line(1,0){2}}
\put(9,0){\line(0,1){4}}
\put(10,0){\line(0,1){4}}
\put(11,3){\line(0,1){1}}
\put(8,2){\makebox(1,1){{\scriptsize $0$}}}
\put(8,3){\makebox(1,1){{\scriptsize $0$}}}
\put(10,0){\makebox(1,1){{\scriptsize $0$}}}
\put(10,1){\makebox(1,1){{\scriptsize $0$}}}
\put(10,2){\makebox(1,1){{\scriptsize $0$}}}
\put(11,3){\makebox(1,1){{\scriptsize $0$}}}
\put(9,0){\makebox(1,1){{\scriptsize $[1]$}}}
\put(10,3){\makebox(1,1){{\scriptsize $[4]$}}}

\put(13,3){\makebox(2,1){$\longrightarrow$}}
\put(13,2){\makebox(2,1){$7$}}

\put(17,3){\line(1,0){3}}
\put(17,4){\line(1,0){3}}
\put(17,3){\line(0,1){1}}
\put(18,3){\line(0,1){1}}
\put(19,3){\line(0,1){1}}
\put(20,3){\line(0,1){1}}

\put(16,3){\makebox(1,1){{\scriptsize $3$}}}
\put(20,3){\makebox(1,1){{\scriptsize $3$}}}
\put(17,3){\makebox(1,1){{\scriptsize $[4]$}}}
\put(18,3){\makebox(1,1){{\scriptsize $[3]$}}}
\put(19,3){\makebox(1,1){{\scriptsize $[2]$}}}

\put(23,1){\line(1,0){1}}
\put(23,2){\line(1,0){1}}
\put(23,3){\line(1,0){1}}
\put(23,4){\line(1,0){1}}
\put(23,1){\line(0,1){3}}
\put(24,1){\line(0,1){3}}

\put(22,3){\makebox(1,1){{\scriptsize $0$}}}
\put(24,1){\makebox(1,1){{\scriptsize $0$}}}
\put(24,2){\makebox(1,1){{\scriptsize $0$}}}
\put(24,3){\makebox(1,1){{\scriptsize $0$}}}
\put(23,1){\makebox(1,1){{\scriptsize $[1]$}}}
\put(23,2){\makebox(1,1){{\scriptsize $[5]$}}}

\put(26,3){\makebox(2,1){$\longrightarrow$}}
\put(26,2){\makebox(2,1){$\emptyset$}}
\end{picture}
\end{center}
which corresponds to the path $\framebox{$7$}\otimes \framebox{$\emptyset$}\otimes \framebox{$2$} \otimes \framebox{$1$}$.
However, the energy of this path is computed to be $-6$.

\end{ex}

\end{itemize}

\item[(BM-6)] Marking $[6]$.

\begin{itemize}

\item[(1)] The selected $i_{5}$-string is q-singular.

Find the singular string of minimum length $(\geq i_{5}+1)$.
If such a string exists, then mark the rightmost box of the founded string by $[6]$.
If such a string does not exist, then delete the marked boxes and return $\framebox{$7$}$ when the selected $i_{5}$-string is in $\nu^{(1)}$ and $\framebox{$8$}$ when the selected $i_{5}$-string is in $\nu^{(2)}$.

\item[(2)] The selected $i_{5}$-string is singular.

Mark the box on the left of $[5]$ by $[6]$.

\end{itemize}
According to the marking rule of [4] the box marking such that the selected $i_{4}$-string in $\nu^{(1)}$ is type-II with $i_{4}^{eff}=i_{6}$ cannot occur.

\item[(BM-7)] Marking $[7]$.

\begin{itemize}

\item[(1)] The selected $i_{6}$-string is in $\nu^{(1)}$.

Find the singular string of minimum length $(\geq i_{6}^{eff})$.
If such a string exists, then mark the rightmost unmarked box of the founded string by $[7]$.
If not, then delete the marked boxes and return $\framebox{$9$}$.
Note that it is possible $[7]$ is marked on the left of $[4]$.

\item[(2)] The selected $i_{6}$-string is in $\nu^{(2)}$.

Find the selected $i_{4}$-string with $i_{4}^{eff}=i_{6}$, which is type-0/I.
If the search is successful, then mark [7] on the left of [4].
If not, then follow the box marking rule \textbf{(BM-2)} with replacing $i_{1}$ by $i_{6}$.
If the search is not successful, then delete the marked boxes and return $\framebox{$10$}$.

\end{itemize}

\item[(BM-8)] Marking $[8]$.

\begin{itemize}

\item[(1)] The selected $i_{7}$-string is in $\nu^{(1)}$.

Find the selected $i_{7}$-string which is also marked by [4].
If such a string exists, then mark [8] on the left of [7].
If the selected string is type-I and the rightmost box of the selected string is marked by [3], then prohibit the further box marking of this string (see Example~\ref{ex:BM-8a}).
We say that this string is \emph{inactivated}.
If the selected $i_{7}$-string is not marked by [4], then follow the box marking rule \textbf{(BM-3)} with replacing $i_{1}$ by $i_{6}$ and $i_{2}$ by $i_{7}$ except that [9] cannot be marked in $\nu^{(2)}$.
If the search is not successful, then delete the marked boxes and return $\framebox{$11$}$.

\begin{ex} \label{ex:BM-8a}
$L=5$ and $\lambda=2\Bar{\Lambda}_{1}$.

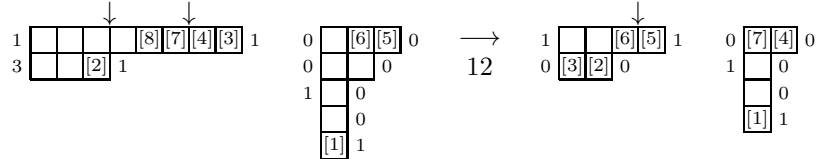
\setlength{\unitlength}{10pt}
\begin{flushleft}
\begin{picture}(31,6)
\put(1,3){\line(1,0){3}}
\put(1,4){\line(1,0){8}}
\put(1,5){\line(1,0){8}}
\put(1,3){\line(0,1){2}}
\put(2,3){\line(0,1){2}}
\put(3,3){\line(0,1){2}}
\put(4,3){\line(0,1){2}}
\put(5,4){\line(0,1){1}}
\put(6,4){\line(0,1){1}}
\put(7,4){\line(0,1){1}}
\put(8,4){\line(0,1){1}}
\put(9,4){\line(0,1){1}}

\put(3,5){\makebox(2,1){$\downarrow$}}
\put(6,5){\makebox(2,1){$\downarrow$}}
\put(0,3){\makebox(1,1){{\scriptsize $3$}}}
\put(0,4){\makebox(1,1){{\scriptsize $1$}}}
\put(4,3){\makebox(1,1){{\scriptsize $1$}}}
\put(9,4){\makebox(1,1){{\scriptsize $1$}}}
\put(3,3){\makebox(1,1){{\scriptsize $[2]$}}}
\put(5,4){\makebox(1,1){{\scriptsize $[8]$}}}
\put(6,4){\makebox(1,1){{\scriptsize $[7]$}}}
\put(7,4){\makebox(1,1){{\scriptsize $[4]$}}}
\put(8,4){\makebox(1,1){{\scriptsize $[3]$}}}

\put(12,0){\line(1,0){1}}
\put(12,1){\line(1,0){1}}
\put(12,2){\line(1,0){1}}
\put(12,3){\line(1,0){2}}
\put(12,4){\line(1,0){3}}
\put(12,5){\line(1,0){3}}
\put(12,0){\line(0,1){5}}
\put(13,0){\line(0,1){5}}
\put(14,3){\line(0,1){2}}
\put(15,4){\line(0,1){1}}
\put(11,2){\makebox(1,1){{\scriptsize $1$}}}
\put(11,3){\makebox(1,1){{\scriptsize $0$}}}
\put(11,4){\makebox(1,1){{\scriptsize $0$}}}
\put(13,0){\makebox(1,1){{\scriptsize $1$}}}
\put(13,1){\makebox(1,1){{\scriptsize $0$}}}
\put(13,2){\makebox(1,1){{\scriptsize $0$}}}
\put(14,3){\makebox(1,1){{\scriptsize $0$}}}
\put(15,4){\makebox(1,1){{\scriptsize $0$}}}
\put(12,0){\makebox(1,1){{\scriptsize $[1]$}}}
\put(13,4){\makebox(1,1){{\scriptsize $[6]$}}}
\put(14,4){\makebox(1,1){{\scriptsize $[5]$}}}

\put(17,4){\makebox(2,1){$\longrightarrow$}}
\put(17,3){\makebox(2,1){$12$}}

\put(21,3){\line(1,0){2}}
\put(21,4){\line(1,0){4}}
\put(21,5){\line(1,0){4}}
\put(21,3){\line(0,1){2}}
\put(22,3){\line(0,1){2}}
\put(23,3){\line(0,1){2}}
\put(24,4){\line(0,1){1}}
\put(25,4){\line(0,1){1}}

\put(23,5){\makebox(2,1){$\downarrow$}}
\put(20,3){\makebox(1,1){{\scriptsize $0$}}}
\put(20,4){\makebox(1,1){{\scriptsize $1$}}}
\put(23,3){\makebox(1,1){{\scriptsize $0$}}}
\put(25,4){\makebox(1,1){{\scriptsize $1$}}}
\put(21,3){\makebox(1,1){{\scriptsize $[3]$}}}
\put(22,3){\makebox(1,1){{\scriptsize $[2]$}}}
\put(23,4){\makebox(1,1){{\scriptsize $[6]$}}}
\put(24,4){\makebox(1,1){{\scriptsize $[5]$}}}

\put(28,1){\line(1,0){1}}
\put(28,2){\line(1,0){1}}
\put(28,3){\line(1,0){1}}
\put(28,4){\line(1,0){2}}
\put(28,5){\line(1,0){2}}
\put(28,1){\line(0,1){4}}
\put(29,1){\line(0,1){4}}
\put(30,4){\line(0,1){1}}

\put(27,3){\makebox(1,1){{\scriptsize $1$}}}
\put(27,4){\makebox(1,1){{\scriptsize $0$}}}
\put(29,1){\makebox(1,1){{\scriptsize $1$}}}
\put(29,2){\makebox(1,1){{\scriptsize $0$}}}
\put(29,3){\makebox(1,1){{\scriptsize $0$}}}
\put(30,4){\makebox(1,1){{\scriptsize $0$}}}
\put(28,1){\makebox(1,1){{\scriptsize $[1]$}}}
\put(28,4){\makebox(1,1){{\scriptsize $[7]$}}}
\put(29,4){\makebox(1,1){{\scriptsize $[4]$}}}
\end{picture}
\end{flushleft}
which corresponds to the path 
$\framebox{$12$}\otimes \framebox{$11$}\otimes \framebox{$1$}\otimes \framebox{$5$}\otimes \framebox{$1$}$.
The charge of the first RC and the energy of this path coincide, which is computed to be $-13$.
If [9] were marked on the left of [8], then the first return would be $13$.
However, $\lambda-\mathrm{wt}(\framebox{$13$})=2\Bar{\Lambda}_{1}-(-3\Bar{\Lambda}_{1}+\Bar{\Lambda}_{2})=5\Bar{\Lambda_{1}}-\Bar{\Lambda_{2}}$, which is not dominant.

\end{ex}

\item[(2)] The selected $i_{7}$-string is in $\nu^{(2)}$.

Find the singular/q-singular string of effective length $i_{7}$ in $\nu^{(1)}$, which is possibly marked by [6], according to the following preferential rule.

\begin{center}
\begin{tabular}{c|cc}
                & type-0 & type-I \\ \hline
singular     & 1 & 2 \\
q-singular & 3 & 
\end{tabular}
\end{center}
If such a string exists, then mark the rightmost unmarked box of the founded string by $[8]$.
If not, then find the singular or q-singular string of minimum length whose effective length $(\geq i_{7}+1)$.
If the search is successful, then mark the rightmost box of the unmarked box of the founded string by $[8]$.
If not, then delete the marked boxes and return $\framebox{$11$}$.
If the selected $i_{8}$-string is type-I and the rightmost box of this string is marked by [5] or [6], then prohibit the further box marking of this string (see Examples~\ref{ex:BM-8b} and \ref{ex:BM-8c}).
We also say that this string is inactivated.

\begin{ex} \label{ex:BM-8b}
$L=5$ and $\lambda=2\Bar{\Lambda}_{1}$ as in Example~\ref{ex:BM-8a}.

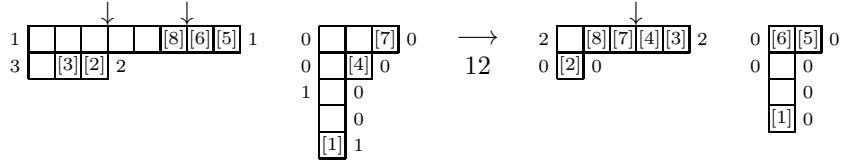
\setlength{\unitlength}{10pt}
\begin{flushleft}
\begin{picture}(31,6)
\put(1,3){\line(1,0){3}}
\put(1,4){\line(1,0){8}}
\put(1,5){\line(1,0){8}}
\put(1,3){\line(0,1){2}}
\put(2,3){\line(0,1){2}}
\put(3,3){\line(0,1){2}}
\put(4,3){\line(0,1){2}}
\put(5,4){\line(0,1){1}}
\put(6,4){\line(0,1){1}}
\put(7,4){\line(0,1){1}}
\put(8,4){\line(0,1){1}}
\put(9,4){\line(0,1){1}}

\put(3,5){\makebox(2,1){$\downarrow$}}
\put(6,5){\makebox(2,1){$\downarrow$}}
\put(0,3){\makebox(1,1){{\scriptsize $3$}}}
\put(0,4){\makebox(1,1){{\scriptsize $1$}}}
\put(4,3){\makebox(1,1){{\scriptsize $2$}}}
\put(9,4){\makebox(1,1){{\scriptsize $1$}}}
\put(2,3){\makebox(1,1){{\scriptsize $[3]$}}}
\put(3,3){\makebox(1,1){{\scriptsize $[2]$}}}
\put(6,4){\makebox(1,1){{\scriptsize $[8]$}}}
\put(7,4){\makebox(1,1){{\scriptsize $[6]$}}}
\put(8,4){\makebox(1,1){{\scriptsize $[5]$}}}

\put(12,0){\line(1,0){1}}
\put(12,1){\line(1,0){1}}
\put(12,2){\line(1,0){1}}
\put(12,3){\line(1,0){2}}
\put(12,4){\line(1,0){3}}
\put(12,5){\line(1,0){3}}
\put(12,0){\line(0,1){5}}
\put(13,0){\line(0,1){5}}
\put(14,3){\line(0,1){2}}
\put(15,4){\line(0,1){1}}
\put(11,2){\makebox(1,1){{\scriptsize $1$}}}
\put(11,3){\makebox(1,1){{\scriptsize $0$}}}
\put(11,4){\makebox(1,1){{\scriptsize $0$}}}
\put(13,0){\makebox(1,1){{\scriptsize $1$}}}
\put(13,1){\makebox(1,1){{\scriptsize $0$}}}
\put(13,2){\makebox(1,1){{\scriptsize $0$}}}
\put(14,3){\makebox(1,1){{\scriptsize $0$}}}
\put(15,4){\makebox(1,1){{\scriptsize $0$}}}
\put(12,0){\makebox(1,1){{\scriptsize $[1]$}}}
\put(13,3){\makebox(1,1){{\scriptsize $[4]$}}}
\put(14,4){\makebox(1,1){{\scriptsize $[7]$}}}

\put(17,4){\makebox(2,1){$\longrightarrow$}}
\put(17,3){\makebox(2,1){$12$}}

\put(21,3){\line(1,0){1}}
\put(21,4){\line(1,0){5}}
\put(21,5){\line(1,0){5}}
\put(21,3){\line(0,1){2}}
\put(22,3){\line(0,1){2}}
\put(23,4){\line(0,1){1}}
\put(24,4){\line(0,1){1}}
\put(25,4){\line(0,1){1}}
\put(26,4){\line(0,1){1}}

\put(23,5){\makebox(2,1){$\downarrow$}}
\put(20,3){\makebox(1,1){{\scriptsize $0$}}}
\put(20,4){\makebox(1,1){{\scriptsize $2$}}}
\put(22,3){\makebox(1,1){{\scriptsize $0$}}}
\put(26,4){\makebox(1,1){{\scriptsize $2$}}}
\put(21,3){\makebox(1,1){{\scriptsize $[2]$}}}
\put(22,4){\makebox(1,1){{\scriptsize $[8]$}}}
\put(23,4){\makebox(1,1){{\scriptsize $[7]$}}}
\put(24,4){\makebox(1,1){{\scriptsize $[4]$}}}
\put(25,4){\makebox(1,1){{\scriptsize $[3]$}}}

\put(29,1){\line(1,0){1}}
\put(29,2){\line(1,0){1}}
\put(29,3){\line(1,0){1}}
\put(29,4){\line(1,0){2}}
\put(29,5){\line(1,0){2}}
\put(29,1){\line(0,1){4}}
\put(30,1){\line(0,1){4}}
\put(31,4){\line(0,1){1}}

\put(28,3){\makebox(1,1){{\scriptsize $0$}}}
\put(28,4){\makebox(1,1){{\scriptsize $0$}}}
\put(30,1){\makebox(1,1){{\scriptsize $0$}}}
\put(30,2){\makebox(1,1){{\scriptsize $0$}}}
\put(30,3){\makebox(1,1){{\scriptsize $0$}}}
\put(31,4){\makebox(1,1){{\scriptsize $0$}}}
\put(29,1){\makebox(1,1){{\scriptsize $[1]$}}}
\put(29,4){\makebox(1,1){{\scriptsize $[6]$}}}
\put(30,4){\makebox(1,1){{\scriptsize $[5]$}}}
\end{picture}
\end{flushleft}
which corresponds to the path 
$\framebox{$12$}\otimes \framebox{$12$}\otimes \framebox{$3$}\otimes \framebox{$2$}\otimes \framebox{$1$}$.
The charge of the first RC and the energy of this path coincide, which is computed to be $-12$.
By the same reasoning as in Example~\ref{ex:BM-8a}, [9] cannot be marked on the left of [8] in the first RC.

\end{ex}

\begin{ex} \label{ex:BM-8c}
$L=5$ and $\lambda=2\Bar{\Lambda}_{1}$ as in Example~\ref{ex:BM-8a}.

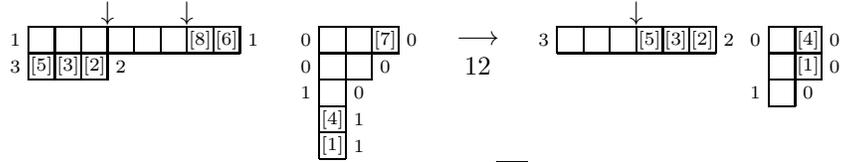
\setlength{\unitlength}{10pt}
\begin{flushleft}
\begin{picture}(32,6)
\put(1,3){\line(1,0){3}}
\put(1,4){\line(1,0){8}}
\put(1,5){\line(1,0){8}}
\put(1,3){\line(0,1){2}}
\put(2,3){\line(0,1){2}}
\put(3,3){\line(0,1){2}}
\put(4,3){\line(0,1){2}}
\put(5,4){\line(0,1){1}}
\put(6,4){\line(0,1){1}}
\put(7,4){\line(0,1){1}}
\put(8,4){\line(0,1){1}}
\put(9,4){\line(0,1){1}}

\put(3,5){\makebox(2,1){$\downarrow$}}
\put(6,5){\makebox(2,1){$\downarrow$}}
\put(0,3){\makebox(1,1){{\scriptsize $3$}}}
\put(0,4){\makebox(1,1){{\scriptsize $1$}}}
\put(4,3){\makebox(1,1){{\scriptsize $2$}}}
\put(9,4){\makebox(1,1){{\scriptsize $1$}}}
\put(1,3){\makebox(1,1){{\scriptsize $[5]$}}}
\put(2,3){\makebox(1,1){{\scriptsize $[3]$}}}
\put(3,3){\makebox(1,1){{\scriptsize $[2]$}}}
\put(7,4){\makebox(1,1){{\scriptsize $[8]$}}}
\put(8,4){\makebox(1,1){{\scriptsize $[6]$}}}

\put(12,0){\line(1,0){1}}
\put(12,1){\line(1,0){1}}
\put(12,2){\line(1,0){1}}
\put(12,3){\line(1,0){2}}
\put(12,4){\line(1,0){3}}
\put(12,5){\line(1,0){3}}
\put(12,0){\line(0,1){5}}
\put(13,0){\line(0,1){5}}
\put(14,3){\line(0,1){2}}
\put(15,4){\line(0,1){1}}
\put(11,2){\makebox(1,1){{\scriptsize $1$}}}
\put(11,3){\makebox(1,1){{\scriptsize $0$}}}
\put(11,4){\makebox(1,1){{\scriptsize $0$}}}
\put(13,0){\makebox(1,1){{\scriptsize $1$}}}
\put(13,1){\makebox(1,1){{\scriptsize $1$}}}
\put(13,2){\makebox(1,1){{\scriptsize $0$}}}
\put(14,3){\makebox(1,1){{\scriptsize $0$}}}
\put(15,4){\makebox(1,1){{\scriptsize $0$}}}
\put(12,0){\makebox(1,1){{\scriptsize $[1]$}}}
\put(12,1){\makebox(1,1){{\scriptsize $[4]$}}}
\put(14,4){\makebox(1,1){{\scriptsize $[7]$}}}

\put(17,4){\makebox(2,1){$\longrightarrow$}}
\put(17,3){\makebox(2,1){$12$}}

\put(21,4){\line(1,0){6}}
\put(21,5){\line(1,0){6}}
\put(21,4){\line(0,1){1}}
\put(22,4){\line(0,1){1}}
\put(23,4){\line(0,1){1}}
\put(24,4){\line(0,1){1}}
\put(25,4){\line(0,1){1}}
\put(26,4){\line(0,1){1}}
\put(27,4){\line(0,1){1}}

\put(23,5){\makebox(2,1){$\downarrow$}}
\put(20,4){\makebox(1,1){{\scriptsize $3$}}}
\put(27,4){\makebox(1,1){{\scriptsize $2$}}}
\put(24,4){\makebox(1,1){{\scriptsize $[5]$}}}
\put(25,4){\makebox(1,1){{\scriptsize $[3]$}}}
\put(26,4){\makebox(1,1){{\scriptsize $[2]$}}}

\put(29,2){\line(1,0){1}}
\put(29,3){\line(1,0){2}}
\put(29,4){\line(1,0){2}}
\put(29,5){\line(1,0){2}}
\put(29,2){\line(0,1){3}}
\put(30,2){\line(0,1){3}}
\put(31,3){\line(0,1){2}}

\put(28,2){\makebox(1,1){{\scriptsize $1$}}}
\put(28,4){\makebox(1,1){{\scriptsize $0$}}}
\put(30,2){\makebox(1,1){{\scriptsize $0$}}}
\put(31,3){\makebox(1,1){{\scriptsize $0$}}}
\put(31,4){\makebox(1,1){{\scriptsize $0$}}}
\put(30,3){\makebox(1,1){{\scriptsize $[1]$}}}
\put(30,4){\makebox(1,1){{\scriptsize $[4]$}}}

\end{picture}
\end{flushleft}
which corresponds to the path 
$\framebox{$12$}\otimes \framebox{$7$}\otimes \framebox{$\emptyset$}\otimes \framebox{$2$}\otimes \framebox{$1$}$.
The charge of the first RC and the energy of this path coincide, which is computed to be $-11$. 
By the same reasoning as in Example~\ref{ex:BM-8a}, [9] cannot be marked on the left of [8] in the first RC.

\end{ex}

\end{itemize}

\item[(BM-9)] Marking $[9]$.

Find the active singular string that is not inactivated of length $(\geq i_{8})$ in $\nu^{(1)}$.
If the founded string whose length is denoted by $l_{9}^{(1)}$ is type-I and the rightmost box of the founded string is marked by $[2]$, then discard this selection and find the string of length $(> l_{9}^{(1)})$.
If the founded string whose length is denoted by $l_{9}^{(1)}$ is type-II and the rightmost box of the founded string is marked by $[3]$, then discard this selection and find the string of length $(> l_{9}^{(1)})$.
The last two rules prohibit the box marking in a single string across two delimiters.
If the string which can be marked by $[9]$ is founded, then mark the rightmost box of the unmarked box of the founded string by $[9]$.
If not, then delete the marked boxes and return $\framebox{$12$}$. 

If the selected $i_{9}$-string is type-0 and the selected $i_{8}$-string is type-I with the same effective length, then discard the $i_{8}$-string selection and do the following box marking.

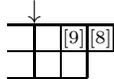
\setlength{\unitlength}{10pt}
\begin{center}
\begin{picture}(5,3)
\put(0,0){\line(1,0){3}}
\put(0,1){\line(1,0){4}}
\put(0,2){\line(1,0){4}}
\put(1,0){\line(0,1){2}}
\put(2,0){\line(0,1){2}}
\put(3,0){\line(0,1){2}}
\put(4,1){\line(0,1){1}}
\put(0,2){\makebox(2,1){$\downarrow$}}
\put(2,1){\makebox(1,1){{\scriptsize $[9]$}}}
\put(3,1){\makebox(1,1){{\scriptsize $[8]$}}}
\end{picture}
\end{center}
This rule guarantees the bijection when [10] is marked in a singular string with $i_{9}^{eff}=i_{10}$ in $\nu^{(2)}$.

\item[(BM-10)] Marking $[10]$.

Find the singular string of minimum length $(\geq i_{9}^{eff})$ in $\nu^{(2)}$.
If such a string exists, then mark the rightmost unmarked box of the founded string by $[10]$ and delete the marked boxes returning  $\framebox{$14$}$.
If not, then delete the marked boxes and return $\framebox{$13$}$.

\end{description}

There are several exceptional rules for box marking.

\begin{ex} \label{ex:boomerang}
$L=4$ and $\lambda=2\Bar{\Lambda}_{2}$.

According to the box marking rules prescribed above, we have

\setlength{\unitlength}{10pt}
\begin{center}
\begin{picture}(16,7)
\put(1,3){\line(1,0){3}}
\put(1,4){\line(1,0){6}}
\put(1,5){\line(1,0){6}}
\put(1,3){\line(0,1){2}}
\put(2,3){\line(0,1){2}}
\put(3,3){\line(0,1){2}}
\put(4,3){\line(0,1){2}}
\put(5,4){\line(0,1){1}}
\put(6,4){\line(0,1){1}}
\put(7,4){\line(0,1){1}}

\put(3,5){\makebox(2,1){$\downarrow$}}
\put(0,3){\makebox(1,1){{\scriptsize $3$}}}
\put(0,4){\makebox(1,1){{\scriptsize $0$}}}
\put(4,3){\makebox(1,1){{\scriptsize $1$}}}
\put(7,4){\makebox(1,1){{\scriptsize $0$}}}
\put(3,3){\makebox(1,1){{\scriptsize $[2]$}}}
\put(4,4){\makebox(1,1){{\scriptsize $[7]$}}}
\put(5,4){\makebox(1,1){{\scriptsize $[4]$}}}
\put(6,4){\makebox(1,1){{\scriptsize $[3]$}}}

\put(10,0){\line(1,0){1}}
\put(10,1){\line(1,0){1}}
\put(10,2){\line(1,0){1}}
\put(10,3){\line(1,0){1}}
\put(10,4){\line(1,0){2}}
\put(10,5){\line(1,0){2}}
\put(10,0){\line(0,1){5}}
\put(11,0){\line(0,1){5}}
\put(12,4){\line(0,1){1}}
\put(9,3){\makebox(1,1){{\scriptsize $0$}}}
\put(9,4){\makebox(1,1){{\scriptsize $1$}}}
\put(11,0){\makebox(1,1){{\scriptsize $0$}}}
\put(11,1){\makebox(1,1){{\scriptsize $0$}}}
\put(11,2){\makebox(1,1){{\scriptsize $0$}}}
\put(11,3){\makebox(1,1){{\scriptsize $0$}}}
\put(12,4){\makebox(1,1){{\scriptsize $0$}}}
\put(10,0){\makebox(1,1){{\scriptsize $[1]$}}}
\put(10,4){\makebox(1,1){{\scriptsize $[6]$}}}
\put(11,4){\makebox(1,1){{\scriptsize $[5]$}}}

\put(14,4){\makebox(2,1){$\longrightarrow$}}
\put(14,3){\makebox(2,1){$11$}}

\end{picture}
\end{center}
However, $\lambda-\mathrm{wt}(\framebox{$11$})=2\Bar{\Lambda}_{2}-(\Bar{\Lambda}_{1}-\Bar{\Lambda}_{2})=-\Bar{\Lambda}_{1}+3\Bar{\Lambda}_{2}$, which is not dominant.
For the above RC, we must do the following box marking.

\setlength{\unitlength}{10pt}
\begin{center}
\begin{picture}(16,7)
\put(1,3){\line(1,0){3}}
\put(1,4){\line(1,0){6}}
\put(1,5){\line(1,0){6}}
\put(1,3){\line(0,1){2}}
\put(2,3){\line(0,1){2}}
\put(3,3){\line(0,1){2}}
\put(4,3){\line(0,1){2}}
\put(5,4){\line(0,1){1}}
\put(6,4){\line(0,1){1}}
\put(7,4){\line(0,1){1}}

\put(3,5){\makebox(2,1){$\downarrow$}}
\put(0,3){\makebox(1,1){{\scriptsize $3$}}}
\put(0,4){\makebox(1,1){{\scriptsize $0$}}}
\put(4,3){\makebox(1,1){{\scriptsize $1$}}}
\put(7,4){\makebox(1,1){{\scriptsize $0$}}}
\put(1,4){\makebox(1,1){{\scriptsize $[9]$}}}
\put(2,4){\makebox(1,1){{\scriptsize $[8]$}}}
\put(3,4){\makebox(1,1){{\scriptsize $[7]$}}}
\put(4,4){\makebox(1,1){{\scriptsize $[4]$}}}
\put(5,4){\makebox(1,1){{\scriptsize $[3]$}}}
\put(6,4){\makebox(1,1){{\scriptsize $[2]$}}}

\put(10,0){\line(1,0){1}}
\put(10,1){\line(1,0){1}}
\put(10,2){\line(1,0){1}}
\put(10,3){\line(1,0){1}}
\put(10,4){\line(1,0){2}}
\put(10,5){\line(1,0){2}}
\put(10,0){\line(0,1){5}}
\put(11,0){\line(0,1){5}}
\put(12,4){\line(0,1){1}}
\put(9,3){\makebox(1,1){{\scriptsize $0$}}}
\put(9,4){\makebox(1,1){{\scriptsize $1$}}}
\put(11,0){\makebox(1,1){{\scriptsize $0$}}}
\put(11,1){\makebox(1,1){{\scriptsize $0$}}}
\put(11,2){\makebox(1,1){{\scriptsize $0$}}}
\put(11,3){\makebox(1,1){{\scriptsize $0$}}}
\put(12,4){\makebox(1,1){{\scriptsize $0$}}}
\put(10,0){\makebox(1,1){{\scriptsize $[1]$}}}
\put(10,4){\makebox(1,1){{\scriptsize $[6]$}}}
\put(11,4){\makebox(1,1){{\scriptsize $[5]$}}}

\put(14,4){\makebox(2,1){$\longrightarrow$}}
\put(14,3){\makebox(2,1){$13$}}

\end{picture}
\end{center}

\end{ex}

The first string in $\nu^{(1)}$ is called a \emph{boomerang string}.
Here we summarize the rules of finding boomerang strings, which have a common feature; the box marking restarts in the boomerang string leading to the larger box marking.

\begin{description}

\item[(BS-1)] The selected $i_{2}$-string is type-II q-singular.

The type-0/I singular string of effective length $i_{2}^{eff}$ under the condition that $\nu^{(2)}$ has a singular string of length $i_{2}^{eff}$.

\item[(BS-2)] The selected $i_{2}$-string is type-0/I q-singular or type-I qq-singular.

The following singular strings are boomerang strings for the $i_{2}$-string. 
If there exist boomerang strings more than two, then select the string of minimum length.

\begin{itemize}
\item[(1)]
The type-II singular string of effective length $i_{2}^{eff}+1$ under the condition that $\nu^{(2)}$ has a singular string of length $i_{2}^{eff}$ or $i_{2}^{eff}+1$.

\item[(2)]
The type-0/I singular string of effective length $i_{2}^{eff}+1$ under the condition that $\nu^{(2)}$ has a singular string of length $i_{2}^{eff}+1$.
\end{itemize}

\item[(BS-3)] The selected $i_{2}$-string is type-0 qq-singular.

The type-0/I singular string of effective length $i_{2}^{eff}+1$ is a boomerang string if $\nu^{(2)}$ has a singular string of length $i_{2}^{eff}+1$.
See Example~\ref{ex:boomerang}.
Note that we always select the type-II singular string of effective length $i_{2}^{eff}+1$ if it exists.
See Example~\ref{ex:BM-3}. 

\item[(BS-4)] The selected $i_{3}$-string with $i_{2}<i_{3}$ is type-II q-singular.

The type-0/I singular string of effective length $i_{3}^{eff}$ is a boomerang string for the $i_{3}$-string if $\nu^{(2)}$ has a singular string of length $i_{3}^{eff}$.

\item[(BS-5)] The selected $i_{3}$-string with $i_{2}<i_{3}$ is type-0/I q-singular.

The singular string of any type and of effective length $i_{3}^{eff}+1$ is a boomerang string for the $i_{3}$-string if $\nu^{(2)}$ has a singular string of length $i_{3}^{eff}+1$.

\item[(BS-6)] The selected $i_{5}$-string in $\nu^{(1)}$ with $i_{3}<i_{5}$ is q-singular and type-I/II.

The singular string of length ($>i_{5}$) and of the effective length $i_{5}^{eff}$ is a boomerang string for the $i_{5}$-string if $\nu^{(2)}$ has a singular string of length $i_{5}^{eff}$.

\end{description}

Here, we summarize the rule of adjusting the values of riggings in $\Tilde{\nu}^{(a)}$ $(a=1,2)$.
The fundamental rule is that we set the box-deleted string in $\Tilde{\nu}^{(a)}$ $(a=1,2)$ to be singular and that we keep the riggings of the box-undeleted strings the same.

\begin{description}
\item[(RA-1)] When $[6]$ is marked in $\nu^{(1)}$ or $\nu^{(2)}$ with $i_{5}< i_{6}$, we set the box-deleted $(i_{6}-n)$-string to be q-singular where $n\geq 1$.
See Example~\ref{ex:BM-8c}.

\item[(RA-2)] When $[3]$ is marked in $\nu^{(1)}$ with $i_{2}< i_{3}$ and if [4] is not marked in $\nu^{(1)}$ and the $i_{3}$-string is not marked by [5], we do as follows according to the type of the selected $i_{3}$-string.

\begin{itemize}
\item[(1)] type-II.

We set the box-deleted $(i_{3}-1)$-string to be q-singular unconditionally.

\item[(2)] type-0/I.

If [4] is marked in $\nu^{(2)}$ with $i_{3}^{eff}=i_{4}$, we set the box-deleted $(i_{3}-1)$-string to be singular.
Otherwise, we set the box-deleted $(i_{3}-1)$-string to be q-singular. 

\end{itemize}

\item[(RA-3)] When $[4]$ is marked in $\nu^{(1)}$ with $i_{3}< i_{4}$ and the selected $i_{4}$-string is marked by $[4]$ only, we do as follows according to the type of the selected $i_{4}$-string.

\begin{itemize}
\item[(1)] type-0.

If $[5]$ is marked in a q-singular string in $\nu^{(2)}$ with $i_{4}^{eff}=i_{5}$ or $[5]$ and $[6]$ are marked in the same singular string in $\nu^{(2)}$ with $i_{4}^{eff}=i_{6}-1$, then we set the box-deleted $(i_{4}-1)$-string to be singular.
Otherwise, we set the box-deleted $(i_{4}-1)$-string to be qq-singular.

\item[(2)] type-I.

If $[5]$ is not marked in $\nu^{(2)}$, then we set the box-deleted $(i_{4}-1)$-string to be qq-singular.
If $[5]$ is marked in a q-singular string in $\nu^{(2)}$ with $i_{4}^{eff}=i_{5}$ or [5] and [6] are marked in the same singular string in $\nu^{(2)}$ with $i_{4}^{eff}=i_{6}-1$, then we set the box-deleted $(i_{4}-1)$-string to be q-singular.
Otherwise, we set the box-deleted $(i_{4}-1)$-string to be qq-singular.

\item[(3)] type-II.

We set the box-deleted $(i_{4}-1)$-string to be qq-singular unconditionally.
\end{itemize}
In either case, if the selected $i_{3}$-string is marked by [3] only, we set the box-deleted $(i_{3}-1)$-string to be q-singular.

\item[(RA-4)] When [4] is marked in $\nu^{(1)}$ with $i_{2}< i_{3}=i_{4}$ and the selected $i_{4}$-string is not marked by [7], we do as follows according to the type of the selected $i_{4}$-string.

\begin{itemize}
\item[(1)] type-0.

If [5] is marked in a q-singular string in $\nu^{(2)}$ with $i_{4}^{eff}=i_{5}$ or [5] and [6] are marked in the same singular string in $\nu^{(2)}$ with $i_{4}^{eff}=i_{6}-1$, then we set the box-deleted $(i_{4}-2)$-string to be singular.
Otherwise, we set the box-deleted $(i_{4}-2)$-string to be q-singular.

\item[(2)] type-I/II.

We set the box-deleted $(i_{4}-2)$-string to be q-singular unconditionally.

\end{itemize}

\item[(RA-5)] When [8] is marked in $\nu^{(1)}$ but [9] is not and the selected $i_{8}$-string is marked by [8] only, we set the box-deleted $(i_{8}-1)$-string in $\Tilde{\nu}^{(1)}$ to be q-singular whether [7] is marked in $\nu^{(1)}$ or $\nu^{(2)}$.

\item[(RA-6)] When [9] is marked in $\nu^{(1)}$ and the rightmost box of the selected $i_{9}$-string is marked by [8], i.e., $i_{8}=i_{9}$, we do as follows according to the type of the selected $i_{9}$-string whether [7] is marked in $\nu^{(1)}$ or $\nu^{(2)}$.

\begin{itemize}
\item[(1)] type-0.

If [10] is marked in $\nu^{(2)}$ with $i_{9}^{eff}=i_{10}$, then we set the box-deleted $(i_{9}-2)$-string to be singular.
Otherwise, we set this string to be q-singular.

\item[(2)] type-I/II.

We set the box-deleted $(i_{9}-2)$-string to be q-singular unconditionally.

\end{itemize}

\item[(RA-7)] When [9] are marked in $\nu^{(1)}$ with $i_{8}< i_{9}$, we do as follows according to the type of the $i_{9}$-string whether [7] is marked in $\nu^{(1)}$ or $\nu^{(2)}$.

\begin{itemize}
\item[(1)] type-0.

If [10] is marked in $\nu^{(2)}$ with $i_{9}^{eff}=i_{10}$, then we set the box-deleted $(i_{9}-1)$-string to be singular.
Otherwise, we set this string to be qq-singular.

\item[(2)] type-I.

If [10] is marked in $\nu^{(2)}$ with $i_{9}^{eff}=i_{10}$, then we set the box-deleted $(i_{9}-1)$-string to be q-singular.
Otherwise, we set this string to be qq-singular.

\item[(3)] type-II.

We set the box-deleted $(i_{9}-1)$-string to be qq-singular unconditionally.

\end{itemize}
In either case, if the selected $i_{8}$-string is marked by [8] only, we set the box-deleted $(i_{8}-1)$-string to be q-singular.

\end{description}

Let $A$ be a statement, then $\chi (A)=1$ if $A$ is true and $\chi (A)=0$ if $A$ is false.
The changes of vacancy numbers are defined by $\Delta p_{i}^{(a)}:=\Tilde{p}_{i}^{(a)}-p_{i}^{(a)}$.
Here we compute the changes of vacancy numbers according to the return $b$ of $\delta_{\theta}$.

\begin{description}
\item[(VC-1)] $b= \framebox{$1$}$.

From Eqs.~\eqref{eq:vacancy1} and \eqref{eq:vacancy2}, we have $\Delta p_{i}^{(1)}=0$ $(i\geq 1)$ and $\Delta p_{i}^{(2)}=-1$ $(i\geq 1)$.

\item[(VC-2)] $b= \framebox{$2$}$.

The changes of vacancy numbers $\Delta p_{i}^{(1)}$ are obtained by overwriting 
\begin{align*}
&\min (i,3i_{1})(m_{i_{1}}^{(2)}-1)+\min (i,3(i_{1}-1))(m_{i_{1}-1}^{(2)}+1) \\
&-\min (i,3i_{1})m_{i_{1}}^{(2)}-\min (i,3(i_{1}-1))m_{i_{1}-1}^{(2)} \\
&=-\min (i,3i_{1})+\min (i,3(i_{1}-1)) \\
&=-3\chi (i\geq 3i_{1})-2\chi (i=3i_{1}-1)-\chi (i=3i_{1}-2)
\end{align*}
on $\Delta p_{i}^{(1)}$ in \textbf{(VC-1)} and $\Delta p_{i}^{(2)}$ are obtained by overwriting 
\begin{align*}
-&2\min (i,i_{1})(m_{i_{1}}^{(2)}-1)-2\min (i,i_{1}-1)(m_{i_{1}-1}^{(2)}+1) \\
&+2\min (i,i_{1})m_{i_{1}}^{(2)}+2\min (i,i_{1}-1)m_{i_{1}-1}^{(2)} \\
&=2\min (i,i_{1})-2\min (i,i_{1}-1)=2\chi (i\geq i_{1})
\end{align*}
on $\Delta p_{i}^{(2)}$ in \textbf{(VC-1)}.
The changed $\Delta p_{i}^{(1)}$ and $\Delta p_{i}^{(2)}$ are listed below.

\begin{center}
\begin{tabular}{cl|cl}
$\Delta p_{i}^{(1)}$ & & $\Delta p_{i}^{(2)}$ & \\ \hline
$-3$ & $(i\geq 3i_{1})$ & 1 & $(i\geq i_{1})$ \\
$-2$ & $(i=3i_{1}-1)$ & & \\
$-1$ & $(i=3i_{1}-2)$ & & 
\end{tabular}
\end{center}

\item[(VC-3)] $b= \framebox{$3$}$.

The changes of vacancy numbers $\Delta p_{i}^{(1)}$ are obtained by overwriting $2\chi (i\geq i_{2})$ on 
$\Delta p_{i}^{(1)}$ in \textbf{(VC-2)} and $\Delta p_{i}^{(2)}$ are obtained by overwriting $-\chi (3i\geq i_{2})$ on $\Delta p_{i}^{(2)}$ in \textbf{(VC-2)}.

\begin{itemize}
\item[(1)] $i_{2}\leq 3i_{1}-1$.

\begin{center}
\begin{tabular}{cccl|cl}
$\Delta p_{i}^{(1)}$ & $i_{2}$:I & $i_{2}$:II &  & $\Delta p_{i}^{(2)}$ & \\ \hline
& $-1$ &$-1$ & $(i\geq 3i_{1})$ & 0 & $(i\geq i_{1})$ \\
& $0$  &$0$ & $(i=3i_{1}-1)$ & & \\
& & $1$ & $(i=3i_{1}-2)$ & & 
\end{tabular}
\end{center}
where $i_{2}$:I (resp. $i_{2}$:II) stands for that the $i_{2}$-string is type-I (resp. type-II).

\item[(2)] $i_{2}\geq3i_{1}$.

\begin{center}
\begin{tabular}{cl|cl}
$\Delta p_{i}^{(1)}$ & & $\Delta p_{i}^{(2)}$ & \\ \hline
$-1$ & $(i\geq i_{2})$ & 0 & $(i\geq i_{2}^{eff})$
\end{tabular}
\end{center}

\end{itemize}

\item[(VC-4)] $b= \framebox{$4$}$.

\begin{itemize}
\item[(1)] $i_{2}=i_{3}$.

The changes of vacancy numbers $\Delta p_{i}^{(1)}$ are obtained by overwriting $4\chi (i\geq i_{3})+2\chi (i=i_{3}-1)$ on $\Delta p_{i}^{(1)}$ in \textbf{(VC-2)} and $\Delta p_{i}^{(2)}$ are obtained by overwriting $-2\chi (3i\geq i_{3})-\chi (3i=i_{3}-1)$ on $\Delta p_{i}^{(2)}$ in \textbf{(VC-2)}.

\begin{itemize}
\item[(a)] $i_{3}^{eff}=i_{1}$.

\begin{center}
\begin{tabular}{cccl|cl}
$\Delta p_{i}^{(1)}$ & $i_{3}$:0 & $i_{3}$:I & & $\Delta p_{i}^{(2)}$ & \\ \hline
& $1$ & $1$ & $(i\geq 3i_{1})$ & $-1$ & $(i\geq i_{1})$ \\
& $0$ & $2$ & $(i=3i_{1}-1)$ & & \\
& $-1$ & $1$ & $(i=3i_{1}-2)$ & & 
\end{tabular}
\end{center}

\item[(b)] $i_{3}^{eff}>i_{1}$.
\begin{center}
\begin{tabular}{cl|cccl}
$\Delta p_{i}^{(1)}$ & & $\Delta p_{i}^{(2)}$ & $i_{3}$:0/I & $i_{3}$:II \\ \hline
$1$ & $(i\geq i_{3})$ & & $-1$ & $-1$ & $(i\geq i_{3}^{eff})$ \\
$-1$ & $(i=i_{3}-1)$ & & $1$ & $0$ & $(i=i_{3}^{eff}-1)$
\end{tabular}
\end{center}
where $i_{3}$:0/I (resp. $i_{3}$:II) stands for that the $i_{3}$-string is type-0/I (resp. type-II).

\end{itemize}

\item[(2)] $i_{2}<i_{3}$.

In this case, $i_{3}\geq 3i_{1}$.
The changes of vacancy numbers $\Delta p_{i}^{(1)}$ are obtained by overwriting $2\chi (i\geq i_{3})$ on 
$\Delta p_{i}^{(1)}$ in \textbf{(VC-3)} and $\Delta p_{i}^{(2)}$ are obtained by overwriting $-\chi (3i\geq i_{3})$ on $\Delta p_{i}^{(2)}$ in \textbf{(VC-3)}.

\begin{center}
\begin{tabular}{cl|cl}
$\Delta p_{i}^{(1)}$ & & $\Delta p_{i}^{(2)}$ & \\ \hline
$1$ & $(i\geq i_{3})$ & $-1$ & $(i\geq i_{3}^{eff})$ 
\end{tabular}
\end{center}

\end{itemize}

\item[(VC-5)] $b= \framebox{$5$}$.

The changes of vacancy numbers $\Delta p_{i}^{(1)}$ are obtained by overwriting 
\[
-3\chi (i\geq 3i_{4})-2\chi (i=3i_{4}-1)-\chi (i=3i_{4}-2)
\]
on $\Delta p_{i}^{(1)}$ in \textbf{(VC-4)} and $\Delta p_{i}^{(2)}$ are obtained by overwriting $2\chi (i\geq i_{4})$ on $\Delta p_{i}^{(2)}$ in \textbf{(VC-4)}.

\begin{itemize}
\item[(1)] $i_{2}=i_{3}$.

\begin{itemize}
\item[(a)] $i_{3}=3i_{4}-1$.

\begin{center}
\begin{tabular}{cccl|cl}
$\Delta p_{i}^{(1)}$ & $i_{1}=i_{4}$ & $i_{1}<i_{4}$ & & $\Delta p_{i}^{(2)}$ & \\ \hline
& $-2$ & $-2$ & $(i\geq 3i_{4})$ & $1$ & $(i\geq i_{1})$ \\
& $0$  & $-1$ & $(i=3i_{4}-1)$ & & \\
& $0$  & $-2$ & $(i=3i_{4}-2)$ & & 
\end{tabular}
\end{center}
Note that ``$i\geq i_{1}$'' in $\Delta p_{i}^{(2)}$ is not a typo.

\item[(b)] otherwise.

\begin{center}
\begin{tabular}{cl|cl}
$\Delta p_{i}^{(1)}$ & & $\Delta p_{i}^{(2)}$ & \\ \hline
$-2$ & $(i\geq 3i_{4})$ & $1$ & $(i\geq i_{4})$ \\
$-1$ & $(i=3i_{4}-1)$ & & \\
$0$ & $(i=3i_{4}-2)$ & & 
\end{tabular}
\end{center}

\end{itemize}
Note that $i_{3}^{eff}<i_{4}$ when the $i_{3}$-string is type-0.

\item[(2)] $i_{2}<i_{3}$.

\begin{itemize}

\item[(a)] $i_{3}^{eff}<i_{4}$.

\begin{center}
\begin{tabular}{cl|cl}
$\Delta p_{i}^{(1)}$ & & $\Delta p_{i}^{(2)}$ & \\ \hline
$-2$ & $(i\geq 3i_{4})$ & $1$ & $(i\geq i_{4})$ \\
$-1$ & $(i=3i_{4}-1)$ & & \\
$0$ & $(i=3i_{4}-2)$ & & 
\end{tabular}
\end{center}

\item[(b)] $i_{3}=3i_{4}-1$.

\begin{center}
\begin{tabular}{cl|cl}
$\Delta p_{i}^{(1)}$ & & $\Delta p_{i}^{(2)}$ & \\ \hline
$-2$ & $(i\geq 3i_{4})$ & $1$ & $(i\geq i_{4})$ \\
$-1$ & $(i=3i_{4}-1)$ & & \\
$-2$ & $(i=3i_{4}-2)$ & & 
\end{tabular}
\end{center}

\item[(c)] $i_{3}=3i_{4}-2$.

\begin{center}
\begin{tabular}{cl|cl}
$\Delta p_{i}^{(1)}$ & & $\Delta p_{i}^{(2)}$ & \\ \hline
$-2$ & $(i\geq 3i_{4})$ & $1$ & $(i\geq i_{4})$ \\
$-1$ & $(i=3i_{4}-1)$ & & \\
$0$ & $(i=3i_{4}-2)$ & & 
\end{tabular}
\end{center}

\end{itemize}
Note that the case when $i_{3}=3i_{4}$ cannot occur.

\end{itemize}

\item[(VC-6)] $b= \framebox{$6$}$.

\begin{itemize}
\item[(1)] $i_{2}=i_{3}=i_{4}$.

The changes of vacancy numbers $\Delta p_{i}^{(1)}$ are obtained by overwriting 
\[
6\chi (i\geq i_{4})+4\chi (i=i_{4}-1)+2\chi (i=i_{4}-2)
\]
on $\Delta p_{i}^{(1)}$ in \textbf{(VC-2)} and $\Delta p_{i}^{(2)}$ are obtained by overwriting 
\[
-3\chi (3i\geq i_{4})-2\chi (3i=i_{4}-1)-\chi (3i=i_{4}-2)
\]
on $\Delta p_{i}^{(2)}$ in \textbf{(VC-2)}.

\begin{itemize}
\item[(a)] The $i_{4}$-string is type-0 with $i_{4}^{eff}=i_{1}$.

\begin{center}
\begin{tabular}{cl|cl}
$\Delta p_{i}^{(1)}$ & & $\Delta p_{i}^{(2)}$ & \\ \hline
$3$ & $(i\geq i_{4})$ & $-2$ & $(i\geq i_{4}^{eff})$ \\
$2$ & $(i=i_{4}-1)$ & & \\
$1$ & $(i=i_{4}-2)$ & & 
\end{tabular}
\end{center}

\item[(b)] The $i_{4}$-string is type-0 with $i_{4}^{eff}>i_{1}$.

\begin{center}
\begin{tabular}{cl|cl}
$\Delta p_{i}^{(1)}$ & & $\Delta p_{i}^{(2)}$ & \\ \hline
$3$ & $(i\geq i_{4})$ & $-2$ & $(i\geq i_{4}^{eff})$ \\
$1$ & $(i=i_{4}-1)$ & $1$ & $(i_{1}\leq i\leq i_{4}^{eff}-1)$ \\
$-1$ & $(i=i_{4}-2)$ & & 
\end{tabular}
\end{center}

\item[(c)] The $i_{4}$-string is type-I.

\begin{center}
\begin{tabular}{cl|cl}
$\Delta p_{i}^{(1)}$ & & $\Delta p_{i}^{(2)}$ & \\ \hline
$3$ & $(i\geq i_{4})$ & $-2$ & $(i\geq i_{4}^{eff})$ \\
$1$ & $(i=i_{4}-1)$   & $0$ & $(i=i_{4}^{eff}-1)$\\
$-1$ & $(i=i_{4}-2)$ & & 
\end{tabular}
\end{center}

\item[(d)] The $i_{4}$-string is type-II.

\begin{center}
\begin{tabular}{cccl}
$\Delta p_{i}^{(1)}$ & $i_{4}^{eff}=i_{1}+1$ & $i_{4}^{eff}>i_{1}+1$ & \\ \hline
& $3$ & $3$ & $(i\geq i_{4})$ \\
& $1$ & $1$ & $(i=i_{4}-1)$ \\
& $0$ & $-1$ & $(i=i_{4}-2)$
\end{tabular}
\end{center}

\begin{center}
\begin{tabular}{cl}
$\Delta p_{i}^{(2)}$ & \\ \hline
$-2$ & $(i\geq i_{4}^{eff})$ \\
$-1$ & $(i=i_{4}^{eff}-1)$
\end{tabular}
\end{center}

\end{itemize}

\item[(2)] $i_{2}<i_{3}=i_{4}$.

The changes of vacancy numbers $\Delta p_{i}^{(1)}$ are obtained by overwriting $4\chi (i\geq i_{4})+2\chi (i=i_{4}-1)$ on $\Delta p_{i}^{(1)}$ in \textbf{(VC-3)} and $\Delta p_{i}^{(2)}$ are obtained by overwriting $-2\chi (3i\geq i_{4})-\chi (3i=i_{4}-1)$ on $\Delta p_{i}^{(2)}$ in \textbf{(VC-3)}.

\begin{center}
\begin{tabular}{cl|cccl}
$\Delta p_{i}^{(1)}$ & & $\Delta p_{i}^{(2)}$ & $i_{4}$:0/I & $i_{4}$:II \\ \hline
$3$ & $(i\geq i_{4})$ & & $-2$ & $-2$ & $(i\geq i_{4}^{eff})$ \\
$1$ & $(i=i_{4}-1)$ & & $0$ & $-1$ & $(i=i_{4}^{eff}-1)$
\end{tabular}
\end{center}

\item[(3)] $i_{3}<i_{4}$.

The changes of vacancy numbers $\Delta p_{i}^{(1)}$ are obtained by overwriting $2\chi (i\geq i_{4})$ on 
$\Delta p_{i}^{(1)}$ in \textbf{(VC-4)} and $\Delta p_{i}^{(2)}$ are obtained by overwriting $-\chi (3i\geq i_{4})$ on $\Delta p_{i}^{(2)}$ in \textbf{(VC-4)}.

\begin{center}
\begin{tabular}{cl|cl}
$\Delta p_{i}^{(1)}$ & & $\Delta p_{i}^{(2)}$ & \\ \hline
$3$ & $(i\geq i_{4})$ & $-2$ & $(i\geq i_{4}^{eff})$ 
\end{tabular}
\end{center}

\end{itemize}

\item[(VC-7)] $b= \framebox{$7$}$.

\begin{itemize}
\item[(1)] $i_{2}=i_{3}=i_{5}$.

In this case, $i_{5}=3i_{4}$.
The changes of vacancy numbers $\Delta p_{i}^{(1)}$ are obtained by overwriting 
\begin{align*}
&6\chi (i\geq i_{5})+4\chi (i=i_{5}-1)+2\chi (i=i_{5}-2) \\
&-3\chi (i\geq 3i_{4})-2\chi (i=3i_{4}-1)-\chi (i=3i_{4}-2) \\
=& 3\chi (i\geq i_{5})+2\chi (i=i_{5}-1)+\chi (i=i_{5}-2)
\end{align*}
on $\Delta p_{i}^{(1)}$ in \textbf{(VC-2)} and $\Delta p_{i}^{(2)}$ are obtained by overwriting 
\begin{align*}
&-3\chi (3i\geq i_{5})-2\chi (3i=i_{5}-1)-\chi (3i=i_{5}-2)+2\chi (i\geq i_{4}) \\
=&-\chi (i\geq i_{4})
\end{align*}
on $\Delta p_{i}^{(2)}$ in \textbf{(VC-2)}.

\begin{center}
\begin{tabular}{cccl|cl}
$\Delta p_{i}^{(1)}$ & $i_{1}=i_{4}$ & $i_{1}<i_{4}$ & & $\Delta p_{i}^{(2)}$ & \\ \hline
& $0$ & $0$ & $(i\geq i_{5})$ & $0$ & $(i\geq i_{4})$ \\
& $0$ & $-1$ & $(i=i_{5}-1)$ & & \\
& $0$ & $-2$ & $(i=i_{5}-2)$ & & 
\end{tabular}
\end{center}

\item[(2)] $i_{2}<i_{3}=i_{5}$.

In this case $i_{5}=3i_{4}$.
The changes of vacancy numbers $\Delta p_{i}^{(1)}$ are obtained by overwriting 
\begin{align*}
&4\chi (i\geq i_{5})+2\chi (i=i_{5}-1) \\
&-3\chi (i\geq 3i_{4})-2\chi (i=3i_{4}-1)-\chi (i=3i_{4}-2) \\
=& \chi (i\geq i_{5})-\chi (i=i_{5}-2)
\end{align*}
on $\Delta p_{i}^{(1)}$ in \textbf{(VC-3)} and $\Delta p_{i}^{(2)}$ are obtained by overwriting 
\[
-2\chi (3i\geq i_{5})-\chi (3i=i_{5}-1)+2\chi (i\geq i_{4})=0
\]
on $\Delta p_{i}^{(2)}$ in \textbf{(VC-3)}.
The changes of vacancy numbers coincide with those in case (1).

\item[(3)] $i_{3}<i_{5}$.

The changes of vacancy numbers $\Delta p_{i}^{(1)}$ are obtained by overwriting $2\chi (i\geq i_{5})$ on 
$\Delta p_{i}^{(1)}$ in \textbf{(VC-5)} and the changes of vacancy numbers $\Delta p_{i}^{(2)}$ are obtained by overwriting $-\chi (3i\geq i_{5})$ on $\Delta p_{i}^{(2)}$ in \textbf{(VC-5)}.

\begin{center}
\begin{tabular}{cl|cl}
$\Delta p_{i}^{(1)}$ & & $\Delta p_{i}^{(2)}$ & \\ \hline
$0$ & $(i\geq i_{5})$ & $0$ & $(i\geq i_{5}^{eff})$ 
\end{tabular}
\end{center}

\end{itemize}

\item[(VC-8)] $b= \framebox{$8$}$.

The changes of vacancy numbers $\Delta p_{i}^{(1)}$ are obtained by overwriting 
\[
-3\chi (i\geq 3i_{5})-2\chi (i=3i_{5}-1)-\chi (i=3i_{5}-2)
\]
on $\Delta p_{i}^{(1)}$ in \textbf{(VC-6)} and $\Delta p_{i}^{(2)}$ are obtained by overwriting $2\chi (i\geq i_{5})$ on $\Delta p_{i}^{(2)}$ in \textbf{(VC-6)}.

\begin{itemize}
\item[(1)] $i_{4}^{eff}<i_{5}$.

\begin{center}
\begin{tabular}{cl|cl}
$\Delta p_{i}^{(1)}$ & & $\Delta p_{i}^{(2)}$ & \\ \hline
$0$ & $(i\geq 3i_{5})$ & $0$ & $(i\geq i_{5})$ \\
$1$ & $(i=3i_{5}-1)$ & & \\
$2$ & $(i=3i_{5}-2)$ & & 
\end{tabular}
\end{center}

\item[(2)] $i_{2}=i_{3}=i_{4}$ and $i_{4}^{eff}=i_{5}$.

In this case, the $i_{4}$-string is type-0/I.

\begin{center}
\begin{tabular}{cccl|cl}
$\Delta p_{i}^{(1)}$ & $i_{1}=i_{5}$ & $i_{1}<i_{5}$ & & $\Delta p_{i}^{(2)}$ & \\ \hline
& $0$ & $0$  & $(i\geq 3i_{5})$ & $0$ & $(i\geq i_{5})$ \\
& $0$ & $-1$ & $(i=3i_{5}-1)$ & & \\
& $0$ & $-2$ & $(i=3i_{5}-2)$ & & 
\end{tabular}
\end{center}

\item[(3)] $i_{2}<i_{3}=i_{4}$ and $i_{4}^{eff}=i_{5}$.

\begin{center}
\begin{tabular}{ccccl|cl}
$\Delta p_{i}^{(1)}$ & $i_{4}$:0 & $i_{4}$:I & $i_{4}$:II & & $\Delta p_{i}^{(2)}$ & \\ \hline
& $0$  & $0$ & $0$ & $(i\geq 3i_{5})$ & $0$ & $(i\geq i_{5})$ \\
& $-1$ & $1$ & $1$ & $(i=3i_{5}-1)$ & & \\
& $-2$ & $0$ & $2$ & $(i=3i_{5}-2)$ & & 
\end{tabular}
\end{center}

\item[(4)] $i_{3}<i_{4}$ and $i_{4}^{eff}=i_{5}$.

\begin{center}
\begin{tabular}{ccccl|cl}
$\Delta p_{i}^{(1)}$ & $i_{4}$:0 & $i_{4}$:I & $i_{4}$:II & & $\Delta p_{i}^{(2)}$ & \\ \hline
& $0$  & $0$ & $0$ & $(i\geq 3i_{5})$ & $0$ & $(i\geq i_{5})$ \\
& $-1$ & $1$ & $1$ & $(i=3i_{5}-1)$ & & \\
& $0$  & $0$ & $2$ & $(i=3i_{5}-2)$ & & 
\end{tabular}
\end{center}
Note that $i_{3}\leq 3i_{5}-2$ when $i_{4}=3i_{5}$ (see \textbf{(BM-4)}).

\end{itemize}

\item[(VC-9)] $b= \framebox{$9$}$.

\begin{itemize}
\item[(1)] $i_{2}=i_{3}=i_{5}=i_{6}$.

In this case, the $i_{6}$-string is type-II and $i_{6}^{eff}=i_{4}+1$.
The changes of vacancy numbers $\Delta p_{i}^{(1)}$ are obtained by overwriting 
\begin{align*}
&8\chi (i\geq i_{6})+6\chi (i=i_{6}-1)+4\chi (i=i_{6}-2)+2\chi (i=i_{6}-3) \\
&-3\chi (i\geq 3i_{4})-2\chi (i=3i_{4}-1)-\chi (i=3i_{4}-2) \\
=& 5\chi (i\geq i_{6})+3\chi (i=i_{6}-1)+2\chi (i=i_{6}-2)+\chi (i=i_{6}-3)
\end{align*}
on $\Delta p_{i}^{(1)}$ in \textbf{(VC-2)} and $\Delta p_{i}^{(2)}$ are obtained by overwriting 
\begin{align*}
&-4\chi (3i\geq i_{6})-3\chi (3i=i_{6}-1)-2\chi (3i=i_{6}-2) \\
&-\chi (3i=i_{6}-3)+2\chi (i\geq i_{4}) =-2\chi (i\geq i_{4}+1)-\chi (i=i_{4})
\end{align*}
on $\Delta p_{i}^{(2)}$ in \textbf{(VC-2)}.

\begin{center}
\begin{tabular}{cccl|cl}
$\Delta p_{i}^{(1)}$ & $i_{1}=i_{4}$ & $i_{1}<i_{4}$ & & $\Delta p_{i}^{(2)}$ & \\ \hline
& $2$ & $2$ & $(i\geq i_{6})$ & $-1$ & $(i\geq i_{4}+1)$ \\
& $0$ & $0$ & $(i=i_{6}-1)$ & $0$ & $(i=i_{4})$ \\
& $0$ & $-1$ & $(i=i_{6}-2)$ & & \\
& $0$ & $-2$ & $(i=i_{6}-3)$ & &
\end{tabular}
\end{center}

\item[(2)] $i_{2}<i_{3}=i_{5}=i_{6}$.

In this case, the $i_{6}$-string is type-II and $i_{6}^{eff}=i_{4}+1$.
The changes of vacancy numbers $\Delta p_{i}^{(1)}$ are obtained by overwriting 
\begin{align*}
&6\chi (i\geq i_{6})+4\chi (i=i_{6}-1)+2\chi (i=i_{6}-2) \\
&-3\chi (i\geq 3i_{4})-2\chi (i=3i_{4}-1)-\chi (i=3i_{4}-2) \\
=& 3\chi (i\geq i_{6})+\chi (i=i_{6}-1)-\chi (i=i_{6}-3)
\end{align*}
on $\Delta p_{i}^{(1)}$ in \textbf{(VC-3)} and $\Delta p_{i}^{(2)}$ are obtained by overwriting 
\begin{align*}
&-3\chi (3i\geq i_{6})-2\chi (3i=i_{6}-1)-\chi (3i=i_{6}-2) \\
&+2\chi (i\geq i_{4}) =-\chi (i\geq i_{4}+1)
\end{align*}
on $\Delta p_{i}^{(2)}$ in \textbf{(VC-3)}.
The changes of vacancy numbers coincide with those in case (1).

\item[(3)] $i_{3}<i_{5}=i_{6}$.

The changes of vacancy numbers $\Delta p_{i}^{(1)}$ are obtained by overwriting $4\chi (i\geq i_{6})+2\chi (i=i_{6}-1)$ on $\Delta p_{i}^{(1)}$ in \textbf{(VC-5)} and $\Delta p_{i}^{(2)}$ are obtained by overwriting $-2\chi (3i\geq i_{6})-\chi (3i=i_{6}-1)$ on $\Delta p_{i}^{(2)}$ in \textbf{(VC-5)}.

\begin{center}
\begin{tabular}{cl|cccl}
$\Delta p_{i}^{(1)}$ & & $\Delta p_{i}^{(2)}$ & $i_{6}$:0/I & $i_{6}$:II & \\ \hline
$2$ & $(i\geq i_{6})$ & & $-1$ & $-1$ & $(i\geq i_{6}^{eff})$ \\
$0$ & $(i=i_{6}-1)$   & & $1$ & $0$ & $(i=i_{6}^{eff}-1)$
\end{tabular}
\end{center}

\item[(4)] $i_{5}<i_{6}$.

The changes of vacancy numbers $\Delta p_{i}^{(1)}$ are obtained by overwriting $2\chi (i\geq i_{6})$ on 
$\Delta p_{i}^{(1)}$ in \textbf{(VC-7)} and $\Delta p_{i}^{(2)}$ are obtained by overwriting $-\chi (3i\geq i_{6})$ on $\Delta p_{i}^{(2)}$ in \textbf{(VC-7)}.

\begin{center}
\begin{tabular}{cl|cl}
$\Delta p_{i}^{(1)}$ & & $\Delta p_{i}^{(2)}$ & \\ \hline
$2$ & $(i\geq i_{6})$ & $-1$ & $(i\geq i_{6}^{eff})$ 
\end{tabular}
\end{center}

\end{itemize}

\item[(VC-10)] $b= \framebox{$10$}$.

\begin{itemize}
\item[(1)] $i_{5}=i_{6}$.

In this case $i_{6}\geq i_{4}^{eff}+1$.
The changes of vacancy numbers $\Delta p_{i}^{(1)}$ are obtained by overwriting 
\[
-6\chi (i\geq 3i_{6})-5\chi (i=3i_{6}-1)-\cdots -\chi (i=3i_{6}-5)
\]
on $\Delta p_{i}^{(1)}$ in \textbf{(VC-6)} and $\Delta p_{i}^{(2)}$ are obtained by overwriting $4\chi (i\geq i_{6})+2\chi (i=i_{6}-1)$ on $\Delta p_{i}^{(2)}$ in \textbf{(VC-6)}.

\begin{center}
\begin{tabular}{cl|cl}
$\Delta p_{i}^{(1)}$ & & $\Delta p_{i}^{(2)}$ & \\ \hline
$-3$ & $(i\geq 3i_{6})$ & $2$ & $(i\geq i_{6})$ \\
$-2$ & $(i=3i_{6}-1)$ & $0$ & $(i=i_{6}-1)$ \\
$-1$ & $(i=3i_{6}-2)$ & & \\
$0$  & $(i=3i_{6}-3)$ & &
\end{tabular}
\end{center}
We omit $\Delta p_{i}^{(1)}$ $(3i_{6}-5\leq i\leq 3i_{6}-4)$.

\item[(2)] $i_{5}<i_{6}$.

The changes of vacancy numbers $\Delta p_{i}^{(1)}$ are obtained by overwriting 
\[
-3\chi (i\geq 3i_{6})-2\chi (i=3i_{6}-1)-\chi (i=3i_{6}-2)
\]
on $\Delta p_{i}^{(1)}$ in \textbf{(VC-8)} and $\Delta p_{i}^{(2)}$ are obtained by overwriting $2\chi (i\geq i_{6})$ on $\Delta p_{i}^{(2)}$ in \textbf{(VC-8)}.

\begin{center}
\begin{tabular}{cl|cl}
$\Delta p_{i}^{(1)}$ & & $\Delta p_{i}^{(2)}$ & \\ \hline
$-3$ & $(i\geq 3i_{6})$ & $2$ & $(i\geq i_{6})$ \\
$-2$ & $(i=3i_{6}-1)$ & $0$ & $(i=i_{6}-1)$ \\
$-1$ & $(i=3i_{6}-2)$ & & 
\end{tabular}
\end{center}

\end{itemize}

\item[(VC-11)] $b= \framebox{$11$}$.

\begin{itemize}
\item[(1)] The $i_{7}$-string is in $\nu^{(1)}$.

In this case [7] is marked in the rightmost box of the $i_{7}$-string.
The changes of vacancy numbers $\Delta p_{i}^{(1)}$ are obtained by overwriting $2\chi (i\geq i_{7})$ on 
$\Delta p_{i}^{(1)}$ in \textbf{(VC-10)} and $\Delta p_{i}^{(2)}$ are obtained by overwriting $-\chi (3i\geq i_{7})$ on $\Delta p_{i}^{(2)}$ in \textbf{(VC-10)}.

\begin{itemize}
\item[(a)] $i_{7}^{eff}=i_{6}$.

\begin{center}
\begin{tabular}{ccccl|cl}
$\Delta p_{i}^{(1)}$ & $i_{7}$:0 & $i_{7}$:I & $i_{7}$:II & & $\Delta p_{i}^{(2)}$ & \\ \hline
& $-1$ & $-1$ & $-1$ & $(i\geq 3i_{6})$ & $1$ & $(i\geq i_{6})$ \\
&         & $0$ & $0$ & $(i\geq 3i_{6}-1)$ & & \\
&         &           & $1$ & $(i\geq 3i_{6}-2)$ & &
\end{tabular}
\end{center}

\item[(b)] $i_{7}^{eff}>i_{6}$.
\begin{center}
\begin{tabular}{cl|cl}
$\Delta p_{i}^{(1)}$ & & $\Delta p_{i}^{(2)}$ & \\ \hline
$-1$ & $(i\geq i_{7})$ & $1$ & $(i\geq i_{7}^{eff})$
\end{tabular}
\end{center}

\end{itemize}

\item[(2)] The $i_{7}$-string is in $\nu^{(2)}$ and $i_{4}=i_{7}$.

We have

\begin{center}
\begin{tabular}{cl|cl}
$\Delta p_{i}^{(1)}$ & & $\Delta p_{i}^{(2)}$ & \\ \hline
$-1$ & $(i\geq 3i_{7})$ & $1$ & $(i\geq i_{7})$ \\
$0$ & $(i=3i_{7}-1)$ & $0$ & $(i=i_{7}-1)$ \\
$1$ & $(i=3i_{7}-2)$ & & \\
$0$   & $(i=3i_{7}-3)$ & &
\end{tabular}
\end{center}
for all possible box marking patterns of $\nu^{(1)}$, 

\setlength{\unitlength}{10pt}
\begin{center}
\begin{picture}(18,3)
\put(0,0){\line(1,0){5}}
\put(0,1){\line(1,0){5}}
\put(1,0){\line(0,1){1}}
\put(2,0){\line(0,1){1}}
\put(3,0){\line(0,1){1}}
\put(4,0){\line(0,1){1}}
\put(5,0){\line(0,1){1}}
\put(3,1){\makebox(2,1){$\downarrow$}}
\put(1,0){\makebox(1,1){{\scriptsize $[6]$}}}
\put(2,0){\makebox(1,1){{\scriptsize $[5]$}}}
\put(3,0){\makebox(1,1){{\scriptsize $[3]$}}}
\put(4,0){\makebox(1,1){{\scriptsize $[2]$}}}
\put(0,2){\makebox(1,1){{\small (a)}}}

\put(8,0){\line(1,0){4}}
\put(8,1){\line(1,0){4}}
\put(9,0){\line(0,1){1}}
\put(10,0){\line(0,1){1}}
\put(11,0){\line(0,1){1}}
\put(12,0){\line(0,1){1}}
\put(10,1){\makebox(2,1){$\downarrow$}}
\put(9,0){\makebox(1,1){{\scriptsize $[6]$}}}
\put(10,0){\makebox(1,1){{\scriptsize $[5]$}}}
\put(11,0){\makebox(1,1){{\scriptsize $[3]$}}}
\put(8,2){\makebox(1,1){{\small (b)}}}

\put(15,0){\line(1,0){3}}
\put(15,1){\line(1,0){3}}
\put(16,0){\line(0,1){1}}
\put(17,0){\line(0,1){1}}
\put(18,0){\line(0,1){1}}
\put(16,1){\makebox(2,1){$\downarrow$}}
\put(16,0){\makebox(1,1){{\scriptsize $[6]$}}}
\put(17,0){\makebox(1,1){{\scriptsize $[5]$}}}
\put(15,2){\makebox(1,1){{\small (c)}}}
\end{picture}
\end{center}
where the $i_{6}$-string is always type-II.
We omit the unmarked strings.
In case (a), $\Delta p_{i}^{(1)}$ $(3i_{7}-5 \leq i\leq 3i_{7}-4)$ are obtained by overwriting $2\chi (i=3i_{7}-4)+\chi (i=3i_{7}-5)$ on $\Delta p_{i}^{(1)}$ in \textbf{(VC-2)} and $\Delta p_{i}^{(1)}$ $(i\leq 3i_{7}-6)$ and $\Delta p_{i}^{(2)}$ $(i\leq i_{7}-2)$ are the same as in \textbf{(VC-2)}.
In case (b), $\Delta p_{i}^{(1)}$ $(3i_{7}-5 \leq i\leq 3i_{7}-4)$ are obtained by overwriting $-\chi (i=3i_{7}-5)$ on $\Delta p_{i}^{(1)}$ in \textbf{(VC-3)} and $\Delta p_{i}^{(1)}$ $(i\leq 3i_{7}-6)$ and $\Delta p_{i}^{(2)}$ $(i\leq i_{7}-2)$ are the same as in \textbf{(VC-3)}.
In case (c), $\Delta p_{i}^{(1)}$ $(3i_{7}-5 \leq i\leq 3i_{7}-4)$ are obtained by overwriting $-2\chi (i=3i_{7}-4)-\chi (i=3i_{7}-5)$ on $\Delta p_{i}^{(1)}$ in \textbf{(VC-4)} and $\Delta p_{i}^{(1)}$ $(i\leq 3i_{7}-6)$ and $\Delta p_{i}^{(2)}$ $(i\leq i_{7}-2)$ are the same as in \textbf{(VC-4)}.

\item[(3)] The $i_{7}$-string is in $\nu^{(2)}$ and $i_{4}<i_{7}$.

\begin{center}
\begin{tabular}{cl|cl}
$\Delta p_{i}^{(1)}$ & & $\Delta p_{i}^{(2)}$ & \\ \hline
$-1$ & $(i\geq 3i_{7})$ & $1$ & $(i\geq i_{7})$ \\
$0$ & $(i=3i_{7}-1)$ & & \\
$1$ & $(i=3i_{7}-2)$ & &
\end{tabular}
\end{center}
The changes of vacancy numbers $\Delta p_{i}^{(1)}$ $(i\leq 3i_{7}-3)$ and $\Delta p_{i}^{(2)}$ $(i\leq i_{7}-1)$ are the same as in \textbf{(VC-9)}.

\end{itemize}

\item[(VC-12)] $b= \framebox{$12$}$.

\begin{itemize}
\item[(1)] The selected $i_{7}$-string is in $\nu^{(1)}$ and $i_{4}=i_{7}=i_{8}$.

In this case $i_{5}=i_{6}$.
Possible box marking patterns in $\nu^{(1)}$ are following two (see \textbf{(BM-8)}).

\setlength{\unitlength}{10pt}
\begin{center}
\begin{picture}(14,3)
\put(0,0){\line(1,0){6}}
\put(0,1){\line(1,0){6}}
\put(1,0){\line(0,1){1}}
\put(2,0){\line(0,1){1}}
\put(3,0){\line(0,1){1}}
\put(4,0){\line(0,1){1}}
\put(5,0){\line(0,1){1}}
\put(6,0){\line(0,1){1}}
\put(3,1){\makebox(2,1){$\downarrow$}}
\put(1,0){\makebox(1,1){{\scriptsize $[8]$}}}
\put(2,0){\makebox(1,1){{\scriptsize $[7]$}}}
\put(3,0){\makebox(1,1){{\scriptsize $[4]$}}}
\put(4,0){\makebox(1,1){{\scriptsize $[3]$}}}
\put(5,0){\makebox(1,1){{\scriptsize $[2]$}}}
\put(0,2){\makebox(1,1){{\small (a)}}}

\put(9,0){\line(1,0){5}}
\put(9,1){\line(1,0){5}}
\put(10,0){\line(0,1){1}}
\put(11,0){\line(0,1){1}}
\put(12,0){\line(0,1){1}}
\put(13,0){\line(0,1){1}}
\put(14,0){\line(0,1){1}}
\put(11,1){\makebox(2,1){$\downarrow$}}
\put(10,0){\makebox(1,1){{\scriptsize $[8]$}}}
\put(11,0){\makebox(1,1){{\scriptsize $[7]$}}}
\put(12,0){\makebox(1,1){{\scriptsize $[4]$}}}
\put(13,0){\makebox(1,1){{\scriptsize $[3]$}}}
\put(9,2){\makebox(1,1){{\small (b)}}}

\end{picture}
\end{center}

\begin{center}
\begin{tabular}{cccl|cl}
$\Delta p_{i}^{(1)}$ & $i_{6}=i_{1}+1$ & $i_{6}>i_{1}+1$ & & $\Delta p_{i}^{(2)}$ & \\ \hline
& $1$ & $1$ & $(i\geq 3i_{6})$ & $0$ & $(i\geq i_{6}-1)$ \\
& $2$ & $2$ & $(i=3i_{6}-1)$ & & \\
& $1$ & $1$ & $(i=3i_{6}-2)$ & & \\
& $0$ & $0$ & $(i=3i_{6}-3)$ & & \\
& $0$ & $-1$ & $(i=3i_{6}-4)$ & & \\
& $0$ & $-2$ & $(i=3i_{6}-5)$ & &
\end{tabular}
\end{center}
The changes of vacancy numbers $\Delta p_{i}^{(1)}$ $(i\leq 3i_{6}-6)$ and $\Delta p_{i}^{(2)}$ $(i\leq i_{6}-2)$ are the same as in \textbf{(VC-2)} (case (a)) and \textbf{(VC-3)} (case (b)).

\item[(2)] The selected $i_{7}$-string is in $\nu^{(1)}$ and $i_{4}<i_{7}=i_{8}$.

\begin{itemize}
\item[(a)] $i_{8}^{eff}=i_{6}$.

\begin{center}
\begin{tabular}{ccccl}
$\Delta p_{i}^{(1)}$ & $i_{8}$:0 & $i_{8}$:I & $i_{8}$:II &  \\ \hline
& $1$ & $1$ & $1$ & $(i\geq 3i_{6})$ \\
& $0$ & $2$ & $2$ & $(i=3i_{6}-1)$ \\
&       & $1$ & $3$ & $(i=3i_{6}-2)$ \\
&       &       & $2$ & $(i=3i_{6}-3)$  
\end{tabular}
\end{center}

\begin{center}
\begin{tabular}{cccl}
$\Delta p_{i}^{(2)}$ & $i_{8}$:0/I & $i_{8}$:II & \\ \hline
& $0$ & $0$ & $(i\geq i_{6})$ \\
& $0$ & $-1$ & $(i=i_{6}-1)$
\end{tabular}
\end{center}

\item[(b)] $i_{8}^{eff}>i_{6}$.

\begin{center}
\begin{tabular}{cl|cccl}
$\Delta p_{i}^{(1)}$ & & $\Delta p_{i}^{(2)}$ & $i_{8}$:0/I & $i_{8}$:II & \\ \hline
$1$ & $(i\geq i_{8})$ & & $0$ & $0$ & $(i\geq i_{8}^{eff})$ \\
$2$ & $(i=i_{8}-1)$   & & $0$ & $-1$ & $(i=i_{8}^{eff}-1)$
\end{tabular}
\end{center}

\end{itemize}
The changes of vacancy numbers $\Delta p_{i}^{(1)}$ $(i\leq i_{8}-2)$ and $\Delta p_{i}^{(2)}$ $(i\leq i_{8}^{eff}-2)$ are the same as in \textbf{(VC-10)}.

\item[(3)] The selected $i_{7}$-string is in $\nu^{(1)}$ and $i_{7}<i_{8}$.

\begin{center}
\begin{tabular}{cl|cl}
$\Delta p_{i}^{(1)}$ & & $\Delta p_{i}^{(2)}$ & \\ \hline
$1$ & $(i\geq i_{8})$ & $0$ & $(i\geq i_{8}^{eff})$ 
\end{tabular}
\end{center}
The changes of vacancy numbers $\Delta p_{i}^{(1)}$ $(i\leq i_{8}-1)$ and $\Delta p_{i}^{(2)}$ $(i\leq i_{8}^{eff}-1)$ are the same as in \textbf{(VC-11)} (1).

\item[(4)] The selected $i_{7}$-string is in $\nu^{(2)}$ and $i_{4}=i_{7}$.

In this case $i_{8}^{eff}=i_{7}$.
The only possible box marking in $\nu^{(1)}$ is

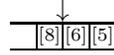
\setlength{\unitlength}{10pt}
\begin{center}
\begin{picture}(4,2)
\put(0,0){\line(1,0){4}}
\put(0,1){\line(1,0){4}}
\put(1,0){\line(0,1){1}}
\put(2,0){\line(0,1){1}}
\put(3,0){\line(0,1){1}}
\put(4,0){\line(0,1){1}}
\put(1,1){\makebox(2,1){$\downarrow$}}
\put(1,0){\makebox(1,1){{\scriptsize $[8]$}}}
\put(2,0){\makebox(1,1){{\scriptsize $[6]$}}}
\put(3,0){\makebox(1,1){{\scriptsize $[5]$}}}

\end{picture}
\end{center}

\begin{center}
\begin{tabular}{cccl|cl}
$\Delta p_{i}^{(1)}$ & $i_{7}=i_{1}+1$ & $i_{7}>i_{1}+1$ & & $\Delta p_{i}^{(2)}$ & \\ \hline
& $1$ & $1$ & $(i\geq 3i_{7})$ & $0$ & $(i\geq i_{7}-1)$ \\
& $2$ & $2$ & $(i=3i_{7}-1)$ & & \\
& $1$ & $1$ & $(i=3i_{7}-2)$ & & \\
& $0$ & $0$ & $(i=3i_{7}-3)$ & & \\
& $0$ & $-1$ & $(i=3i_{7}-4)$ & & \\
& $0$ & $0$ & $(i=3i_{7}-5)$ & &
\end{tabular}
\end{center}
The changes of vacancy numbers $\Delta p_{i}^{(1)}$ $(i\leq 3i_{7}-6)$ and $\Delta p_{i}^{(2)}$ $(i\leq i_{7}-2)$ are the same as in \textbf{(VC-4)}.

\item[(5)] The selected $i_{7}$-string is in $\nu^{(2)}$ and $i_{4}<i_{7}$.

\begin{itemize}
\item[(a)] The box marking in $\nu^{(1)}$ is 

\setlength{\unitlength}{10pt}
\begin{center}
\begin{picture}(4,2)
\put(0,0){\line(1,0){4}}
\put(0,1){\line(1,0){4}}
\put(1,0){\line(0,1){1}}
\put(2,0){\line(0,1){1}}
\put(3,0){\line(0,1){1}}
\put(4,0){\line(0,1){1}}
\put(1,1){\makebox(2,1){$\downarrow$}}
\put(1,0){\makebox(1,1){{\scriptsize $[8]$}}}
\put(2,0){\makebox(1,1){{\scriptsize $[6]$}}}
\put(3,0){\makebox(1,1){{\scriptsize $[5]$}}}

\end{picture}
\end{center}

\begin{center}
\begin{tabular}{cl|cl}
$\Delta p_{i}^{(1)}$ & & $\Delta p_{i}^{(2)}$ & \\ \hline
$1$ & $(i\geq 3i_{7})$ & $0$ & $(i\geq i_{7}-1)$ \\
$2$ & $(i=3i_{7}-1)$ & & \\
$1$ & $(i=3i_{7}-2)$ & & \\
$0$ & $(i=3i_{7}-3)$ & & 
\end{tabular}
\end{center}
The changes of vacancy numbers $\Delta p_{i}^{(1)}$ $(i\leq 3i_{7}-4)$ and $\Delta p_{i}^{(2)}$ $(i\leq i_{7}-2)$ are the same as in \textbf{(VC-5)}.

\item[(b)] The box marking in $\nu^{(1)}$ is 

\setlength{\unitlength}{10pt}
\begin{center}
\begin{picture}(4,2)
\put(0,0){\line(1,0){3}}
\put(0,1){\line(1,0){3}}
\put(1,0){\line(0,1){1}}
\put(2,0){\line(0,1){1}}
\put(3,0){\line(0,1){1}}
\put(0,1){\makebox(2,1){$\downarrow$}}
\put(1,0){\makebox(1,1){{\scriptsize $[8]$}}}
\put(2,0){\makebox(1,1){{\scriptsize $[6]$}}}

\end{picture}
\end{center}

\begin{center}
\begin{tabular}{cl|cl}
$\Delta p_{i}^{(1)}$ & & $\Delta p_{i}^{(2)}$ & \\ \hline
$1$ & $(i\geq 3i_{7})$ & $-$ & \\
$2$ & $(i=3i_{7}-1)$ & & \\
$1$ & $(i=3i_{7}-2)$ & & 
\end{tabular}
\end{center}
The changes of vacancy numbers $\Delta p_{i}^{(1)}$ $(i\leq 3i_{7}-3)$ and $\Delta p_{i}^{(2)}$ are the same as in \textbf{(VC-7)}.

\item[(c)] $i_{8}^{eff}>i_{7}$.

\begin{center}
\begin{tabular}{cl|cl}
$\Delta p_{i}^{(1)}$ & & $\Delta p_{i}^{(2)}$ & \\ \hline
$1$ & $(i\geq i_{8})$ & $0$ & $(i\geq i_{8}^{eff})$ 
\end{tabular}
\end{center}
The changes of vacancy numbers $\Delta p_{i}^{(1)}$ $(i\leq i_{8}-1)$ and $\Delta p_{i}^{(2)}$ $(i\leq i_{8}^{eff}-1)$ are the same as in \textbf{(VC-11)} (2) or (3).

\end{itemize}

\end{itemize}

\item[(VC-13)] $b= \framebox{$13$}$.

\begin{itemize}
\item[(1)] The selected $i_{7}$-string is in $\nu^{(2)}$ with $i_{4}=i_{7}$.

The only possible box marking of the selected $i_{9}$-string is

\setlength{\unitlength}{10pt}
\begin{center}
\begin{picture}(5,2)
\put(0,0){\line(1,0){5}}
\put(0,1){\line(1,0){5}}
\put(1,0){\line(0,1){1}}
\put(2,0){\line(0,1){1}}
\put(3,0){\line(0,1){1}}
\put(4,0){\line(0,1){1}}
\put(5,0){\line(0,1){1}}
\put(1,1){\makebox(2,1){$\downarrow$}}
\put(1,0){\makebox(1,1){{\scriptsize $[9]$}}}
\put(2,0){\makebox(1,1){{\scriptsize $[8]$}}}
\put(3,0){\makebox(1,1){{\scriptsize $[6]$}}}
\put(4,0){\makebox(1,1){{\scriptsize $[5]$}}}

\end{picture}
\end{center}

\begin{center}
\begin{tabular}{cccl|cl}
$\Delta p_{i}^{(1)}$ & $i_{7}=i_{1}+1$ & $i_{7}>i_{1}+1$ & & $\Delta p_{i}^{(2)}$ & \\ \hline
& $3$ & $3$  & $(i\geq i_{9})$ & $-1$ & $(i\geq i_{7})$ \\
& $2$ & $2$  & $(i=i_{9}-1)$ & $0$ & $(i=i_{7}-1)$ \\
& $1$ & $1$  & $(i=i_{9}-2)$ & & \\
& $0$ & $0$  & $(i=i_{9}-3)$ & & \\
& $0$ & $-1$ & $(i=i_{9}-4)$ & & \\
& $0$ & $0$  & $(i=i_{9}-5)$ & &
\end{tabular}
\end{center}
The changes of vacancy numbers $\Delta p_{i}^{(1)}$ $(i\leq i_{9}-6)$ and $\Delta p_{i}^{(2)}$ $(i\leq i_{7}-2)$ are the same as in \textbf{(VC-4)}.

\item[(2)] The selected $i_{7}$-string is in $\nu^{(2)}$ with $i_{4}<i_{7}$ and the box marking in $\nu^{(1)}$ is either

\setlength{\unitlength}{10pt}
\begin{center}
\begin{picture}(13,3)
\put(0,0){\line(1,0){5}}
\put(0,1){\line(1,0){5}}
\put(1,0){\line(0,1){1}}
\put(2,0){\line(0,1){1}}
\put(3,0){\line(0,1){1}}
\put(4,0){\line(0,1){1}}
\put(5,0){\line(0,1){1}}
\put(1,1){\makebox(2,1){$\downarrow$}}
\put(1,0){\makebox(1,1){{\scriptsize $[9]$}}}
\put(2,0){\makebox(1,1){{\scriptsize $[8]$}}}
\put(3,0){\makebox(1,1){{\scriptsize $[6]$}}}
\put(4,0){\makebox(1,1){{\scriptsize $[5]$}}}
\put(0,2){\makebox(1,1){{\small (a)}}}

\put(6,0){\makebox(2,1){$\text{or}$}}

\put(9,0){\line(1,0){4}}
\put(9,1){\line(1,0){4}}
\put(10,0){\line(0,1){1}}
\put(11,0){\line(0,1){1}}
\put(12,0){\line(0,1){1}}
\put(13,0){\line(0,1){1}}
\put(9,1){\makebox(2,1){$\downarrow$}}
\put(10,0){\makebox(1,1){{\scriptsize $[9]$}}}
\put(11,0){\makebox(1,1){{\scriptsize $[8]$}}}
\put(12,0){\makebox(1,1){{\scriptsize $[6]$}}}
\put(9,2){\makebox(1,1){{\small (b)}}}
\end{picture}
\end{center}

\begin{center}
\begin{tabular}{cl|cl}
$\Delta p_{i}^{(1)}$ & & $\Delta p_{i}^{(2)}$ & \\ \hline
$3$ & $(i\geq i_{9})$ & $-1$ & $(i\geq i_{7})$ \\
$2$ & $(i=i_{9}-1)$ & $0$ & $(i=i_{7}-1)$ \\
$1$ & $(i=i_{9}-2)$ & & \\
$0$ & $(i=i_{9}-3)$ & & 
\end{tabular}
\end{center}
The changes of vacancy numbers $\Delta p_{i}^{(1)}$ $(i\leq i_{9}-4)$ and $\Delta p_{i}^{(2)}$ $(i\leq i_{7}-2)$ are the same as in \textbf{(VC-5)} (case (a)) and \textbf{(VC-7)} (case (b)).

\item[(3)] The selected $i_{6}$-string is in $\nu^{(2)}$ with $i_{5}=i_{6}$ and the box marking in $\nu^{(1)}$ is either

\setlength{\unitlength}{10pt}
\begin{center}
\begin{picture}(17,3)
\put(0,0){\line(1,0){7}}
\put(0,1){\line(1,0){7}}
\put(1,0){\line(0,1){1}}
\put(2,0){\line(0,1){1}}
\put(3,0){\line(0,1){1}}
\put(4,0){\line(0,1){1}}
\put(5,0){\line(0,1){1}}
\put(6,0){\line(0,1){1}}
\put(7,0){\line(0,1){1}}
\put(3,1){\makebox(2,1){$\downarrow$}}
\put(1,0){\makebox(1,1){{\scriptsize $[9]$}}}
\put(2,0){\makebox(1,1){{\scriptsize $[8]$}}}
\put(3,0){\makebox(1,1){{\scriptsize $[7]$}}}
\put(4,0){\makebox(1,1){{\scriptsize $[4]$}}}
\put(5,0){\makebox(1,1){{\scriptsize $[3]$}}}
\put(6,0){\makebox(1,1){{\scriptsize $[2]$}}}
\put(0,2){\makebox(1,1){{\small (a)}}}

\put(8,0){\makebox(2,1){$\text{or}$}}

\put(11,0){\line(1,0){6}}
\put(11,1){\line(1,0){6}}
\put(12,0){\line(0,1){1}}
\put(13,0){\line(0,1){1}}
\put(14,0){\line(0,1){1}}
\put(15,0){\line(0,1){1}}
\put(16,0){\line(0,1){1}}
\put(17,0){\line(0,1){1}}
\put(13,1){\makebox(2,1){$\downarrow$}}
\put(12,0){\makebox(1,1){{\scriptsize $[9]$}}}
\put(13,0){\makebox(1,1){{\scriptsize $[8]$}}}
\put(14,0){\makebox(1,1){{\scriptsize $[7]$}}}
\put(15,0){\makebox(1,1){{\scriptsize $[4]$}}}
\put(16,0){\makebox(1,1){{\scriptsize $[3]$}}}
\put(11,2){\makebox(1,1){{\small (b)}}}

\end{picture}
\end{center}

\begin{center}
\begin{tabular}{cccl|cl}
$\Delta p_{i}^{(1)}$ & $i_{6}=i_{1}+1$ & $i_{6}>i_{1}+1$ & & $\Delta p_{i}^{(2)}$ & \\ \hline
& $3$ & $3$  & $(i\geq i_{9})$ & $-1$ & $(i\geq i_{6})$ \\
& $2$ & $2$  & $(i=i_{9}-1)$ & $0$ & $(i=i_{6}-1)$ \\
& $1$ & $1$  & $(i=i_{9}-2)$ & & \\
& $0$ & $0$  & $(i=i_{9}-3)$ & & \\
& $0$ & $-1$ & $(i=i_{9}-4)$ & & \\
& $0$ & $-2$  & $(i=i_{9}-5)$ & &
\end{tabular}
\end{center}
The changes of vacancy numbers $\Delta p_{i}^{(1)}$ $(i\leq i_{9}-6)$ and $\Delta p_{i}^{(2)}$ $(i\leq i_{6}-2)$ are the same as in \textbf{(VC-2)} (case (a)) and \textbf{(VC-3)} (case (b)).

\item[(4)] The rightmost box of the $i_{9}$-string is marked by [7].

\begin{center}
\begin{tabular}{cccl|cl}
$\Delta p_{i}^{(1)}$ & $i_{9}^{eff}=i_{6}$ & $i_{9}^{eff}>i_{6}$ & & $\Delta p_{i}^{(2)}$ & \\ \hline
& $3$ & $3$  & $(i\geq i_{9})$ & $-1$ & $(i\geq i_{9}^{eff})$ \\
& $2$ & $1$  & $(i=i_{9}-1)$ & & \\
& $1$ & $-1$  & $(i=i_{9}-2)$ & &
\end{tabular}
\end{center}
The changes of vacancy numbers $\Delta p_{i}^{(1)}$ $(i\leq i_{9}-3)$ and $\Delta p_{i}^{(2)}$ $(i\leq i_{9}^{eff}-1)$ are the same as in \textbf{(VC-10)}.

\item[(5)] The rightmost box of the $i_{9}$-string is marked by [8].

\begin{center}
\begin{tabular}{cl|cccl}
$\Delta p_{i}^{(1)}$ & & $\Delta p_{i}^{(2)}$ & $i_{9}$:0/I & $i_{9}$:II & \\ \hline
$3$ & $(i\geq i_{9})$ & & $-1$ & $-1$ & $(i\geq i_{9}^{eff})$ \\
$1$ & $(i=i_{9}-1)$   & & $1$ & $0$ & $(i=i_{9}^{eff}-1)$
\end{tabular}
\end{center}
The changes of vacancy numbers $\Delta p_{i}^{(1)}$ $(i\leq i_{9}-2)$ and $\Delta p_{i}^{(2)}$ $(i\leq i_{9}^{eff}-2)$ are the same as in \textbf{(VC-11)}.

\item[(6)] The $i_{9}$-string is marked by [9] only.

\begin{center}
\begin{tabular}{cl|cl}
$\Delta p_{i}^{(1)}$ & & $\Delta p_{i}^{(2)}$ & \\ \hline
$3$ & $(i\geq i_{9})$ & $-1$ & $(i\geq i_{9}^{eff})$ 
\end{tabular}
\end{center}
The changes of vacancy numbers $\Delta p_{i}^{(1)}$ $(i\leq i_{9}-1)$ and $\Delta p_{i}^{(2)}$ $(i\leq i_{9}^{eff}-1)$ are the same as in \textbf{(VC-12)}.

\end{itemize}

\item[(VC-14)] $b= \framebox{$14$}$.

\begin{itemize}
\item[(1)] $i_{1}<i_{10}$.

The changes of vacancy numbers $\Delta p_{i}^{(1)}$ are obtained by overwriting 
\[
-3\chi (i\geq 3i_{10})-2\chi (i=3i_{10}-1)-\chi (i=3i_{10}-2)
\]
on $\Delta p_{i}^{(1)}$ in \textbf{(VC-13)} and $\Delta p_{i}^{(2)}$ are obtained by overwriting $2\chi (i\geq i_{10})$ on $\Delta p_{i}^{(2)}$ in \textbf{(VC-13)}.

\item[(2)] $i_{1}=i_{10}$.

The only possible box marking of $\nu^{(1)}$ (left) and $\nu^{(2)}$ (right) is 

\setlength{\unitlength}{10pt}
\begin{center}
\begin{picture}(14,4)
\put(0,2){\line(1,0){7}}
\put(0,3){\line(1,0){7}}
\put(1,2){\line(0,1){1}}
\put(2,2){\line(0,1){1}}
\put(3,2){\line(0,1){1}}
\put(4,2){\line(0,1){1}}
\put(5,2){\line(0,1){1}}
\put(6,2){\line(0,1){1}}
\put(7,2){\line(0,1){1}}
\put(3,3){\makebox(2,1){$\downarrow$}}
\put(6,0.9){\makebox(1,1){$\uparrow$}}
\put(1,2){\makebox(1,1){{\scriptsize $[9]$}}}
\put(2,2){\makebox(1,1){{\scriptsize $[8]$}}}
\put(3,2){\makebox(1,1){{\scriptsize $[7]$}}}
\put(4,2){\makebox(1,1){{\scriptsize $[4]$}}}
\put(5,2){\makebox(1,1){{\scriptsize $[3]$}}}
\put(6,2){\makebox(1,1){{\scriptsize $[2]$}}}
\put(5,0){\makebox(3,1){{\scriptsize $3i_{1}$}}}

\put(10,2){\line(1,0){3}}
\put(10,3){\line(1,0){3}}
\put(10,4){\line(1,0){3}}
\put(11,2){\line(0,1){2}}
\put(12,2){\line(0,1){2}}
\put(13,2){\line(0,1){2}}
\put(12,0.9){\makebox(1,1){$\uparrow$}}
\put(11,2){\makebox(1,1){{\scriptsize $[10]$}}}
\put(11,3){\makebox(1,1){{\scriptsize $[6]$}}}
\put(12,2){\makebox(1,1){{\scriptsize $[1]$}}}
\put(12,3){\makebox(1,1){{\scriptsize $[5]$}}}
\put(11,0){\makebox(3,1){{\scriptsize $i_{1}$}}}
\end{picture}
\end{center}

\begin{center}
\begin{tabular}{cl|cl}
$\Delta p_{i}^{(1)}$ & & $\Delta p_{i}^{(2)}$ & \\ \hline
0 & $(i\geq 1)$ & $1$ & $(i\geq i_{1})$ \\
& & 0 & $(i=i_{1}-1)$ \\
& & $-1$ & $(1\leq i\leq i_{1}-2)$ \\
\end{tabular}
\end{center}

\end{itemize} 

\item[(VC-$\emptyset$)] $b= \framebox{$\emptyset$}$.

\begin{center}
\begin{tabular}{cl|cl}
$\Delta p_{i}^{(1)}$ & & $\Delta p_{i}^{(2)}$ & \\ \hline
$0$ & $(i\geq 1)$ & $0$ & $(i\geq 1)$
\end{tabular}
\end{center}

\end{description}

\subsection{Inverse algorithm $\Tilde{\delta}_{\theta}$}
The reader may skip this subsection as we do not use it in the following sections.
This is provided just for completeness.

For a given RC $(\Tilde{\nu},\Tilde{J})$ and $b\in B$ the inverse algorithm $\Tilde{\delta}_{\theta}$ of $\delta_{\theta}$ is described in this subsection.
The inverse algorithm consists of adding boxes to strings in $\Tilde{\nu}$ and adjusting the numbers of riggings.
We omit the rule of adjusting the numbers of riggings as it is obvious by the algorithm of $\delta_{\theta}$. 
We follow the convention that riggings in strings of the same length are sorted in an decreasing order (from the top).
We assume that $\Tilde{\nu}^{(a)}$ has singular strings of length zero with zero vacancy (and therefore zero rigging).
In the following, strings of length $l$ under some conditions are selected and the $l$-string is referred to the selected string unless stated otherwise.
We denote by $[l]_{eff}$ the effective length of the (selected or unselected) string of length $l$ in $\Tilde{\nu}^{(1)}$.

\begin{flushleft}
Case 1. $b= \framebox{$\emptyset$}$.
\end{flushleft}
Add three boxes to the string of length zero in $\Tilde{\nu}^{(1)}$ and add one box to two strings of length zero in $\Tilde{\nu}^{(2)}$.

\begin{flushleft}
Case 2. $b= \framebox{$2$}$.
\end{flushleft}
Find the singular string of maximum length in $\Tilde{\nu}^{(2)}$ and add one box in this string.
This corresponds to the box marking in this string in $\nu^{(2)}$ by $[1]$.
The augmented string is called the box-added $i_{1}$-string.
This terminology is generalized in an obvious manner.

\begin{flushleft}
Case 3. $b= \framebox{$3$}$.
\end{flushleft}
Find the singular string of maximum length in $\Tilde{\nu}^{(1)}$, add one box in this string, and follow the algorithm of Case 2 such that the length of the box-added $i_{1}$-string does not exceed the effective length of the box-added $i_{2}$-string.

\begin{flushleft}
Case 4. $b= \framebox{$4$}$.
\end{flushleft}
Find the singular/q-singular string of maximum length $l^{(1)}$ in $\Tilde{\nu}^{(1)}$.
We choose the singular string when there exist singular and q-singular strings of the same length.

If the $l^{(1)}$-string is singular, then add two boxes to this string and follow the algorithm of Case 2.
If the $l^{(1)}$-string is q-singular, then add one box to this string and follow the algorithm of Case 3.

\begin{flushleft}
Case 5. $b= \framebox{$5$}$.
\end{flushleft}
Find the singular string of maximum length $l^{(2)}$ in $\Tilde{\nu}^{(2)}$ and add one box in this string.
Then find singular/q-singular string of maximum length $l^{(1)}$ in $\Tilde{\nu}^{(1)}$.
We choose the singular string when there exist singular and q-singular strings of the same length and do the following cases.

\begin{itemize}
\item[(1)] The $l^{(1)}$-string is singular of length $3l^{(2)}+1$.

Add one box to the $l^{(1)}$-string and follow the algorithm, of Case 3.
This corresponds to the following box marking in $\nu^{(1)}$.
\setlength{\unitlength}{10pt}
\begin{center}
\begin{picture}(3,2)
\put(0,0){\line(1,0){3}}
\put(0,1){\line(1,0){3}}
\put(1,0){\line(0,1){1}}
\put(2,0){\line(0,1){1}}
\put(3,0){\line(0,1){1}}
\put(2,0){\makebox(1,1){{\scriptsize $[3]$}}}
\put(0,1){\makebox(2,1){$\downarrow$}}
\end{picture}
\end{center}
with $i_{3}^{eff}=i_{4}$.

\item[(2)] The $l^{(1)}$-string is singular of length $3l^{(2)}$ or $3l^{(2)}-1$.

Add two boxes to the $l^{(1)}$-string and follow the algorithm of Case 2.
This corresponds to the following box marking in $\nu^{(1)}$
\setlength{\unitlength}{10pt}
\begin{center}
\begin{picture}(10,2)
\put(0,0){\line(1,0){3}}
\put(0,1){\line(1,0){3}}
\put(1,0){\line(0,1){1}}
\put(2,0){\line(0,1){1}}
\put(3,0){\line(0,1){1}}
\put(0,1){\makebox(2,1){$\downarrow$}}
\put(1,0){\makebox(1,1){{\scriptsize $[3]$}}}
\put(2,0){\makebox(1,1){{\scriptsize $[2]$}}}

\put(4,0){\makebox(2,1){$\text{or}$}}

\put(7,0){\line(1,0){3}}
\put(7,1){\line(1,0){3}}
\put(8,0){\line(0,1){1}}
\put(9,0){\line(0,1){1}}
\put(10,0){\line(0,1){1}}
\put(8,1){\makebox(2,1){$\downarrow$}}
\put(8,0){\makebox(1,1){{\scriptsize $[3]$}}}
\put(9,0){\makebox(1,1){{\scriptsize $[2]$}}}

\end{picture}
\end{center}
with $i_{3}^{eff}=i_{4}$.

\item[(3)] The $l^{(1)}$-string is q-singular of length $3l^{(2)}$.

Add one box to the founded string and follow the algorithm of Case 3.
This corresponds to the following box marking in $\nu^{(1)}$
\setlength{\unitlength}{10pt}
\begin{center}
\begin{picture}(2,2)
\put(0,0){\line(1,0){2}}
\put(0,1){\line(1,0){2}}
\put(1,0){\line(0,1){1}}
\put(2,0){\line(0,1){1}}
\put(1,0){\makebox(1,1){{\scriptsize $[3]$}}}
\put(0,1){\makebox(2,1){$\downarrow$}}
\end{picture}
\end{center}
with $i_{3}^{eff}=i_{4}$.

\end{itemize}

If the box adding not successful, then follow the algorithm of Case 4.
This corresponds to the box marking in $\nu$ with $i_{3}^{eff}<i_{4}$.

\begin{flushleft}
Case 6. $b= \framebox{$6$}$.
\end{flushleft}
Find the singular/q-singular/qq-singular string of maximum length $l^{(1)}$ in $\Tilde{\nu}^{(1)}$.
We choose the singular string when there exists singular and q/qq-singular strings of the same length.
Similarly, we choose the q-singular string when there exists q-singular and qq-singular strings of the same length.
That is, we choose the string of the same length by the following preferential rule 
\begin{center}
\begin{tabular}{ccc}
singular & q-singular & qq-singular \\ \hline
1 & 2 & 3
\end{tabular}
\end{center}
In addition, we impose the following preferential rule for the $l^{(1)}$-string.
That is, if the $l^{(1)}$-string is qq-singular and there exists a singular string of length $l^{(1)}-1$ in $\Tilde{\nu}^{(1)}$, then reset the $l^{(1)}$-string to be the $(l^{(1)}-1)$-string.

The box adding goes as follows.
If the $l^{(1)}$-string is singular (resp. q-singular), then add three (resp. two) boxes to the $l^{(1)}$-string and follow the algorithm of Case 2 (resp. Case 3).
If the $l^{(1)}$ is qq-singular, then add one box to the $l^{(1)}$-string and follow the algorithm of Case 4.

\begin{flushleft}
Case 7. $b= \framebox{$7$}$.
\end{flushleft}
Find the singular string of maximum length $l^{(1)}$ in $\Tilde{\nu}^{(1)}$ and do as follows according to the type of the $l^{(1)}$-string.

\begin{itemize}
\item[(1)] type-0. 

Find the singular string of length $[l^{(1)}]_{eff}$ in $\Tilde{\nu}^{(2)}$.
If such a string exists, then add three boxes to the $l^{(1)}$-string in $\Tilde{\nu}^{(1)}$ and one box to the founded string in $\Tilde{\nu}^{(2)}$ and follow the algorithm of Case 2.
This corresponds to the following box marking in $\nu^{(1)}$
\setlength{\unitlength}{10pt}
\begin{center}
\begin{picture}(5,2)
\put(0,0){\line(1,0){4}}
\put(0,1){\line(1,0){4}}
\put(1,0){\line(0,1){1}}
\put(2,0){\line(0,1){1}}
\put(3,0){\line(0,1){1}}
\put(4,0){\line(0,1){1}}
\put(1,0){\makebox(1,1){{\scriptsize $[5]$}}}
\put(2,0){\makebox(1,1){{\scriptsize $[3]$}}}
\put(3,0){\makebox(1,1){{\scriptsize $[2]$}}}
\put(0,1){\makebox(2,1){$\downarrow$}}
\end{picture}
\end{center}
with $i_{5}^{eff}=i_{4}$.

\item[(2)] type-I. 

Find the singular string of length $[l^{(1)}]_{eff}-1$ in $\Tilde{\nu}^{(2)}$.
If such a string exists, then add one box to the $l^{(1)}$-string in $\Tilde{\nu}^{(1)}$ and follow the algorithm of Case 5 where the length of the box-added $i_{4}$-string in $\Tilde{\nu}^{(2)}$ is $[l^{(1)}]_{eff}$.

\item[(3)] type-II. 

Find the singular string of length $[l^{(1)}]_{eff}-1$ in $\Tilde{\nu}^{(2)}$.
If such a string exists, then add two boxes to the $l^{(1)}$-string in $\Tilde{\nu}^{(1)}$, one box to the founded string in $\Tilde{\nu}^{(2)}$ and follow the algorithm of Case 3.
This corresponds to the following box marking in $\nu^{(1)}$
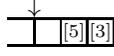
\setlength{\unitlength}{10pt}
\begin{center}
\begin{picture}(5,2)
\put(0,0){\line(1,0){4}}
\put(0,1){\line(1,0){4}}
\put(1,0){\line(0,1){1}}
\put(2,0){\line(0,1){1}}
\put(3,0){\line(0,1){1}}
\put(4,0){\line(0,1){1}}
\put(2,0){\makebox(1,1){{\scriptsize $[5]$}}}
\put(3,0){\makebox(1,1){{\scriptsize $[3]$}}}
\put(0,1){\makebox(2,1){$\downarrow$}}
\end{picture}
\end{center}
with $i_{5}^{eff}=i_{4}$.

\end{itemize}

If the box adding is not successful, then add one box to the $l^{(1)}$-string in $\Tilde{\nu}^{(1)}$ and follow the algorithm of Case 5.
This corresponds to the box marking in $\nu$ with $i_{5}^{eff}>i_{4}$.

\begin{flushleft}
Case 8. $b= \framebox{$8$}$.
\end{flushleft}
Find the singular string of maximum length $l^{(2)}$ in $\Tilde{\nu}^{(2)}$ and add one box to this string.
The resulting string is the box-added $i_{5}$-string.
Then, find the string of maximum length $l^{(1)}$ in $\Tilde{\nu}^{(1)}$ such that it is 
\begin{itemize}
\item[(1)] a singular string of type-I/II (resp. type-0) with $[l^{(1)}]_{eff}=l^{(2)}+1$ (resp. $[l^{(1)}]_{eff}=l^{(2)}$) or

\item[(2)] a q-singular string of type-II (resp. type-0) with $[l^{(1)}]_{eff}=l^{(2)}+1$ (resp. $[l^{(1)}]_{eff}=l^{(2)}$) or 

\item[(3)] a singular string of type-I/II with $[l^{(1)}]_{eff}=l^{(2)}$ or 

\item[(4)] a q-singular string of type-I with with $[l^{(1)}]_{eff}=l^{(2)}$ oir

\item[(5)] a qq-singular string of type-0 with $[l^{(1)}]_{eff}=l^{(2)}$.
\end{itemize}
We choose the string of the same length by the following preferential rule 
\begin{center}
\begin{tabular}{ccc}
singular & q-singular & qq-singular \\ \hline
1 & 2 & 3
\end{tabular}
\end{center}
and do as follows according to the type of the $l^{(1)}$-string.
All correspond to the box marking in $\nu$ with $i_{4}^{eff}=i_{5}$.

\begin{itemize}
\item[(1)] type-0.

If the $l^{(1)}$-string is singular (resp. q-singular), then add three (resp. two) boxes to this string and follow the algorithm of Case 2 (resp. Case 3).
If the $l^{(1)}$-string is qq-singular, then add one box to this string and follow the algorithm of Case 4.

\item[(2)] type-I.

If the $l^{(1)}$-string is singular with with $[l^{(1)}]_{eff}=l^{(2)}+1$, then add one box to this string and follow the algorithm of Case 4 ignoring the q-singular string of length $l^{(1)}-1$ in $\Tilde{\nu}^{(1)}$.
This is because the following box marking in $\nu^{(1)}$

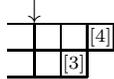
\setlength{\unitlength}{10pt}
\begin{center}
\begin{picture}(5,3)
\put(0,0){\line(1,0){3}}
\put(0,1){\line(1,0){4}}
\put(0,2){\line(1,0){4}}
\put(1,0){\line(0,1){2}}
\put(2,0){\line(0,1){2}}
\put(3,0){\line(0,1){2}}
\put(4,1){\line(0,1){1}}
\put(0,2){\makebox(2,1){$\downarrow$}}
\put(2,0){\makebox(1,1){{\scriptsize $[3]$}}}
\put(3,1){\makebox(1,1){{\scriptsize $[4]$}}}
\end{picture}
\end{center}
with $i_{4}^{eff}=i_{5}$ is forbidden.
If the selected string is singular (resp. q-singular) with $[l^{(1)}]_{eff}=l^{(2)}$, then add three (resp. two) boxes to this string and follow the algorithm of Case 2 (resp. Case 3).

\item[(3)] type-II.

If the $l^{(1)}$-string is singular (resp. q-singular) with $[l^{(1)}]_{eff}=l^{(2)}+1$, then add two (resp. one) box(es) to this string and follow the algorithm of Case 3 (resp. Case 4).
If the $l^{(1)}$-string is singular with $[l^{(1)}]_{eff}=l^{(2)}$, then add three boxes to this string and follow the algorithm of Case 2.
\end{itemize} 

If the box adding is not successful, then follow the algorithm of Case 6 such that the effective length of the box-added $i_{4}$-string is smaller than $l^{(2)}+1$, which corresponds to the box marking in $\nu$ with $i_{4}^{eff}<i_{5}$.

\begin{flushleft}
Case 9. $b= \framebox{$9$}$.
\end{flushleft}
Find the singular string of maximum length $l^{(2)}$ in $\Tilde{\nu}^{(2)}$.
Then, find the singular/q-singular string of maximum length $l^{(1)}$ in $\Tilde{\nu}^{(1)}$.
We choose the singular string when there exist singular and q-singular strings of the same length.
In addition, we impose the following preferential rule for the $l^{(1)}$-string.
\begin{itemize}
\item
If the $l^{(1)}$-string in $\Tilde{\nu}^{(1)}$ is type-I q-singular with $[l^{(1)}]_{eff}=l^{(2)}+1$ and there exists a type-II singular string of the same effective length in $\Tilde{\nu}^{(1)}$, then reset the $l^{(1)}$-string to be the founded type-II singular string.
\item
If the $l^{(1)}$-string is type-I/II q-singular string with $[l^{(1)}]_{eff}=l^{(2)}+1$ and there exists a type-0 singular string of effective length $l^{(2)}$, then reset the $l^{(1)}$-string to be the founded type-0 singular string.
\end{itemize}

The box adding goes as follows.
If the $l^{(1)}$-string is q-singular, then add one box to this string and follow the algorithm of Case 7.
If the $l^{(1)}$-string is singular, then do as follows according to the type of the $l^{(1)}$-string.

\begin{itemize}
\item[(1)] type-0.

If $[l^{(1)}]_{eff}=l^{(2)}$, then add four boxes to this string, one box to the $l^{(2)}$-string, and follow the algorithm of Case 2.

\item[(2)] type-II.

If $[l^{(1)}]_{eff}=l^{(2)}+1$, then add three boxes to this string, one box to the $l^{(2)}$-string, and follow the algorithm of Case 3.

\item[(3)] type-I.

If $[l^{(1)}]_{eff}=l^{(2)}$, then add two boxes to this string, one box to the $l^{(2)}$-string, and follow the algorithm of Case 4 ignoring the q-singular string of length $l^{(1)}-1$ in $\Tilde{\nu}^{(1)}$.
This is due to the rule of \textbf{(BS-5)}.

\end{itemize}

The corresponding box marking in $\nu^{(1)}$ is one of the following three.

\setlength{\unitlength}{10pt}
\begin{center}
\begin{picture}(20,2)
\put(0,0){\line(1,0){5}}
\put(0,1){\line(1,0){5}}
\put(1,0){\line(0,1){1}}
\put(2,0){\line(0,1){1}}
\put(3,0){\line(0,1){1}}
\put(4,0){\line(0,1){1}}
\put(5,0){\line(0,1){1}}
\put(3,1){\makebox(2,1){$\downarrow$}}
\put(1,0){\makebox(1,1){{\scriptsize $[6]$}}}
\put(2,0){\makebox(1,1){{\scriptsize $[5]$}}}
\put(3,0){\makebox(1,1){{\scriptsize $[3]$}}}
\put(4,0){\makebox(1,1){{\scriptsize $[2]$}}}

\put(6,0){\makebox(2,1){$\text{or}$}}

\put(9,0){\line(1,0){4}}
\put(9,1){\line(1,0){4}}
\put(10,0){\line(0,1){1}}
\put(11,0){\line(0,1){1}}
\put(12,0){\line(0,1){1}}
\put(13,0){\line(0,1){1}}
\put(11,1){\makebox(2,1){$\downarrow$}}
\put(10,0){\makebox(1,1){{\scriptsize $[6]$}}}
\put(11,0){\makebox(1,1){{\scriptsize $[5]$}}}
\put(12,0){\makebox(1,1){{\scriptsize $[3]$}}}

\put(14,0){\makebox(2,1){$\text{or}$}}

\put(17,0){\line(1,0){3}}
\put(17,1){\line(1,0){3}}
\put(18,0){\line(0,1){1}}
\put(19,0){\line(0,1){1}}
\put(20,0){\line(0,1){1}}
\put(18,1){\makebox(2,1){$\downarrow$}}
\put(18,0){\makebox(1,1){{\scriptsize $[6]$}}}
\put(19,0){\makebox(1,1){{\scriptsize $[5]$}}}

\end{picture}
\end{center}
where $i_{6}^{eff}=i_{4}$ in the left two and $i_{6}^{eff}=i_{4}+1$ in the rightmost.

If the box adding is not successful, then add two boxes to the $l^{(1)}$-string and follow the algorithm of Case 5.

\begin{flushleft}
Case 10. $b= \framebox{$10$}$.
\end{flushleft}
Find the singular/q-singular string of maximum length $l^{(2)}$ in $\Tilde{\nu}^{(2)}$.
We choose the singular string when there exist singular and q-singular strings of the same length.

If the $l^{(2)}$-string is q-singular, then add one box to this string and follow the algorithm of Case 8.

If the $l^{(2)}$-string is singular, then add two boxes to the $l^{(2)}$-string and find the string of maximum length $l^{(1)}$ in $\Tilde{\nu}^{(1)}$ such that it is
\begin{itemize}
\item[(1)]
a singular string of type-I/II with $[l^{(1)}]_{eff}=l^{(2)}+1$ or a singular string (of any type) with $[l^{(1)}]_{eff}=l^{(2)}$ or

\item[(2)]
a q-singular string of type-II (resp. type-0/I) with $[l^{(1)}]_{eff}=l^{(2)}+1$ (resp. $[l^{(1)}]_{eff}=l^{(2)}$) or

\item[(3)]
a qq-singular string of type-0 with $[l^{(1)}]_{eff}=l^{(2)}$.

\end{itemize}
In addition, we impose the following preferential rule for the $l^{(1)}$-string.

\begin{itemize}
\item
If the $l^{(1)}$-string is type-II q-singular and there exists a type-0 singular string of effective length $[l^{(1)}]_{eff}-1$, then reset the $l^{(1)}$-string to be the founded singular string.
\item
If the $l^{(1)}$-string is qq-singular and there exists a type-I singular string of the same effective length, then reset the $l^{(1)}$-string to be the founded singular string.
\end{itemize}

The box adding goes as follows.

\begin{itemize}
\item[(1)] The $l^{(1)}$-string is singular.

Do the following according to the type of the $l^{(1)}$-string.

\begin{itemize}
\item[(a)] type-0.

Add three boxes to this string and follow the algorithm of Case 2.

\item[(b)] type-I.

If $[l^{(1)}]_{eff}=l^{(2)}$, then add three boxes to this string and follow the algorithm of Case 2.
If $[l^{(1)}]_{eff}=l^{(2)}+1$, then add one box to this string and follow the algorithm of Case 4.

\item[(c)] type-II.

If $[l^{(1)}]_{eff}=l^{(2)}$, then add three boxes to this string and follow the algorithm of Case 2.
If $[l^{(1)}]_{eff}=l^{(2)}+1$, then add two boxes to this string and follow the algorithm of Case 3.
\end{itemize}

\item[(2)] The $l^{(1)}$-string is q-singular.
If the $l^{(1)}$-string is type-0/I (resp. type-II), then add two (resp. one) box(es) to this string and follow the algorithm of Case 3 (resp. Case 4).

\item[(3)] The $l^{(1)}$-string is qq-singular.

Add one box to this string and follow the algorithm of Case 4.

\end{itemize}

If the box adding is not successful, then follow the algorithm of Case 6 such that the effective length of the box-added $i_{4}$-string does not exceed $l^{(2)}$.

\begin{flushleft}
Case 11. $b= \framebox{$11$}$.
\end{flushleft}
Find the singular string of maximum length $l^{(1)}$ in $\Tilde{\nu}^{(1)}$.
Find the singular string of maximum length $l_{s}^{(2)}$ in $\Tilde{\nu}^{(2)}$.
Find the q-singular string of maximum length $l_{q}^{(2)}$ satisfying $[ l^{(1)}+1]_{eff}\geq l_{q}^{(2)}$.
If such a q-singular string does not exist, then set $l_{q}^{(2)}=0$.
If $l_{q}^{(2)}>l_{s}^{(2)}$, then add one box to the $l^{(1)}$-string in $\Tilde{\nu}^{(1)}$, add one box to the $l_{q}^{(2)}$-string in $\Tilde{\nu}^{(2)}$ and follow the algorithm of Case 8.
This corresponds to the following box marking in $\nu^{(1)}$ (left) and $\nu^{(2)}$ (right). 
\setlength{\unitlength}{10pt}
\begin{center}
\begin{picture}(7,1)
\put(0,0){\line(1,0){2}}
\put(0,1){\line(1,0){2}}
\put(1,0){\line(0,1){1}}
\put(2,0){\line(0,1){1}}
\put(1,0){\makebox(1,1){{\scriptsize $[7]$}}}


\put(5,0){\line(1,0){2}}
\put(5,1){\line(1,0){2}}
\put(6,0){\line(0,1){1}}
\put(7,0){\line(0,1){1}}
\put(6,0){\makebox(1,1){{\scriptsize $[6]$}}}
\end{picture}
\end{center}
If $l_{s}^{(2)}>l_{q}^{(2)}$, then do as follows.
If $[l^{(1)}+1]_{eff}\geq l_{s}^{(2)}$, the add one box to the $l^{(1)}$-string in $\Tilde{\nu}^{(1)}$, add two boxes to the $l_{s}^{(2)}$ singular string in $\Tilde{\nu}^{(2)}$ and follow the algorithm of Case 6.
This corresponds to the following box marking in $\nu^{(1)}$ (left) and $\nu^{(2)}$ (right). 
\setlength{\unitlength}{10pt}
\begin{center}
\begin{picture}(8,1)
\put(0,0){\line(1,0){2}}
\put(0,1){\line(1,0){2}}
\put(1,0){\line(0,1){1}}
\put(2,0){\line(0,1){1}}
\put(1,0){\makebox(1,1){{\scriptsize $[7]$}}}


\put(5,0){\line(1,0){3}}
\put(5,1){\line(1,0){3}}
\put(6,0){\line(0,1){1}}
\put(7,0){\line(0,1){1}}
\put(8,0){\line(0,1){1}}
\put(6,0){\makebox(1,1){{\scriptsize $[6]$}}}
\put(7,0){\makebox(1,1){{\scriptsize $[5]$}}}
\end{picture}
\end{center}

If the box adding is not successful, then do as follows.
Find the singular string of maximum length $l^{(2)}$ in $\Tilde{\nu}^{(2)}$.
Find the singular string of maximum length $l^{(1)}$ in $\Tilde{\nu}^{(1)}$.
If the $l^{(1)}$-string is type-0, $[l^{(1)}+4]_{eff}=l^{(2)}+2$, and there are no q-singular strings of length $l$ ($l^{(1)}+1\leq l\leq l^{(1)}+3$), then add four boxes to the $l^{(1)}$-string, add two boxes to the $l^{(2)}$-string, and follow the algorithm of Case 2.
This corresponds to the following box marking $\nu^{(1)}$ (left) and $\nu^{(2)}$ (right).

\setlength{\unitlength}{10pt}
\begin{center}
\begin{picture}(11,2)
\put(0,0){\line(1,0){5}}
\put(0,1){\line(1,0){5}}
\put(1,0){\line(0,1){1}}
\put(2,0){\line(0,1){1}}
\put(3,0){\line(0,1){1}}
\put(4,0){\line(0,1){1}}
\put(5,0){\line(0,1){1}}
\put(3,1){\makebox(2,1){$\downarrow$}}
\put(1,0){\makebox(1,1){{\scriptsize $[6]$}}}
\put(2,0){\makebox(1,1){{\scriptsize $[5]$}}}
\put(3,0){\makebox(1,1){{\scriptsize $[3]$}}}
\put(4,0){\makebox(1,1){{\scriptsize $[2]$}}}

\put(8,0){\line(1,0){3}}
\put(8,1){\line(1,0){3}}
\put(9,0){\line(0,1){1}}
\put(10,0){\line(0,1){1}}
\put(11,0){\line(0,1){1}}
\put(9,0){\makebox(1,1){{\scriptsize $[7]$}}}
\put(10,0){\makebox(1,1){{\scriptsize $[4]$}}}

\end{picture}
\end{center}

If the box adding is not successful, then do as follows.
Find the singular string of maximum length $l_{1}^{(2)}$ in $\Tilde{\nu}^{(2)}$.
Find the singular string of maximum length $l^{(1)}$ in $\Tilde{\nu}^{(1)}$.
If $l^{(1)}$-string is type-0, $[l^{(1)}+4]_{eff}\leq l^{(2)}+1$, and there are no q-singular strings of length $l$ ($l^{(1)}+1\leq l\leq 3l_{1}^{(2)}$), then find the singular string of length $l_{2}^{(2)}$ in $\Tilde{\nu}^{(2)}$ satisfying $l_{2}^{(2)}+1=[l^{(1)}+4]_{eff}-1$.
If such a singular string exists, then add four boxes to the $l^{(1)}$-string, add one box the $l_{1}^{(2)}$- and $l_{2}^{(2)}$-string and follow the algorithm of Case 2.
This corresponds to the following box marking in $\nu^{(1)}$ (left) and $\nu^{(2)}$ (right)

\setlength{\unitlength}{10pt}
\begin{center}
\begin{picture}(12,3)
\put(0,1){\line(1,0){5}}
\put(0,2){\line(1,0){5}}
\put(1,1){\line(0,1){1}}
\put(2,1){\line(0,1){1}}
\put(3,1){\line(0,1){1}}
\put(4,1){\line(0,1){1}}
\put(5,1){\line(0,1){1}}
\put(3,2){\makebox(2,1){$\downarrow$}}
\put(1,1){\makebox(1,1){{\scriptsize $[6]$}}}
\put(2,1){\makebox(1,1){{\scriptsize $[5]$}}}
\put(3,1){\makebox(1,1){{\scriptsize $[3]$}}}
\put(4,1){\makebox(1,1){{\scriptsize $[2]$}}}

\put(8,0){\line(1,0){2}}
\put(8,1){\line(1,0){4}}
\put(8,2){\line(1,0){4}}
\put(9,0){\line(0,1){1}}
\put(10,0){\line(0,1){1}}
\put(11,1){\line(0,1){1}}
\put(12,1){\line(0,1){1}}
\put(9,0){\makebox(1,1){{\scriptsize $[4]$}}}
\put(11,1){\makebox(1,1){{\scriptsize $[7]$}}}

\end{picture}
\end{center}
where $i_{6}^{eff}=i_{4}+1\leq i_{7}$.

If the box adding is not successful, then do as follows.
Find the singular string of maximum length $l_{1}^{(2)}$ in $\Tilde{\nu}^{(2)}$.
Find the singular string of maximum length $l^{(1)}$ in $\Tilde{\nu}^{(1)}$.
If the $l^{(1)}$-string is type-II, $[l^{(1)}+3]_{eff}\leq l_{1}^{(2)}+1$, and there are no q-singular strings of length $l$ ($l^{(1)}+1\leq l\leq 3l_{1}^{(2)}$), then find the singular string of length $l_{2}^{(2)}$ in $\Tilde{\nu}^{(2)}$ satisfying $l_{2}^{(2)}+1=[l^{(1)}+3]_{eff}-1$.
If such a singular string exists, then add three boxes to the $l^{(1)}$-string, add one box the $l_{1}^{(2)}$- and $l_{2}^{(2)}$-string and follow the algorithm of Case 3.
This corresponds to the following box marking in $\nu^{(1)}$ (left) and $\nu^{(2)}$ (right)

\setlength{\unitlength}{10pt}
\begin{center}
\begin{picture}(11,3)
\put(0,1){\line(1,0){4}}
\put(0,2){\line(1,0){4}}
\put(1,1){\line(0,1){1}}
\put(2,1){\line(0,1){1}}
\put(3,1){\line(0,1){1}}
\put(4,1){\line(0,1){1}}
\put(2,2){\makebox(2,1){$\downarrow$}}
\put(1,1){\makebox(1,1){{\scriptsize $[6]$}}}
\put(2,1){\makebox(1,1){{\scriptsize $[5]$}}}
\put(3,1){\makebox(1,1){{\scriptsize $[3]$}}}

\put(7,0){\line(1,0){2}}
\put(7,1){\line(1,0){4}}
\put(7,2){\line(1,0){4}}
\put(8,0){\line(0,1){1}}
\put(9,0){\line(0,1){1}}
\put(10,1){\line(0,1){1}}
\put(11,1){\line(0,1){1}}
\put(8,0){\makebox(1,1){{\scriptsize $[4]$}}}
\put(10,1){\makebox(1,1){{\scriptsize $[7]$}}}

\end{picture}
\end{center}
where $i_{6}^{eff}=i_{4}+1\leq i_{7}$.

If the box adding is not successful, then do as follows.
Find the singular string of maximum length $l^{(2)}$ in $\Tilde{\nu}^{(2)}$.
Find the singular string of maximum length $l^{(1)}$ in $\Tilde{\nu}^{(1)}$.
If the $l^{(1)}$-string is type-I and $[l^{(1)}+2]_{eff}=l^{(2)}+2$, then add two boxes to the $l^{(1)}$- and $l^{(2)}$-strings and follow the algorithm of Case 4 ignoring the q-singular string of length $l^{(1)}-1$ in $\Tilde{\nu}^{(1)}$.
The box adding in this case corresponds to the following box marking in $\nu^{(1)}$ (left) and $\nu^{(2)}$ (right).

\setlength{\unitlength}{10pt}
\begin{center}
\begin{picture}(9,2)
\put(0,0){\line(1,0){3}}
\put(0,1){\line(1,0){3}}
\put(1,0){\line(0,1){1}}
\put(2,0){\line(0,1){1}}
\put(3,0){\line(0,1){1}}
\put(1,1){\makebox(2,1){$\downarrow$}}
\put(1,0){\makebox(1,1){{\scriptsize $[6]$}}}
\put(2,0){\makebox(1,1){{\scriptsize $[5]$}}}

\put(6,0){\line(1,0){3}}
\put(6,1){\line(1,0){3}}
\put(7,0){\line(0,1){1}}
\put(8,0){\line(0,1){1}}
\put(9,0){\line(0,1){1}}
\put(7,0){\makebox(1,1){{\scriptsize $[7]$}}}
\put(8,0){\makebox(1,1){{\scriptsize $[4]$}}}

\end{picture}
\end{center}

If the box adding is not successful, then do as follows.
Find the singular string of maximum length $l^{(2)}$ in $\Tilde{\nu}^{(2)}$.
Find the singular string of maximum length $l^{(1)}$ in $\Tilde{\nu}^{(1)}$.
If the $l^{(1)}$-string is type-I and $[l^{(1)}+2]_{eff}=l^{(2)}+1$, then add two boxes to the $l^{(1)}$-string and one box to the $l^{(2)}$-strings and follow the algorithm of Case 5.
This corresponds to the following box marking in $\nu^{(1)}$ (left) and $\nu^{(2)}$ (right)

\setlength{\unitlength}{10pt}
\begin{center}
\begin{picture}(8,2)
\put(0,0){\line(1,0){3}}
\put(0,1){\line(1,0){3}}
\put(1,0){\line(0,1){1}}
\put(2,0){\line(0,1){1}}
\put(3,0){\line(0,1){1}}
\put(1,1){\makebox(2,1){$\downarrow$}}
\put(1,0){\makebox(1,1){{\scriptsize $[6]$}}}
\put(2,0){\makebox(1,1){{\scriptsize $[5]$}}}

\put(6,0){\line(1,0){2}}
\put(6,1){\line(1,0){2}}
\put(7,0){\line(0,1){1}}
\put(8,0){\line(0,1){1}}
\put(7,0){\makebox(1,1){{\scriptsize $[7]$}}}

\end{picture}
\end{center}
where $i_{6}^{eff}=i_{7}$.

If the box adding is not successful, then do as follows.
Find the singular string of maximum length $l^{(2)}$ in $\Tilde{\nu}^{(2)}$.
Find the q-singular string of maximum length $l^{(1)}$ in $\Tilde{\nu}^{(1)}$.
If the $l^{(1)}$-string is type-0 and $[l^{(1)}+1]_{eff}=l^{(2)}+1$, then add one box to the $l^{(1)}$- and the $l^{(2)}$-strings and follow the algorithm of Case 7.
This corresponds to the following box marking in $\nu^{(1)}$ (left) and $\nu^{(2)}$ (right)

\setlength{\unitlength}{10pt}
\begin{center}
\begin{picture}(7,2)
\put(0,0){\line(1,0){2}}
\put(0,1){\line(1,0){2}}
\put(1,0){\line(0,1){1}}
\put(2,0){\line(0,1){1}}
\put(0,1){\makebox(2,1){$\downarrow$}}
\put(1,0){\makebox(1,1){{\scriptsize $[6]$}}}


\put(5,0){\line(1,0){2}}
\put(5,1){\line(1,0){2}}
\put(6,0){\line(0,1){1}}
\put(7,0){\line(0,1){1}}
\put(6,0){\makebox(1,1){{\scriptsize $[7]$}}}
\end{picture}
\end{center}
where $i_{6}^{eff}=i_{7}$.

If the box adding is not successful, then do as follows.
Find the singular string of maximum length $l^{(2)}$ in $\Tilde{\nu}^{(2)}$.
Then, find the string of maximum length $l^{(1)}$ such that the $l^{(1)}$-string is singular with $[l^{(1)}+2]_{eff}<l^{(2)}+1$ or the $l^{(1)}$-string is q-singular with $[l^{(1)}+1]_{eff}<l^{(2)}+1$.
If the $l^{(1)}$-string is singular, then add two boxes to the $l^{(1)}$-string and one box to the $l^{(2)}$-strings and follow the algorithm of Case 5.
If the $l^{(1)}$-string is q-singular, then add one box to the $l^{(1)}$- and the $l^{(2)}$-strings and follow the algorithm of Case 7.
The corresponding box marking in $\nu^{(1)}$ is

\setlength{\unitlength}{10pt}
\begin{center}
\begin{picture}(8,2)
\put(0,0){\line(1,0){3}}
\put(0,1){\line(1,0){3}}
\put(1,0){\line(0,1){1}}
\put(2,0){\line(0,1){1}}
\put(3,0){\line(0,1){1}}
\put(1,0){\makebox(1,1){{\scriptsize $[6]$}}}
\put(2,0){\makebox(1,1){{\scriptsize $[5]$}}}

\put(4,0){\makebox(2,1){$\text{or}$}}

\put(7,0){\line(1,0){2}}
\put(7,1){\line(1,0){2}}
\put(8,0){\line(0,1){1}}
\put(9,0){\line(0,1){1}}
\put(8,0){\makebox(1,1){{\scriptsize $[6]$}}}

\end{picture}
\end{center}
with $i_{6}^{eff}<i_{7}$.

\begin{flushleft}
Case 12. $b= \framebox{$12$}$.
\end{flushleft}
Find the singular/q-singular string of maximum length $l^{(1)}$ in $\Tilde{\nu}^{(1)}$.
We choose the singular string when there exist singular and q-singular strings of the same length.
In addition, we impose the following preferential rules.
\begin{itemize}
\item
If the $l^{(1)}$-string is q-singular (of any type), then find the type-0 singular string of type-0 of length $l^{(1)\prime}$ whose effective length $[l^{(1)}]_{eff}-1$ in $\Tilde{\nu}^{(1)}$ such that there exists a singular string of length $l^{(2)}$ satisfying $[l^{(1)\prime}+5]_{eff}=l^{(2)}+2$.
If such an $l^{(1)\prime}$-string exists, then reset the $l^{(1)}$-string to be $l^{(1)\prime}$-string.
\item
If the $l^{(1)}$-string is q-singular of type I, then find the type-II of length $l^{(1)\prime}=l^{(1)}-1$ such that there exists a singular string of length $l^{(2)}$ satisfying $[l^{(1)\prime}+4]_{eff}=l^{(2)}+2$.
If such an $l^{(1)\prime}$-string exists, then reset the $l^{(1)}$-string to be $l^{(1)\prime}$-string.
\end{itemize}

The box adding goes as follows.

If the $l^{(1)}$-string is singular of type-0, then find a singular string of length $l^{(2)}$ in $\Tilde{\nu}^{(2)}$ satisfying $[l^{(1)}+5]_{eff}=l^{(2)}+2$.
If such a string exists, then add five boxes to the singular $l^{(1)}$-string, two boxes to the $l^{(2)}$-string, and follow the algorithm of Case 2.
This corresponds to the following box marking $\nu^{(1)}$ (left) and $\nu^{(2)}$ (right).

\setlength{\unitlength}{10pt}
\begin{center}
\begin{picture}(12,2)
\put(0,0){\line(1,0){6}}
\put(0,1){\line(1,0){6}}
\put(1,0){\line(0,1){1}}
\put(2,0){\line(0,1){1}}
\put(3,0){\line(0,1){1}}
\put(4,0){\line(0,1){1}}
\put(5,0){\line(0,1){1}}
\put(6,0){\line(0,1){1}}
\put(3,1){\makebox(2,1){$\downarrow$}}
\put(1,0){\makebox(1,1){{\scriptsize $[8]$}}}
\put(2,0){\makebox(1,1){{\scriptsize $[7]$}}}
\put(3,0){\makebox(1,1){{\scriptsize $[4]$}}}
\put(4,0){\makebox(1,1){{\scriptsize $[3]$}}}
\put(5,0){\makebox(1,1){{\scriptsize $[2]$}}}

\put(9,0){\line(1,0){3}}
\put(9,1){\line(1,0){3}}
\put(10,0){\line(0,1){1}}
\put(11,0){\line(0,1){1}}
\put(12,0){\line(0,1){1}}
\put(10,0){\makebox(1,1){{\scriptsize $[6]$}}}
\put(11,0){\makebox(1,1){{\scriptsize $[5]$}}}
\end{picture}
\end{center}

If the $l^{(1)}$-string is singular of type-II, then find a singular string of length $l^{(2)}$ in $\Tilde{\nu}^{(2)}$ satisfying $[l^{(1)}+4]_{eff}=l^{(2)}+2$.
If such a string exists, then add four boxes to the singular $l^{(1)}$-string, two boxes to the $l^{(2)}$-string, and follow the algorithm of Case 3.
This corresponds to the following box marking $\nu^{(1)}$ (left) and $\nu^{(2)}$ (right).

\setlength{\unitlength}{10pt}
\begin{center}
\begin{picture}(11,2)
\put(0,0){\line(1,0){5}}
\put(0,1){\line(1,0){5}}
\put(1,0){\line(0,1){1}}
\put(2,0){\line(0,1){1}}
\put(3,0){\line(0,1){1}}
\put(4,0){\line(0,1){1}}
\put(5,0){\line(0,1){1}}
\put(2,1){\makebox(2,1){$\downarrow$}}
\put(1,0){\makebox(1,1){{\scriptsize $[8]$}}}
\put(2,0){\makebox(1,1){{\scriptsize $[7]$}}}
\put(3,0){\makebox(1,1){{\scriptsize $[4]$}}}
\put(4,0){\makebox(1,1){{\scriptsize $[3]$}}}

\put(8,0){\line(1,0){3}}
\put(8,1){\line(1,0){3}}
\put(9,0){\line(0,1){1}}
\put(10,0){\line(0,1){1}}
\put(11,0){\line(0,1){1}}
\put(9,0){\makebox(1,1){{\scriptsize $[6]$}}}
\put(10,0){\makebox(1,1){{\scriptsize $[5]$}}}
\end{picture}
\end{center}

If the $l^{(1)}$-string is singular of type-I, then find a singular string of length $l^{(2)}$ in $\Tilde{\nu}^{(2)}$ satisfying $[l^{(1)}+3]_{eff}=l^{(2)}+1$.
If such a string exists, then add three boxes to the singular $l^{(1)}$-string, one box to the $l^{(2)}$-string, and follow the algorithm of Case 5.
This corresponds to the following box marking $\nu^{(1)}$ (left) and $\nu^{(2)}$ (right).

\setlength{\unitlength}{10pt}
\begin{center}
\begin{picture}(9,2)
\put(0,0){\line(1,0){4}}
\put(0,1){\line(1,0){4}}
\put(1,0){\line(0,1){1}}
\put(2,0){\line(0,1){1}}
\put(3,0){\line(0,1){1}}
\put(4,0){\line(0,1){1}}
\put(1,1){\makebox(2,1){$\downarrow$}}
\put(1,0){\makebox(1,1){{\scriptsize $[8]$}}}
\put(2,0){\makebox(1,1){{\scriptsize $[6]$}}}
\put(3,0){\makebox(1,1){{\scriptsize $[5]$}}}

\put(7,0){\line(1,0){2}}
\put(7,1){\line(1,0){2}}
\put(8,0){\line(0,1){1}}
\put(9,0){\line(0,1){1}}
\put(8,0){\makebox(1,1){{\scriptsize $[7]$}}}
\end{picture}
\end{center}

If the $l^{(1)}$-string is singular of type-I, then find a singular string of length $l^{(2)}$ in $\Tilde{\nu}^{(2)}$ satisfying $[l^{(1)}+3]_{eff}=l^{(2)}+2$.
If such a string exists, then add three boxes to the singular $l^{(1)}$-string, two boxes to the $l^{(2)}$-string, and follow the algorithm of Case 4 ignoring the q-singular string of length $l^{(1)}-1$ in $\Tilde{\nu}^{(1)}$.
This corresponds to the following box marking $\nu^{(1)}$ (left) and $\nu^{(2)}$ (right).

\setlength{\unitlength}{10pt}
\begin{center}
\begin{picture}(10,2)
\put(0,0){\line(1,0){4}}
\put(0,1){\line(1,0){4}}
\put(1,0){\line(0,1){1}}
\put(2,0){\line(0,1){1}}
\put(3,0){\line(0,1){1}}
\put(4,0){\line(0,1){1}}
\put(1,1){\makebox(2,1){$\downarrow$}}
\put(1,0){\makebox(1,1){{\scriptsize $[8]$}}}
\put(2,0){\makebox(1,1){{\scriptsize $[6]$}}}
\put(3,0){\makebox(1,1){{\scriptsize $[5]$}}}

\put(7,0){\line(1,0){3}}
\put(7,1){\line(1,0){3}}
\put(8,0){\line(0,1){1}}
\put(9,0){\line(0,1){1}}
\put(10,0){\line(0,1){1}}
\put(8,0){\makebox(1,1){{\scriptsize $[7]$}}}
\put(9,0){\makebox(1,1){{\scriptsize $[4]$}}}
\end{picture}
\end{center}

If the $l^{(1)}$-string is singular of type-0/II, then find a singular (resp. q-singular) string of length $l^{(2)}$ in $\Tilde{\nu}^{(2)}$ satisfying $[l^{(1)}+2]_{eff}=l^{(2)}+2$ (resp. $[l^{(1)}+2]_{eff}=l^{(2)}+1$).
If such a string exists, then add two boxes to the $l^{(1)}$-string and follow the algorithm of Case 10 such that $i_{8}^{eff}=i_{6}$.
This corresponds to the following box marking in $\nu^{(1)}$.

\setlength{\unitlength}{10pt}
\begin{center}
\begin{picture}(3,1)
\put(0,0){\line(1,0){3}}
\put(0,1){\line(1,0){3}}
\put(1,0){\line(0,1){1}}
\put(2,0){\line(0,1){1}}
\put(3,0){\line(0,1){1}}
\put(1,0){\makebox(1,1){{\scriptsize $[8]$}}}
\put(2,0){\makebox(1,1){{\scriptsize $[7]$}}}
\end{picture}
\end{center}

If the $l^{(1)}$-string is q-singular of type-0, then find a singular string of length $l^{(2)}$ in $\Tilde{\nu}^{(2)}$ satisfying $[l^{(1)}+2]_{eff}=l^{(2)}+1$.
If such a string exists, then add two boxes to the $l^{(1)}$-string, one box to the $l^{(2)}$-string, and follow the algorithm of Case 7.
This corresponds to the following box marking in $\nu^{(1)}$ (left) and $\nu^{(2)}$ (right).

\setlength{\unitlength}{10pt}
\begin{center}
\begin{picture}(8,2)
\put(0,0){\line(1,0){3}}
\put(0,1){\line(1,0){3}}
\put(1,0){\line(0,1){1}}
\put(2,0){\line(0,1){1}}
\put(3,0){\line(0,1){1}}
\put(0,1){\makebox(2,1){$\downarrow$}}
\put(1,0){\makebox(1,1){{\scriptsize $[8]$}}}
\put(2,0){\makebox(1,1){{\scriptsize $[6]$}}}

\put(6,0){\line(1,0){2}}
\put(6,1){\line(1,0){2}}
\put(7,0){\line(0,1){1}}
\put(8,0){\line(0,1){1}}
\put(7,0){\makebox(1,1){{\scriptsize $[7]$}}}

\end{picture}
\end{center}

If the $l^{(1)}$-string is q-singular of type-I/II, then find a singular string of length $l^{(2)}$ in $\Tilde{\nu}^{(2)}$ satisfying $[l^{(1)}+2]_{eff}=l^{(2)}+2$ or $[l^{(1)}+2]_{eff}=l^{(2)}+1$.
If such a string exists, then add two boxes to the $l^{(1)}$-string and follow the algorithm of Case 11.
This corresponds to the following box marking in $\nu^{(2)}$.

\setlength{\unitlength}{10pt}
\begin{center}
\begin{picture}(8,2)
\put(0,0){\line(1,0){3}}
\put(0,1){\line(1,0){3}}
\put(1,0){\line(0,1){1}}
\put(2,0){\line(0,1){1}}
\put(3,0){\line(0,1){1}}
\put(1,0){\makebox(1,1){{\scriptsize $[7]$}}}
\put(2,0){\makebox(1,1){{\scriptsize $[4]$}}}

\put(4,0){\makebox(2,1){$\text{or}$}}

\put(7,0){\line(1,0){2}}
\put(7,1){\line(1,0){2}}
\put(8,0){\line(0,1){1}}
\put(9,0){\line(0,1){1}}
\put(8,0){\makebox(1,1){{\scriptsize $[7]$}}}

\end{picture}
\end{center}
with $i_{8}^{eff}=i_{7}$ and the $i_{8}$-string in  $\nu^{(1)}$ is marked by [8] only.

If the box adding so far is not successful, then do as follows.
If the $l^{(1)}$-string is singular, then add two boxes to the $l^{(1)}$-string and follow the algorithm of Case 10 such that $i_{8}^{eff}>i_{6}$.
If the $l^{(1)}$-string is q-singular, then add one box to the $l^{(1)}$-string and follow the algorithm of Case 11.

\begin{flushleft}
Case 13. $b= \framebox{$13$}$.
\end{flushleft}
Find the singular/q-singular/qq-singular string of maximum length $l^{(1)}$ in $\Tilde{\nu}^{(1)}$
We choose the string of the same length by the following preferential rule 
\begin{center}
\begin{tabular}{ccc}
singular & q-singular & qq-singular \\ \hline
1 & 2 & 3
\end{tabular}
\end{center}
In addition, we impose the following preferential rules.
\begin{itemize}
\item
If the $l^{(1)}$-string is q-singular/qq-singular of any type, then find a type-0 singular string of length $l^{(1)\prime}$ in $\Tilde{\nu}^{(1)}$ with $[l^{(1)\prime}]_{eff}=[l^{(1)}]_{eff}-1$ such that there exists a singular string of length $l^{(2)}$ in $\Tilde{\nu}^{(2)}$ satisfying $[l^{(1)\prime}+6]_{eff}=l^{(2)}+2$.
If such an $l^{(1)\prime}$-string exists, then reset the $l^{(1)}$-string to be the $l^{(1)\prime}$-string.
\item
If the $l^{(1)}$-string is qq-singular of type-II, then find a singular string of length $l^{(1)\prime}=l^{(1)}-4$ in $\Tilde{\nu}^{(1)}$ such that there exists a singular string of length $l^{(2)}$ in $\Tilde{\nu}^{(2)}$ satisfying $[l^{(1)\prime}+6]_{eff}=l^{(2)}+2$.
If such an $l^{(1)\prime}$-string exists, then reset the $l^{(1)}$-string to be the $l^{(1)\prime}$-string.
\item
If the $l^{(1)}$-string is q-singular of type-I, then find a singular string of length $l^{(1)\prime}=l^{(1)}-1$ in $\Tilde{\nu}^{(1)}$ such that there exists a singular string of length $l^{(2)}$ in $\Tilde{\nu}^{(2)}$ satisfying $[l^{(1)\prime}+5]_{eff}=l^{(2)}+2$.
If such an $l^{(1)\prime}$-string exists, then reset the $l^{(1)}$-string to be the $l^{(1)\prime}$-string.
\end{itemize}

The box adding goes as follows.

If the $l^{(1)}$-string is qq-singular, then add one box to the $l^{(1)}$-string and follow the algorithm of Case 12.

If the $l^{(1)}$-string is singular of type-0, then find a singular string of length $l^{(2)}$ in $\Tilde{\nu}^{(2)}$ satisfying $[l^{(1)}+6]_{eff}=l^{(2)}+2$.
If such a string exists, then add six boxes to the $l^{(1)}$-string, two boxes to the $l^{(2)}$-string, and follow the algorithm of Case 2.
This corresponds to the following box marking in $\nu^{(1)}$ (left) and $\nu^{(2)}$ (right).

\setlength{\unitlength}{10pt}
\begin{center}
\begin{picture}(13,2)
\put(0,0){\line(1,0){7}}
\put(0,1){\line(1,0){7}}
\put(1,0){\line(0,1){1}}
\put(2,0){\line(0,1){1}}
\put(3,0){\line(0,1){1}}
\put(4,0){\line(0,1){1}}
\put(5,0){\line(0,1){1}}
\put(6,0){\line(0,1){1}}
\put(7,0){\line(0,1){1}}
\put(3,1){\makebox(2,1){$\downarrow$}}
\put(1,0){\makebox(1,1){{\scriptsize $[9]$}}}
\put(2,0){\makebox(1,1){{\scriptsize $[8]$}}}
\put(3,0){\makebox(1,1){{\scriptsize $[7]$}}}
\put(4,0){\makebox(1,1){{\scriptsize $[4]$}}}
\put(5,0){\makebox(1,1){{\scriptsize $[3]$}}}
\put(6,0){\makebox(1,1){{\scriptsize $[2]$}}}

\put(10,0){\line(1,0){3}}
\put(10,1){\line(1,0){3}}
\put(11,0){\line(0,1){1}}
\put(12,0){\line(0,1){1}}
\put(13,0){\line(0,1){1}}
\put(11,0){\makebox(1,1){{\scriptsize $[6]$}}}
\put(12,0){\makebox(1,1){{\scriptsize $[5]$}}}

\end{picture}
\end{center}

If the $l^{(1)}$-string is singular of type-II, then find a singular string of length $l^{(2)}$ in $\Tilde{\nu}^{(2)}$ satisfying $[l^{(1)}+5]_{eff}=l^{(2)}+2$.
If such a string exists, then add five boxes to the $l^{(1)}$-string, two boxes to the $l^{(2)}$-string, and follow the algorithm of Case 3.
This corresponds to the following box marking in $\nu^{(1)}$ (left) and $\nu^{(2)}$ (right).

\setlength{\unitlength}{10pt}
\begin{center}
\begin{picture}(12,2)
\put(0,0){\line(1,0){6}}
\put(0,1){\line(1,0){6}}
\put(1,0){\line(0,1){1}}
\put(2,0){\line(0,1){1}}
\put(3,0){\line(0,1){1}}
\put(4,0){\line(0,1){1}}
\put(5,0){\line(0,1){1}}
\put(6,0){\line(0,1){1}}
\put(2,1){\makebox(2,1){$\downarrow$}}
\put(1,0){\makebox(1,1){{\scriptsize $[9]$}}}
\put(2,0){\makebox(1,1){{\scriptsize $[8]$}}}
\put(3,0){\makebox(1,1){{\scriptsize $[7]$}}}
\put(4,0){\makebox(1,1){{\scriptsize $[4]$}}}
\put(5,0){\makebox(1,1){{\scriptsize $[3]$}}}

\put(9,0){\line(1,0){3}}
\put(9,1){\line(1,0){3}}
\put(10,0){\line(0,1){1}}
\put(11,0){\line(0,1){1}}
\put(12,0){\line(0,1){1}}
\put(10,0){\makebox(1,1){{\scriptsize $[6]$}}}
\put(11,0){\makebox(1,1){{\scriptsize $[5]$}}}

\end{picture}
\end{center}

If the $l^{(1)}$-string is singular of type-I, then find a singular string of length $l^{(2)}$ in $\Tilde{\nu}^{(2)}$ satisfying $[l^{(1)}+4]_{eff}=l^{(2)}+1$.
If such a string exists, then add four boxes to the $l^{(1)}$-string, one box to the $l^{(2)}$-string, and follow the algorithm of Case 5.
This corresponds to the following box marking in $\nu^{(1)}$ (left) and $\nu^{(2)}$ (right).

\setlength{\unitlength}{10pt}
\begin{center}
\begin{picture}(10,2)
\put(0,0){\line(1,0){5}}
\put(0,1){\line(1,0){5}}
\put(1,0){\line(0,1){1}}
\put(2,0){\line(0,1){1}}
\put(3,0){\line(0,1){1}}
\put(4,0){\line(0,1){1}}
\put(5,0){\line(0,1){1}}
\put(1,1){\makebox(2,1){$\downarrow$}}
\put(1,0){\makebox(1,1){{\scriptsize $[9]$}}}
\put(2,0){\makebox(1,1){{\scriptsize $[8]$}}}
\put(3,0){\makebox(1,1){{\scriptsize $[6]$}}}
\put(4,0){\makebox(1,1){{\scriptsize $[5]$}}}

\put(8,0){\line(1,0){2}}
\put(8,1){\line(1,0){2}}
\put(9,0){\line(0,1){1}}
\put(10,0){\line(0,1){1}}
\put(9,0){\makebox(1,1){{\scriptsize $[7]$}}}
\end{picture}
\end{center}

If the $l^{(1)}$-string is singular of type-I, then find a singular string of length $l^{(2)}$ in $\Tilde{\nu}^{(2)}$ satisfying $[l^{(1)}+4]_{eff}=l^{(2)}+2$ .
If such a string exists, then add four boxes to the $l^{(1)}$-string, two boxes to the $l^{(2)}$-string, and follow the algorithm of Case 4 ignoring the q-singular string of length $l^{(1)}-1$ in $\Tilde{\nu}^{(1)}$.
This corresponds to the following box marking in $\nu^{(1)}$ (left) and $\nu^{(2)}$ (right).

\setlength{\unitlength}{10pt}
\begin{center}
\begin{picture}(11,2)
\put(0,0){\line(1,0){5}}
\put(0,1){\line(1,0){5}}
\put(1,0){\line(0,1){1}}
\put(2,0){\line(0,1){1}}
\put(3,0){\line(0,1){1}}
\put(4,0){\line(0,1){1}}
\put(5,0){\line(0,1){1}}
\put(1,1){\makebox(2,1){$\downarrow$}}
\put(1,0){\makebox(1,1){{\scriptsize $[9]$}}}
\put(2,0){\makebox(1,1){{\scriptsize $[8]$}}}
\put(3,0){\makebox(1,1){{\scriptsize $[6]$}}}
\put(4,0){\makebox(1,1){{\scriptsize $[5]$}}}

\put(8,0){\line(1,0){3}}
\put(8,1){\line(1,0){3}}
\put(9,0){\line(0,1){1}}
\put(10,0){\line(0,1){1}}
\put(11,0){\line(0,1){1}}
\put(9,0){\makebox(1,1){{\scriptsize $[7]$}}}
\put(10,0){\makebox(1,1){{\scriptsize $[4]$}}}
\end{picture}
\end{center}

If the box adding is not successful, then do as follows.
If the $l^{(1)}$-string is singular of type-0, then find a singular string of length $l^{(2)}$ in $\Tilde{\nu}^{(2)}$ satisfying $[l^{(1)}+3]_{eff}=l^{(2)}+2$ or a q-singular string of length $l^{(2)}$ satisfying $[l^{(1)}+3]_{eff}=l^{(2)}+1$ .
If such a string exists, then add three boxes to the $l^{(1)}$-string, and follow the algorithm of Case 10.
This corresponds to the following box marking in $\nu^{(1)}$ (left) and $\nu^{(2)}$ (right)

\setlength{\unitlength}{10pt}
\begin{center}
\begin{picture}(9,2)
\put(0,0){\line(1,0){4}}
\put(0,1){\line(1,0){4}}
\put(1,0){\line(0,1){1}}
\put(2,0){\line(0,1){1}}
\put(3,0){\line(0,1){1}}
\put(4,0){\line(0,1){1}}
\put(1,1){\makebox(2,1){$\downarrow$}}
\put(1,0){\makebox(1,1){{\scriptsize $[9]$}}}
\put(2,0){\makebox(1,1){{\scriptsize $[8]$}}}
\put(3,0){\makebox(1,1){{\scriptsize $[7]$}}}

\put(7,0){\line(1,0){2}}
\put(7,1){\line(1,0){2}}
\put(8,0){\line(0,1){1}}
\put(9,0){\line(0,1){1}}
\put(8,0){\makebox(1,1){{\scriptsize $x$}}}
\end{picture}
\end{center}
where
\setlength{\unitlength}{10pt}
\begin{picture}(1,1)
\put(0,0){\line(1,0){1}}
\put(0,1){\line(1,0){1}}
\put(0,0){\line(0,1){1}}
\put(1,0){\line(0,1){1}}
\put(0,0){\makebox(1,1){{\scriptsize $x$}}}
\end{picture}
 is
\setlength{\unitlength}{10pt}
\begin{picture}(2,1)
\put(0,0){\line(1,0){2}}
\put(0,1){\line(1,0){2}}
\put(0,0){\line(0,1){1}}
\put(1,0){\line(0,1){1}}
\put(2,0){\line(0,1){1}}
\put(0,0){\makebox(1,1){{\scriptsize $[6]$}}}
\put(1,0){\makebox(1,1){{\scriptsize $[5]$}}}
\end{picture}
or
\setlength{\unitlength}{10pt}
\begin{picture}(1,1)
\put(0,0){\line(1,0){1}}
\put(0,1){\line(1,0){1}}
\put(0,0){\line(0,1){1}}
\put(1,0){\line(0,1){1}}
\put(0,0){\makebox(1,1){{\scriptsize $[6]$}}}
\end{picture}
with $i_{9}^{eff}=i_{6}$.

If the box adding is not successful, then do as follows.
If the $l^{(1)}$-string is singular of any type, then find a singular string of length $l^{(2)}$ in $\Tilde{\nu}^{(2)}$ satisfying $[l^{(1)}+3]_{eff}-1\geq l^{(2)}+2$ or a q-singular string of length $l^{(2)}$ satisfying $[l^{(1)}+3]_{eff}-1\geq l^{(2)}+1$.
If such a string exists, then add three boxes to the $l^{(1)}$-string, and follow the algorithm of Case 10.
This corresponds to the box marking $\nu^{(1)}$ and $\nu^{(2)}$ depicted above but with $i_{9}^{eff}>i_{6}$.
If the $l^{(1)}$-string is q-singular of type-0, then find a singular string of length $l^{(2)}$ in $\Tilde{\nu}^{(2)}$ satisfying $[l^{(1)}+3]_{eff}=l^{(2)}+1$.
If such a string exists, then add three boxes to the $l^{(1)}$-string, one box to the $l^{(2)}$-string, and follow the algorithm of Case 7.
This corresponds to the following box marking in $\nu^{(1)}$ (left) and $\nu^{(2)}$ (right).

\setlength{\unitlength}{10pt}
\begin{center}
\begin{picture}(9,2)
\put(0,0){\line(1,0){4}}
\put(0,1){\line(1,0){4}}
\put(1,0){\line(0,1){1}}
\put(2,0){\line(0,1){1}}
\put(3,0){\line(0,1){1}}
\put(4,0){\line(0,1){1}}
\put(0,1){\makebox(2,1){$\downarrow$}}
\put(1,0){\makebox(1,1){{\scriptsize $[9]$}}}
\put(2,0){\makebox(1,1){{\scriptsize $[8]$}}}
\put(3,0){\makebox(1,1){{\scriptsize $[6]$}}}

\put(7,0){\line(1,0){2}}
\put(7,1){\line(1,0){2}}
\put(8,0){\line(0,1){1}}
\put(9,0){\line(0,1){1}}
\put(8,0){\makebox(1,1){{\scriptsize $[7]$}}}
\end{picture}
\end{center}
If the $l^{(1)}$-string is q-singular of type-II, then find a singular string of length $l^{(2)}$ in $\Tilde{\nu}^{(2)}$ satisfying $[l^{(1)}+2]_{eff}=l^{(2)}+1$ or $[l^{(1)}+2]_{eff}=l^{(2)}+2$.
If such a string exists, then add two boxes to the $l^{(1)}$-string and follow the algorithm of Case 11 such that the length of the box-added $i_{7}$-string is equal to $[l^{(1)}+2]_{eff}$.
This corresponds to the following box marking in $\nu^{(1)}$ (left) and $\nu^{(2)}$ (right)

\setlength{\unitlength}{10pt}
\begin{center}
\begin{picture}(9,2)
\put(0,0){\line(1,0){3}}
\put(0,1){\line(1,0){3}}
\put(1,0){\line(0,1){1}}
\put(2,0){\line(0,1){1}}
\put(3,0){\line(0,1){1}}
\put(2,1){\makebox(2,1){$\downarrow$}}
\put(1,0){\makebox(1,1){{\scriptsize $[9]$}}}
\put(2,0){\makebox(1,1){{\scriptsize $[8]$}}}

\put(6,0){\line(1,0){2}}
\put(6,1){\line(1,0){2}}
\put(7,0){\line(0,1){1}}
\put(8,0){\line(0,1){1}}
\put(7,0){\makebox(1,1){{\scriptsize $x$}}}
\end{picture}
\end{center}
where
\setlength{\unitlength}{10pt}
\begin{picture}(1,1)
\put(0,0){\line(1,0){1}}
\put(0,1){\line(1,0){1}}
\put(0,0){\line(0,1){1}}
\put(1,0){\line(0,1){1}}
\put(0,0){\makebox(1,1){{\scriptsize $x$}}}
\end{picture}
 is
\setlength{\unitlength}{10pt}
\begin{picture}(2,1)
\put(0,0){\line(1,0){2}}
\put(0,1){\line(1,0){2}}
\put(0,0){\line(0,1){1}}
\put(1,0){\line(0,1){1}}
\put(2,0){\line(0,1){1}}
\put(0,0){\makebox(1,1){{\scriptsize $[7]$}}}
\put(1,0){\makebox(1,1){{\scriptsize $[4]$}}}
\end{picture}
or
\setlength{\unitlength}{10pt}
\begin{picture}(1,1)
\put(0,0){\line(1,0){1}}
\put(0,1){\line(1,0){1}}
\put(0,0){\line(0,1){1}}
\put(1,0){\line(0,1){1}}
\put(0,0){\makebox(1,1){{\scriptsize $[7]$}}}
\end{picture}
with $i_{9}^{eff}=i_{7}$.

If the box adding so far is not successful, then do as follows.
If the $l^{(1)}$-string is q-singular of any type, then find a singular string of length $l^{(2)}$ in $\Tilde{\nu}^{(2)}$ satisfying $[l^{(1)}+2]_{eff}-1\geq l^{(2)}+1$.
If such a string exists, then add two boxes to the $l^{(1)}$-string and follow the algorithm of Case 11 such that the length of the box-added $i_{7}$-string is smaller than $[l^{(1)}+2]_{eff}$.


\begin{flushleft}
Case 14. $b= \framebox{$14$}$.
\end{flushleft}

Find two longest singular strings of length $l_{1}^{(2)}$ and $l_{2}^{(2)}(\leq l_{1}^{(2)})$ in $\Tilde{\nu}^{(2)}$.
If $l_{1}^{(2)}=l_{2}^{(2)}$, then find the singular string of length $l^{(1)}$ in $\Tilde{\nu}^{(1)}$ satisfying $[l^{(1)}+6]_{eff}=l_{1}^{(2)}$.
If such a string exists, then add six boxes to the $l^{(1)}$-string, two boxes to the $l_{1}^{(2)}$- and $l_{2}^{(2)}$-strings, and terminate the box adding.
This corresponds to the following box marking in $\nu^{(1)}$ (left) and $\nu^{(2)}$ (right).

\setlength{\unitlength}{10pt}
\begin{center}
\begin{picture}(13,2)
\put(0,0){\line(1,0){7}}
\put(0,1){\line(1,0){7}}
\put(1,0){\line(0,1){1}}
\put(2,0){\line(0,1){1}}
\put(3,0){\line(0,1){1}}
\put(4,0){\line(0,1){1}}
\put(5,0){\line(0,1){1}}
\put(6,0){\line(0,1){1}}
\put(7,0){\line(0,1){1}}
\put(3,1){\makebox(2,1){$\downarrow$}}
\put(1,0){\makebox(1,1){{\scriptsize $[9]$}}}
\put(2,0){\makebox(1,1){{\scriptsize $[8]$}}}
\put(3,0){\makebox(1,1){{\scriptsize $[7]$}}}
\put(4,0){\makebox(1,1){{\scriptsize $[4]$}}}
\put(5,0){\makebox(1,1){{\scriptsize $[3]$}}}
\put(6,0){\makebox(1,1){{\scriptsize $[2]$}}}

\put(10,0){\line(1,0){3}}
\put(10,1){\line(1,0){3}}
\put(10,2){\line(1,0){3}}
\put(11,0){\line(0,1){2}}
\put(12,0){\line(0,1){2}}
\put(13,0){\line(0,1){2}}
\put(11,0){\makebox(1,1){{\scriptsize $[10]$}}}
\put(11,1){\makebox(1,1){{\scriptsize $[6]$}}}
\put(12,0){\makebox(1,1){{\scriptsize $[1]$}}}
\put(12,1){\makebox(1,1){{\scriptsize $[5]$}}}
\end{picture}
\end{center}


If the box adding is not successful, then do as follows.
Find the singular string of maximum length $l_{1}^{(2)}$ in $\Tilde{\nu}^{(2)}$.
Then find the singular string of length $l_{2}^{(2)}$ in $\Tilde{\nu}^{(2)}$ satisfying $l_{2}^{(2)}=l_{1}^{(2)}-1$ and the type-0 singular string of length $l^{(1)}$ in $\Tilde{\nu}^{(1)}$ satisfying $[l^{(1)}+6]_{eff}=l_{1}^{(2)}+1$.
If such strings exist, then add six boxes to the $l^{(1)}$-string, one box to the $l_{1}^{(2)}$-string, two boxes to the $l_{2}^{(2)}$-string, and follow the algorithm of Case 2.
This corresponds to the following box marking in $\nu^{(1)}$ (left) and $\nu^{(2)}$ (right). 

\setlength{\unitlength}{10pt}
\begin{center}
\begin{picture}(13,2)
\put(0,0){\line(1,0){7}}
\put(0,1){\line(1,0){7}}
\put(1,0){\line(0,1){1}}
\put(2,0){\line(0,1){1}}
\put(3,0){\line(0,1){1}}
\put(4,0){\line(0,1){1}}
\put(5,0){\line(0,1){1}}
\put(6,0){\line(0,1){1}}
\put(7,0){\line(0,1){1}}
\put(3,1){\makebox(2,1){$\downarrow$}}
\put(1,0){\makebox(1,1){{\scriptsize $[9]$}}}
\put(2,0){\makebox(1,1){{\scriptsize $[8]$}}}
\put(3,0){\makebox(1,1){{\scriptsize $[7]$}}}
\put(4,0){\makebox(1,1){{\scriptsize $[4]$}}}
\put(5,0){\makebox(1,1){{\scriptsize $[3]$}}}
\put(6,0){\makebox(1,1){{\scriptsize $[2]$}}}

\put(10,0){\line(1,0){3}}
\put(10,1){\line(1,0){3}}
\put(10,2){\line(1,0){3}}
\put(11,0){\line(0,1){2}}
\put(12,0){\line(0,1){2}}
\put(13,0){\line(0,1){2}}
\put(11,0){\makebox(1,1){{\scriptsize $[6]$}}}
\put(12,0){\makebox(1,1){{\scriptsize $[5]$}}}
\put(12,1){\makebox(1,1){{\scriptsize $[10]$}}}
\end{picture}
\end{center}

If the box adding is not successful, then do as follows.
Find the singular string of maximum length $l_{1}^{(2)}$ in $\Tilde{\nu}^{(2)}$.
Then find the singular string of length $l_{2}^{(2)}$ in $\Tilde{\nu}^{(2)}$ satisfying $l_{2}^{(2)}=l_{1}^{(2)}-1$ and the type-II singular string of length $l^{(1)}$ in $\Tilde{\nu}^{(1)}$ satisfying $[l^{(1)}+5]_{eff}=l_{1}^{(2)}+1$.
If such strings exist, then add five boxes to the $l^{(1)}$-string, one box to the $l_{1}^{(2)}$-string, two boxes to the $l_{2}^{(2)}$-string, and follow the algorithm of Case 3.
This corresponds to the following box marking in $\nu^{(1)}$ (left) and $\nu^{(2)}$ (right).  

\setlength{\unitlength}{10pt}
\begin{center}
\begin{picture}(13,2)
\put(0,0){\line(1,0){6}}
\put(0,1){\line(1,0){6}}
\put(1,0){\line(0,1){1}}
\put(2,0){\line(0,1){1}}
\put(3,0){\line(0,1){1}}
\put(4,0){\line(0,1){1}}
\put(5,0){\line(0,1){1}}
\put(6,0){\line(0,1){1}}
\put(2,1){\makebox(2,1){$\downarrow$}}
\put(1,0){\makebox(1,1){{\scriptsize $[9]$}}}
\put(2,0){\makebox(1,1){{\scriptsize $[8]$}}}
\put(3,0){\makebox(1,1){{\scriptsize $[7]$}}}
\put(4,0){\makebox(1,1){{\scriptsize $[4]$}}}
\put(5,0){\makebox(1,1){{\scriptsize $[3]$}}}

\put(9,0){\line(1,0){3}}
\put(9,1){\line(1,0){3}}
\put(9,2){\line(1,0){3}}
\put(10,0){\line(0,1){2}}
\put(11,0){\line(0,1){2}}
\put(12,0){\line(0,1){2}}
\put(10,0){\makebox(1,1){{\scriptsize $[6]$}}}
\put(11,0){\makebox(1,1){{\scriptsize $[5]$}}}
\put(11,1){\makebox(1,1){{\scriptsize $[10]$}}}
\end{picture}
\end{center}

If the box adding is not successful, then do as follows.
Find the singular string of maximum length $l_{1}^{(2)}$ in $\Tilde{\nu}^{(2)}$.
Then find the singular string of length $l_{2}^{(2)}$ in $\Tilde{\nu}^{(2)}$ satisfying $l_{2}^{(2)}=l_{1}^{(2)}-1$ and the type-I singular string of length $l^{(1)}$ in $\Tilde{\nu}^{(1)}$ satisfying $[l^{(1)}+4]_{eff}=l_{1}^{(2)}+1$.
If such strings exist, then add four boxes to the $l^{(1)}$-string, one box to the $l_{1}^{(2)}$-string, two boxes to the $l_{2}^{(2)}$-string, and follow the algorithm of Case 4 ignoring the q-singular string of length $l^{(1)}-1$ in $\Tilde{\nu}^{(1)}$.
This corresponds to the following box marking in $\nu^{(1)}$ (left) and $\nu^{(2)}$ (right). 

\setlength{\unitlength}{10pt}
\begin{center}
\begin{picture}(10,2)
\put(0,0){\line(1,0){5}}
\put(0,1){\line(1,0){5}}
\put(1,0){\line(0,1){1}}
\put(2,0){\line(0,1){1}}
\put(3,0){\line(0,1){1}}
\put(4,0){\line(0,1){1}}
\put(5,0){\line(0,1){1}}
\put(1,1){\makebox(2,1){$\downarrow$}}
\put(1,0){\makebox(1,1){{\scriptsize $[9]$}}}
\put(2,0){\makebox(1,1){{\scriptsize $[8]$}}}
\put(3,0){\makebox(1,1){{\scriptsize $[6]$}}}
\put(4,0){\makebox(1,1){{\scriptsize $[5]$}}}

\put(8,0){\line(1,0){3}}
\put(8,1){\line(1,0){3}}
\put(8,2){\line(1,0){3}}
\put(9,0){\line(0,1){2}}
\put(10,0){\line(0,1){2}}
\put(11,0){\line(0,1){2}}
\put(9,0){\makebox(1,1){{\scriptsize $[7]$}}}
\put(10,0){\makebox(1,1){{\scriptsize $[4]$}}}
\put(10,1){\makebox(1,1){{\scriptsize $[10]$}}}
\end{picture}
\end{center}
In this case, if there exists a type-II singular string of length $l^{(1)\prime}=l^{(1)}-1$ in $\Tilde{\nu}^{(1)}$, then discard the box adding above, reset the $l^{(1)}$-string to be the $l^{(1)\prime}$-string, and follow the previous box marking.

If the box adding is not successful, then do as follows.
Find the singular string of maximum length $l_{1}^{(2)}$ in $\Tilde{\nu}^{(2)}$.
Then find the singular string of length $l_{2}^{(2)}$ in $\Tilde{\nu}^{(2)}$ satisfying $l_{2}^{(2)}=l_{1}^{(2)}$ and the type-I singular string of length $l^{(1)}$ in $\Tilde{\nu}^{(1)}$ satisfying $[l^{(1)}+4]_{eff}=l_{1}^{(2)}+1$.
If such strings exist, then add four boxes to the $l^{(1)}$-string, one box to the $l_{1}^{(2)}$- and $l_{2}^{(2)}$-strings, and follow the algorithm of Case 5.
This corresponds to the following box marking in $\nu^{(1)}$ (left) and $\nu^{(2)}$ (right).

\setlength{\unitlength}{10pt}
\begin{center}
\begin{picture}(10,2)
\put(0,0){\line(1,0){5}}
\put(0,1){\line(1,0){5}}
\put(1,0){\line(0,1){1}}
\put(2,0){\line(0,1){1}}
\put(3,0){\line(0,1){1}}
\put(4,0){\line(0,1){1}}
\put(5,0){\line(0,1){1}}
\put(1,1){\makebox(2,1){$\downarrow$}}
\put(1,0){\makebox(1,1){{\scriptsize $[9]$}}}
\put(2,0){\makebox(1,1){{\scriptsize $[8]$}}}
\put(3,0){\makebox(1,1){{\scriptsize $[6]$}}}
\put(4,0){\makebox(1,1){{\scriptsize $[5]$}}}

\put(8,0){\line(1,0){2}}
\put(8,1){\line(1,0){2}}
\put(8,2){\line(1,0){2}}
\put(9,0){\line(0,1){2}}
\put(10,0){\line(0,1){2}}
\put(9,0){\makebox(1,1){{\scriptsize $[7]$}}}
\put(9,1){\makebox(1,1){{\scriptsize $[10]$}}}
\end{picture}
\end{center}

If the box adding is not successful, then do as follows.
Find the singular string of maximum length $l_{1}^{(2)}$ in $\Tilde{\nu}^{(2)}$.
Then find the singular string of length $l^{(1)}$ in $\Tilde{\nu}^{(1)}$ satisfying $[l^{(1)}+3]_{eff}=l_{1}^{(2)}+1$.
If such strings exist, then add three boxes to the $l^{(1)}$-string, one box to the $l_{1}^{(2)}$-string, and follow the algorithm of Case 10.
This corresponds to the following box marking in $\nu^{(1)}$ (left) and $\nu^{(2)}$ (right) 

\setlength{\unitlength}{10pt}
\begin{center}
\begin{picture}(9,1)
\put(0,0){\line(1,0){4}}
\put(0,1){\line(1,0){4}}
\put(1,0){\line(0,1){1}}
\put(2,0){\line(0,1){1}}
\put(3,0){\line(0,1){1}}
\put(4,0){\line(0,1){1}}
\put(1,0){\makebox(1,1){{\scriptsize $[9]$}}}
\put(2,0){\makebox(1,1){{\scriptsize $[8]$}}}
\put(3,0){\makebox(1,1){{\scriptsize $[7]$}}}

\put(7,0){\line(1,0){2}}
\put(7,1){\line(1,0){2}}
\put(8,0){\line(0,1){1}}
\put(9,0){\line(0,1){1}}
\put(8,0){\makebox(1,1){{\scriptsize $[10]$}}}
\end{picture}
\end{center}
with $i_{9}^{eff}=i_{10}$. When the box-added $i_{9}$-string is type-0, the length of the box-added $i_{6}$-string is smaller than or equal to $i_{9}^{eff}$ and when the box-added $i_{9}$-string is type-I/II, the length of the box-added $i_{6}$-string is strictly smaller than $i_{9}^{eff}$.

If the box adding is not successful, then do as follows.
Find the singular string of maximum length $l_{1}^{(2)}$ in $\Tilde{\nu}^{(2)}$.
Then find the type-II singular string of length $l^{(1)}$ in $\Tilde{\nu}^{(1)}$ satisfying $[l^{(1)}+2]_{eff}=l_{1}^{(2)}+1$.
If such strings exist, then add two boxes to the $l^{(1)}$-string, one box to the $l_{1}^{(2)}$-string, and follow the algorithm of Case 11.
This corresponds to the following box marking in $\nu^{(1)}$ (left) and $\nu^{(2)}$ (right)

\setlength{\unitlength}{10pt}
\begin{center}
\begin{picture}(8,2)
\put(0,0){\line(1,0){3}}
\put(0,1){\line(1,0){3}}
\put(1,0){\line(0,1){1}}
\put(2,0){\line(0,1){1}}
\put(3,0){\line(0,1){1}}
\put(2,1){\makebox(2,1){$\downarrow$}}
\put(1,0){\makebox(1,1){{\scriptsize $[9]$}}}
\put(2,0){\makebox(1,1){{\scriptsize $[8]$}}}

\put(6,0){\line(1,0){2}}
\put(6,1){\line(1,0){2}}
\put(7,0){\line(0,1){1}}
\put(8,0){\line(0,1){1}}
\put(7,0){\makebox(1,1){{\scriptsize $[10]$}}}
\end{picture}
\end{center}
with $i_{9}^{eff}=i_{10}$.

If the box adding is not successful, then do as follows.
Find the singular string of maximum length $l_{1}^{(2)}$ in $\Tilde{\nu}^{(2)}$.
Then find the type-I singular string of length $l^{(1)}$ in $\Tilde{\nu}^{(1)}$ satisfying $[l^{(1)}+1]_{eff}=l_{1}^{(2)}+1$.
If such strings exist, then add one boxes to the $l^{(1)}$-string, one box to the $l_{1}^{(2)}$-string, and follow the algorithm of Case 12.
This corresponds to the following box marking in $\nu^{(1)}$ (left) and $\nu^{(2)}$ (right) 

\setlength{\unitlength}{10pt}
\begin{center}
\begin{picture}(7,2)
\put(0,0){\line(1,0){2}}
\put(0,1){\line(1,0){2}}
\put(1,0){\line(0,1){1}}
\put(2,0){\line(0,1){1}}
\put(1,1){\makebox(2,1){$\downarrow$}}
\put(1,0){\makebox(1,1){{\scriptsize $[9]$}}}

\put(5,0){\line(1,0){2}}
\put(5,1){\line(1,0){2}}
\put(6,0){\line(0,1){1}}
\put(7,0){\line(0,1){1}}
\put(6,0){\makebox(1,1){{\scriptsize $[10]$}}}
\end{picture}
\end{center}
with $i_{9}^{eff}=i_{10}$.
In this case, if there exists a type-II singular string of length $l^{(1)\prime}=l^{(1)}-1$ in $\Tilde{\nu}^{(1)}$, then discard the box adding above, reset the $l^{(1)}$-string to be the $l^{(1)\prime}$-string, and follow the previous box marking.

If the box adding so far is not successful, do as follows.
Let $l_{1}^{(2)}$ and $l_{2}^{(2)}(\leq l_{1}^{(2)})$ be the two largest length of singular strings in $\Tilde{\nu}^{(2)}$.
Find the q-singular/qq-singular string of maximum length $l^{(1)}$ in $\Tilde{\nu}^{(1)}$ such that it is
\begin{itemize}
\item[(1)]
a q-singular string of type-0 with $[l^{(1)}+3]_{eff}=l_{1}^{(2)}+1=l_{2}^{(2)}+1$ or
\item[(2)]
a q-singular string of type-0/I with $[l^{(1)}+2]_{eff}=l_{1}^{(2)}+1$ or 
\item[(3)]
a q-singular string of type-II with $[l^{(1)}+1]_{eff}=l_{1}^{(2)}+1$ or
\item[(4)]
a qq-singular string of type-0 with $[l^{(1)}+1]_{eff}=l_{1}^{(2)}+1$.
\end{itemize}
We choose the q-singular string when there exist q-singular and qq-singular strings of the same length.
In addition, we impose the following preferential rules.
\begin{itemize}
\item
If the $l^{(1)}$-string is q-singular/qq-singular of type-0 or q-singular of type-I, then find a type-0 singular string of length $l^{(1)\prime}$ in $\Tilde{\nu}^{(1)}$ with $[l^{(1)\prime}]_{eff}=[l^{(1)}]_{eff}-1$ such that there exist singular strings of length $l_{1}^{(2)}$ and $l_{2}^{(2)}$ in $\Tilde{\nu}^{(2)}$ satisfying $[l^{(1)\prime}+6]_{eff}=l_{1}^{(2)}+1=l_{2}^{(2)}+2$.
If such an $l^{(1)\prime}$-string exists, then reset the $l^{(1)}$-string to be the $l^{(1)\prime}$-string.
\item
If the $l^{(1)}$-string is q-singular of type-II, then find a string of length $l^{(1)\prime}=l^{(1)}-4$ in $\Tilde{\nu}^{(1)}$ such that there exist singular strings of length $l_{1}^{(2)}$ and $l_{2}^{(2)}$ in $\Tilde{\nu}^{(2)}$ satisfying $[l^{(1)\prime}+6]_{eff}=l_{1}^{(2)}+1=l_{2}^{(2)}+2$.
If such an $l^{(1)\prime}$-string exists, then reset the $l^{(1)}$-string to be the $l^{(1)\prime}$-string.
\item
If the $l^{(1)}$-string is q-singular of type-0, then find a string of length $l^{(1)\prime}=l^{(1)}-2$ in $\Tilde{\nu}^{(1)}$ such that there exist singular strings of length $l_{1}^{(2)}$ and $l_{2}^{(2)}$ in $\Tilde{\nu}^{(2)}$ satisfying $[l^{(1)\prime}+5]_{eff}=l_{1}^{(2)}+1=l_{2}^{(2)}+2$.
If such an $l^{(1)\prime}$-string exists, then reset the $l^{(1)}$-string to be the $l^{(1)\prime}$-string.
\end{itemize}

If the box adding is not successful, then do as follows.
If the $l^{(1)}$-string is singular, then do as before.

Find two longest singular strings of length $l_{1}^{(2)}$ and $l_{2}^{(2)}$.
If $l_{1}^{(2)}=l_{2}^{(2)}$, then find the type-0 q-singular string of length $l^{(1)}(\geq 3)$ in $\Tilde{\nu}^{(1)}$ satisfying $[l^{(1)}+3]_{eff}=l_{1}^{(2)}+1$.
If such a string exists, then add three boxes to the $l^{(1)}$-string, one box to the $l_{1}^{(2)}$- and $l_{2}^{(2)}$-strings, and follow the algorithm of Case 7.
This corresponds to the following box marking in $\nu^{(1)}$ (left) and $\nu^{(2)}$ (right).

\setlength{\unitlength}{10pt}
\begin{center}
\begin{picture}(9,2)
\put(0,0){\line(1,0){4}}
\put(0,1){\line(1,0){4}}
\put(1,0){\line(0,1){1}}
\put(2,0){\line(0,1){1}}
\put(3,0){\line(0,1){1}}
\put(4,0){\line(0,1){1}}
\put(0,1){\makebox(2,1){$\downarrow$}}
\put(1,0){\makebox(1,1){{\scriptsize $[9]$}}}
\put(2,0){\makebox(1,1){{\scriptsize $[8]$}}}
\put(3,0){\makebox(1,1){{\scriptsize $[6]$}}}

\put(7,0){\line(1,0){2}}
\put(7,1){\line(1,0){2}}
\put(7,2){\line(1,0){2}}
\put(8,0){\line(0,1){2}}
\put(9,0){\line(0,1){2}}
\put(8,0){\makebox(1,1){{\scriptsize $[7]$}}}
\put(8,1){\makebox(1,1){{\scriptsize $[10]$}}}
\end{picture}
\end{center}

If the box adding is not successful, then do as follows.
Find the singular strings of length $l_{1}^{(2)}$ in $\Tilde{\nu}^{(2)}$.
Then find the type-0/I q-singular string of length $l^{(1)}$ in $\Tilde{\nu}^{(1)}$ satisfying $[l^{(1)}+2]_{eff}=l_{1}^{(2)}+1$.
If such a string exists, then add two boxes to the $l^{(1)}$-string, one box to the $l_{1}^{(2)}$-string, and follow the algorithm of Case 11.
This corresponds to the following box marking in $\nu^{(1)}$ (left) and $\nu^{(2)}$ (right).  

\setlength{\unitlength}{10pt}
\begin{center}
\begin{picture}(8,1)
\put(0,0){\line(1,0){3}}
\put(0,1){\line(1,0){3}}
\put(1,0){\line(0,1){1}}
\put(2,0){\line(0,1){1}}
\put(3,0){\line(0,1){1}}
\put(1,0){\makebox(1,1){{\scriptsize $[9]$}}}
\put(2,0){\makebox(1,1){{\scriptsize $[8]$}}}

\put(6,0){\line(1,0){2}}
\put(6,1){\line(1,0){2}}
\put(7,0){\line(0,1){1}}
\put(8,0){\line(0,1){1}}
\put(7,0){\makebox(1,1){{\scriptsize $[10]$}}}
\end{picture}
\end{center}
with $i_{9}^{eff}=i_{10}$.

If the box adding is not successful, then do as follows.
Find the singular strings of length $l_{1}^{(2)}$ in $\Tilde{\nu}^{(2)}$.
Then find the type-II q-singular string of length $l^{(1)}$ in $\Tilde{\nu}^{(1)}$ satisfying $[l^{(1)}+1]_{eff}=l_{1}^{(2)}+1$.
If such a string exists, then add one box to the $l^{(1)}$-string, one box to the $l_{1}^{(2)}$-string, and follow the algorithm of Case 12.
This corresponds to the following box marking in $\nu^{(1)}$ (left) and $\nu^{(2)}$ (right).

\setlength{\unitlength}{10pt}
\begin{center}
\begin{picture}(8,2)
\put(0,0){\line(1,0){3}}
\put(0,1){\line(1,0){3}}
\put(1,0){\line(0,1){1}}
\put(2,0){\line(0,1){1}}
\put(3,0){\line(0,1){1}}
\put(0,1){\makebox(2,1){$\downarrow$}}
\put(2,0){\makebox(1,1){{\scriptsize $[9]$}}}

\put(6,0){\line(1,0){2}}
\put(6,1){\line(1,0){2}}
\put(7,0){\line(0,1){1}}
\put(8,0){\line(0,1){1}}
\put(7,0){\makebox(1,1){{\scriptsize $[10]$}}}
\end{picture}
\end{center}
with $i_{9}^{eff}=i_{10}$.

If the box adding is not successful, then do as follows.
Find the singular strings of length $l_{1}^{(2)}$ in $\Tilde{\nu}^{(2)}$.
Then find the type-0 qq-singular string of length $l^{(1)}(\geq 4)$ in $\Tilde{\nu}^{(1)}$ satisfying $[l^{(1)}+1]_{eff}=l_{1}^{(2)}+1$.
If such a string exists, then add one box to the $l^{(1)}$-string, one box to the $l_{1}^{(2)}$-string, and follow the algorithm of Case 12.
This corresponds to the following box marking in $\nu^{(1)}$ (left) and $\nu^{(2)}$ (right)

\setlength{\unitlength}{10pt}
\begin{center}
\begin{picture}(7,2)
\put(0,0){\line(1,0){2}}
\put(0,1){\line(1,0){2}}
\put(1,0){\line(0,1){1}}
\put(2,0){\line(0,1){1}}
\put(0,1){\makebox(2,1){$\downarrow$}}
\put(1,0){\makebox(1,1){{\scriptsize $[9]$}}}

\put(5,0){\line(1,0){2}}
\put(5,1){\line(1,0){2}}
\put(6,0){\line(0,1){1}}
\put(7,0){\line(0,1){1}}
\put(6,0){\makebox(1,1){{\scriptsize $[10]$}}}
\end{picture}
\end{center}
with $i_{9}^{eff}=i_{10}$.

If the box adding is not successful, then do as follows.
Find the singular string of maximum length in $\Tilde{\nu}^{(2)}$.
Then add one box to this string and follow the algorithm of Case 13.
This corresponds to the box marking in $\nu$ depicted above but with $i_{9}^{eff}<i_{10}$.

\section{Proof of Theorem~\ref{th:main}}

Theorem~\ref{th:main} is proved in this section.
The following notation is used.
Let $(\nu,J)\in \mathrm{RC}(\lambda,L)$, $b=\gamma (\nu,J)$, $\rho=\lambda-\mathrm{wt}(b)$, and $(\Tilde{\nu},\Tilde{J})=\delta_{\theta} (\nu,J)$.
For $(\nu,J)\in \mathrm{RC}(\lambda,L)$, we define $\Delta c(\nu,J)=c(\nu,J)-c(\delta_{\theta} (\nu,J))$.
The following lemma is essentially the same as~\cite[Lemma 5.1]{OSS03a}.

\begin{lem} \label{lem:stat}
To prove that Eq.~\eqref{eq:main} holds, it suffices to show that it holds for $L=1$ and that for $L\geq 2$ with $\Phi (\nu,J)=b_{1}\otimes \cdots \otimes b_{L}$, we have
\begin{equation} \label{eq:dc}
\Delta c(\nu,J)=-\alpha_{1}^{(2)}+\chi (b_{1}=\emptyset)
\end{equation}
and
\begin{equation} \label{eq:Ha}
H(b_{1}\otimes b_{2})=\Tilde{\alpha} _{1}^{(2)}-\alpha_{1}^{(2)}+\chi (b_{1}=\emptyset)-\chi (b_{2}=\emptyset),
\end{equation}
where $\alpha _{1}^{(2)}$ and $\Tilde{\alpha}_{1}^{(2)}$ are the lengths of the first columns in $\nu ^{(2)}$ and $\Tilde{\nu}^{(2)}$ respectively, and $\delta_{\theta} (\nu,J)=(\Tilde{\nu},\Tilde{J})$.
\end{lem}

There are five things that must be verified:

\begin{itemize}
\item[(I)] $\rho$ is dominant.
\item[(II)] $(\Tilde{\nu},\Tilde{J})\in \mathrm{RC}(\rho,L-1)$.
\item[(III)] $b$ can be appended to $(\Tilde{\nu},\Tilde{J})$.
\item[(IV)] \eqref{eq:dc} in Lemma~\ref{lem:stat} holds.
\item[(V)] \eqref{eq:Ha} in Lemma~\ref{lem:stat} holds.
\end{itemize}
Parts (I) and (II) show that $\delta_{\theta}$ is well-defined.
Part (III) shows that $\delta_{\theta}$ has an inverse.
Parts (IV) and (V) suffice to prove that $\Phi$ preserves statistics.
We omit the proof of (III) as it is very similar to the proof of well-definedness.

We need several preliminary lemmas on the convexity and nonnegativity of the vacancy numbers $p_{i}^{(a)}$.

\begin{lem} \label{lem:large_i}
For large $i$, we have
\[
p_{i}^{(a)}=\lambda_{a}
\]
where $\lambda_{a}$ is defined by $\lambda=\sum_{a\in I_{0}}\lambda_{a}\Bar{\Lambda}_{a}$.
\end{lem}

\begin{proof}
This follows from the formula for the vacancy number Eq.~\eqref{eq:vacancy} and the constraint Eq.~\eqref{eq:admissible}.
\end{proof}

Direct calculations show that

\begin{align}
-p_{3i-2}^{(1)}+2p_{3i-1}^{(1)}-p_{3i}^{(1)}&=-2m_{3i-1}^{(1)}, \label{eq:ddp1a} \\
-p_{3i-1}^{(1)}+2p_{3i}^{(1)}-p_{3i+1}^{(1)}&=-2m_{3i}^{(1)}+m_{i}^{(2)}, \label{eq:ddp1b}\\
-p_{3i}^{(1)}+2p_{3i+1}^{(1)}-p_{3i+2}^{(1)}&=-2m_{3i+1}^{(1)}, \label{eq:ddp1c}
\end{align}
and
\begin{align}
-p_{i-1}^{(2)}+2p_{i}^{(2)}-p_{i+1}^{(2)}=&3m_{3i}^{(1)}+2(m_{3i-1}^{(1)}+m_{3i+1}^{(1)})
+m_{3i-2}^{(1)}+m_{3i+2}^{(1)} \label{eq:ddp2} \\
-&2m_{i}^{(2)}. \nonumber
\end{align}
In particular, these equations imply the convexity condition
\begin{equation} \label{eq:convexity}
p_{i}^{(a)}\geq \frac{1}{2}(p_{i-1}^{(a)}+p_{i+1}^{(a)})\quad \text{if } m_{i}^{(a)}=0.
\end{equation}

The following two lemmas follow immediately from Lemma~\ref{lem:large_i} and the convexity condition Eq.~\eqref{eq:convexity}.

\begin{lem} \label{lem:vacancy01}
Let $\nu$ be a configuration and let $n$ be $0$ or $1$.
The following are equivalent:
\begin{itemize}
\item[(1)]
$p_{i}^{(a)}\geq n$ for all $i\in \mathbb{Z}_{>0}$, $a\in I_{0}$;
\item[(2)]
$p_{i}^{(a)}\geq n$ for all $i\in \mathbb{Z}_{>0}$, $a\in I_{0}$ such that $m_{i}^{(a)}>0$.
\end{itemize}
\end{lem}

\begin{lem} \label{lem:vacancy23}
Let $\nu$ be a configuration and let $n$ be $2$ or $3$.
Let $p_{i_{0}}^{(a)}\geq n$.
The following are equivalent:
\begin{itemize}
\item[(1)]
$p_{i}^{(a)}\geq n$ for all $i\geq i_{0}$, $a\in I_{0}$;
\item[(2)]
$p_{i}^{(a)}\geq n$ for all $i\geq i_{0}$, $a\in I_{0}$ such that $m_{i}^{(a)}>0$.
\end{itemize}
\end{lem}

\begin{lem} \label{lem:vacancy1}
Suppose that
\[
p_{i}^{(a)}\geq 1, \ldots, p_{j}^{(a)}\geq 1 \quad (i<j)
\]
with $m_{i}^{(a)}=\cdots =m_{j}^{(a)}=0$.
If $p_{k}^{(a)}=1$ for some $k$ $(i\leq k\leq j)$, then 
\[
p_{i}^{(a)}= 1, \ldots, p_{j}^{(a)}= 1 \quad (i<j).
\]
\end{lem}

\begin{proof}
This follows immediately from the convexity condition Eq.~\eqref{eq:convexity}.
\end{proof}

\begin{proof}[Proof of (I)]
Here we show $\rho=\lambda-\mathrm{wt}(b)$ is dominant with $l\geq 1$.
Let us assume that $\nu^{(1)}\neq \emptyset$ and $\nu^{(2)}\neq \emptyset$ and let $l^{(a)}$ be the largest part in $\nu^{(a)}$ $(a=1,2)$.
In particular, $m_{l^{(1)}}^{(1)}\neq 0$ and $m_{l^{(2)}}^{(2)}\neq 0$.
The proof for the case when $\nu^{(1)}= \emptyset$ or $\nu^{(2)}= \emptyset$ is much simpler and we omit the details.

\begin{flushleft}
Case 0. $b= \framebox{$\emptyset$}$.
\end{flushleft}
Since $\mathrm{wt}(\framebox{$\emptyset$})=0$, there is nothing to prove.

\begin{flushleft}
Case 1. $b= \framebox{$1$}$.
\end{flushleft}
Since $\mathrm{wt}(\framebox{$1$})=\Bar{\Lambda}_{2}$, we must show that $\lambda_{2}\geq 1$.
Suppose that $\lambda_{2}=0$ so that $p_{\infty}^{(2)}=0$.
The convexity condition implies  $p_{l^{(2)}}^{(2)}=0$ so that $\nu^{(2)}$ has a singular string of length $l^{(2)}$, which contradicts $b= \framebox{$1$}$.
Hence $\lambda_{2}\geq 1$.

\begin{flushleft}
Case 2. $b= \framebox{$2$}$.
\end{flushleft}
Since $\mathrm{wt}(\framebox{$2$})=3\Bar{\Lambda}_{1}-\Bar{\Lambda}_{2}$, we must show that $\lambda_{1}\geq 3$.
We first suppose that 
\[
\lambda_{1}=-2\sum_{j=1}^{l^{(1)}}j m_{j}^{(1)}+3\sum_{j=1}^{l^{(2)}}j m_{j}^{(2)}=0.
\]
Then
\[
p_{l^{(1)}}^{(1)}=-2\sum_{j=1}^{l^{(1)}}j m_{j}^{(1)}+\sum_{j=1}^{l^{(2)}}\min(l^{(1)},3j) m_{j}^{(2)} 
=\sum_{j=1}^{l^{(2)}}\left( \min(l^{(1)},3j)-3j \right) m_{j}^{(2)}.
\]
If $l^{(1)}<3l^{(2)}$, then $p_{l^{(1)}}^{(1)}<0$.
Hence $l^{(1)}\geq 3l^{(2)}$, which yields  $p_{l^{(1)}}^{(1)}=0$.
This implies that $\nu^{(1)}$ has a singular string of effective length $(\geq l^{(2)})$, which contradicts $b= \framebox{$2$}$.
Secondly, we suppose that $\lambda_{1}=1$.
Then
\[
p_{l^{(1)}}^{(1)}=1+\sum_{j=1}^{l^{(2)}}\left( \min(l^{(1)},3j)-3j \right) m_{j}^{(2)}.
\]
If $m_{l^{(2)}}^{(2)}=1$, then $p_{l^{(1)}}^{(1)}<0$ when $l^{(1)}\leq 3l^{(2)}-2$.
If $m_{l^{(2)}}^{(2)}\geq 2$, then $p_{l^{(1)}}^{(1)}<0$ when $l^{(1)}\leq 3l^{(2)}-1$.
Hence $l^{(1)}\geq 3l^{(2)}-1$ if $m_{l^{(2)}}^{(2)}=1$ and $l^{(1)}\geq 3l^{(2)}$ if $m_{l^{(2)}}^{(2)}\geq 2$.
Simple calculations yield two possible cases:

\begin{itemize}
\item[(1)] $m_{l^{(2)}}^{(2)}=1$, $l^{(1)}=3l^{(2)}-1$, and $p_{l^{(1)}}^{(1)}=0$.
\item[(2)] $m_{l^{(2)}}^{(2)}\geq 1$, $l^{(1)}\geq 3l^{(2)}$, and $p_{l^{(1)}}^{(1)}=1$.
\end{itemize}
Case (1) contradicts $b= \framebox{$2$}$ because $\nu^{(1)}$ has a type-I singular string of effective length $l^{(2)}$.
Case (2) also contradicts $b= \framebox{$2$}$ because $\nu^{(1)}$ has a string (singular or q-singular) of length $(\geq 3l^{(2)})$.
Hence $\lambda_{1}\geq 2$.
Thirdly, we suppose that $\lambda_{1}=2$.
By the similar calculations as above yield four possible cases:

\begin{itemize}
\item[(1)] $m_{l^{(2)}}^{(2)}=1$, $l^{(1)}=3l^{(2)}-2$, and $p_{l^{(1)}}^{(1)}=0$.
\item[(2)] $m_{l^{(2)}}^{(2)}=1$, $l^{(1)}=3l^{(2)}-1$, and $p_{l^{(1)}}^{(1)}=1$.
\item[(3)] $m_{l^{(2)}}^{(2)}=2$, $l^{(1)}=3l^{(2)}-1$, and $p_{l^{(1)}}^{(1)}=0$.
\item[(4)] $m_{l^{(2)}}^{(2)}\geq 1$, $l^{(1)}\geq 3l^{(2)}$, and $p_{l^{(1)}}^{(1)}=2$.
\end{itemize}
Case (1) (resp. (3)) implies that $\nu^{(1)}$ has a type-II (resp. I) singular string of effective length $l^{(2)}$.
Case (2) implies that  $\nu^{(1)}$ has a type-I string (singular or q-singular) of effective length $l^{(2)}$.
Case (4) implies that  $\nu^{(1)}$ has a string (singular or q-singular or qq-singular) of length $(\geq 3l^{(2)})$.
All contradict $b= \framebox{$2$}$ by the algorithm of $\delta_{\theta}$.
Hence, $\lambda_{1}\geq 3$.

\begin{flushleft}
Case 3. $b= \framebox{$3$}$.
\end{flushleft}
In this case, after $[2]$ was marked in $\nu^{(1)}$ the box marking has terminated.
Since $\mathrm{wt}(\framebox{$3$})=\Bar{\Lambda}_{1}$, we must show that $\lambda_{1}\geq 1$.
Suppose that $\lambda_{1}=0$.
As in Case 2, it follows that $\nu^{(1)}$ has a singular string of length ($\leq 3i_{1}$) contradicting $b= \framebox{$3$}$.
Hence $\lambda_{1}\geq 1$.
Note that even if [2] was marked in this singular string, [3] can also be marked in this string.

\begin{flushleft}
Case 4. $b= \framebox{$4$}$.
\end{flushleft}
In this case, after $[3]$ was marked in $\nu^{(1)}$ the box marking has terminated.
Since $\mathrm{wt}(\framebox{$4$})=-\Bar{\Lambda}_{1}+\Bar{\Lambda}_{2}$, we must show that $\lambda_{2}\geq 1$.
We first suppose that
\[
\lambda_{2}=L+\sum_{j=1}^{l^{(1)}}j m_{j}^{(1)}-2\sum_{j=1}^{l^{(2)}}j m_{j}^{(2)}=0.
\]
Then we have $p_{l^{(2)}+1}^{(2)}$ by convexity and 
\[
p_{l^{(2)}}^{(2)}=\sum_{j=1}^{l^{(1)}}\left( \min(3l^{(2)},j)-j \right) m_{j}^{(1)}.
\]
If $l^{(1)}>3l^{(2)}$, then $p_{l^{(2)}}^{(2)}<0$.
Hence $l^{(1)}\leq 3l^{(2)}$, which yields $p_{l^{(2)}}^{(2)}=0$ so that all $l^{(2)}$-strings in $\nu^{(2)}$ are singular.
If [3] is marked in a string of of effective length $(<l^{(2)})$, then [4] can be marked in a singular string of length $l^{(2)}$ in $\nu^{(2)}$.
This is a contradiction.
Note that [1] must have been marked in a singular string of length $(< l^{(2)})$ in this case.
Therefore, it suffices to consider the case when [3] had been marked in a string of effective length $l^{(2)}$ in $\nu^{(1)}$.

\begin{itemize}
\item[(1)] $l^{(1)}=3l^{(2)}$.

We show that $m_{l^{(2)}}^{(2)}\geq 2$.
Suppose that $m_{l^{(2)}}^{(2)}=1$.
From Eq.~\eqref{eq:ddp2}, we have
\begin{equation} \label{eq:Icase4}
-p_{l^{(2)}-1}^{(2)}=3m_{3l^{(2)}}^{(1)}+2m_{3l^{(2)}-1}^{(1)}+m_{3l^{(2)}-2}^{(1)}-2.
\end{equation}
However, this does not hold because $m_{3l^{(2)}}^{(1)}=m_{l^{(1)}}^{(1)}> 0$.
Hence  $m_{l^{(2)}}^{(2)}\geq 2$ so that [4] can be marked in an $l^{(2)}$-string in $\nu^{(2)}$ even if [1] is already marked in another $l^{(2)}$-string in $\nu^{(2)}$, which contradicts $b= \framebox{$4$}$. 

\item[(2)] $l^{(1)}=3l^{(2)}-1$.

Suppose that $m_{l^{(2)}}^{(2)}=1$.
Equation~\eqref{eq:Icase4} still holds, from which we have that $p_{l^{(2)}-1}^{(2)}=0$ because $m_{l~{(1)}}^{(1)}=m_{3l^{(2)}-1}^{(1)}>0$. 
Suppose that $m_{j}^{(2)}\neq 0$ for some $j (\leq l^{(2)}-1)$.
Then, $p_{j}^{(2)}\geq 1$ as otherwise [1] would have been marked in a singular string of length ($\leq l^{(2)}-2$) contradicting $b= \framebox{$4$}$.
However, $p_{j}^{(2)}\geq 1$ implies that $p_{l^{(2)}-1}^{(2)}\geq 1$ by convexity, which in turn contradicts Eq.~\eqref{eq:Icase4}.
Therefore it suffices to consider the case when $m_{j}^{(2)}=0$ $(j\leq l^{(2)}-1)$.
We show that $m_{l^{(2)}}^{(2)}\geq 2$.
Suppose that $m_{l^{(2)}}^{(2)}=1$.
Then from Eq.~\eqref{eq:vacancy1}, we have $p_{l^{(1)}}^{(1)}<0$ so that $m_{l^{(2)}}^{(2)}\geq 2$, which implies that [4] can be marked in an $l^{(2)}$-string in $\nu^{(2)}$ even if [1] is already marked in another $l^{(2)}$-string in $\nu^{(2)}$, which contradicts $b= \framebox{$4$}$. 

\item[(3)] $l^{(1)}=3l^{(2)}-2$.

In this case, [1] been marked in a singular string of length ($\leq l^{(2)}-1$) in $\nu^{(2)}$.
Therefore, [4] can be marked in an $l^{(2)}$-string in $\nu^{(2)}$, which contradicts $b= \framebox{$4$}$. 

\end{itemize}

All cases (a), (b), and (c) contradict $b= \framebox{$4$}$.
Hence $\lambda_{2}\geq 1$.

\begin{flushleft}
Case 5. $b= \framebox{$5$}$.
\end{flushleft}
In this case, after $[4]$ was marked in $\nu^{(2)}$ the box marking has terminated.
Since $\mathrm{wt}(\framebox{$5$})=2\Bar{\Lambda}_{1}-\Bar{\Lambda}_{2}$, we must show that $\lambda_{1}\geq 2$.
We first suppose that $\lambda_{1}=0$.
Then we have $l^{(1)}\geq 3l^{(2)}$ and $p_{l^{(1)}}^{(1)}=0$ as in Case 2.
Note that all $l^{(1)}$-strings in $\nu^{(1)}$ are singular.

\begin{itemize}
\item[(1)]
$l^{(1)}=3l^{(2)}$.

Note that $[3]$ is not marked in an $l^{(1)}$-string (singular) in $\nu^{(1)}$.
Otherwise, $[4]$ would be marked in the same string because it is singular.
Therefore, $[5]$ can be marked in an $l^{(1)}$-string in $\nu^{(1)}$.

\item[(2)]
$l^{(1)}=3l^{(2)}+1$.

Note that $[3]$ can be marked in an $l^{(1)}$-string in $\nu^{(1)}$.
In this case, the effective length is reduced by one ($l^{(2)}+1\rightarrow l^{(2)}$) and [4] is marked in an $l^{(2)}$-string in $\nu^{(2)}$ followed by the box marking of an $l^{(1)}$-string in $\nu^{(1)}$ by $[5]$.




\item[(3)]
$l^{(1)}\geq 3l^{(2)}+2$.

$[3]$ is not marked in the $l^{(1)}$-string.
Otherwise, $[4]$ could not be marked in $\nu^{(2)}$ because the effective length of $l^{(1)}$-strings is larger than $l^{(2)}$.
It is obvious that $[5]$ can be marked in the $l^{(1)}$-string. 
\end{itemize}
In three cases above, $[5]$ can be marked in the $l^{(1)}$-string contradicting $b= \framebox{$5$}$.
Next, we suppose that $\lambda_{1}=1$.
As in Case 2, we have two possible cases:

\begin{itemize}
\item[(1)] $m_{l^{(2)}}^{(2)}=1$, $l^{(1)}=3l^{(2)}-1$, and $p_{l^{(1)}}^{(1)}=0$.
\item[(2)] $m_{l^{(2)}}^{(2)}\geq 1$, $l^{(1)}\geq 3l^{(2)}$, and $p_{l^{(1)}}^{(1)}=1$.
\end{itemize}
Note that $l^{(1)}$-strings in $\nu^{(1)}$ are singular or q-singular.
Here we assume that $[3]$ is marked in an $l^{(1)}$-string as otherwise $[5]$ could be marked in this string contradicting $b= \framebox{$5$}$.
Note that case (1) must be excluded.
This is shown as follows.
If $[1]$ were marked in a singular row of length $l^{(2)}$, then $[4]$ could not be marked in $\nu^{(2)}$ because $m_{l^{(2)}}^{(2)}=1$.
If $[1]$ were marked in a singular row of length $(< l^{(2)})$, then $[4]$ would be marked on the left of the box marked by $[3]$ in the $l^{(1)}$-string.
Thus, case (b) survives.
If one of the $l^{(1)}$-strings is singular, the same arguments as those in the case of $\lambda_{1}=0$ hold and lead to a contradiction.
Therefore all the $l^{(1)}$-strings must be q-singular so that $[4]$ is not marked in an $l^{(1)}$-string.
This implies $[5]$ can be marked in an $l^{(1)}$-string contradicting $b= \framebox{$5$}$.
Hence we have $\lambda_{1}\geq 2$.

\begin{flushleft}
Case 6. $b= \framebox{$6$}$.
\end{flushleft}
In this case, after $[4]$ was marked in $\nu^{(1)}$ the box marking has terminated.
Since $\mathrm{wt}(\framebox{$6$})=-3\Bar{\Lambda}_{1}+2\Bar{\Lambda}_{2}$, we must show that $ \lambda_{2}\geq 2$.
We first suppose that $\lambda_{2}=0$.
As in Case 4 this results in a contradiction.
Secondly we suppose that $\lambda_{2}=1$.
Then $p_{l^{(2)}+1}^{(2)}=1$ by convexity and
\[
p_{l^{(2)}}^{(2)}=1+\sum_{j=1}^{l^{(1)}}\left( \min(3l^{(2)},j)-j \right) m_{j}^{(1)}.
\]
We have following two possible cases in a similar fashion as in Case 2:. 

\begin{itemize}
\item[(1)]
$m_{l^{(1)}}^{(1)}=1$, $l^{(1)}=3l^{(2)}+1$, and $p_{l^{(2)}}^{(2)}=0$.
\item[(2)]
$m_{l^{(1)}}^{(1)}\geq 1$, $l^{(1)}\leq 3l^{(2)}$, and $p_{l^{(2)}}^{(2)}=1$.
\end{itemize}
Let us show that we have a contradiction in both cases.

In case (1), we show that $m_{l^{(2)}}^{(2)}\geq 2$, which in turn leads to a contradiction.
Suppose that $m_{l^{(2)}}^{(2)}=1$.
From Eq.~\eqref{eq:ddp2}, we have
\[
-p_{l^{(2)}-1}^{(2)}-1=3m_{3l^{(2)}}^{(1)}+2(m_{3l^{(2)}+1}^{(1)}+m_{3l^{(2)}-1}^{(1)})+m_{3l^{(2)}-2}^{(1)}-2.
\]
However, this does not hold because $m_{l^{)1)}}^{(1)}=m_{3l^{(2)}+1}^{(1)}>0$.
Therefore, $m_{l^{(2)}}^{(2)}\geq 2$.
If the $l^{(1)}$-string is q-singular at best, then [4] must have been marked in a string of length ($\leq 3l^{(2)}$) so that [5] can be marked in one of the $l^{(2)}$-strings in $\nu^{(2)}$ even if [5] cannot be marked in a string of length ($\leq l^{(2)}-1$).
If the $l^{(1)}$-string is singular and [3] was marked in a string of length ($\leq 3l^{(2)}$), then [4] would not be the $l^{(1)}$-string in $\nu^{(1)}$ because the effective length of the $l^{(1)}$-string is larger than $l^{(2)}$ so that [4] must have been marked in a string of length ($\leq 3l^{(2)}$) in $\nu^{(1)}$ which leads to a contradiction as before.
Thus, it remains to consider the case when the $l^{(1)}$-string is singular and [3] has been marked in this string.
We have two possible ways of box marking in $\nu^{(1)}$.

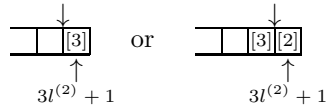
\setlength{\unitlength}{10pt}
\begin{center}
\begin{picture}(12,4)
\put(0,2){\line(1,0){3}}
\put(0,3){\line(1,0){3}}
\put(1,2){\line(0,1){1}}
\put(2,2){\line(0,1){1}}
\put(3,2){\line(0,1){1}}
\put(1,3){\makebox(2,1){$\downarrow$}}
\put(2,0.9){\makebox(1,1){$\uparrow$}}
\put(2,2){\makebox(1,1){{\scriptsize $[3]$}}}
\put(1,0){\makebox(3,1){{\scriptsize $3l^{(2)}+1$}}}

\put(4,2){\makebox(2,1){$\text{or}$}}

\put(7,2){\line(1,0){4}}
\put(7,3){\line(1,0){4}}
\put(8,2){\line(0,1){1}}
\put(9,2){\line(0,1){1}}
\put(10,2){\line(0,1){1}}
\put(11,2){\line(0,1){1}}
\put(9,3){\makebox(2,1){$\downarrow$}}
\put(10,0.9){\makebox(1,1){$\uparrow$}}
\put(9,2){\makebox(1,1){{\scriptsize $[3]$}}}
\put(10,2){\makebox(1,1){{\scriptsize $[2]$}}}
\put(9,0){\makebox(3,1){{\scriptsize $3l^{(2)}+1$}}}
\end{picture}
\end{center}
In both cases, the effective length is reduced by one ($l^{(2)}+1\rightarrow l^{(2)}$) and $[4]$ is marked in a singular string of length $l^{(2)}$ in $\nu^{(2)}$, which is also a contradiction.
Hence,  $\lambda_{2}\geq 2$.

In case (2), we also show that $m_{l^{(2)}}^{(2)}\geq 2$ leading to a contradiction.
Suppose that $m_{l^{(2)}}^{(2)}=1$.
From Eq.~\eqref{eq:ddp2}, we have
\[
-p_{l^{(2)}-1}^{(2)}+1=3m_{3l^{(2)}}^{(1)}+2m_{3l^{(2)}-1}^{(1)}+m_{3l^{(2)}-2}^{(1)}-2\geq 1,
\]
which forces that $p_{l^{(2)}-1}^{(2)}=0$, which implies that $m_{l^{(2)}-1}^{(2)}\neq 0$ by convexity.
Therefore [1] must be marked in a singular string of length ($\leq l^{(2)}-1$) in $\nu^{(2)}$ and hence [5] can be marked in a string in a string (singular or q-singular) of length $l^{(2)}$ in $\nu^{(2)}$ even if [5] cannot be marked in a string of length ($\leq l^{(2)}-1$), which contradicts $b= \framebox{$6$}$.
Hence $m_{l^{(2)}}^{(2)}\geq 2$.
The remaining argument is similar to case (1) and results in a contradiction.
Hence,  $\lambda_{2}\geq 2$.

\begin{flushleft}
Case 7. $b= \framebox{$7$}$.
\end{flushleft}
Since $\mathrm{wt}(\framebox{$7$})=0$, there is nothing to prove.

\begin{flushleft}
Case 8. $b= \framebox{$8$}$.
\end{flushleft}
Since $\mathrm{wt}(\framebox{$8$})=0$, there is nothing to prove.

\begin{flushleft}
Case 9. $b= \framebox{$9$}$.
\end{flushleft}
In this case, after $[6]$ was marked in $\nu^{(1)}$ the box marking has terminated.
Since $\mathrm{wt}(\framebox{$9$})=-2\Bar{\Lambda}_{1}+\Bar{\Lambda}_{2}$, we must show that $\lambda_{2}\geq 1$.
The proof is analogous to that in Case 4.

\begin{flushleft}
Case 10. $b= \framebox{$10$}$.
\end{flushleft}
In this case, after $[6]$ was marked in $\nu^{(2)}$ the box marking has terminated.
Since $\mathrm{wt}(\framebox{$10$})=3\Bar{\Lambda}_{1}-2\Bar{\Lambda}_{2}$, we must show that $\lambda_{1}\geq 3$.
We first suppose that $\lambda_{1}=0$.
Then we have that $l^{(1)}\geq 3l^{(2)}$ and $p_{l^{(1)}}^{(1)}=0$ as in Case 2.
Note that all $l^{(1)}$-strings in $\nu^{(1)}$ are singular. 
If $[6]$ is marked in a singular string of length $(< l^{(2)})$, then it is obvious that $[7]$ can be marked in an $l^{(1)}$-string even if [7] cannot be marked in a string of length ($\leq l^{(1)}-1$).
So we assume that $[6]$ is marked in an $l^{(2)}$-string and the selected $l^{(2)}$-string is singular.
If $[4]$ is marked in a singular string of length $(< l^{(1)})$, it is obvious that $[7]$ can be marked in an $l^{(1)}$-string.
If $[4]$ were marked in a singular string of length $(>3l^{(2)})$, then $[5]$ (and therefore $[6]$) would not be marked in $\nu^{(2)}$.
So it suffices to consider the case when $l^{(1)}=3l^{(2)}$ and $[4]$ is marked in an $l^{(1)}$-string. 
We have two possible ways of box marking in $\nu^{(1)}$.

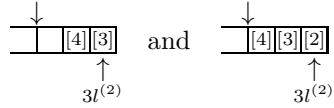
\setlength{\unitlength}{10pt}
\begin{center}
\begin{picture}(12,4)
\put(0,2){\line(1,0){4}}
\put(0,3){\line(1,0){4}}
\put(1,2){\line(0,1){1}}
\put(2,2){\line(0,1){1}}
\put(3,2){\line(0,1){1}}
\put(4,2){\line(0,1){1}}
\put(0,3){\makebox(2,1){$\downarrow$}}
\put(3,0.9){\makebox(1,1){$\uparrow$}}
\put(2,2){\makebox(1,1){{\scriptsize $[4]$}}}
\put(3,2){\makebox(1,1){{\scriptsize $[3]$}}}
\put(2,0){\makebox(3,1){{\scriptsize $3l^{(2)}$}}}

\put(5,2){\makebox(2,1){$\text{and}$}}

\put(8,2){\line(1,0){4}}
\put(8,3){\line(1,0){4}}
\put(9,2){\line(0,1){1}}
\put(10,2){\line(0,1){1}}
\put(11,2){\line(0,1){1}}
\put(12,2){\line(0,1){1}}
\put(8,3){\makebox(2,1){$\downarrow$}}
\put(11,0.9){\makebox(1,1){$\uparrow$}}
\put(9,2){\makebox(1,1){{\scriptsize $[4]$}}}
\put(10,2){\makebox(1,1){{\scriptsize $[3]$}}}
\put(11,2){\makebox(1,1){{\scriptsize $[2]$}}}
\put(10,0){\makebox(3,1){{\scriptsize $3l^{(2)}$}}}
\end{picture}
\end{center}
In both cases, $[7]$ can be marked in the string marked by $[4]$, which contradicts $b= \framebox{$10$}$.
Secondly, we suppose that $\lambda_{1}=1$.
As in Case 2, we have two possible cases.

\begin{itemize}
\item[(1)] $m_{l^{(2)}}^{(2)}=1$, $l^{(1)}=3l^{(2)}-1$, and $p_{l^{(1)}}^{(1)}=0$,
\item[(2)] $m_{l^{(2)}}^{(2)}\geq 1$, $l^{(1)}\geq 3l^{(2)}$, and $p_{l^{(1)}}^{(1)}=1$.
\end{itemize}
As before it suffices to consider the case when $[4]$ is marked in an $l^{(1)}$-string and the selected $l^{(1)}$- and $l^{(2)}$-strings are singular.
By the same arguments as in Case 2, both cases (1) and (2) result in a contradiction.
Thirdly, we suppose that $\lambda_{1}=2$.
As in Case 2, we have four possible cases.

\begin{itemize}
\item[(1)] $m_{l^{(2)}}^{(2)}=1$, $l^{(1)}=3l^{(2)}-2$, and $p_{l^{(1)}}^{(1)}=0$,
\item[(2)] $m_{l^{(2)}}^{(2)}=1$, $l^{(1)}=3l^{(2)}-1$, and $p_{l^{(1)}}^{(1)}=1$,
\item[(3)] $m_{l^{(2)}}^{(2)}=2$, $l^{(1)}=3l^{(2)}-1$, and $p_{l^{(1)}}^{(1)}=0$,
\item[(4)] $m_{l^{(2)}}^{(2)}\geq 1$, $l^{(1)}\geq 3l^{(2)}$, and $p_{l^{(1)}}^{(1)}=2$.
\end{itemize}
As before it suffices to consider the case when $[4]$ is marked in an $l^{(1)}$-string and the selected $l^{(1)}$- and $l^{(2)}$-strings are singular.
By the same arguments as in Case 2, these cases result in a contradiction.
Hence, $\lambda_{1}\geq 3$.

\begin{flushleft}
Case 11. $b= \framebox{$11$}$.
\end{flushleft}
In this case, after $[7]$ was marked in $\nu^{(1)}$ or $\nu^{(2)}$ the box marking has terminated.
Since $\mathrm{wt}(\framebox{$11$})=\Bar{\Lambda}_{1}-\Bar{\Lambda}_{2}$, we must show that $\lambda_{1}\geq 1$.
Suppose that $\lambda_{1}=0$.
Then we have that $l^{(1)}\geq 3l^{(2)}$ and $p_{l^{(1)}}^{(1)}=0$ as in Case 2.
Note that all $l^{(1)}$-strings in $\nu^{(1)}$ are singular.

\begin{itemize}
\item[(1)]
$[7]$ is marked in $\nu^{(2)}$.

If [7] is marked in a singular string of length $(< l^{(2)})$, then it is obvious that [8] can be marked in an $l^{(1)}$-string in $\nu^{(1)}$ ((6) must have been marked in a string of length $(<3l^{(2)})$).
So we assume that [7] is marked in an $l^{(2)}$-string and that the string is singular.
If [6] is marked in a singular string of length $(<3l^{(2)})$ (i.e., it is not an $l^{(1)}$-string), it is obvious that [8] can be marked in an $l^{(1)}$-string.
If [6] were marked in a singular string of length $(>3l^{(2)})$, then [7] would not be marked in $\nu^{(2)}$.
So it suffices to consider the case when $l^{(1)}=3l^{(2)}$ and [6] is marked in an $l^{(1)}$-string and [7] is marked in an $l^{(2)}$-string and both selected strings are singular.
We have three possible ways of box marking in $\nu^{(1)}$.

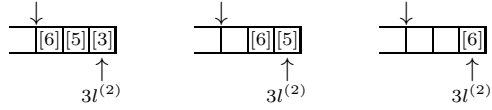
\setlength{\unitlength}{10pt}
\begin{center}
\begin{picture}(19,4)
\put(0,2){\line(1,0){4}}
\put(0,3){\line(1,0){4}}
\put(1,2){\line(0,1){1}}
\put(2,2){\line(0,1){1}}
\put(3,2){\line(0,1){1}}
\put(4,2){\line(0,1){1}}
\put(0,3){\makebox(2,1){$\downarrow$}}
\put(3,0.9){\makebox(1,1){$\uparrow$}}
\put(1,2){\makebox(1,1){{\scriptsize $[6]$}}}
\put(2,2){\makebox(1,1){{\scriptsize $[5]$}}}
\put(3,2){\makebox(1,1){{\scriptsize $[3]$}}}
\put(2,0){\makebox(3,1){{\scriptsize $3l^{(2)}$}}}

\put(7,2){\line(1,0){4}}
\put(7,3){\line(1,0){4}}
\put(8,2){\line(0,1){1}}
\put(9,2){\line(0,1){1}}
\put(10,2){\line(0,1){1}}
\put(11,2){\line(0,1){1}}
\put(7,3){\makebox(2,1){$\downarrow$}}
\put(10,0.9){\makebox(1,1){$\uparrow$}}
\put(9,2){\makebox(1,1){{\scriptsize $[6]$}}}
\put(10,2){\makebox(1,1){{\scriptsize $[5]$}}}
\put(9,0){\makebox(3,1){{\scriptsize $3l^{(2)}$}}}

\put(14,2){\line(1,0){4}}
\put(14,3){\line(1,0){4}}
\put(15,2){\line(0,1){1}}
\put(16,2){\line(0,1){1}}
\put(17,2){\line(0,1){1}}
\put(18,2){\line(0,1){1}}

\put(14,3){\makebox(2,1){$\downarrow$}}
\put(17,0.9){\makebox(1,1){$\uparrow$}}
\put(17,2){\makebox(1,1){{\scriptsize $[6]$}}}
\put(16,0){\makebox(3,1){{\scriptsize $3l^{(2)}$}}}
\end{picture}
\end{center}
It is obvious that  [8] can be marked on the left of the box marked by [6].
Note that the box marking depicted below
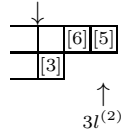
\setlength{\unitlength}{10pt}
\begin{center}
\begin{picture}(4,5)
\put(0,2){\line(1,0){2}}
\put(0,3){\line(1,0){4}}
\put(0,4){\line(1,0){4}}

\put(1,2){\line(0,1){2}}
\put(2,2){\line(0,1){2}}
\put(3,3){\line(0,1){1}}
\put(4,3){\line(0,1){1}}

\put(0,4){\makebox(2,1){$\downarrow$}}
\put(3,0.9){\makebox(1,1){$\uparrow$}}
\put(1,2){\makebox(1,1){{\scriptsize $[3]$}}}
\put(2,3){\makebox(1,1){{\scriptsize $[6]$}}}
\put(3,3){\makebox(1,1){{\scriptsize $[5]$}}}
\put(2,0){\makebox(3,1){{\scriptsize $3l^{(2)}$}}}
\end{picture}
\end{center}
cannot occur because the $l^{(1)}$-string is a boomerang string for the string marked by [3] (q-singular) in the above figure (see \textbf{(BS-4)}).

\item[(2)]
[7] is marked in $\nu^{(1)}$.

If [7] is marked in a string of length $(< l^{(1)})$, then it is obvious that [8] can be marked in a singular $l^{(1)}$-string.
So we assume that [7] is marked in an $l^{(1)}$-string.

\begin{itemize}
\item[(a)]
$l^{(1)}=3l^{(2)}$.

We have three possible ways of box marking in $\nu^{(1)}$.
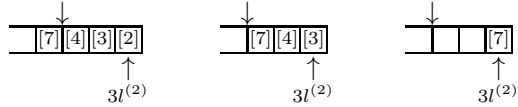
\setlength{\unitlength}{10pt}
\begin{center}
\begin{picture}(20,4)
\put(0,2){\line(1,0){5}}
\put(0,3){\line(1,0){5}}
\put(1,2){\line(0,1){1}}
\put(2,2){\line(0,1){1}}
\put(3,2){\line(0,1){1}}
\put(4,2){\line(0,1){1}}
\put(5,2){\line(0,1){1}}
\put(1,3){\makebox(2,1){$\downarrow$}}
\put(4,0.9){\makebox(1,1){$\uparrow$}}
\put(1,2){\makebox(1,1){{\scriptsize $[7]$}}}
\put(2,2){\makebox(1,1){{\scriptsize $[4]$}}}
\put(3,2){\makebox(1,1){{\scriptsize $[3]$}}}
\put(4,2){\makebox(1,1){{\scriptsize $[2]$}}}
\put(3,0){\makebox(3,1){{\scriptsize $3l^{(2)}$}}}

\put(8,2){\line(1,0){4}}
\put(8,3){\line(1,0){4}}
\put(9,2){\line(0,1){1}}
\put(10,2){\line(0,1){1}}
\put(11,2){\line(0,1){1}}
\put(12,2){\line(0,1){1}}
\put(8,3){\makebox(2,1){$\downarrow$}}
\put(11,0.9){\makebox(1,1){$\uparrow$}}
\put(9,2){\makebox(1,1){{\scriptsize $[7]$}}}
\put(10,2){\makebox(1,1){{\scriptsize $[4]$}}}
\put(11,2){\makebox(1,1){{\scriptsize $[3]$}}}
\put(10,0){\makebox(3,1){{\scriptsize $3l^{(2)}$}}}

\put(15,2){\line(1,0){4}}
\put(15,3){\line(1,0){4}}
\put(16,2){\line(0,1){1}}
\put(17,2){\line(0,1){1}}
\put(18,2){\line(0,1){1}}
\put(19,2){\line(0,1){1}}
\put(15,3){\makebox(2,1){$\downarrow$}}
\put(18,0.9){\makebox(1,1){$\uparrow$}}
\put(18,2){\makebox(1,1){{\scriptsize $[7]$}}}
\put(17,0){\makebox(3,1){{\scriptsize $3l^{(2)}$}}}
\end{picture}
\end{center}

The box marking of 
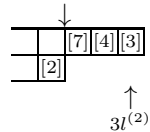
\setlength{\unitlength}{10pt}
\begin{center}
\begin{picture}(5,5)
\put(0,2){\line(1,0){2}}
\put(0,3){\line(1,0){5}}
\put(0,4){\line(1,0){5}}

\put(1,2){\line(0,1){2}}
\put(2,2){\line(0,1){2}}
\put(3,3){\line(0,1){1}}
\put(4,3){\line(0,1){1}}
\put(5,3){\line(0,1){1}}

\put(1,4){\makebox(2,1){$\downarrow$}}
\put(4,0.9){\makebox(1,1){$\uparrow$}}
\put(1,2){\makebox(1,1){{\scriptsize $[2]$}}}
\put(2,3){\makebox(1,1){{\scriptsize $[7]$}}}
\put(3,3){\makebox(1,1){{\scriptsize $[4]$}}}
\put(4,3){\makebox(1,1){{\scriptsize $[3]$}}}
\put(3,0){\makebox(3,1){{\scriptsize $3l^{(2)}$}}}
\end{picture}
\end{center}
cannot occur because $l^{(1)}$-string is a boomerang string for the string marked by [2] (qq-singular) in the above figure (see \textbf{(BS-3)}).


In all cases, [8] can be marked on the left of the box marked by [7].

\item[(b)]
$l^{(1)}=3l^{(2)}+1$.

The only possible way of marking [7] is
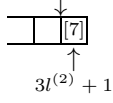
\setlength{\unitlength}{10pt}
\begin{center}
\begin{picture}(5,4)
\put(0,2){\line(1,0){3}}
\put(0,3){\line(1,0){3}}
\put(1,2){\line(0,1){1}}
\put(2,2){\line(0,1){1}}
\put(3,2){\line(0,1){1}}

\put(1,3){\makebox(2,1){$\downarrow$}}
\put(2,0.9){\makebox(1,1){$\uparrow$}}
\put(2,2){\makebox(1,1){{\scriptsize $[7]$}}}
\put(1,0){\makebox(3,1){{\scriptsize $3l^{(2)}+1$}}}
\end{picture}
\end{center}
where [4] must have been marked in a string of length $(\leq 3l^{(2)}-2)$ as otherwise [7] could be marked in the left of the box marked by [4].
It is obvious [8] can be marked on the left of the box marked by [7].We omit the details.

\item[(c)]
$l^{(1)}\geq 3l^{(2)}+2$.

It is obvious that [8] can be marked on the left of the box marked by [7].
\end{itemize}

\end{itemize}
Both cases (1) and (2) contradict $b= \framebox{$11$}$ and therefore $\lambda_{1}\geq 1$.

\begin{flushleft}
Case 12. $b= \framebox{$12$}$.
\end{flushleft}
Since $\mathrm{wt}(\framebox{$12$})=-\Bar{\Lambda}_{1}$, there is nothing to prove.

\begin{flushleft}
Case 13. $b= \framebox{$13$}$.
 \end{flushleft}
In this case, after [9] was marked in $\nu^{(1)}$ the box marking has terminated.
Since $\mathrm{wt}(\framebox{$13$})=-3\Bar{\Lambda}_{1}+\Bar{\Lambda}_{2}$, we must show that $\lambda_{2}\geq 1$.
The proof is analogous to that in Case 4.

\begin{flushleft}
Case 14. $b= \framebox{$14$}$.
\end{flushleft}
Since $\mathrm{wt}(\framebox{$14$})=-\Bar{\Lambda}_{2}$, there is nothing to prove.

\end{proof}

\begin{proof}[Proof of (II)]

\mbox{}

\begin{flushleft}
Case 1. $b= \framebox{$1$}$.
\end{flushleft} 
Since $\Delta{p}_{i}^{(2)}=-1$ $(i\geq 1)$ (see \textbf{(VC-1)}), we must show that $p_{i}^{(2)}\geq 1$ $(i\geq 1)$.
If $m_{i}^{(2)}=0$ $(i\geq 1)$, then $\nu^{(2)}=\emptyset$ so that we assume $m_{j}^{(2)}\neq 0$ for some $j$.
Since the $j$-string is q-singular at best, $p_{j}^{(2)}\geq 1$ so that $p_{i}^{(2)}\geq 1$ $(i\geq 1)$ by Lemma~\ref{lem:vacancy01}.
Hence the admissibility in the new RC is guaranteed.

\begin{flushleft}
Case 2. $b= \framebox{$2$}$.
\end{flushleft} 
Firstly we show that $p_{i}^{(2)}\geq 1$ $(i\leq i_{1}-1)$.
In the proof of Case 1, we assumed that  $m_{j}^{(2)}\neq 0$ for some $j$.
In the following we assume that $i_{1}\geq 2$ and $m_{i}^{(2)}= 0$ $(i\leq i_{1}-1)$.
Suppose that $p_{i_{1}-1}^{(2)}=0$.
By convexity we have $p_{i}^{(2)}=0$ $(i\leq i_{1})$.
By Eq.~\eqref{eq:ddp2}, we have
\[
0=3m_{3(i_{1}-1)}^{(1)}+2(m_{3(i_{1}-1)-1}^{(1)}+m_{3(i_{1}-1)+1}^{(1)})
+m_{3(i_{1}-1)-2}^{(1)}+m_{3(i_{1}-1)+2}^{(1)}
\]
so that $m_{3i_{1}-1}^{(1)}=\cdots =m_{3i_{1}-5}^{(1)}=0$.
Similarly we have $m_{i}^{(1)}=0$ $(i\leq 3i_{1}-1)$ so that 
\begin{align*}
p_{3(i_{1}-1)}^{(1)}=&-2\sum_{j\geq 3i_{1}}\min (3(i_{1}-1),j)m_{j}^{(1)}+\sum_{j\geq i_{1}}\min (3(i_{1}-1),3j)m_{j}^{(2)} \\
=&-6(i_{1}-1)\sum_{j\geq 3i_{1}}m_{j}^{(1)}+3(i_{1}-1)\sum_{j\geq i_{1}}m_{j}^{(2)}.
\end{align*}
Since $p_{3(i_{1}-1)}^{(1)}\geq 0$, we have 
$\sum_{j\geq 3i_{1}}m_{j}^{(1)}\leq \frac{1}{2}\sum_{j\geq i_{1}}m_{j}^{(2)}$
which yields 
\[
p_{i_{1}}^{(2)}-p_{i_{1}-1}^{(2)}=3\sum_{j\geq 3i_{1}}m_{j}^{(1)}-2\sum_{j\geq i_{1}}m_{j}^{(2)} 
\leq -\frac{1}{2}\sum_{j\geq i_{1}}m_{j}^{(2)}<0.
\]
The last inequality is due to the fact that $m_{i_{1}}^{(2)}>0$.
Therefore $p_{i_{1}}^{(2)}<p_{i_{1}-1}^{(2)}=0$, which is a contradiction.
Hence $p_{i_{1}-1}^{(2)}\geq 1$ and therefore $p_{i}^{(2)}\geq 1$ $(i\leq i_{1}-1)$ by convexity.
Note that $m_{i}^{(2)}= 0$ $(i\leq i_{1}-1)$.

Secondly we show that $p_{3i_{1}}^{(1)}\geq 3$,  $p_{3i_{1}-1}^{(1)}\geq 2$, and $p_{3i_{1}-2}^{(1)}\geq 1$.
We assume that $m_{3i_{1}}^{(1)}=m_{3i_{1}-1}^{(1)}=m_{3i_{1}-2}^{(1)}=0$.
Otherwise the proof is much easier.
If $m_{3i_{1}+1}^{(1)}\neq 0$, then $(3i_{1}+1)$-strings in $\nu^{(1)}$ must be qqq-singular at best so that $p_{3i_{1}+1}^{(1)}\geq 3$.
By convexity we have $p_{3i_{1}}^{(1)}\geq 3$,  $p_{3i_{1}-1}^{(1)}\geq 2$, and $p_{3i_{1}-2}^{(1)}\geq 1$.
Therefore we further assume that $m_{3i_{1}+1}^{(1)}= 0$.
Let us show that $p_{3i_{1}-2}^{(1)}\geq 1$.
From Eq.~\eqref{eq:ddp1b}, we have
\begin{equation} \label{eq:IIcase2}
-p_{3i_{1}-1}^{(1)}+2p_{3i_{1}}^{(1)}-p_{3i_{1}+1}^{(1)}=m_{i_{1}}^{(2)}>0.
\end{equation}
Suppose that $p_{3i_{1}-2}^{(1)}=0$, then $p_{3i_{1}-1}^{(1)}=p_{3i_{1}}^{(1)}=0$ by convexity.
This contradicts Eq.~\eqref{eq:IIcase2}.
Hence $p_{3i_{1}-2}^{(1)}\geq 1$.
This implies $p_{3i_{1}-1}^{(1)}\geq 1$, $p_{3i_{1}}^{(1)}\geq 1$, and $p_{3i_{1}+1}^{(1)}\geq 1$ by convexity.
However, $p_{3i_{1}}^{(1)}= 1$ contradicts Eq.~\eqref{eq:IIcase2} so that $p_{3i_{1}}^{(1)}\geq 2$ and hence $p_{3i_{1}-1}^{(1)}\geq 2$ by convexity.
Now we show that $p_{3i_{1}}^{(1)}=2$ is not the case.
Suppose that  $p_{3i_{1}}^{(1)}=2$.
Then  $(p_{3i_{1}-1}^{(1)},p_{3i_{1}+1}^{(1)})=(2,2)$ or $(3,1)$ by convexity, which contradicts Eq.~\eqref{eq:IIcase2}.
Hence $p_{3i_{1}}^{(1)}\geq 3$.
If $m_{j}^{(1)}>0$ for some $j(\geq 3i_{1}+1)$, then $p_{j}^{(1)}\geq 3$ so that $p_{i}^{(1)}\geq 3$ $(i\geq 3i_{1})$ by Lemma~\ref{lem:vacancy23}. 
If not, then $p_{i}^{(1)}\geq 3$ $(i\geq 3i_{1})$ by convexity.
We have shown that $p_{i}^{(1)}\geq 3$ $(i\geq 3i_{1})$, $p_{3i_{1}-1}^{(1)}\geq 2$, $p_{3i_{1}-2}^{(1)}\geq 1$, and  $p_{i}^{(2)}\geq 1$ $(i\leq i_{1}-1)$.
Therefore, the admissibility of the new RC is guaranteed (see \textbf{(VC-2)}).

\begin{flushleft}
Case 3. $b= \framebox{$3$}$.
\end{flushleft} 
Although it suffices to show that $p_{i}^{(1)}\geq 1$ $(i\geq 3i_{1})$ when $i_{2}=3i_{1}-2$ or $i_{2}=3i_{1}-1$ and $p_{i}^{(1)}\geq 1$ $(i\geq i_{2})$ when $i_{2}\geq 3i_{1}$ (see \textbf{(VC-3)}), we show that $p_{i}^{(1)}\geq 2$ $(i\geq \max (i_{2},3i_{1}))$.
We assume that $m_{i}^{(1)}=0$ $(i\geq i_{2}+1)$.
Otherwise, the proof is much easier.
From Eq.~\eqref{eq:ddp1b}, we have
\begin{equation} \label{eq:IIcase3}
-p_{3i_{1}-1}^{(1)}+2p_{3i_{1}}^{(1)}-p_{3i_{1}+1}^{(1)}=m_{i_{1}}^{(2)}>0. 
\end{equation}

\begin{itemize}
\item[(1)] $i_{2}=3i_{1}-2$.

From Eq.~\eqref{eq:IIcase3}, we have $p_{3i_{1}}^{(1)}\geq 1$.
The convexity relation yields $p_{3i_{1}-1}^{(1)}\geq 1$ and $p_{3i_{1}+1}^{(1)}\geq 1$.
If $p_{3i_{1}}^{(1)}= 1$, then Eq.~\eqref{eq:IIcase3} leads to a contradiction so that $p_{3i_{1}}^{(1)}\geq 2$.

\item[(2)] $i_{2}=3i_{1}-1$.

Since the selected $i_{2}$-string is q-singular, $p_{3i_{1}-1}^{(1)}\geq 1$ so that Eq.~\eqref{eq:IIcase3} yields $p_{3i_{1}}^{(1)}\geq 2$ as in case (1).

\item[(3)] $i_{2}\geq 3i_{1}$.

Since the selected $i_{2}$-string is qq-singular, $p_{i_{2}}^{(1)}\geq 2$.

\end{itemize}
Therefore,  $p_{i}^{(1)}\geq 2$ $(i\geq \max (i_{2},3i_{1}))$.
In particular,  $\Tilde{p}_{i}^{(1)}\geq 1$ $(i\geq \max (i_{2},3i_{1}))$.

\begin{flushleft}
Case 4. $b= \framebox{$4$}$.
\end{flushleft}
It is easily verified that $p_{i}^{(1)}+\Delta p_{i}^{(1)}\geq 0$ $(i\leq i_{3}-1)$.
Since $\Delta p_{i}^{(2)}=-1$ $(i\geq i_{3}^{eff})$ we must show that $p_{i}^{(2)}\geq 1$ $(i\geq i_{3}^{eff})$.
Since [4] or [5] are not marked in $\nu^{(2)}$, $p_{j}^{(2)}\geq 1$ for $j(\geq i_{1}+1)$ such that $m_{j}^{(2)}\neq 0$.
Therefore it suffices to show that $p_{i_{3}^{eff}}^{(2)}\geq 1$.

\begin{itemize}
\item[(1)] $i_{3}^{eff}=i_{1}$.

If $m_{i_{1}}^{(2)}\geq 2$, then $p_{i_{1}}^{(2)}\geq 1$.
Assume that $m_{i_{1}}^{(2)}=1$.
In Case 2, we have shown that $p_{i_{1}-1}^{(2)}\geq 1$.
From Eq.~\eqref{eq:ddp2}, we have 
\begin{align*}
&-p_{i_{1}-1}^{(2)}+2p_{i_{1}}^{(2)}-p_{i_{1}+1}^{(2)} \\
=&3m_{3i_{1}}^{(1)}+2(m_{3i_{1}-1}^{(1)}+m_{3i_{1}+1}^{(1)})
+m_{3i_{1}-2}^{(1)}+m_{3i_{1}+2}^{(1)}-2.
\end{align*}
Suppose that $p_{i_{1}}^{(2)}=0$.
Then the left-hand side is smaller than $0$.
On the other hand, the right-hand side is greater than or equal to $0$ because $m_{3i_{1}}^{(1)}\neq 0$ or $m_{3i_{1}-1}^{(1)}\neq 0$. 
This is a contradiction so that $p_{i_{1}}^{(2)}\geq 1$.

\item[(2)] $i_{3}^{eff}>i_{1}$.

The argument is similar to case (1) and we have $p_{i}^{(2)}\geq 1$ $(i\geq i_{3}^{eff})$.

\end{itemize}

When $i_{2}<i_{3}$, the box-deleted $(i_{3}-1)$-string in $\Tilde{\nu}^{(1)}$ is set to be q-singular by \textbf{(RA-2)}.
Therefore, we must check that 
\begin{equation} \label{eq:IIcase4}
p_{i_{3}-1}^{(1)}+\Delta p_{i_{3}-1}^{(1)}\geq 1.
\end{equation}

\begin{itemize}
\item[(1)] $i_{3}=3i_{1}$.

In this case, $i_{2}=3i_{1}-2$.
If $m_{3i_{1}-1}^{(1)}=0$, then $p_{3i_{1}-1}^{(2)}\geq 1$ by convexity because $p_{3i_{1}}^{(1)}\geq 1$.
If $m_{3i_{1}-1}^{(1)}\neq 0$, then $p_{3i_{1}-1}^{(2)}\geq 1$ by the box marking of this case.
By \textbf{(VC-3)}, $\Delta p_{3i_{1}-1}^{(1)}=0$.
Hence Eq.~\eqref{eq:IIcase4} is satisfied.

\item[(2)] $i_{3}\geq 3i_{1}+1$.

As shown in Case 3, $p_{i_{3}-1}^{(1)}\geq 2$ and $\Delta p_{i_{3}-1}^{(1)}=-1$ by \textbf{(VC-3)}.
Hence Eq.~\eqref{eq:IIcase4} is satisfied.

\end{itemize}

\begin{flushleft}
Case 5. $b= \framebox{$5$}$.
\end{flushleft} 
We must show that $p_{i}^{(1)}\geq 2$ $(i\geq 3i_{4})$ and $p_{3i_{4}-1}^{(1)}\geq 1$.
In order to show that $p_{i}^{(1)}\geq 2$ $(i\geq 3i_{4})$, it suffices to show $p_{3i_{4}}^{(1)}\geq 2$ by noting Lemma~\ref{lem:vacancy23}.
Furthermore, we must show that $p_{3i_{4}-2}^{(1)}\geq 2$ when the $i_{3}$-string is type-I with $i_{3}\geq 3i_{1}+2$ and $i_{3}^{eff}=i_{4}$; $i_{3}=3i_{4}-1$ (see \textbf{(VC-5)}).

Firstly, let us show that $p_{3i_{4}}^{(1)}\geq 2$ and $p_{3i_{4}-1}^{(1)}\geq 1$.
If $m_{3i_{4}}^{(1)}\neq 0$, then $p_{3i_{4}}^{(1)}\geq 2$ as otherwise [5] would be marked in the $3i_{4}$-string in $\nu^{(1)}$. 
Suppose that $m_{3i_{4}}^{(1)}= 0$.
Then, from Eq.~\eqref{eq:ddp1b} we have
\begin{equation} \label{eq:IIcase5}
-p_{3i_{4}-1}^{(1)}+2p_{3i_{4}}^{(1)}-p_{3i_{4}+1}^{(1)}=m_{i_{4}}^{(2)}\geq 1. 
\end{equation}
Hence $p_{3i_{4}}^{(1)}\geq 1$.
We show that $p_{3i_{4}}^{(1)}= 1$ is not the case.
Suppose that $p_{3i_{4}}^{(1)}= 1$.
Then, $(p_{3i_{4}-1}^{(1)},p_{3i_{4}+1}^{(1)})=(0,0), (0,1)$ or $(1,0)$ by Eq.~\eqref{eq:IIcase5} and by convexity.
If $p_{3i_{4}+1}^{(1)}=0$, then $m_{3i_{4}+1}^{(1)}=0$ since otherwise [5] would be marked in the $(3i_{4}+1)$-string if $m_{3i_{4}+1}^{(1)}\neq 0$.
However, the convexity relation
\[
p_{3i_{4}+1}^{(1)}\geq \frac{1}{2}(p_{3i_{4}}^{(1)}+p_{3i_{4}+2}^{(1)})=\frac{1}{2}(1+p_{3i_{4}+2}^{(1)})
\]
implies $p_{3i_{4}+1}^{(1)}\geq 1$.
This contradicts our assumption so that $(p_{3i_{4}-1}^{(1)},p_{3i_{4}+1}^{(1)})=(0,1)$.
Since [5] is not marked in a $(3i_{4}-1)$-string, $m_{3i_{4}-1}^{(1)}=0$.
From Eq.~\eqref{eq:ddp1a}, we have 
\[
-p_{3i_{4}-2}^{(1)}+2p_{3i_{4}-1}^{(1)}-p_{3i_{4}}^{(1)}=-p_{3i_{4}-2}^{(1)}-1=0.
\]
However this does not hold.
Hence $p_{3i_{4}}^{(1)}\geq 2$.
If $m_{3i_{4}-1}^{(1)}\neq 0$, then $p_{3i_{4}-1}^{(1)}\geq 1$ because [5] is not marked in the $(3i_{4}-1)$-string in $\nu^{(1)}$.
If not, then $p_{3i_{4}-1}^{(1)}\geq 1$ by convexity.

Secondly, we show that $p_{3i_{4}-2}^{(1)}\geq 2$ when the $i_{3}$-string is type-I with $i_{3}\geq 3i_{1}+2$ and $i_{3}^{eff}=i_{4}$; $i_{3}=3i_{4}-1$.

\begin{itemize}
\item[(1)] $i_{2}=i_{3}$.

It suffices to consider the case when $m_{3i_{1}-1}^{(1)}=\cdots =m_{3i_{4}-2}^{(1)}=0$ by noting Lemma~\ref{lem:vacancy23}.
Since $p_{3i_{4}-1}^{(1)}\geq 1$, we have $p_{3i_{1}-1}^{(1)}\geq 1, \ldots, p_{3i_{4}-2}^{(1)}\geq 1$ by convexity.
We show that $p_{3i_{4}-2}^{(1)}=1$ is not the case.
Suppose that $p_{3i_{4}-2}^{(1)}=1$.
Then $p_{3i_{1}-1}^{(1)}=\cdots =p_{3i_{4}-2}^{(1)}=1$ by Lemma~\ref{lem:vacancy1}.
Hence
\[
0=p_{3i_{4}-2}^{(1)}-p_{3i_{4}-3}^{(1)}=-2\sum_{j\geq i_{3}}m_{j}^{(1)}+\sum_{j\geq i_{4}}m_{j}^{(2)}
\] 
and
\[
0=p_{3i_{1}}^{(1)}-p_{3i_{1}-1}^{(1)}=-2\sum_{j\geq i_{3}}m_{j}^{(1)}+\sum_{j\geq i_{1}}m_{j}^{(2)}
\]  
so that $\sum_{j\geq i_{4}}m_{j}^{(2)}=\sum_{j\geq i_{1}}m_{j}^{(2)}$, which is a contradiction because $m_{i_{1}}^{(2)}\neq 0$.
Hence $p_{3i_{4}-2}^{(1)}\geq 2$.

\item[(2)] $i_{2}<i_{3}$.

We first suppose that he box marking in $\nu^{(1)}$ is

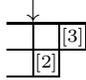
\setlength{\unitlength}{10pt}
\begin{center}
\begin{picture}(11,3)
\put(0,0){\line(1,0){2}}
\put(0,1){\line(1,0){3}}
\put(0,2){\line(1,0){3}}
\put(1,0){\line(0,1){2}}
\put(2,0){\line(0,1){2}}
\put(3,1){\line(0,1){1}}
\put(0,2){\makebox(2,1){$\downarrow$}}
\put(1,0){\makebox(1,1){{\scriptsize $[2]$}}}
\put(2,1){\makebox(1,1){{\scriptsize $[3]$}}}
\end{picture}
\end{center}
Then $p_{3i_{4}-2}^{(1)}=p_{i_{2}}^{(1)}\geq 2$, where the $i_{2}$ (resp. $i_{3}$)-string is qq-singular (resp. q-singular).
Next, suppose that $i_{2}\leq 3i_{4}$.
It suffices to consider the case when $m_{i}^{(1)}=0$ $(i_{2}+1\leq i \leq 3i_{4}-2)$ as in Case (1).
If $i_{2}\geq 3i_{1}$, then the selected $i_{2}$-string is qq-singular and $p_{i_{2}}^{(1)}\geq 2$ so that $p_{3i_{4}-2}^{(1)}\geq 2$.
Therefore we further assume that $i_{2}\leq 3i_{1}-1$.
Suppose that $p_{3i_{4}-2}^{(1)}=1$.
Then $p_{i_{2}+1}^{(1)}=\cdots =p_{3i_{4}-2}^{(1)}=1$ by Lemma~\ref{lem:vacancy1}.
If $i_{2}=3i_{1}-1$, then $p_{i_{2}}^{(1)}\geq 1$.
This is the equality because the convexity relation
\[
1=p_{i_{2}+1}^{(1)}\geq \frac{1}{2}(p_{i_{2}}^{(1)}+p_{i_{2}+2}^{(1)})=\frac{1}{2}(p_{i_{2}}^{(1)}+1)
\]
implies $p_{i_{2}}^{(1)}=1$.
Hence, $p_{3i_{1}-1}^{(1)}=\cdots =p_{3i_{4}-2}^{(1)}=1$.
The remaining argument is the same as in case (a) and we have  $p_{3i_{4}-2}^{(1)}\geq 2$.

\end{itemize}

\begin{flushleft}
Case 6. $b= \framebox{$6$}$.
\end{flushleft} 
It is easily verified that $p_{i}^{(1)}+\Delta p_{i}^{(1)}\geq 0$ $(i\leq i_{4}-1)$.
We must show that $p_{i_{4}^{eff}}^{(2)}\geq 2$; $p_{i}^{(2)}\geq 2$ $(i\geq i_{4}^{eff})$ follows from this.
Furthermore, we must show that $p_{i_{4}^{eff}-1}^{(2)}\geq 1$ when the $i_{4}$-string is type-II with $i_{2}=i_{3}=i_{4}$ or $i_{2}<i_{3}=i_{4}$ (see \textbf{(VC-6)}).

Firstly, we show that $p_{i_{4}^{eff}}^{(2)}\geq 2$.

\begin{itemize} 
\item[(1)] The $i_{4}$-string is type-0/I.

If $m_{i_{4}^{eff}}^{(2)}\neq 0$, then the $i_{4}^{eff}$-string in $\nu^{(2)}$ must be qq-singular at best so that $p_{i_{4}^{eff}}^{(2)}\geq 2$.
Now we assume that $m_{i_{4}^{eff}}^{(2)}= 0$.
From Eq.~\eqref{eq:ddp2}, we have 
\begin{align} \label{eq:IIcase6a}
&-p_{i_{4}^{eff}-1}^{(2)}+2p_{i_{4}^{eff}}^{(2)}-p_{i_{4}^{eff}+1}^{(2)} \\
=&3m_{3i_{4}^{eff}}^{(1)}+2(m_{3i_{4}^{eff}-1}^{(1)}+m_{3i_{4}^{eff}+1}^{(1)})
+m_{3i_{4}^{eff}-2}^{(1)}+m_{3i_{4}^{eff}+2}^{(1)}. \nonumber 
\end{align}
Obviously, $p_{i_{4}^{eff}}^{(2)}\geq 1$ because $m_{3i_{4}^{eff}}^{(1)}$ or  $m_{3i_{4}^{eff}-1}^{(1)}$ is not $0$.
We show that $p_{i_{4}^{eff}}^{(2)}= 1$ is not the case.
Suppose that $p_{i_{4}^{eff}}^{(2)}= 1$.
If $m_{i_{4}^{eff}+1}^{(2)}\neq 0$, then $p_{i_{4}^{eff}+1}^{(2)}\geq 2$.
If $m_{i_{4}^{eff}+1}^{(2)}= 0$, then $p_{i_{4}^{eff}+1}^{(2)}\geq 1$ by convexity.
In either case, the left-hand side of Eq.~\eqref{eq:IIcase6a} is smaller than or equal to $1$ while the right-hand side is greater than or equal to $2$, which is a contradiction.
Hence $p_{i_{4}^{eff}}^{(2)}\geq 2$.

\item[(2)] The $i_{4}$-string is type-II.

If $m_{i_{4}^{eff}-1}^{(2)}\neq 0$, then [4] would be marked in $\nu^{(2)}$ contradicting $b= \framebox{$6$}$ so that $m_{i_{4}^{eff}-1}^{(2)}=0$.
The remaining argument is similar to case (1).

\end{itemize}

Secondly, we show that $p_{i_{4}^{eff}-1}^{(2)}\geq 1$ when the $i_{4}$-string is type-II with $i_{2}=i_{3}=i_{4}$ or $i_{2}<i_{3}=i_{4}$. 
We consider two cases (1) $m_{i_{4}^{eff}-1}^{(2)}=0$ and (2) $m_{i_{4}^{eff}-1}^{(2)}\neq 0$ separately.

\begin{itemize}

\item[(1)] $m_{i_{4}^{eff}-1}^{(2)}=0$.

If $m_{i_{4}^{eff}}^{(2)}\neq 0$, then $p_{i_{4}^{eff}}^{(2)}\geq 2$ and therefore $p_{i_{4}^{eff}-1}^{(2)}\geq 1$ by convexity.
This is because [5] is not marked in the $i_{4}^{eff}$-string in $\nu^{(2)}$.
Now, we assume that  $m_{i_{4}^{eff}}^{(2)}=0$.
From Eq.~\eqref{eq:ddp2}, we have
\begin{align} \label{eq:IIcase6b}
&-p_{i_{4}^{eff}-1}^{(2)}+2p_{i_{4}^{eff}}^{(2)}-p_{i_{4}^{eff}+1}^{(2)} \\
=&3m_{3i_{4}^{eff}}^{(1)}+2(m_{3i_{4}^{eff}-1}^{(1)}+m_{3i_{4}^{eff}+1}^{(1)})+m_{3i_{4}^{eff}-2}^{(1)}+m_{3i_{4}^{eff}+2}^{(1)}. \nonumber
\end{align}
Suppose that $p_{i_{4}^{eff}-1}^{(2)}=0$.
Then  $p_{i_{4}^{eff}}^{(2)}=0$ by convexity, which contradicts Eq.~\eqref{eq:IIcase6b} because $m_{i_{4}}^{(1)}=m_{3i_{4}^{eff}-2}^{(1)}\neq 0$.
Hence $p_{i_{4}^{eff}-1}^{(2)}\geq 1$.

\item[(2)] $m_{i_{4}^{eff}-1}^{(2)}\neq 0$.

The $i_{4}$-string is

\setlength{\unitlength}{10pt}
\begin{center}
\begin{picture}(11,2)
\put(0,0){\line(1,0){4}}
\put(0,1){\line(1,0){4}}
\put(1,0){\line(0,1){1}}
\put(2,0){\line(0,1){1}}
\put(3,0){\line(0,1){1}}
\put(4,0){\line(0,1){1}}
\put(2,1){\makebox(2,1){$\downarrow$}}
\put(1,0){\makebox(1,1){{\scriptsize $[4]$}}}
\put(2,0){\makebox(1,1){{\scriptsize $[3]$}}}
\put(3,0){\makebox(1,1){{\scriptsize $[2]$}}}

\put(5,0){\makebox(2,1){$\text{or}$}}

\put(8,0){\line(1,0){3}}
\put(8,1){\line(1,0){3}}
\put(9,0){\line(0,1){1}}
\put(10,0){\line(0,1){1}}
\put(11,0){\line(0,1){1}}
\put(9,1){\makebox(2,1){$\downarrow$}}
\put(9,0){\makebox(1,1){{\scriptsize $[4]$}}}
\put(10,0){\makebox(1,1){{\scriptsize $[3]$}}}
\end{picture}
\end{center}

Since the marking [2] and [3] causes the effective length reduction $i_{4}^{eff}\rightarrow i_{4}^{eff}-1$, the $(i_{4}^{eff}-1)$-string in $\nu^{(2)}$ must be q-singular at best.
Otherwise [4] would be marked in this string.
Hence $p_{i_{4}^{eff}-1}^{(2)}\geq 1$.

\end{itemize}

When $i_{3}<i_{4}$, the box-deleted $(i_{4}-1)$-string in $\Tilde{\nu}^{(1)}$ is set to be qq-singular by \textbf{(RA-3)}, we must check that 
\[
p_{i_{4}-1}^{(1)}+\Delta p_{i_{4}-1}^{(1)}\geq 2.
\]
Here, $\Delta p_{i_{4}-1}^{(1)}=1$ (see \textbf{(VC-4)}) and it is obvious that $p_{i_{4}-1}^{(1)}\geq 1$ so that the above inequality is satisfied.
When $i_{2}<i_{3}=i_{4}$, the box-deleted $(i_{4}-2)$-string in $\Tilde{\nu}^{(1)}$ is set to be q-singular by \textbf{(RA-4)}, we must check that
\begin{equation} \label{eq:IIcase6c}
p_{i_{4}-2}^{(1)}+\Delta p_{i_{4}-2}^{(1)}\geq 1.
\end{equation}
If $i_{4}\geq 3i_{1}+2$, then $p_{i_{4}-2}^{(1)}\geq 2$ (see Case 3) and $\Delta p_{i_{4}-2}^{(1)}=-1$ (see \textbf{(VC-3)}) so that Eq.~\eqref{eq:IIcase6c} is satisfied.
If $i_{4}=3i_{1}+1$, then $\Delta p_{i_{4}-2}^{(1)}=\Delta p_{3i_{1}-1}^{(1)}=0$ (see \textbf{(VC-3)}).
Let us show that $p_{3i_{1}-1}^{(1)}\geq 1$.
Since this is obvious if $m_{3i_{1}}^{(1)}\neq 0$ or $m_{3i_{1}-1}^{(1)}\neq 0$, we assume that $m_{3i_{1}}^{(1)}=0$ and $m_{3i_{1}-1}^{(1)}=0$.
From Eq.~\eqref{eq:ddp1b}, we have
\[
-p_{3i_{1}-1}^{(1)}+2p_{3i_{1}}^{(1)}-p_{3i_{1}+1}^{(1)}=m_{i_{1}}^{(2)}>0
\]
so that $p_{3i_{1}}^{(1)}\geq 1$ and therefore $p_{3_{1}-1}^{(1)}\geq 1$ by convexity.
Hence, Eq.~\eqref{eq:IIcase6c} is satisfied.

\begin{flushleft}
Case 7. $b= \framebox{$7$}$.
\end{flushleft}
It is not hard to check that $p_{i}^{(1)}+\Delta p_{i}^{(1)}\geq 0$ $(i\leq i_{5}-1)$.

\begin{flushleft}
Case 8. $b= \framebox{$8$}$.
\end{flushleft}
It is is not hard to check that $p_{i}^{(1)}+\Delta p_{i}^{(1)}\geq 0$ $(i\leq 3i_{5}-1)$.

\begin{flushleft}
Case 9. $b= \framebox{$9$}$.
\end{flushleft}
It is not hard to check that $p_{1}^{(1)}+\Delta p_{i}^{(1)}\geq 0$ $(i\leq i_{6}-1)$.
Since $\Delta p_{i}^{(2)}=-1$ $(i\geq i_{6}^{eff})$ (see \textbf{(VC-9)}), we must show that $p_{i_{6}^{eff}}^{(2)}\geq 1$; $p_{i_{6}^{i}}^{(2)}\geq 1$ $(i\geq i_{6}^{eff})$ follows from this.
If $m_{i_{6}^{eff}}^{(2)}\neq 0$, then $p_{i_{6}^{eff}}^{(2)}\geq 1$.
So we assume that  $m_{i_{6}^{eff}}^{(2)}= 0$.
From Eq.~\eqref{eq:ddp2}, we have 
\begin{align*}
&-p_{i_{6}^{eff}-1}^{(2)}+2p_{i_{6}^{eff}}^{(2)}-p_{i_{6}^{eff}+1}^{(2)} \\
=&3m_{3i_{6}^{eff}}^{(1)}+2(m_{3i_{6}^{eff}-1}^{(1)}+m_{3i_{6}^{eff}+1}^{(1)})
+m_{3i_{6}^{eff}-2}^{(1)}+m_{3i_{6}^{eff}+2}^{(1)}.
\end{align*}
Suppose that $p_{i_{6}^{eff}}^{(2)}=0$.
Then $p_{i_{6}^{eff}-1}^{(2)}=p_{i_{6}^{eff}+1}^{(2)}=0$ by convexity. 
The left-hand side is $0$ while the right-hand side is positive because $m_{i_{6}}^{(1)}\neq 0$, which is a contradiction.
Hence, we have $p_{i_{6}^{eff}}^{(2)}\geq 1$.

When $i_{5}<i_{6}$, we must check that $p_{i_{6}-1}^{(1)}\geq 1$ because the box-deleted $(i_{6}-1)$-string in $\Tilde{\nu}^{(1)}$ is set to be q-singular by \textbf{(RA-1)} and $\Delta p_{i_{6}-1}^{(1)}=0$ (see \textbf{(VC-7)}).
However this is obvious because $i_{5}$-string is q-singular and [6] is not marked in a string of length $i$ $(i_{5}\leq i\leq i_{6}-1)$.

We omit the proof in the case when $b= \framebox{$i$}$ with $10\leq i\leq 14$.
The proof is similar to those in the previous cases.
When $b= \framebox{$\emptyset$}$, there is nothing to prove.  

\end{proof}

\begin{proof}[Proof of (IV)]

Let $\Tilde{\nu}=(\Tilde{m}_{i}^{(a)})_{(a,i)\in\mathcal{H}_{0}}$.

\begin{flushleft}
Case 0. $b= \framebox{$\emptyset$}$.
\end{flushleft}
In this case, $\Tilde{m}_{1}^{(2)}=m_{1}^{(2)}-2$ and $\Tilde{m}_{3}^{(1)}=m_{3}^{(1)}-1$.

\begin{align*}
\Delta c(\nu)=&3\left( (m_{3}^{(1)})^{2}-(\Tilde{m}_{3}^{(1)})^{2}\right)+2\sum_{j\neq 3}\min (3,j)(m_{3}^{(1)}-\Tilde{m}_{3}^{(1)})m_{j}^{(1)} \\
-&3(m_{3}^{(1)}m_{1}^{(2)}-\Tilde{m}_{3}^{(1)}\Tilde{m}_{1}^{(2)})-\sum_{j\neq 1}\min (3,3j)(m_{3}^{(1)}-\Tilde{m}_{3}^{(1)})m_{j}^{(2)} \\
&-\sum_{i\neq 3}\min (i,3)m_{i}^{(1)}(m_{1}^{(2)}-\Tilde{m}_{1}^{(2)}) \\
+&(m_{1}^{(2)})^{2}-(\Tilde{m}_{1}^{(2)})^{2}+4\sum_{j\neq 1} \min (1,j)(m_{1}^{(2)}-\Tilde{m}_{1}^{(2)})m_{j}^{(2)} \\
-&L\sum_{j\geq 1}m_{j}^{(2)}+(L-1)\left( \sum_{j\geq 1}m_{j}^{(2)}-2 \right) \\
=&\sum_{j\geq 1}m_{j}^{(2)}-\alpha_{1}^{(2)}-2L+1.
\end{align*}

The change of the sum of riggings is 
\[
\Delta |J|=J^{(1,3)}+2J^{(2,1)}=p_{3}^{(1)}+2p_{1}^{(2)},
\]
where
\[
p_{3}^{(1)}=-6\sum_{j\geq 3}m_{j}^{(1)}-4m_{2}^{(1)}-2m_{1}^{(1)}+3\sum_{j\geq 1}m_{j}^{(2)}
\]
and
\[
p_{1}^{(2)}=L+3\sum_{j\geq 3}m_{j}^{(1)}+2m_{2}^{(1)}+m_{1}^{(1)}-2\sum_{j\geq 1}m_{j}^{(2)}.
\]
Hence altogether $\Delta c(\nu,J)=-\alpha_{1}^{(2)}+1$.

\begin{flushleft}
Case 1. $b= \framebox{$1$}$.
\end{flushleft}
Let $(\Tilde{\nu},\Tilde{J})=\delta _{b}(\nu,J)$.
Then
\[
\Delta c(\nu)=c(\nu)-c(\Tilde{\nu})=-\sum_{i}m_{i}^{(2)}=-\alpha_{1}^{(2)}.
\]
The riggings are unchanged so that $\Delta c(\nu,J)=-\alpha_{1}^{(2)}$.

In what follows we assume that $i_{1}>1$.
The proof in the case when $i_{1}=1$ is similar.
We write $\Delta _{i\rightarrow j}c(\nu)=c(\Tilde{\nu}_{i})-c(\Tilde{\nu}_{j})$, $\Delta _{i\rightarrow j} |J|=|\Tilde{J}_{i}|-|\Tilde{J}_{j}|$, and $\Delta _{i\rightarrow j}c(\nu,J)=c(\Tilde{\nu}_{i},\Tilde{J}_{i})-c(\Tilde{\nu}_{j},\Tilde{J}_{j})$, where $\delta_{b_{i}}(\nu,J)=(\Tilde{\nu}_{i},\Tilde{J}_{i})$ and $\delta_{b_{j}}(\nu,J)=(\Tilde{\nu}_{j},\Tilde{J}_{j})$ with  
$b_{i}= \framebox{$i$}$ and $b_{j}= \framebox{$j$}$ $(i<j)$.
We also write $\Delta_{1}c(\nu,J)=c(\nu,J)-c(\Tilde{\nu},\Tilde{J})$ where  $(\Tilde{\nu},\Tilde{J})=\delta _{b}(\nu,J)$ with $b= \framebox{$1$}$.

\begin{flushleft}
Case 2. $b= \framebox{$2$}$.
\end{flushleft}
In this case, $\Tilde{m}_{i_{1}}^{(2)}=m_{i_{1}}^{(2)}-1$ and $\Tilde{m}_{i_{1}-1}^{(2)}=m_{i_{1}-1}^{(2)}+1$.
\[
\Delta_{1\rightarrow 2}c(\nu)=-3\sum_{j\geq 3i_{1}}m_{j}^{(1)}-2m_{3i_{1}-1}^{(1)}-m_{3i_{1}-2}^{(1)}+2\sum_{j\geq i_{1}}m_{j}^{(2)}-1
\] 
and
\[
\Delta_{1\rightarrow 2}|J|=p_{i_{1}}^{(2)}-(p_{i_{1}-1}^{(2)}+\Delta p_{i_{1}-1}^{(2)}),
\]
where $\Delta p_{i_{1}-1}^{(2)}=-1$ by \textbf{(VC-1)}.
The direct calculation yields 
\[
\Delta_{1\rightarrow 2}|J|=3\sum_{j\geq 3i_{1}}m_{j}^{(1)}+2m_{3i_{1}-1}^{(1)}+m_{3i_{1}-2}^{(1)}-2\sum_{j\geq i_{1}}m_{j}^{(2)}+1.
\] 
Hence  $\Delta_{1\rightarrow 2} c(\nu,J)=0$ so that $\Delta c(\nu,J)=\Delta_{1}c(\nu,J)+\Delta_{1\rightarrow 2}c(\nu,J)=-\alpha_{1}^{(2)}$.

As easily verified that $\Delta_{i\rightarrow j}c(\nu,J)$ does not contain terms involving $m_{i}^{(a)}$.
Let us write
\[
\Delta_{i\rightarrow j}|J|=\sum_{a}\sum_{k}(p_{i_{k}}^{(a)}+\Delta r_{i_{k}}^{(a)}-(p_{i_{k}-n_{k}}^{(a)}+\Delta p_{i_{k}-n_{k}}^{(a)}+\Delta r_{i_{k}-n_{k}}^{(a)})),
\]
where $\Delta r_{k}^{(a)}$ is the rigging adjustment for the selected $i_{k}$-string in $\nu^{(a)}$, i.e., $\Delta r_{i_{k}}^{(a)}=0,-1$, and $-2$ when the selected $i_{k}$-string in $\nu^{(a)}$ is singular, q-singular, and qq-singular, respectively so that $p_{i_{k}}^{(a)}+\Delta r_{i_{k}}^{(a)}$ is the value of the rigging of the selected $i_{k}$-string in $\nu^{(a)}$ and $n_{k}$ is the numbers of box marking in the $i_{k}$-string in $\nu^{(a)}$, $p_{i_{k}-n_{k}}^{(a)}+\Delta p_{i_{k}-n_{k}}^{(a)}=\Tilde{p}_{i_{k}-n_{k}}^{(a)}$, and $\Delta r_{i_{k}-n_{k}}^{(a)}$ is the rigging adjustment for the $\Tilde{p}_{i_{k}-n_{k}}^{(a)}$-string in $\Tilde{\nu}^{(a)}$ so that $p_{i_{k}-n_{k}}^{(a)}+\Delta p_{i_{k}-n_{k}}^{(a)}+\Delta r_{i_{k}-n_{k}}^{(a)}$ is the value of rigging of the box-deleted $\Tilde{p}_{i_{k}-n_{k}}^{(a)}$-string in $\Tilde{\nu}^{(a)}$.
Then, 
\[
\Delta_{i\rightarrow j}c(\nu,J)=t-\sum_{a}\sum_{k}(\Delta r_{i_{k}}^{(a)}-(\Delta p_{i_{k}-n_{k}}^{(a)}+\Delta r_{i_{k}-n_{k}}^{(a)})),
\]
where $t$ is the ``constant'' term in $\Delta_{i\rightarrow j}c(\nu)$, which does not contain $m_{l}^{(a)}$.

\begin{flushleft}
Case 3. $b= \framebox{$3$}$.
\end{flushleft}
In this case one caution is in order.
When $i_{2}^{eff}=i_{1}$, we must consider the changes of $m_{i_{2}}^{(1)}$, $m_{i_{2}-1}^{(1)}$, $m_{i_{1}}^{(2)}$, and  $m_{i_{1}-1}^{(2)}$ simultaneously in $c(\nu)$.

\begin{itemize}
\item[(1)] $i_{2}^{eff}=i_{1}$.

We compute $\Delta_{1\rightarrow 3}c(\nu,J)$.
\begin{align*}
\Delta _{1\rightarrow 3}c(\nu,J)=-1&+\Delta r_{i_{1}}^{(2)}-(\Delta p_{i_{1}-1}^{(2)}+\Delta r_{i_{1}-1}^{(2)}) \\
&+\Delta r_{i_{2}}^{(1)}-(\Delta p_{i_{2}-1}^{(1)}+\Delta r_{i_{2}-1}^{(1)}),
\end{align*}
where $\Delta p_{i_{1}-1}^{(2)}=-1$ by \textbf{(VC-1)} and $\Delta r_{i_{1}}^{(2)}=\Delta r_{i_{2}-1}^{(1)}=0$.
The values of the rigging adjustment $\Delta r_{i_{2}}^{(1)}$ and $\Delta p_{i_{2}-1}^{(1)}$ are listed below.
In this table the type is refereed to the type of the $i_{2}$-string.
\begin{center}
\begin{tabular}{c|c|c}
type & $\Delta r_{i_{2}}^{(1)}$ & $\Delta p_{i_{2}-1}^{(1)}$ \\ \hline
0 & $-2$ & $-2$ \\
I  & $-1$ & $-1$ \\
II & $0$ & $0$
\end{tabular}
\end{center}
The values of $\Delta r_{i_{2}}^{(1)}$ is due to the fact that the selected $i_{2}$-string is singular (resp. q-singular) when it is type-II (resp. type-I) and is qq-singular when it is type-0.
The values of $\Delta p_{i_{2}-1}^{(1)}$ come from \textbf{(VC-2)}.
Hence altogether $\Delta_{1\rightarrow 3}c(\nu,J)=0$ so that $\Delta c(\nu,J)=\Delta_{1}c(\nu,J)+\Delta_{1\rightarrow 3}c(\nu,J)=-\alpha_{1}^{(2)}$.

\item[(2)] $i_{1}<i_{2}^{eff}$.

Since $\Delta_{1\rightarrow 3}c(\nu,J)=\Delta_{1\rightarrow 2}c(\nu,J)+\Delta_{2\rightarrow 3}c(\nu,J)$, we compute $\Delta_{2\rightarrow 3}c(\nu,J)$.
\begin{equation*}
\Delta _{1\rightarrow 3}c(\nu,J)=-1+\Delta r_{i_{2}}^{(1)}-(\Delta p_{i_{2}-1}^{(1)}+\Delta r_{i_{2}-1}^{(1)}),
\end{equation*}
where $\Delta r_{i_{2}}^{(1)}=-2$ because the selected $i_{2}$-string in $\nu^{(1)}$ is qq-singular and $\Delta r_{i_{2}-1}^{(1)}=0$.
By \textbf{(VC-2)}, $\Delta p_{i_{2}-1}^{(1)}=-3$.
Hence altogether $\Delta_{2\rightarrow 3}c(\nu,J)=0$ so that $\Delta c(\nu,J)=\Delta_{1}c(\nu,J)+\Delta_{1\rightarrow 2}c(\nu,J)+\Delta_{2\rightarrow 3}c(\nu,J)=-\alpha_{1}^{(2)}$.

\end{itemize}

\begin{flushleft}
Case 4. $b= \framebox{$4$}$.
\end{flushleft}

\begin{itemize}

\item[(1)] $i_{2}=i_{3}$.

In this case, $\Tilde{m}_{i_{3}}^{(1)}=m_{i_{3}}^{(1)}-1$ and $\Tilde{m}_{i_{3}-2}^{(1)}=m_{i_{3}-2}^{(1)}+1$.
As before, we must consider the changes of  $m_{i_{3}}^{(1)}$,  $m_{i_{3}-2}^{(1)}$, $m_{i_{1}}^{(2)}$ and $m_{i_{1}-1}^{(2)}$ simultaneously in $c(\nu)$ when $i_{2}\leq 3i_{1}+1$.

\begin{itemize}

\item[(a)] $i_{2}\leq 3i_{1}+1$.

We compute $\Delta_{1\rightarrow 4}c(\nu,J)$.
\begin{align*}
\Delta _{1\rightarrow 4}c(\nu,J)=t&+\Delta r_{i_{1}}^{(2)}-(\Delta p_{i_{1}-1}^{(2)}+\Delta r_{i_{1}-1}^{(2)}) \\
&+\Delta r_{i_{3}}^{(1)}-(\Delta p_{i_{3}-2}^{(1)}+\Delta r_{i_{3}-2}^{(1)}),
\end{align*}
where $t=-2$ when the $i_{3}$-string is type-II $(i_{3}=3i_{1}+1)$ and $t=-1$ otherwise.
As before $\Delta r_{i_{1}}^{(2)}=\Delta r_{i_{1}-1}^{(2)}=0$, $\Delta p_{i_{1}-1}^{(2)}=-1$, and $\Delta r_{i_{3}-2}^{(1)}=0$ because the box-deleted $(i_{3}-2)$-string is set to be singular.
The values of $\Delta r_{i_{3}}^{(1)}$ and $\Delta p_{i_{3}-2}^{(1)}$ (see \textbf{(VC-2)}) are listed below.
In this table the type is refereed to the type of the $i_{3}$-string.
Note that the $i_{3}$-string is singular when it is type-I ($i_{3}=3i_{1}-1$) while it is q-singular when the type is not I.
\begin{center}
\begin{tabular}{c|c|c}
type & $\Delta_{i_{3}}^{rig}$ & $\Delta p_{i_{3}-2}^{(1)}$ \\ \hline
0 & $-1$ & $-1$ \\
I  & $0$ & $0$ \\
II & $-1$ & $-2$
\end{tabular}
\end{center}
Hence altogether $\Delta_{1\rightarrow 4}c(\nu,J)=0$ so that $\Delta c(\nu,J)=\Delta_{1}c(\nu,J)+\Delta_{1\rightarrow 4}c(\nu,J)=-\alpha_{1}^{(2)}$.

\item[(b)] $i_{2}\geq 3i_{1}+2$.

Since $\Delta_{1\rightarrow 4}c(\nu,J)=\Delta_{1\rightarrow 2}c(\nu,J)+\Delta_{2\rightarrow 4}c(\nu,J)$, we compute $\Delta_{2\rightarrow 4}c(\nu,J)$.
\begin{equation*}
\Delta _{1\rightarrow 3}c(\nu,J)=-2+\Delta r_{i_{3}}^{(1)}-(\Delta p_{i_{3}-2}^{(1)}+\Delta r_{i_{3}-2}^{(1)}),
\end{equation*}
where $\Delta p_{i_{3}-2}^{(1)}=-3$ by \textbf{(VC-2)}, $\Delta r_{i_{3}-2}^{(1)}=0$, and $\Delta r_{i_{3}}^{(1)}=-1$ because the selected $i_{3}$-string is q-singular.
Hence $\Delta_{2\rightarrow 4}c(\nu,J)=0$ so that $\Delta c(\nu,J)=\Delta_{1}c(\nu,J)+\Delta_{1\rightarrow 2}c(\nu,J)+\Delta_{2\rightarrow 4}c(\nu,J)=-\alpha_{1}^{(2)}$.
\end{itemize}

\item[(2)] $i_{2}<i_{3}$.

\begin{itemize}

\item[(a)] $i_{3}^{eff}=i_{1}$.

We compute $\Delta_{1\rightarrow 4}c(\nu,J)$.
In this case $i_{2}=3i_{1}-2$ and $i_{3}=3i_{1}$.
The $i_{2}$-string is singular and the $i_{3}$-string is q-singular (see \textbf{(BM-3)}); $\Delta r_{i_{2}}^{(1)}=0$ and $\Delta r_{i_{3}}^{(1)}=-1$.
\begin{align*}
\Delta _{1\rightarrow 4}c(\nu,J)=-1&+\Delta r_{i_{1}}^{(2)}-(\Delta p_{i_{1}-1}^{(2)}+\Delta r_{i_{1}-1}^{(2)}) \\
&+\Delta r_{i_{2}}^{(1)}-(\Delta p_{i_{2}-1}^{(1)}+\Delta r_{i_{2}-1}^{(1)}) \\
&+\Delta r_{i_{3}}^{(1)}-(\Delta p_{i_{3}-1}^{(1)}+\Delta r_{i_{3}-1}^{(1)}),
\end{align*}
where $\Delta r_{i_{1}}^{(2)}=\Delta r_{i_{1}-1}^{(2)}=0$, $\Delta p_{i_{1}-1}^{(2)}=-1$ as before.
By \textbf{(RA-2)} $\Delta r_{i_{2}-1}^{(1)}=-1$ and $\Delta r_{i_{3}-1}^{(1)}=-1$.
By \textbf{(VC-3)} $\Delta p_{i_{2}-1}^{(1)}=0$ and $\Delta p_{i_{3}-1}^{(1)}=0$.
Hence altogether $\Delta_{1\rightarrow 4}c(\nu,J)=0$ so that $\Delta c(\nu,J)=\Delta_{1}c(\nu,J)+\Delta_{1\rightarrow 4}c(\nu,J)=-\alpha_{1}^{(2)}$.

\item[(b)] $i_{3}^{eff}>i_{1}$.

We compute $\Delta_{3\rightarrow 4}c(\nu,J)$.
\begin{equation*}
\Delta _{3\rightarrow 4}c(\nu,J)=-1+\Delta r_{i_{3}}^{(1)}-(\Delta p_{i_{3}-1}^{(1)}+\Delta r_{i_{3}-1}^{(1)}),
\end{equation*}
where $\Delta p_{i_{3}-1}^{(1)}=-1$ by \textbf{(VC-3)} and $\Delta r_{i_{3}}^{(1)}=\Delta r_{i_{3}-1}^{(1)}=-1$.
Hence $\Delta_{3\rightarrow 4}c(\nu,J)=0$ so that $\Delta c(\nu,J)=\Delta_{1}c(\nu,J)+\Delta_{1\rightarrow 3}c(\nu,J)+\Delta_{3\rightarrow 4}c(\nu,J)=-\alpha_{1}^{(2)}$.

\end{itemize}

\end{itemize}

The verifications of  $\Delta c(\nu,J)=-\alpha_{1}^{(2)}$ for $b= \framebox{$i$}$ $(5\leq i\leq 14)$ are similar and we omit the details.

\end{proof}




Define the following subsets of $B_{0}\otimes B_{0}$.
\begin{align*}
S_{0}=&\left\{ \framebox{$1$}\otimes\framebox{$1$}, \framebox{$14$}\otimes\framebox{$14$} \right\} \bigsqcup \left\{\framebox{$i$}\otimes\framebox{$j$}  \relmiddle| i=1,2;2\leq j \leq 14 \right\} \\ 
&\bigsqcup \left\{\framebox{$i$}\otimes\framebox{$j$}  \relmiddle| i=3,4,6; 6\leq j \leq 14; j\neq 7 \right\} \\
&\bigsqcup \left\{\framebox{$i$}\otimes\framebox{$j$}  \relmiddle| i=5,8,10; 10\leq j \leq 14 \right\} \\
&\bigsqcup \left\{\framebox{$i$}\otimes\framebox{$j$}  \relmiddle| i=7,9,11,12,13; j=13,14 \right\},
\end{align*}
\begin{align*}
S_{1}=&\left\{ \framebox{$2$}\otimes\framebox{$1$} \right\}  
\bigsqcup \left\{\framebox{$i$}\otimes\framebox{$j$}  \relmiddle| i=3,4,6;1\leq j \leq 7; j\neq 6 \right\} \\ 
&\bigsqcup \left\{\framebox{$i$}\otimes\framebox{$j$}  \relmiddle| i=5,8,10; 2\leq j \leq 9 \right\} \\
&\bigsqcup \left\{\framebox{$i$}\otimes\framebox{$j$}  \relmiddle| i=7,9,11,12,13; 6\leq j \leq 12; j\neq 7 \right\} \\
&\bigsqcup \left\{\framebox{$14$}\otimes\framebox{$j$}  \relmiddle| 10 \leq j \leq 13 \right\},
\end{align*}
and
\begin{align*}
S_{2}=&\left\{\framebox{$i$}\otimes\framebox{$1$}  \relmiddle| i=5,8,10 \right\} \\
&\bigsqcup \left\{\framebox{$i$}\otimes\framebox{$j$}  \relmiddle| i=7,9,11,12,13; 1\leq j \leq 7; j\neq 6 \right\} \\
&\bigsqcup \left\{\framebox{$14$}\otimes\framebox{$j$}  \relmiddle| 1 \leq j \leq 9 \right\}.
\end{align*}
The subset $S_{0}$ (resp. $S_{1}$) is $B(2\Bar{\Lambda}_{2})$ (resp. $B(3\Bar{\Lambda}_{1})$) in Eq.~ \eqref{eq:decomposition} and the subset $S_{2}$ is the disjoint union of $B(2\Bar{\Lambda}_{1})$, $B(\Bar{\Lambda}_{2})$, and $B(0)$ in Eq.~\eqref{eq:decomposition} so that
\begin{equation} \label{eq:Hb}
H(b_{1}\otimes b_{2})=
\begin{cases}
0 & \text{if } b_{1}\otimes b_{2} \in S_{0}, \\
-1 & \text{if } b_{1}\otimes b_{2} \in S_{1}, \\
-2 & \text{otherwise}.
\end{cases}
\end{equation}

\begin{proof}[Proof of (V)]
The proof is reduced to showing the following lemma.
\end{proof}

\begin{lem}
$b_{1}\otimes b_{2}$ belongs to $S_{2}$ or takes the form of $\framebox{$\emptyset$}\otimes b$ if and only if $\alpha_{1}^{(2)}-\Tilde{\alpha}_{1}^{(2)}=2$, where $b=\framebox{$j$}$ with $1\leq j \leq 14$ or $j=\emptyset$.
$b_{1}\otimes b_{2}$ belongs to $S_{1}$ if and only if $\alpha_{1}^{(2)}-\Tilde{\alpha}_{1}^{(2)}=1$.
$b_{1}\otimes b_{2}$ belongs to $S_{0}$ or takes the form of $b\otimes \framebox{$\emptyset$}$ with $b=\framebox{$i$}$ with $1\leq i \leq 14$ if and only if $\alpha_{1}^{(2)}-\Tilde{\alpha}_{1}^{(2)}=0$.
\end{lem}
\begin{proof}
Assertions are the direct consequences of the following three lemmas.
\end{proof}

\begin{lem} \label{lem:da=2}
If $\alpha_{1}^{(2)}-\Tilde{\alpha}_{1}^{(2)}=2$. then $b_{1}\otimes b_{2}$ belongs to $S_{2}$ or it takes the form of $\framebox{$\emptyset$}\otimes b$ with $b=\framebox{$i$}$ with $1\leq i \leq 14$ or $i=\emptyset$. 
\end{lem}

\begin{lem} \label{lem:da=1}
If $\alpha_{1}^{(2)}-\Tilde{\alpha}_{1}^{(2)}=1$. then $b_{1}\otimes b_{2}$ belongs to $S_{1}$.
In particular, it does not take the form of  $b\otimes \framebox{$\emptyset$}$ or $\framebox{$\emptyset$}\otimes b$.
\end{lem}

\begin{lem} \label{lem:da=0}
If $\alpha_{1}^{(2)}-\Tilde{\alpha}_{1}^{(2)}=0$. then $b_{1}\otimes b_{2}$ belongs to $S_{0}$ or takes the form of $b\otimes \framebox{$\emptyset$}$ with $b=\framebox{$i$}$ with $1\leq i \leq 14$. 
\end{lem}

We give the proof of Lemma \ref{lem:da=2} only.
Proofs of  Lemma \ref{lem:da=1} and  Lemma \ref{lem:da=0} are similar.  

\begin{proof}[Proof of Lemma \ref{lem:da=2}]
Since $\alpha_{1}^{(2)}-\Tilde{\alpha}_{1}^{(2)}=2$, $i\geq 5$ but $i\neq 6$.
It is obvious that $b_{1}= \framebox{$\emptyset$}$ is also possible.

\begin{flushleft}
Case 1. $b_{1}= \framebox{$\emptyset$}$.
\end{flushleft}
From \textbf{(VC-$\emptyset$)}, it is obvious that $1\leq j\leq 14$ or $j=\emptyset$.

\begin{flushleft}
Case 2. $b_{1}= \framebox{$5$}$
\end{flushleft}
Since $b_{1}= \framebox{$5$}$, after [4] was marked in $\nu^{(2)}$ the box marking has terminated.
Since $\alpha_{1}^{(2)}-\Tilde{\alpha}_{1}^{(2)}=2$, two strings of length one are marked by [1] and [4] in $\nu^{(2)}$ and they are deleted in $\Tilde{\nu}^{(1)}$.
Since $\Delta p_{i}^{(2)}=1$ $(i\geq 1)$ (see \textbf{(VC-5)}), all strings of $\Tilde{\nu}^{(2)}$ are not singular so that $j=1$.
In particular $j\neq \emptyset$.

\begin{flushleft} 
Case 3. $b_{1}= \framebox{$7$}$.
\end{flushleft}
Since $b_{1}= \framebox{$7$}$, after [5] was marked in $\nu^{(1)}$ the box marking has terminated.
We claim that $1\leq j \leq 7$ but $j\neq 6$ and $j\neq \emptyset$.
To do so, we show that neither [4] nor [6] can be marked in $\Tilde{\nu}^{(1)}$.
From \textbf{(VC-7)}, the changes of vacancy numbers are summarized as 

\begin{center}
\begin{tabular}{cl|cl}
$\Delta p_{i}^{(1)}$ & & $\Delta p_{i}^{(2)}$ & \\ \hline
$0$ & $(i\geq i_{5})$ & $0$ & $(i\geq i_{5}^{eff})$ \\
& & $1$ & $(1\leq i\leq i_{5}^{eff}-1)$
\end{tabular}
\end{center}
Note that strings of length ($\geq i_{5}$) in $\nu^{(1)}$ and therefore in $\Tilde{\nu}^{(1)}$ are q-singular at best.
Otherwise [6] would be marked in $\nu^{(1)}$ contradicting $b_{1}= \framebox{$7$}$.
The selected $i_{1}$- and $i_{4}$-strings of length one are deleted in $\Tilde{\nu}^{(2)}$ so that it is obvious from the table above that [1] can be marked in a string of length ($\geq i_{5}^{eff}$) if possible.
If [1] cannot be marked in $\Tilde{\nu}^{(2)}$, then $b_{2}=\framebox{$1$}$.
If [1] is marked in $\Tilde{\nu}^{(2)}$, then the box marking in $\Tilde{\nu}^{(1)}$ is possible only for strings of length ($\geq 3i_{5}^{eff}-2$).
Therefore, in order to mark [4] or [6] there must be a singular string of length $(\geq 3i_{5}^{eff})$ in $\Tilde{\nu}^{(1)}$.
However, such a string does not exist because string of length $(\geq i_{5})$ in $\Tilde{\nu}^{(1)}$ are q-singular at best.
Hence the claim follows. 
 
\begin{flushleft}  
Case 4. $b_{1}= \framebox{$8$}$.
\end{flushleft}
Since $b_{1}= \framebox{$8$}$, after [5] is marked in $\nu^{(2)}$ the box marking has terminated.
Two strings of length one are marked by [1] and [5] in $\nu^{(2)}$.
Only one 1-string (the selected $i_{1}$-string) in $\nu^{(2)}$ is singular.
Strings of length $(\geq 2)$ in $\nu^{(2)}$ are not singular.
Otherwise [6] would be marked in $\nu^{(2)}$ contradicting $b_{1}= \framebox{$8$}$.
Since $\Delta p_{i}^{(2)}= 0$ $(i\geq 1)$ (see \textbf{(VC-8)}), all strings of $\Tilde{\nu}^{(2)}$ are still not singular so that $j=1$ and $j\neq \emptyset$.
 
\begin{flushleft}   
Case 5. $b_{1}= \framebox{$9$}$.
\end{flushleft}
Since $b_{1}= \framebox{$9$}$, after [6] was marked in $\nu^{(1)}$ the box marking has terminated.
We claim that $1\leq j \leq 7$ but $j\neq 6$ and $j\neq \emptyset$.
To do so, we show that neither [4] nor [6] can be marked in $\Tilde{\nu}^{(1)}$.
From \textbf{(VC-9)}, we have two cases.

\begin{itemize}

\item[(1)] $i_{5}=i_{6}$.


\begin{itemize}
\item[(a)] The selected $i_{6}$-string in $\nu^{(1)}$ is type-II.

\begin{center}
\begin{tabular}{cl|cl}
$\Delta p_{i}^{(1)}$ & & $\Delta p_{i}^{(2)}$ & \\ \hline
2 & $(i\geq i_{6})$ & $-1$ & $(i\geq i_{6}^{eff})$ \\
0 & $(i=i_{6}-1)$ & 0 & $(i=i_{6}^{eff}-1)$ \\
& & 1 & $(1\leq i\leq i_{6}^{eff}-2)$
\end{tabular}
\end{center}
Since the selected $i_{1}$- and $i_{4}$-strings are deleted in $\Tilde{\nu}^{(2)}$, it is obvious from the table above that [1] can be marked in a string of length ($\geq i_{6}^{eff}-1$) if possible.
If [1] cannot be marked in $\Tilde{\nu}^{(2)}$, then $b_{2}=\framebox{$1$}$.
If [1] is marked in $\Tilde{\nu}^{(2)}$, then the box marking in $\Tilde{\nu}^{(1)}$ is possible only for strings of length ($\geq 3(i_{6}^{eff}-1)-2=i_{6}-3$).
Therefore, in order to mark [4] or [6] there must a singular string of length ($\geq i_{6}-1$) in $\Tilde{\nu}^{(1)}$.
However such a string does not exist in $\Tilde{\nu}^{(1)}$.
This is shown as follows.
The $(i_{6}-1)$-strings (type-0) in $\nu^{(1)}$ are qq-singular at best.
Otherwise [5] would be marked in an $(i_{6}-1)$-string.
Since $\Delta p_{i_{6}-1}^{(1)}=0$ and $\Delta p_{i}^{(1)}=2$ $(i\geq i_{6})$, strings of length $(\geq i_{6}-1)$ in $\Tilde{\nu}^{(1)}$ are qq-singular at best.

\item[(b)] The selected $i_{6}$-string in $\nu^{(1)}$ is type-0/I.

\begin{center}
\begin{tabular}{cl|cl}
$\Delta p_{i}^{(1)}$ & & $\Delta p_{i}^{(2)}$ & \\ \hline
2 & $(i\geq i_{6})$ & $-1$ & $(i\geq i_{6}^{eff})$ \\
0 & $(i=i_{6}-1)$ & 1 & $(1\leq i\leq i_{6}^{eff}-1)$
\end{tabular}
\end{center}
From the table above, box marking in $\Tilde{\nu}^{(1)}$ is possible only for strings of length ($\geq 3i_{6}^{eff}-2$).
The remaining argument is the same as in (a). 

\end{itemize}

\item[(2)] $i_{5}<i_{6}$.

\begin{center}
\begin{tabular}{cl|cl}
$\Delta p_{i}^{(1)}$ & & $\Delta p_{i}^{(2)}$ & \\ \hline
2 & $(i\geq i_{6})$ & $-1$ & $(i\geq i_{6}^{eff})$ 
\end{tabular}
\end{center}
The changes of vacancy numbers $\Delta p_{i}^{(1)}$ $(i\leq i_{6}-1)$ and $\Delta p_{i}^{(2)}$ $(i\leq i_{6}^{eff}-1)$ are the same as those in Case 3.
Note that the box-deleted $(i_{6}-1)$-string in $\Tilde{\nu}^{(1)}$ is set to be q-singular by \textbf{(RA-1)}.
The argument is the same as that in Case 3.

\end{itemize}

\begin{flushleft}
Case 6. $b_{1}= \framebox{$10$}$.
\end{flushleft}   
Since $b_{1}= \framebox{$10$}$, after [6] was marked in $\nu^{(2)}$ the box marking has terminated.

\begin{itemize}
\item[(1)] $i_{5}=i_{6}=2$.

The effective length of the $i_{4}$-string in $\nu^{(1)}$ is one.
Otherwise, [7] would be marked on the left of [4] in $\nu^{(1)}$.
Therefore, only one 1-string (the $i_{1}$-string) in $\nu^{(2)}$ is singular.
Otherwise [5] would be marked in the singular 1-string in $\nu^{(2)}$ and the box marking would terminate.
Since $\Delta p_{1}^{(2)}=0$ (see \textbf{(VC-10)}), strings of length one in $\Tilde{\nu}^{(2)}$ are not singular.
Furthermore, $\Delta p_{i}^{(2)}=2$ $(i\geq 2)$ (see \textbf{(VC-10)}), strings of length ($\geq 2$) in $\Tilde{\nu}^{(2)}$ are also not singular.
Hence $j=1$.
In particular, $j\neq \emptyset$.

\item[(2)] $i_{5}=1$ and $i_{6}\geq 2$.

Since strings of length $i$ $(1\leq i \leq i_{6}-1)$ in $\nu^{(1)}$ except the $i_{1}$-string are not singular and $\Delta p_{i}^{(2)}=0$ $(1\leq i \leq i_{6}-1)$ (see \textbf{(VC-10)}), strings of length $i$ $(2\leq i \leq i_{6}-1)$ in $\Tilde{\nu}^{(1)}$ are also not singular.
Furthermore, $\Delta p_{i}^{(2)}=2$ $(i\geq i_{6})$ (see \textbf{(VC-10)}) so that all strings in $\Tilde{\nu}^{(2)}$ are not singular.
Hence $j=1$.
In particular, $j\neq \emptyset$.

\end{itemize}

\begin{flushleft}
Case 7. $b_{1}= \framebox{$11$}$.
\end{flushleft}
Since $b_{1}= \framebox{$11$}$, after [7] was marked in $\nu^{(1)}$ or $\nu^{(2)}$ the box marking has terminated.
We claim that $1\leq j \leq 7$ but $j\neq 6$ and $j\neq\emptyset$.
To do so, we show that neither [4] nor [6] can be marked in $\Tilde{\nu}^{(1)}$.

Suppose that [7] is marked in $\nu^{(1)}$.
Note that [7] is always marked in the string which is unmarked so far.
The changes of vacancy numbers $\Delta p_{i}^{(2)}$ in Case 6 are changed to $\Delta p_{i}^{(2)}=1$ only for $i\geq i_{7}^{eff}$ in this case.
Hence $j=1$ and $j\neq \emptyset$ as in Case 6.
So we assume that [6] is always marked in $\nu^{(1)}$ and [7] is always marked in $\nu^{(2)}$.
Furthermore if $i_{6}^{eff}<i_{7}$, then $\Delta p_{i}^{(1)}=-1$ $(i\geq 3i_{7})$, $\Delta p_{3i_{7}-1}^{(1)}=0$, and $\Delta p_{3i_{7}-2}^{(1)}=1$  from \textbf{(VC-11)}.
Here we note the fact that [8] cannot be marked in $\nu^{(1)}$.
This fact implies that strings of length $3i_{7}-2$ in $\nu^{(1)}$ are singular at best, strings of length $3i_{7}-1$ in $\nu^{(1)}$ are q-singular at best, and strings of length ($\geq 3i_{7}$) are qq-singular at best by the algorithm $\delta_{\theta}$.
Hence strings of length ($\geq 3i_{7}-2$) in $\Tilde{\nu}^{(1)}$ are q-singular at best.
The singularity of strings of length ($\leq 3i_{7}-3$) is the same as in Case 5.
Hence $1\leq j\leq 7$ but $j\neq 6$ and $j\neq \emptyset$.
So we assume that $i_{6}^{eff}=i_{7}$.
From \textbf{(VC-11)}, we have following two cases.

\begin{itemize}

\item[(1)] $i_{5}=i_{6}$.

\begin{center}
\begin{tabular}{cl|cl}
$\Delta p_{i}^{(1)}$ & & $\Delta p_{i}^{(2)}$ & \\ \hline
$-1$ & $(i\geq 3i_{7})$ & 1 & $(i\geq i_{7})$ \\
0 & $(i=i_{6}+1)$ & $0$ & $(i=i_{7}-1)$ \\
1 & $(i=i_{6})$ & 1 & $(1\leq i \leq i_{7}-2)$ \\
0 & $(i=i_{6}-1)$ & &
\end{tabular}
\end{center}
Note that the $i_{6}$-string must be type-II.
Otherwise [8] would be marked on the left of [6].
Since the selected $i_{1}$- and $i_{4}$-strings are deleted in $\Tilde{\nu}^{(2)}$, it is obvious that [1] can be marked in a string of length ($\geq i_{7}-1$).
Therefore if [1] is marked in $\Tilde{\nu}^{(2)}$, then the box marking is possible only for strings of length ($\geq 3(i_{7}-1)-2=i_{6}-3$).
Therefor, in order to mark [4] or [6] there must be a singular string of length ($\geq i_{6}-1$) in $\Tilde{\nu}^{(1)}$.
However, such a string does not exist in $\Tilde{\nu}^{(1)}$.
This is shown as follows.
The $(i_{6}-1)$-strings in $\nu^{(1)}$ are qq-singular at best.
Otherwise [5] would be marked in an $(i_{6}-1)$-string.
Furthermore, since [8] cannot be marked in $\nu^{(1)}$, strings of length $3i_{7}-1(=i_{6}+1)$ are q-singular at best and strings of length ($\geq 3i_{7}$) are qq-singular at best.
Hence, from the table above, strings of length ($\geq i_{6}-1$) are not singular so that the box marking of [4] or [6] in $\Tilde{\nu}^{(1)}$ is impossible.

\item[(2)] $i_{5}<i_{6}$.

\begin{center}
\begin{tabular}{cl|cl}
$\Delta p_{i}^{(1)}$ & & $\Delta p_{i}^{(2)}$ & \\ \hline
$-1$ & $(i\geq 3i_{7})$ & 1 & $(i\geq i_{7})$ \\
0 & $(i=i_{6}+1)$ & $0$ & $(i_{5}^{eff}\leq i\leq i_{7}-1)$ \\
1 & $(i=i_{6})$ & 1 & $(1\leq i \leq i_{5}^{eff}-1)$ \\
0 & $(i=i_{6}-1)$ & &
\end{tabular}
\end{center}
Note that the $i_{6}$-string must be type-II.
Otherwise [8] would be marked on the left of [6].
Since the selected $i_{1}$- and $i_{4}$-strings are deleted in $\Tilde{\nu}^{(2)}$, it is obvious that [1] can be marked in a string of length ($\geq i_{5}^{eff}$).
Therefore if [1] is marked in $\Tilde{\nu}^{(2)}$, then the box marking is possible only for strings of length ($\geq 3i_{5}^{eff}-2$) and in order to [4] or [6] there must be a singular string of length ($\geq 3i_{5}^{eff}$).
However, such a string does not exist in $\Tilde{\nu}^{(1)}$.
This is shown as follows.
Strings of length $i$ ($i_{5}\leq i\leq i_{6}-1$) in $\nu^{(1)}$ are q-singular at best and their singularity does not change in $\Tilde{\nu}^{(1)}$.
Furthermore, since [8] cannot be marked in $\nu^{(1)}$, strings of length $3i_{7}-1(=i_{6}+1)$ are q-singular at best and strings of length ($\geq 3i_{7}$) are qq-singular at best.
Hence from the table above, strings of length ($\geq i_{5}$) in $\Tilde{\nu}^{(1)}$ are not singular so that the box marking of [4] or [6] in $\Tilde{\nu}^{(1)}$ is impossible.

\end{itemize}

\begin{flushleft}
Case 8. $b_{1}= \framebox{$12$}$.
\end{flushleft}
Since $b_{1}= \framebox{$12$}$, after [8] was marked in $\nu^{(1)}$ the box marking has terminated.
We claim that $1\leq j \leq 7$ but $j\neq 6$ and $j\neq \emptyset$.
To do so, we show that neither [4] nor [6] can be marked in $\Tilde{\nu}^{(1)}$.

Firstly, we suppose that [8] is marked in the string which is unmarked so far in $\nu^{(1)}$, where [7] is marked in $\nu^{(1)}$ or $\nu^{(2)}$.
From \textbf{(VC-12)} (5), $\Delta p_{i}^{(1)}=1$ $(i\geq i_{8})$ and $\Delta p_{i}^{(2)}=0$ $(i\geq i_{8}^{eff})$.
The changes of vacancy numbers $\Delta p_{i}^{(1)}$ ($i\leq i_{8}-1$) and $\Delta p_{i}^{(2)}$ ($i\leq i_{8}^{eff}-1$) are the same as those in Case 7.
The box-deleted $(i_{8}-1)$-string in $\Tilde{\nu}^{(1)}$ is set to be q-singular by the rule of rigging adjustment \textbf{(RA-5)}.
Therefore, the box marking of [4] or [6] in $\Tilde{\nu}^{(1)}$ is impossible as in Case 7.

Secondly, we suppose that [8] is marked in a marked string.
From \textbf{(VC-12)}, we have following four cases.

\begin{itemize}
\item[(1)] The $i_{7}$-string is in $\nu^{(1)}$ with $i_{8}=i_{7}=i_{4}$ ($i_{5}=i_{6}=2$) or the $i_{7}$-string is in $\nu^{(2)}$ with $i_{4}=i_{7}(=2)$.
The box marking in $\nu^{(1)}$ is one of the following three.

\setlength{\unitlength}{10pt}
\begin{center}
\begin{picture}(23,3)
\put(0,1){\line(1,0){5}}
\put(0,2){\line(1,0){5}}
\put(0,1){\line(0,1){1}}
\put(1,1){\line(0,1){1}}
\put(2,1){\line(0,1){1}}
\put(3,1){\line(0,1){1}}
\put(4,1){\line(0,1){1}}
\put(5,1){\line(0,1){1}}
\put(2,2){\makebox(2,1){$\downarrow$}}
\put(0,1){\makebox(1,1){{\scriptsize $[8]$}}}
\put(1,1){\makebox(1,1){{\scriptsize $[7]$}}}
\put(2,1){\makebox(1,1){{\scriptsize $[4]$}}}
\put(3,1){\makebox(1,1){{\scriptsize $[3]$}}}
\put(4,1){\makebox(1,1){{\scriptsize $[2]$}}}

\put(6,1){\makebox(2,1){$\text{or}$}}

\put(9,0){\line(1,0){1}}
\put(9,1){\line(1,0){5}}
\put(9,2){\line(1,0){5}}
\put(9,0){\line(0,1){2}}
\put(10,0){\line(0,1){2}}
\put(11,1){\line(0,1){1}}
\put(12,1){\line(0,1){1}}
\put(13,1){\line(0,1){1}}
\put(14,1){\line(0,1){1}}
\put(11,2){\makebox(2,1){$\downarrow$}}
\put(9,0){\makebox(1,1){{\scriptsize $[2]$}}}
\put(10,1){\makebox(1,1){{\scriptsize $[8]$}}}
\put(11,1){\makebox(1,1){{\scriptsize $[7]$}}}
\put(12,1){\makebox(1,1){{\scriptsize $[4]$}}}
\put(13,1){\makebox(1,1){{\scriptsize $[3]$}}}

\put(15,1){\makebox(2,1){$\text{or}$}}

\put(18,0){\line(1,0){2}}
\put(18,1){\line(1,0){5}}
\put(18,2){\line(1,0){5}}
\put(18,0){\line(0,1){2}}
\put(19,0){\line(0,1){2}}
\put(20,0){\line(0,1){2}}
\put(21,1){\line(0,1){1}}
\put(22,1){\line(0,1){1}}
\put(23,1){\line(0,1){1}}
\put(20,2){\makebox(2,1){$\downarrow$}}
\put(18,0){\makebox(1,1){{\scriptsize $[3]$}}}
\put(19,0){\makebox(1,1){{\scriptsize $[2]$}}}
\put(20,1){\makebox(1,1){{\scriptsize $[8]$}}}
\put(21,1){\makebox(1,1){{\scriptsize $[6]$}}}
\put(22,1){\makebox(1,1){{\scriptsize $[5]$}}}

\end{picture}
\end{center}

\begin{center}
\begin{tabular}{cl|cl}
$\Delta p_{i}^{(1)}$ & & $\Delta p_{i}^{(2)}$ & \\ \hline
$1$ & $(i\geq 6)$ & $0$ & $(i\geq 1)$ \\
$2$ & $(i=5)$ & & \\
$1$ & $(i=4)$ & &  \\
$0$ & $(i=3)$ & &
\end{tabular}
\end{center}
Since 3-strings in $\nu^{(1)}$ are not singular, it is obvious that the box marking of [4] or [6] in $\Tilde{\nu}^{(1)}$ is impossible.

\item[(2)] $i_{8}=i_{7}\neq i_{4}$.


\begin{itemize}
\item[(a)] $i_{8}^{eff}=i_{6}$.

\begin{center}
\begin{tabular}{cl|cl}
$\Delta p_{i}^{(1)}$ & & $\Delta p_{i}^{(2)}$ & \\ \hline
$1$ & $(i\geq 3i_{6})$ & $0$ & $(i\geq i_{6})$ 
\end{tabular}
\end{center}

\item[(b)] $i_{8}^{eff}>i_{6}$.

\begin{center}
\begin{tabular}{cl|cl}
$\Delta p_{i}^{(1)}$ & & $\Delta p_{i}^{(2)}$ & \\ \hline
$1$ & $(i\geq i_{8})$ & $0$ & $(i\geq i_{8}^{eff})$ 
\end{tabular}
\end{center}

\end{itemize}
The changes of vacancy numbers $\Delta p_{i}^{(2)}$ $(i\leq i_{6}-1)$ (case (a)) and $\Delta p_{i}^{(2)}$ $(i\leq i_{8}^{eff}-1)$ (case (b)) are the same as those in Case 6 so that the box marking in $\Tilde{\nu}^{(1)}$ is possible only for strings of length $(\geq 3i_{6}-2)$ (case (a)) and $(\geq 3i_{8}^{eff}-2)$ (case (b)).
It is obvious that the box marking of [4] or [6] in $\Tilde{\nu}^{(1)}$ is impossible.

\item[(3)] $i_{8}=i_{6}=i_{5}\neq i_{3}$.

In this case, the $i_{8}$-string is type-I and $i_{8}^{eff}=i_{7}$; $i_{8}=3i_{7}-1$.

\begin{center}
\begin{tabular}{cl|cl}
$\Delta p_{i}^{(1)}$ & & $\Delta p_{i}^{(2)}$ & \\ \hline
$1$ & $(i\geq i_{8}+1)$ & $0$ & $(i\geq i_{7}-1)$ \\
$2$ & $(i=i_{8})$ & $1$ & $(1\leq i\leq i_{7}-2)$ \\
$1$ & $(i=i_{8}-1)$ & &  \\
$0$ & $(i=i_{8}-2)$ & &
\end{tabular}
\end{center}
The box marking in $\Tilde{\nu}^{(1)}$ is possible only for strings of length $(\geq 3(i_{7}-1)-2=i_{8}-4)$.
Therefore, in order to mark [4] or [6] there must be a singular string of length ($\geq i_{8}-2$) in $\Tilde{\nu}^{(1)}$.
However, such a string does not exist.
This is shown as follows.
The $(i_{8}-2)$-strings in $\nu^{(1)}$ are not singular as otherwise [4] would be marked in this string.
Therefore, strings of length ($\geq i_{8}-2$) in $\Tilde{\nu}^{(1)}$ are not singular.

\item[(4)] $i_{8}=i_{6}\neq i_{5}$.

In this case, the $i_{8}$-string is type-I and $i_{8}^{eff}=i_{7}$; $i_{8}=3i_{7}-1$.
Note that the box-deleted $(i_{8}-2)$-string is set to be q-singular by \textbf{(RA-1)}.

\begin{itemize}
\item[(a)] $i_{3}=i_{5}(=3)$.

\begin{center}
\begin{tabular}{cl|cl}
$\Delta p_{i}^{(1)}$ & & $\Delta p_{i}^{(2)}$ & \\ \hline
$1$ & $(i\geq i_{8}+1)$ & $0$ & $(i\geq i_{7}-1)$ \\
$2$ & $(i=i_{8})$ & $1$ & $(1\leq i\leq i_{7}-2)$ \\
$1$ & $(i=i_{8}-1)$ & &  
\end{tabular}
\end{center}

\item[(b)] $i_{3}<i_{5}$.

\begin{center}
\begin{tabular}{cl|cl}
$\Delta p_{i}^{(1)}$ & & $\Delta p_{i}^{(2)}$ & \\ \hline
$1$ & $(i\geq i_{8}+1)$ & $0$ & $(i\geq i_{5}^{eff})$ \\
$2$ & $(i=i_{8})$ & $1$ & $(1\leq i\leq i_{5}^{eff}-1)$ \\
$1$ & $(i=i_{8}-1)$ & &  
\end{tabular}
\end{center}

\end{itemize}
In both cases, it is obvious that the box marking of [4] or [6] in $\Tilde{\nu}^{(1)}$ is impossible.

\end{itemize}

\begin{flushleft}
Case 9. $b_{1}= \framebox{$13$}$.
\end{flushleft}
Since $b_{1}= \framebox{$13$}$, after [9] was marked in $\nu^{(1)}$ the box marking has terminated.
We claim that $1\leq j \leq 7$ but $j\neq 6$ and $j\neq \emptyset$.
The proof is very similar to Case 8 and we omit the details.

\begin{flushleft}
Case 10. $b_{1}= \framebox{$14$}$.
\end{flushleft}
We claim that  $1\leq j \leq 9$ and $j\neq \emptyset$.

\begin{itemize}
\item[(1)] $i_{1}<i_{10}$.


\begin{itemize}
\item[(a)] $i_{9}^{eff}< i_{10}$.

\begin{center}
\begin{tabular}{cl|cl}
$\Delta p_{i}^{(1)}$ & & $\Delta p_{i}^{(2)}$ & \\ \hline
$0$ & $(i\geq 3i_{10})$ & 1 & $(i\geq i_{10})$ \\
1 & $(i=3i_{10}-1)$ & & \\
2 & $(i=3i_{10}-2)$ & & \\
3 & $(i_{9}\leq i \leq 3i_{10}-3)$ & &
\end{tabular}
\end{center}
The changes of vacancy numbers $\Delta p_{i}^{(1)}$ $(i\leq 3(i_{10}-1))$ are the same as in Case 13.
Therefore, [4] or [6] can be marked in a singular string of length ($\geq 3i_{10}$) in $\Tilde{\nu}^{(1)}$.
If [4] is marked in such a string, then $j\geq 6$.
Although [5] can be marked in a q-singular string of length ($\geq i_{10}$) in $\Tilde{\nu}^{(2)}$, [6] cannot so that $j=8$ is possible but $j\ngeq 10$.
In particular, $j\neq \emptyset$.
If [6] is marked in such a string, then $j\geq 9$.
Since [7] cannot be marked in $\Tilde{\nu}^{(2)}$ because strings of length ($\geq i_{10}$) in $\Tilde{\nu}^{(2)}$ are q-singular at best so that $j\ngeq 11$.
Hence altogether we have $1\leq j\leq 9$ and $j\neq \emptyset$.

\item[(b)] $i_{9}^{eff}= i_{10}$.

Possible ways of box marking in $\nu^{(1)}$ are following three;

\setlength{\unitlength}{10pt}
\begin{center}
\begin{picture}(17,3)
\put(0,0){\line(1,0){4}}
\put(0,1){\line(1,0){4}}
\put(1,0){\line(0,1){1}}
\put(2,0){\line(0,1){1}}
\put(3,0){\line(0,1){1}}
\put(4,0){\line(0,1){1}}

\put(0,1){\makebox(2,1){$\downarrow$}}

\put(1,0){\makebox(1,1){{\scriptsize $[9]$}}}
\put(2,0){\makebox(1,1){{\scriptsize $[8]$}}}
\put(3,0){\makebox(1,1){{\scriptsize $[7]$}}}
\put(0,2){\makebox(1,1){{\small (i)}}}

\put(8,0){\line(1,0){3}}
\put(8,1){\line(1,0){3}}
\put(9,0){\line(0,1){1}}
\put(10,0){\line(0,1){1}}
\put(11,0){\line(0,1){1}}

\put(9,0){\makebox(1,1){{\scriptsize $[9]$}}}
\put(10,0){\makebox(1,1){{\scriptsize $[8]$}}}
\put(8,2){\makebox(1,1){{\small (ii)}}}

\put(15,0){\line(1,0){2}}
\put(15,1){\line(1,0){2}}
\put(16,0){\line(0,1){1}}
\put(17,0){\line(0,1){1}}

\put(16,0){\makebox(1,1){{\scriptsize $[9]$}}}
\put(15,2){\makebox(1,1){{\small (iii)}}}

\end{picture}
\end{center}
We omit the unmarked or marked strings except the selected $i_{9}$-string.
In either case, $\Delta p_{i}^{(2)}=1$ $(i\geq i_{10})$.
The changes of vacancy numbers $\Delta p_{i}^{(1)}$ are obtained by overwriting
\[
-3\chi (i\geq 3i_{10})-2\chi (i=3i_{10}-1)-\chi (i=3i_{10}-2)
\]
on $\Delta p_{i}^{(1)}$ in Case 13 where $\Delta p_{i}^{(1)}\leq 3$ and in particular $\Delta p_{i}^{(1)}=3$ $(i\geq i_{9})$.
Therefore, $\Delta p_{i}^{(1)}\leq 0$ for some strings of effective length $i_{10}$ and $\Delta p_{i}^{(1)}=0$ $(i\geq 3i_{10})$.
In case (i), the selected $i_{9}$-string is type-0 and the box-deleted $(i_{9}-3)$-string in $\Tilde{\nu}^{(1)}$ is set to be singular but [4] or [6] cannot be marked in this string as in Case 9 so that [4] or [6] can be marked in only a singular string of effective length ($\geq i_{10}$) in $\Tilde{\nu}^{(1)}$.
The remaining argument is the same as in case (a) so that $1\leq j\leq 9$ and $j\neq \emptyset$.
In cases (ii) and (iii), the box-deleted $(i_{9}-1)$- or $(i_{9}-2)$-string is set to be q/qq-singular by \textbf{(RA-6)} or \textbf{(RA-7)} so that [4] or [6] cannot marked in this string.
Similarly, we have $1\leq j\leq 9$ and $j\neq \emptyset$.

\end{itemize}

\item[(2)] $i_{1}=i_{10}$.

From \textbf{(VC-14)}, it is obvious that $j=1$.

\end{itemize}

\end{proof}

\section{Verifications of the Rules of Forbidden Box Marking}

In this section we verify the rules of forbidden box marking.
They are necessary to ensure Eq.~\eqref{eq:dc}.

In \textbf{(BM-2)}, we excluded the box marking in $\nu^{(1)}$ depicted below.

\setlength{\unitlength}{10pt}
\begin{center}
\begin{picture}(11,3)
\put(0,0){\line(1,0){3}}
\put(0,1){\line(1,0){4}}
\put(0,2){\line(1,0){4}}
\put(1,0){\line(0,1){2}}
\put(2,0){\line(0,1){2}}
\put(3,0){\line(0,1){2}}
\put(4,1){\line(0,1){1}}
\put(0,2){\makebox(2,1){$\downarrow$}}
\put(2,0){\makebox(1,1){{\scriptsize $[2]$}}}
\put(3,1){\makebox(1,1){{\scriptsize $[3]$}}}

\put(5,1){\makebox(2,1){$\text{and}$}}

\put(8,0){\line(1,0){2}}
\put(8,1){\line(1,0){3}}
\put(8,2){\line(1,0){3}}
\put(9,0){\line(0,1){2}}
\put(10,0){\line(0,1){2}}
\put(11,1){\line(0,1){1}}
\put(8,2){\makebox(2,1){$\downarrow$}}
\put(9,0){\makebox(1,1){{\scriptsize $[2]$}}}
\put(10,1){\makebox(1,1){{\scriptsize $[3]$}}}
\end{picture}
\end{center}
where $i_{2}^{eff}>i_{1}$ and the $i_{2}$ (resp. $i_{3}$)-string is qq-singular (resp. singular).
Suppose that [4] is not marked in $\nu^{(2)}$.
Then,
\begin{align*}
\Delta _{2\rightarrow 4}c(\nu,J)=-2&+\Delta r_{i_{2}}^{(1)}-(\Delta p_{i_{2}-1}^{(1)}+\Delta r_{i_{2}-1}^{(1)}) \\
&+\Delta r_{i_{3}}^{(1)}-(\Delta p_{i_{3}-1}^{(1)}+\Delta r_{i_{3}-1}^{(1)}),
\end{align*}
where $\Delta r_{i_{2}}^{(1)}=-2$, $\Delta r_{i_{3}}^{(1)}=0$, $\Delta r_{i_{2}-1}^{(1)}=0$ and $\Delta r_{i_{3}-1}^{(1)}=-1$ by \textbf{(RA-2)}.
The changes of vacancy numbers are $\Delta p_{i_{2}-1}^{(1)}=-3$ and $\Delta p_{i_{3}-1}^{(1)}=-1$ (see \textbf{(VC-2)} and \textbf{(VC-3)}).
Therefore $\Delta _{2\rightarrow 4}c(\nu,J)=-1$ so that $\Delta c(\nu,J)\neq \alpha_{1}^{(2)}$.

We excluded the box marking in $\nu^{(1)}$ depicted below.

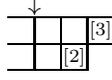
\setlength{\unitlength}{10pt}
\begin{center}
\begin{picture}(4,3)
\put(0,0){\line(1,0){3}}
\put(0,1){\line(1,0){4}}
\put(0,2){\line(1,0){4}}
\put(1,0){\line(0,1){2}}
\put(2,0){\line(0,1){2}}
\put(3,0){\line(0,1){2}}
\put(4,1){\line(0,1){1}}
\put(0,2){\makebox(2,1){$\downarrow$}}
\put(2,0){\makebox(1,1){{\scriptsize $[2]$}}}
\put(3,1){\makebox(1,1){{\scriptsize $[3]$}}}
\end{picture}
\end{center}
where $i_{2}^{eff}>i_{1}$ and the $i_{2}$ (resp. $i_{3}$)-string is qq-singular (resp. q-singular).
Suppose that [4] is marked in $\nu^{(2)}$ with $i_{3}^{eff}=i_{4}$.
Then,
\begin{align*}
\Delta _{2\rightarrow 5}c(\nu,J)=-1&+\Delta r_{i_{2}}^{(1)}-(\Delta p_{i_{2}-1}^{(1)}+\Delta r_{i_{2}-1}^{(1)}) \\
&+\Delta r_{i_{3}}^{(1)}-(\Delta p_{i_{3}-1}^{(1)}+\Delta r_{i_{3}-1}^{(1)}) \\
&+\Delta r_{i_{4}}^{(2)}-(\Delta p_{i_{4}-1}^{(2)}+\Delta r_{i_{4}-1}^{(2)}),
\end{align*}
where $\Delta r_{i_{2}}^{(1)}=-2$, $\Delta r_{i_{3}}^{(1)}=-1$, $\Delta r_{i_{4}}^{(2)}=0$, $\Delta r_{i_{2}-1}^{(1)}=0$, $\Delta r_{i_{3}-1}^{(1)}=0$ by \textbf{(RA-2)}, and $\Delta r_{i_{4}-1}^{(2)}=0$,
The changes of vacancy numbers $\Delta p_{i_{2}-1}^{(1)}$ and $\Delta p_{i_{3}-1}^{(1)}$ are obtained by overwriting 
\[
-3\chi (i\geq 3i_{5})-2\chi (i=3i_{5}-1)-\chi (i=3i_{5}-2)
\]
on the values of previous case.
That is,  $\Delta p_{i_{2}-1}^{(1)}=-4$ and $\Delta p_{i_{3}-1}^{(1)}=-3$.
By \textbf{(VC-3)} and \textbf{(VC-4)} $\Delta p_{i_{4}-1}^{(2)}=0$.
Therefore, $\Delta _{2\rightarrow 5}c(\nu,J)=3$ so that $\Delta c(\nu,J)\neq \alpha_{1}^{(2)}$.

We excluded the box marking (a) in $\nu^{(1)}$ depicted below and prescribe that the box marking must be (b).

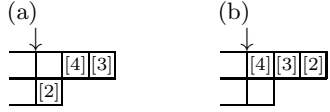
\setlength{\unitlength}{10pt}
\begin{center}
\begin{picture}(12,4)
\put(0,0){\line(1,0){2}}
\put(0,1){\line(1,0){4}}
\put(0,2){\line(1,0){4}}
\put(1,0){\line(0,1){2}}
\put(2,0){\line(0,1){2}}
\put(3,1){\line(0,1){1}}
\put(4,1){\line(0,1){1}}
\put(0,2){\makebox(2,1){$\downarrow$}}
\put(1,0){\makebox(1,1){{\scriptsize $[2]$}}}
\put(2,1){\makebox(1,1){{\scriptsize $[4]$}}}
\put(3,1){\makebox(1,1){{\scriptsize $[3]$}}}
\put(0,3){\makebox(1,1){{\small (a)}}}

\put(8,0){\line(1,0){2}}
\put(8,1){\line(1,0){4}}
\put(8,2){\line(1,0){4}}
\put(9,0){\line(0,1){2}}
\put(10,0){\line(0,1){2}}
\put(11,1){\line(0,1){1}}
\put(12,1){\line(0,1){1}}
\put(8,2){\makebox(2,1){$\downarrow$}}
\put(9,1){\makebox(1,1){{\scriptsize $[4]$}}}
\put(10,1){\makebox(1,1){{\scriptsize $[3]$}}}
\put(11,1){\makebox(1,1){{\scriptsize $[2]$}}}
\put(8,3){\makebox(1,1){{\small (b)}}}
\end{picture}
\end{center}
where $i_{2}^{eff}>i_{1}$ and the first (resp. second) string is singular (resp. qq-singular).
It is easily verified that (a) and (b) yields the same RC unless [5] and [6] are marked in $\nu^{(2)}$ with $i_{5}=i_{6}$ and $i_{4}^{eff}=i_{6}$.
The changes of vacancy numbers in (a) are $\Delta p_{i_{2}-1}^{(1)}=-3$ and $\Delta p_{i_{4}-2}^{(1)}=-1$ (see \textbf{(VC-2)} and \textbf{(VC-3)}).
Now suppose that [5] and [6] are marked in $\nu^{(2)}$ with $i_{5}=i_{6}$ and $i_{4}^{eff}=i_{6}$.
In case (a), the new vacancy changes $\Delta p_{i_{2}-1}^{(1)}$ and $\Delta p_{i_{4}-2}^{(1)}$ are obtained by overwriting

\begin{align*}
&-6\chi (i\geq 3i_{6})-5\chi (i=3i_{6}-1)-4\chi (i=3i_{6}-2) \\
&-3\chi (i=3i_{6}-3)-2\chi (i=3i_{6}-4)-\chi (i=3i_{6}-5)
\end{align*}
on the old ones.
That is, $\Delta p_{i_{2}-1}^{(1)}=-6$ and $\Delta p_{i_{4}-2}^{(1)}=-5$.
$\Delta p_{i_{6}-2}^{(2)}=0$, which is unchanged (see \textbf{(VC-3)}).
$\Delta_{2\rightarrow 10}c(\nu,J)$ is computed to be

\begin{align*}
-2&+\Delta r_{i_{2}}^{(1)}-(\Delta p_{i_{2}-1}^{(1)}+\Delta r_{i_{2}-1}^{(1)}) \\
&+\Delta r_{i_{4}}^{(1)}-(\Delta p_{i_{4}-2}^{(1)}+\Delta r_{i_{4}-2}^{(1)}) \\
&+\Delta r_{i_{6}}^{(2)}-(\Delta p_{i_{6}-2}^{(2)}+\Delta r_{i_{6}-2}^{(2)}),
\end{align*}
where $\Delta r_{i_{2}}^{(1)}=-2$ and other rigging adjustments are zero so that $\Delta_{2\rightarrow 10}c(\nu,J)=7$.
In this case, we must do the following box marking in $\nu^{(1)}$.

\setlength{\unitlength}{10pt}
\begin{center}
\begin{picture}(7,2)
\put(0,0){\line(1,0){7}}
\put(0,1){\line(1,0){7}}
\put(1,0){\line(0,1){1}}
\put(2,0){\line(0,1){1}}
\put(3,0){\line(0,1){1}}
\put(4,0){\line(0,1){1}}
\put(5,0){\line(0,1){1}}
\put(6,0){\line(0,1){1}}
\put(7,0){\line(0,1){1}}
\put(3,1){\makebox(2,1){$\downarrow$}}
\put(1,0){\makebox(1,1){{\scriptsize $[9]$}}}
\put(2,0){\makebox(1,1){{\scriptsize $[8]$}}}
\put(3,0){\makebox(1,1){{\scriptsize $[7]$}}}
\put(4,0){\makebox(1,1){{\scriptsize $[4]$}}}
\put(5,0){\makebox(1,1){{\scriptsize $[3]$}}}
\put(6,0){\makebox(1,1){{\scriptsize $[2]$}}}
\end{picture}
\end{center}
Suppose that $i_{6}\geq i_{1}+2$.
Then $\Delta p_{i_{9}-6}^{(1)}=-3$ and $\Delta p_{i_{6}-2}^{(2)}=1$, 
\[
\Delta_{2\rightarrow 13}c(\nu,J)=-2-\Delta p_{i_{9}-6}^{(1)}-\Delta p_{i_{6}-2}^{(2)}=0
\]
so that $\Delta c(\nu,J)=\Delta_{1}c(\nu,J)+\Delta_{1\rightarrow 2}c(\nu,J)+\Delta_{2\rightarrow 13}c(\nu,J)=-\alpha_{1}^{(2)}$.
The computation in the case when $i_{6}=i_{1}+1$ is similar.
In this case, we must consider the changes of $m_{i_{1}}^{(2)}$ and $m_{i_{1}-1}^{(2)}$ as well. 

We also excluded the box marking (a) in $\nu^{(1)}$ depicted below and prescribed that the box marking must be (b).

\setlength{\unitlength}{10pt}
\begin{center}
\begin{picture}(12,4)
\put(0,0){\line(1,0){3}}
\put(0,1){\line(1,0){4}}
\put(0,2){\line(1,0){4}}
\put(1,0){\line(0,1){2}}
\put(2,0){\line(0,1){2}}
\put(3,0){\line(0,1){2}}
\put(4,1){\line(0,1){1}}
\put(0,2){\makebox(2,1){$\downarrow$}}
\put(1,0){\makebox(1,1){{\scriptsize $[3]$}}}
\put(2,0){\makebox(1,1){{\scriptsize $[2]$}}}
\put(3,1){\makebox(1,1){{\scriptsize $[4]$}}}
\put(0,3){\makebox(1,1){{\small (a)}}}

\put(8,0){\line(1,0){3}}
\put(8,1){\line(1,0){4}}
\put(8,2){\line(1,0){4}}
\put(9,0){\line(0,1){2}}
\put(10,0){\line(0,1){2}}
\put(11,0){\line(0,1){2}}
\put(12,1){\line(0,1){1}}
\put(8,2){\makebox(2,1){$\downarrow$}}
\put(9,1){\makebox(1,1){{\scriptsize $[4]$}}}
\put(10,1){\makebox(1,1){{\scriptsize $[3]$}}}
\put(11,1){\makebox(1,1){{\scriptsize $[2]$}}}
\put(8,3){\makebox(1,1){{\small (b)}}}
\end{picture}
\end{center}
where $i_{2}^{eff}>i_{1}$ and the first (resp. second) string is singular (resp. q-singular).
The verification of this case is similar.

In \textbf{(BM-3)}, we excluded the following box marking (see Example~\ref{ex:BM-3}).

\setlength{\unitlength}{10pt}
\begin{center}
\begin{picture}(3,3)
\put(0,0){\line(1,0){2}}
\put(0,1){\line(1,0){3}}
\put(0,2){\line(1,0){3}}
\put(1,0){\line(0,1){2}}
\put(2,0){\line(0,1){2}}
\put(3,1){\line(0,1){1}}
\put(1,2){\makebox(2,1){$\downarrow$}}

\put(1,0){\makebox(1,1){{\scriptsize $[2]$}}}
\put(2,1){\makebox(1,1){{\scriptsize $[3]$}}}
\end{picture}
\end{center}
where the $i_{2}$-string is qq-singular and the $i_{3}$-string is singular.
Suppose that [4] is not marked in $\nu^{(2)}$.
Then
\[
\Delta_{3\rightarrow 4}c(\nu,J)=-1+\Delta r_{i_{3}}^{(1)}-(\Delta p_{i_{3}-1}^{(1)}+\Delta r_{i_{3}-1}^{(1)}),
\]
where $\Delta r_{i_{3}}^{(1)}=0$, $\Delta r_{i_{3}-1}^{(1)}=-1$ by \textbf{(RA-2)}, and $\Delta p_{i_{3}-1}^{(1)}=-1$ (see \textbf{(VC-3)}) so that $\Delta_{3\rightarrow 4}c(\nu,J)=1$, which yields $\Delta c(\nu,J)\neq \alpha_{1}^{(2)}$.

In \textbf{(BM-4)}, we excluded the following box marking.

\setlength{\unitlength}{10pt}
\begin{center}
\begin{picture}(4,3)
\put(0,0){\line(1,0){3}}
\put(0,1){\line(1,0){4}}
\put(0,2){\line(1,0){4}}
\put(1,0){\line(0,1){2}}
\put(2,0){\line(0,1){2}}
\put(3,0){\line(0,1){2}}
\put(4,1){\line(0,1){1}}
\put(0,2){\makebox(2,1){$\downarrow$}}
\put(2,0){\makebox(1,1){{\scriptsize $[3]$}}}
\put(3,1){\makebox(1,1){{\scriptsize $[4]$}}}
\end{picture}
\end{center}
where $i_{3}^{eff}>i_{1}$ and the first (resp. second) string is singular (resp. q-singular).
Suppose that [5] is marked in a q-singular string in $\nu^{(2)}$ with $i_{4}^{eff}=i_{5}$.
Then
\begin{align*}
\Delta _{3\rightarrow 8}c(\nu,J)=-1&+\Delta r_{i_{3}}^{(1)}-(\Delta p_{i_{3}-1}^{(1)}+\Delta r_{i_{3}-1}^{(1)}) \\
&+\Delta r_{i_{4}}^{(1)}-(\Delta p_{i_{4}-1}^{(1)}+\Delta r_{i_{4}-1}^{(1)}) \\
&+\Delta r_{i_{5}}^{(2)}-(\Delta p_{i_{5}-1}^{(2)}+\Delta r_{i_{5}-1}^{(2)}),
\end{align*}
where $\Delta r_{i_{3}}^{(1)}=-1$, $\Delta r_{i_{4}}^{(1)}=0$, $\Delta r_{i_{5}}=-1$, $\Delta r_{i_{3}-1}=-1$ by \textbf{(RA-3)}, $\Delta r_{i_{4}-1}^{(1)}=0$ by \textbf{(RA-3)}, and $\Delta r_{i_{5}-1}^{(2)}=0$.
If [5] is not marked in $\nu^{(2)}$, then $\Delta p_{i_{3}-1}^{(1)}=-1$ and $\Delta p_{i_{4}-1}^{(1)}=1$.
When [5] is marked in $\nu^{(2)}$  with $i_{4}^{eff}=i_{5}$, the new vacancy changes in $\nu^{(1)}$ are obtained by overwriting
\[
-3\chi (i\geq 3i_{5})-2\chi (i=3i_{5}-1)-\chi (i=3i_{5}-2)
\]
on old ones.
Therefore, $\Delta p_{i_{3}-1}^{(1)}=-2$ and $\Delta p_{i_{4}-1}^{(1)}=-1$.
$\Delta p_{i_{5}-1}^{(2)}=-1$, which is unchanged (see \textbf{(VC-3)} and \textbf{(VC-4)}).
Hence, $\Delta _{3\rightarrow 8}c(\nu,J)=2$, which yields $\Delta c(\nu,J)\neq \alpha_{1}^{(2)}$.

In \textbf{(BM-5)}, we excluded the following box marking.

\setlength{\unitlength}{10pt}
\begin{center}
\begin{picture}(14,3)
\put(0,0){\line(1,0){2}}
\put(0,1){\line(1,0){2}}
\put(1,0){\line(0,1){1}}
\put(2,0){\line(0,1){1}}
\put(0,1){\makebox(2,1){$\downarrow$}}
\put(1,0){\makebox(1,1){{\scriptsize $[5]$}}}
\put(0,2){\makebox(1,1){{\small (a)}}}

\put(5,0){\line(1,0){3}}
\put(5,1){\line(1,0){3}}
\put(6,0){\line(0,1){1}}
\put(7,0){\line(0,1){1}}
\put(8,0){\line(0,1){1}}
\put(5,1){\makebox(2,1){$\downarrow$}}
\put(7,0){\makebox(1,1){{\scriptsize $[5]$}}}
\put(5,2){\makebox(1,1){{\small (b)}}}

\put(11,0){\line(1,0){3}}
\put(11,1){\line(1,0){3}}
\put(12,0){\line(0,1){1}}
\put(13,0){\line(0,1){1}}
\put(14,0){\line(0,1){1}}
\put(11,1){\makebox(2,1){$\downarrow$}}
\put(12,0){\makebox(1,1){{\scriptsize $[5]$}}}
\put(13,0){\makebox(1,1){{\scriptsize $[3]$}}}
\put(11,2){\makebox(1,1){{\small (c)}}}

\end{picture}
\end{center}
where $i_{5}^{eff}=i_{4}$.
In cases (a) and (b),
\begin{align*}
\Delta _{5\rightarrow 7}c(\nu,J)=-1&+\Delta r_{i_{4}}^{(2)}-(\Delta p_{i_{4}-1}^{(2)}+\Delta r_{i_{4}-1}^{(2)}) \\
&+\Delta r_{i_{5}}^{(1)}-(\Delta p_{i_{5}-1}^{(1)}+\Delta r_{i_{5}-1}^{(1)}),
\end{align*}
where $\Delta r_{i_{4}}^{(2)}=0$, $\Delta r_{i_{5}}^{(1)}=-1$, $\Delta r_{i_{4}-1}^{(2)}=0$, and $\Delta r_{i_{5}-1}^{(1)}=0$.
In case (a) , the box marking by [4] and [5] does not affect the change of vacancy number of the $(i_{5}-1)$-string in $\nu^{(1)}$ and that of the $(i_{4}-1)$-string in $\nu^{(2)}$ so that $\Delta p_{i_{4}-1}^{(2)}=-1$ and $\Delta p_{i_{5}-1}^{(1)}=-1$ (see \textbf{(VC-4)}) and therefore $\Delta _{5\rightarrow 7}c(\nu,J)=-2$, which yields $\Delta c(\nu,J)\neq \alpha_{1}^{(2)}$.
In case (b), $\Delta p_{i_{4}-1}^{(2)}=-1$ (see \textbf{(VC-4)}) and $\Delta p_{i_{5}-1}^{(1)}=0$ (see \textbf{(VC-5)}) so that $\Delta _{5\rightarrow 7}c(\nu,J)=-1$, which yields $\Delta c(\nu,J)\neq \alpha_{1}^{(2)}$.
In case (c), 
\begin{align*}
\Delta _{3\rightarrow 7}c(\nu,J)=-1&+\Delta r_{i_{4}}^{(2)}-(\Delta p_{i_{4}-1}^{(2)}+\Delta r_{i_{4}-1}^{(2)}) \\
&+\Delta r_{i_{5}}^{(1)}-(\Delta p_{i_{5}-2}^{(1)}+\Delta r_{i_{5}-2}^{(1)}),
\end{align*}
where  $\Delta r_{i_{4}}^{(2)}=0$, $\Delta r_{i_{5}}^{(1)}=-1$, $\Delta r_{i_{4}-1}^{(2)}=0$, and $\Delta r_{i_{5}-2}^{(1)}=0$.
Since the box marking by [3], [4], and [5] and the deletion these boxes do not affect the change of the vacancy number of the $(i_{5}-2)$-string in $\Tilde{\nu}^{(1)}$ so that $\Delta p_{i_{5}-2}^{(1)}=-1$ (see \textbf{(VC-2)}).
$\Delta p_{i_{4}-1}^{(2)}$ is obtained by overwriting
\[
-2\chi (3i\geq i_{5})-\chi (3i=i_{5}-1)+2\chi (i\geq i_{4})
\]
on $\Delta p_{i_{4}-1}^{(2)}$ in \textbf{(VC-2)}.
Again it is unchanged and $\Delta p_{i_{4}-1}^{(2)}=0$.
Hence $\Delta _{3\rightarrow 7}c(\nu,J)=-1$, which yields $\Delta c(\nu,J)\neq \alpha_{1}^{(2)}$.

In \textbf{(BM-8)}, we excluded the following box marking in $\nu^{(1)}$.

\setlength{\unitlength}{10pt}
\begin{center}
\begin{picture}(6,2)
\put(0,0){\line(1,0){6}}
\put(0,1){\line(1,0){6}}
\put(1,0){\line(0,1){1}}
\put(2,0){\line(0,1){1}}
\put(3,0){\line(0,1){1}}
\put(4,0){\line(0,1){1}}
\put(5,0){\line(0,1){1}}
\put(6,0){\line(0,1){1}}
\put(3,1){\makebox(2,1){$\downarrow$}}
\put(1,0){\makebox(1,1){{\scriptsize $[9]$}}}
\put(2,0){\makebox(1,1){{\scriptsize $[8]$}}}
\put(3,0){\makebox(1,1){{\scriptsize $[7]$}}}
\put(4,0){\makebox(1,1){{\scriptsize $[4]$}}}
\put(5,0){\makebox(1,1){{\scriptsize $[3]$}}}
\end{picture}
\end{center}
where $i_{9}^{eff}=i_{6}$ ($i_{9}=3i_{6}-1$) (see Example~\ref{ex:BM-8a}).

\begin{align*}
\Delta _{3\rightarrow 13}c(\nu,J)=-1&+\Delta r_{i_{6}}^{(2)}-(\Delta p_{i_{6}-2}^{(2)}+\Delta r_{i_{6}-2}^{(2)}) \\
&+\Delta r_{i_{9}}^{(1)}-(\Delta p_{i_{9}-5}^{(1)}+\Delta r_{i_{9}-5}^{(1)}),
\end{align*}
where $\Delta r_{i_{6}}^{(2)}=\Delta r_{i_{6}-2}^{(2)}=\Delta r_{i_{9}}^{(1)}=\Delta r_{i_{9}-5}^{(1)}=0$.
The change of vacancy number $\Delta p_{i_{9}-5}^{(1)}$ is obtained by overwriting

\begin{align*}
&10\chi (i\geq i_{9})+8\chi (i=i_{9}-1)+\cdots +2\chi (i=i_{9}-4) \\
&-6\chi (i\geq 3i_{6})-5\chi (i=3i_{6}-1)-\cdots -\chi (i=3i_{6}-5)
\end{align*}
on $\Delta p_{i_{9}-5}^{(1)}$ in \textbf{(VC-3)} so that $\Delta p_{i_{9}-5}^{(1)}=-1$.
The change of vacancy number $\Delta p_{i_{6}-2}^{(2)}$ is obtained by overwriting

\begin{align*}
&-5\chi (3i\geq i_{4})-4\chi (3i=i_{9}-1)-\cdots -2\chi (3i=i_{9}-4) \\
&+4\chi (i\geq i_{6})+2\chi (i=i_{6}-1)
\end{align*}
on  $\Delta p_{i_{6}-2}^{(2)}$ in \textbf{(VC-3)} so that $\Delta p_{i_{6}-2}^{(2)}=0$.
Hence $\Delta _{3\rightarrow 13}c(\nu,J)=-1$, which yields $\Delta c(\nu,J)\neq \alpha_{1}^{(2)}$.

The verification of the following forbidden box marking in $\nu^{(1)}$

\setlength{\unitlength}{10pt}
\begin{center}
\begin{picture}(13,2)
\put(0,0){\line(1,0){5}}
\put(0,1){\line(1,0){5}}
\put(1,0){\line(0,1){1}}
\put(2,0){\line(0,1){1}}
\put(3,0){\line(0,1){1}}
\put(4,0){\line(0,1){1}}
\put(5,0){\line(0,1){1}}
\put(2,1){\makebox(2,1){$\downarrow$}}
\put(1,0){\makebox(1,1){{\scriptsize $[9]$}}}
\put(2,0){\makebox(1,1){{\scriptsize $[8]$}}}
\put(3,0){\makebox(1,1){{\scriptsize $[6]$}}}
\put(4,0){\makebox(1,1){{\scriptsize $[5]$}}}

\put(6,0){\makebox(2,1){$\text{or}$}}

\put(9,0){\line(1,0){4}}
\put(9,1){\line(1,0){4}}
\put(10,0){\line(0,1){1}}
\put(11,0){\line(0,1){1}}
\put(12,0){\line(0,1){1}}
\put(13,0){\line(0,1){1}}
\put(10,1){\makebox(2,1){$\downarrow$}}
\put(10,0){\makebox(1,1){{\scriptsize $[9]$}}}
\put(11,0){\makebox(1,1){{\scriptsize $[8]$}}}
\put(12,0){\makebox(1,1){{\scriptsize $[6]$}}}
\end{picture}
\end{center}
with $i_{9}^{eff}=i_{7}$ (see Example~\ref{ex:BM-8b} or \ref{ex:BM-8c}) is similar.

In \textbf{(BM-9)}, we excluded the following box marking in $\nu^{(1)}$.

\setlength{\unitlength}{10pt}
\begin{center}
\begin{picture}(4,3)
\put(0,0){\line(1,0){3}}
\put(0,1){\line(1,0){4}}
\put(0,2){\line(1,0){4}}
\put(1,0){\line(0,1){2}}
\put(2,0){\line(0,1){2}}
\put(3,0){\line(0,1){2}}
\put(4,1){\line(0,1){1}}
\put(0,2){\makebox(2,1){$\downarrow$}}
\put(2,0){\makebox(1,1){{\scriptsize $[8]$}}}
\put(3,1){\makebox(1,1){{\scriptsize $[9]$}}}
\end{picture}
\end{center}
where the first (resp. second) string is singular (resp. q-singular).
If [7] is marked in $\nu^{(1)}$, then $i_{7}\leq i_{8}-2$ and if [7] is marked in $\nu^{(2)}$, then $i_{8}^{eff}>i_{7}$ (see the preferential rule of \textbf{(BM-8)} (2)).
Suppose that [10] is marked in a singular string in $\nu^{(2)}$ with $i_{9}^{eff}=i_{10}$.
Then,
\begin{align*}
\Delta _{11\rightarrow 14}c(\nu,J)=-1&+\Delta r_{i_{8}}^{(1)}-(\Delta p_{i_{8}-1}^{(1)}+\Delta r_{i_{8}-1}^{(1)}) \\
&+\Delta r_{i_{9}}^{(1)}-(\Delta p_{i_{9}-1}^{(1)}+\Delta r_{i_{9}-1}^{(1)}) \\
&+\Delta r_{i_{10}}^{(2)}-(\Delta p_{i_{10}-1}^{(2)}+\Delta r_{i_{10}-1}^{(2)}),
\end{align*}
where $\Delta r_{i_{8}}^{(1)}=-1$, $\Delta r_{i_{9}}^{(1)}=0$, $\Delta r_{i_{10}}^{(2)}=0$, $\Delta r_{i_{8}-1}^{(1)}=-1$ by \textbf{(RA-7)}, $\Delta r_{i_{9}-1}^{(1)}=0$ by \textbf{(RA-7)}, and $\Delta r_{i_{10}-1}^{(2)}=0$.
The changes of vacancy numbers $\Delta p_{i}^{(1)}$ are obtained by overwriting
\[
-3\chi (i\geq 3i_{10})-2\chi (i=3i_{10}-1)-\chi (i=3i_{10}-2)
\]
on  $\Delta p_{i}^{(1)}$ in \textbf{(VC-13)} so that  $\Delta p_{i_{8}-1}^{(1)}=-2$ and  $\Delta p_{i_{9}-1}^{(1)}=-1$.
By \textbf{(VC-12)} and \textbf{(VC-13)}, $\Delta p_{i_{10}-1}^{(2)}=0$.
Hence, $\Delta _{11\rightarrow 14}c(\nu,J)=2$, which yields $\Delta c(\nu,J)\neq \alpha_{1}^{(2)}$.

The verifications of rules of boomerang strings are similar and we omit the details.

\subsection*{Acknowledgements}
The author would like to express his gratitude to Professor Masato Okado for helpful discussions.
He also would like to thank Professor Travis Scrimshaw for helpful comments.
This work was partly supported by Osaka City University Advanced Mathematical Institute (MEXT Joint Usage/Research Center on Mathematics and Theoretical Physics JPMXP0619217849).


\begin{thebibliography}{99}

\bibitem{BFKL06}
G. Benkart, I. Frenkel, S.-J. Kang, and H. Lee,
Level 1 perfect crystals and path realizations of basic representations at $q=0$,
Int. Math. Res. Not. \textbf{2006} 10312 (2006).

\bibitem{BS17}
D. Bump and A. Schilling, 
``Crystal Bases: Representations and Combinatorics,''
World Scientific Press (2017).

\bibitem{DS06}
L. Deka and A. Schilling,
New fermionic formula for unrestricted Kostoka polynomials,
J. Combin. Theory Ser. A \textbf{113} 1435--1461 (2006). 





\bibitem{HKO+99}
G. Hatayama, A, Kuniba, M. Okado, T. Takagi, and Y. Yamada,
Remarks on fermionic formula.
Contemp. Math. \textbf{248} 243--291 (1999).


\bibitem{HKO+02b}
G. Hatayama, A. Kuniba, M. Okado, T. Takagi, and Z. Tsuboi,
Paths, crystals and fermionic formulae,
Prog. Math. Phys. \textbf{23} 205--272 (2002).

\bibitem{Her10}
D. Hernandez,
Kirillov-Reshetikhin conjecture: the general case,
Int. Math. Res. Not. \textbf{1} 149--193 (2010).



\bibitem{Kac90}
V. G. Kac,
``Infinite-dimensional Lie algebras,''
Cambridge University Press, Cambridge, third edition (1990).


\bibitem{Kas90}
M. Kashiwara,
Crystalizing the $q$-analogue of universal enveloping algebras,
Comm. Math. Phys. \textbf{133} 249--260 (1990).

\bibitem{Kas91}
M. Kashiwara,
On crystal bases of the $q$-analogue of universal enveloping algebras,
Duke Math. J. \textbf{63} 465--516 (1991).

\bibitem{KKM+92a}
S.-J. Kang, M. Kashiwara, K. C. Misra, T. Miwa, T. Nakashima, and A. Nakayashiki,
Affine crystals and vertex models,
Int. J. Mod. Phys. \textbf{A7} (suppl. 1A) 449--484 (1992).

\bibitem{KKM+92b}
S.-J. Kang, M. Kashiwara, K. C. Misra, T. Miwa, T. Nakashima, and A. Nakayashiki, 
Perfect crystals of quantum affine Lie algebras,
Duke Math. J. \textbf{68} 499--607 (1992).





\bibitem{KKR86}
S. Kerov, A. Kirillov, and N. Reshetikhin,
Combinatrics, the Bethe anzats and representations of the symmetric group,
Zap. Nauchn. Sem. (LOMI) \textbf{155} 5--64 (1986).
(English translation: J. Sov. Math. \textbf{41} 916--924 (1988).)

\bibitem{KR88}
A. Kirillov, and N. Reshetikhin,
The Bethe ansatz and the combinatrics of Young tableaux,
J. Sov. Math. \textbf{41} 925--955 (1988).

\bibitem{KSS02}
A. Kirillov, A. Schilling, and M. Shimozono,
A bijection between Littlewood-Richardson tableaux and rigged configurations,
Selecta Math. (N. S.) \textbf{8} 67--135 (2002).

\bibitem{KNT02}
A. Kuniba, T. Nakanishi, and Z. Tsuboi,
The canonical solutions of the Q-systems and the Kirillov-Reshetikhin conjecture,
Comm. Math. Phys. \textbf{227} 155--190 (2002).




\bibitem{LS78}
A. Lascoux and M. Sch$\ddot{\mathrm{u}}$tzenberger,
Sur une conjecture de H.O. Foulkes,
CR Acad. Sci. Paris \textbf{286A} 323--324 (1978).




 




\bibitem{MMO10}
K. C. Misra, M. Mohamad, and M. Okado,
Zero action on perfect crystals for $U_{q}(G_{2}^{(1)})$,
SIGMA \textbf{6} 022 (2010).

\bibitem{MOW12}
K. C. Misra, M. Okado, and E. A. Wilson,
Soliton cellular automata associated with $G_{2}^{(1)}$ crystal base,
J. Math. Phys. \textbf{53} 013510 (2012).

\bibitem{Moh12b}
Mahathir bin Mohamad,
Soliton cellular automata constructed from a $U_{q}(\mathfrak{g})$-crystal $B^{n,1}$ and Kirillov-Reshetikhin type bijection for $U_{q}(E_{6}^{(1)})$-crystal $B^{6,1}$,
Thesis (Ph.D.)-Osaka University (2012).


\bibitem{Nao12}
K. Naoi,
Fusion products of Kirillov-Reshetikhin modules and the $X=M$ conjecture,
Adv. Math. \textbf{231} 1546--1571 (2012).






\bibitem{Nak03b}
H. Nakajima,
$t$-analogue of the $q$-characters of Kirillov-Reshetikhin modules of quantum affine algebras,
Represent. Theory \textbf{7} 259--274 (2003).

\bibitem{NY97}
A. Nakayashiki and Y. Yamada,
Kostka polynomials and energy functions in solvable lattice models,
Selecta Math. (N. S.) \textbf{3} 547--599 (1997).





\bibitem{OSS13}
M. Okado, R. Sakamoto, and A. Schilling,
Affine crystal structure on rigged configurations of type $D_{n}^{(1)}$,
J. Algebraic Combin. \textbf{37} 571--599 (2013).

\bibitem{OSS+17}
M. Okado, R. Sakamoto, A. Schilling, and T. Scrimshaw,
Type $D_{n}^{(1)}$ rigged configuration bijection,
J. Algebraic Combin. \textbf{46} 341--401 (2017).

\bibitem{OS12}
M. Okado and N. Sano,
KKR type bijection for the exceptional affine algebra $E_{6}^{(1)}$,
In ``Algebraic groups and quantum groups,'' volume 565 of Contemp. Math. 227--242. Amer. Math. Soc., Providence, RI (2012).

\bibitem{OSS03a}
M. Okado, A. Schilling, and M. Shimozono,
A crystal to rigged configuration bijection for nonexceptional affine algebras,
In ``Algebraic combinatorics and quantum groups,'' Edited by N. Jing, World Sci. Publ. 85--124 (2003).

\bibitem{OSS03b}
M. Okado, A. Schilling, and M. Shimozono,
Virtual crystals and fermionic formulas of type $D_{n+1}^{(2)}$, $A_{2n}^{(2)}$, and $C_{n}^{(1)}$,
Represent. Theory \textbf{7} 101--163 (electronic) (2003).

\bibitem{OSS03c}
M. Okado, A. Schilling, and M. Shimozono,
Virtual crystals and Kleber's algorithm,
Comm. Math. Phys. \textbf{238} 187--209 (2003).

\bibitem{OSS18}
M. Okado, A. Schilling, and T. Scrimshaw,
Rigged configuration bijection and the proof of the $X=M$ conjecture for nonexceptional affine types,
J. Algebra \textbf{516} 1--37 (2018).

\bibitem{Sak14}
R. Sakamoto,
Rigged configurations and Kashiwara operators,
SIGMA \textbf{10} 028 (2014).

\bibitem{Sch05}
A. Schilling,
A bijection between type $D_{n}^{(1)}$ crystals and rigged configurations,
J. Algebra \textbf{285} 292--334 (2005).

\bibitem{SS06a}
A. Schilling and M. Shimozono,
$X=M$ for symmetric powers,
J. Algebra \textbf{295} 562--610 (2006).


\bibitem{SS15b}
A. Schilling and T. Scrimshaw,
Crystal structure on rigged configurations and the filling map for nonexceptional affine types,
Electon. J. Combin. \textbf{22} $\sharp$P1.73 (2015).

\bibitem{Scr16}
T. Scrimshaw,
A crystal to rigged configuration bijection and the filling map for type $D_{4}^{(3)}$,
J. Algebra \textbf{448C} 294--349 (2016).

\bibitem{Scr20}
T. Scrimshaw,
Uniform description of the rigged configuration bijection,
Selecta Math. (N. S.) \textbf{26} 42 (2020).






\end{thebibliography}
\end{document}